\documentclass[11pt,reqno]{amsart}
\usepackage{etex}
\usepackage{textcmds} 
\usepackage{amssymb, amsfonts, amstext, verbatim, mathrsfs}
\usepackage[centertags]{amsmath}
\usepackage{microtype}
\usepackage[all]{xy}
\usepackage[modulo]{lineno}
\usepackage[usenames]{color}
\usepackage{enumitem}
\usepackage{xspace}
\usepackage{rotating}
\usepackage[margin=3cm]{geometry}
\usepackage{dsfont}
\usepackage{bm}
\usepackage{color}
\usepackage{subfigure}
\usepackage{array}
\usepackage[all]{xy}
\usepackage{euscript}
\usepackage[T1]{fontenc}
\usepackage{mathbbol}
\usepackage{nccmath}
\usepackage{stmaryrd}
\usepackage{youngtab}
\usepackage{tikz,pgfplots}
\usepackage{blkarray}
\usepackage{ulem}
\usepackage{mathtools}
\usepackage{tikzsymbols}
\usepackage{enumitem}

\usepackage{float}

\usepackage{graphics,graphicx}  

\usepackage[colorlinks=true,linkcolor=black,citecolor=blue,urlcolor=blue,citebordercolor={0 0 1},urlbordercolor={0 0 1},linkbordercolor={0 0 1}]{hyperref} 

\tikzset{node distance=3cm, auto}

\numberwithin{figure}{subsection}  

\numberwithin{equation}{subsection}  

\numberwithin{table}{subsection}  

\makeatletter
\let\c@equation\c@figure
\makeatother

\makeatletter
\let\c@table\c@figure
\makeatother

\makeatletter
\let\c@algorithm\c@figure
\makeatother

\newtheorem{thm}{Theorem}
\newtheorem{lemm}[thm]{Lemma}
\newtheorem{prop}[thm]{Proposition}
\newtheorem{cor}[thm]{Corollary}
\newtheorem{CONJEC}[thm]{Conjecture}
\newtheorem{DEFN}[thm]{Definition}

\newtheorem{EXAMple}[thm]{Example}
\newtheorem{REMark}[thm]{Remark}
\newtheorem{rmk}[thm]{Remark}
\newtheorem{CLAim}[thm]{Claim}

\newcommand{\eend}{{\rm end}}

\newcommand{\N}{{\mathbb{N}}}
\newcommand{\Ee}{\mathcal{E}}

\newenvironment{itemlist}
   { \begin{list} {$\bullet$}
         { \setlength{\topsep}{.5ex}  \setlength{\itemsep}{.5ex} \setlength{\leftmargin}{2.5ex} } }
   { \end{list} }

\newcommand{\bbm}{{\bf{m}}}
\newcommand{\bw}{{\bf{w}}}
\newcommand{\bn}{{\bf{n}}}

\newcommand{\bE}{{\bf{E}}}
\newcommand{\bB}{{\bf{B}}}

\newcommand{\NI}{{\noindent}}

\newcommand{\Ss}{{\mathcal S}}

\newcommand{\ov}{\overline}
\newcommand{\al}{{\alpha}}

\newcommand{\be}{{\beta}}

\newcommand{\om}{{\omega}}

\newcommand{\la}{{\lambda}}

\newcommand{\si}{{\sigma}}

\newcommand{\less} {{\smallsetminus}}
\newcommand{\p}{{\partial}}
\newcommand{\MS}{{\medskip}}

\newcommand{\er}{{\Diamond}}
\newcommand{\Z}{\mathbb{Z}}

\newcommand{\R}{\mathbb{R}}
\newcommand{\Q}{\mathbb{Q}}
\newcommand{\C}{\mathbb{C}}
\newcommand{\CP}{\mathbb{CP}}

\newcommand{\eps}{\varepsilon}

\newcommand{\op}[1]{{\operatorname{#1}}}

\newcommand{\Sh}{{\op{Sh}}}

\newcommand{\acc}{\mathrm{acc}}
\newcommand{\sembeds}{\stackrel{s}{\hookrightarrow}}

\makeatletter
\newcommand{\dashover}[2][\mathop]{#1{\mathpalette\df@over{{\dashfill}{#2}}}}
\newcommand{\fillover}[2][\mathop]{#1{\mathpalette\df@over{{\solidfill}{#2}}}}
\newcommand{\df@over}[2]{\df@@over#1#2}
\newcommand\df@@over[3]{%
  \vbox{
    \offinterlineskip
    \ialign{##\cr
      #2{#1}\cr
      \noalign{\kern1pt}
      $\m@th#1#3$\cr
    }
  }%
}
\newcommand{\dashfill}[1]{%
  \kern-.5pt
  \xleaders\hbox{\kern.5pt\vrule height.4pt width \dash@width{#1}\kern.5pt}\hfill
  \kern-.5pt
}
\newcommand{\dash@width}[1]{%
  \ifx#1\displaystyle
    2pt
  \else
    \ifx#1\textstyle
      1.5pt
    \else
      \ifx#1\scriptstyle
        1.25pt
      \else
        \ifx#1\scriptscriptstyle
          1pt
        \fi
      \fi
    \fi
  \fi
}
\newcommand{\solidfill}[1]{\leaders\hrule\hfill}
\makeatother

\DeclarePairedDelimiter\ceil{\lceil}{\rceil}
\DeclarePairedDelimiter\floor{\lfloor}{\rfloor}

\title{Infinite staircases for Hirzebruch surfaces}
\date{\today}

\author{Maria Bertozzi}
\address{MB: Ruhr Universit\"at Bochum}
\email{maria.bertozzi@rub.de}

\author{Tara S. Holm}
\address{TSH: Cornell University}
\email{tsh@math.cornell.edu}

\author{Emily Maw}
\address{EM: University College London}
\email{emilyharrietmaw@gmail.com}

\author{Dusa McDuff}
\address{DMc: Barnard College, Columbia University}
\email{dusa@math.columbia.edu}

\author{Grace T. Mwakyoma}
\address{GTM: Instituto Superior T\'ecnico}
\email{grace.mwakyoma@tecnico.ulisboa.pt,  gracetmwakyoma@gmail.com}

\author{Ana Rita Pires}
\address{ARP: University of Edinburgh}
\email{apires@ed.ac.uk}

\author{Morgan Weiler}
\address{MW: Rice University}
\email{morgan.weiler@rice.edu}

\begin{document}
%
%
%
%
%
%
%

\maketitle              

\begin{abstract} 
We consider the embedding capacity functions $c_{H_b}(z)$ for symplectic embeddings of ellipsoids of eccentricity $z$ into the family of nontrivial rational Hirzebruch surfaces $H_b$ with symplectic form parametrized by $b\in [0,1)$.  
This function was known to have an infinite staircase in the monotone cases ($b= 0$ and $ b= 1/3$).  It is also known that for each $b$ there is at most one value of $z$ that can be the accumulation point of such a staircase.  
In this manuscript, we identify three sequences of  open, disjoint,  blocked $b$-intervals, consisting of $b$-parameters where the embedding capacity function for $H_b$ does not contain an infinite staircase. There is one sequence in each of the intervals $(0,1/5)$, $(1/5,1/3)$,  and $(1/3,1)$.  We then establish six sequences of associated infinite staircases, one occurring at each endpoint of the blocked $b$-intervals.  The staircase numerics are variants of those in the Fibonacci staircase for the projective plane (the case $b=0$). We also show that there is no staircase  at the point $b=1/5$, even though this value is not blocked. The focus of this paper is to develop techniques, both graphical and numeric, that allow identification of potential staircases, and then to understand the obstructions 
well enough to prove that the purported staircases really do have the required properties.    
A subsequent paper will explore   in more depth the set of $b$ that admit  infinite staircases. 

symplectic capacity, embedded contact homology, quantitative symplectic geometry
\end{abstract}

\tableofcontents

\section{Introduction}\label{sec:intro}

\subsection{Overview of results} \label{ss:overv}
Given a four-dimensional symplectic manifold $(X,\omega)$, we define its {\bf ellipsoid embedding} function to be
$$c_X(z):=\inf\left\{\lambda\ \Big|\ E(1,z)\sembeds \lambda X\right\},$$
where $z\ge 1$ 
is a real variable,  $\lambda X: = (X,\lambda\omega)$ is the symplectic scaling of $X$,  an ellipsoid $E(c,d)\subset \mathbb{C}^2$ is the set
$$
E(c,d)=\left\{(\zeta_1,\zeta_2)\in\mathbb{C}^2 \ \big|\  \pi \left(  \frac{|\zeta_1|^2}{c}+\frac{|\zeta_2|^2}{d} \right)<1 \right\},
$$
and we write $E\sembeds \la X$ if there is a symplectic embedding of $E$ into $\la X$.
One simple property of this function is that it is bounded below by the \textbf{volume curve}, because symplectomorphisms preserve volume:
$$
c_X(z)\geq\sqrt{\frac{z}{\text{vol(X)}}},
$$
where $\text{vol(X)}$ is the appropriately normalized volume of $X$.
Similarly, the invariance of symplectomorphisms under scaling $\om\mapsto \la \om$ of the symplectic form
implies   the following  scaling property 
\begin{align}\label{eq:scale} 
c_{X}(\la z) \le \la c_{X}(z),\qquad z\ge 1, \la \ge 1;
\end{align}
indeed, if $E(1,z) \sembeds X$ then $E(1,\la z)\subset \la(E(1,z))\sembeds \la X$.

The first ellipsoid embedding function to be computed was that of a ball $B^4$, by McDuff and Schlenk in \cite{ball}. They 
describe the following striking behavior: on the interval $[1,\tau^4]$ the function is piecewise linear, alternating 
between horizontal line segments and line
segments that extend to lines passing through the origin in increasingly smaller intervals, creating what looks 
like an infinite staircase accumulating at $\tau^4$, where $\tau = \left(1+\sqrt 5\right)/2$ is the golden ratio. 
Similar behavior has been described for other targets $X$: the 
symmetric polydisk $B^2(1)\times B^2(1)$
 in 
\cite{frenkelmuller}, 
certain convex toric domains  in \cite{AADT}, and certain infinite families of polydisks $B^2(1)\times B^2(\beta)$
with irrational ratios $\beta$ in \cite{usher}.  For { Delzant} toric domains,  \cite{AADT} establishes that the open  manifold has the same embedding function as its closure; for example the functions for the $4$-ball and the projective plane $\C P^2$ are the same.  
Further, for { all convex} toric domains $X$ { of finite type}, the capacity function exhibits packing stability: 
it coincides with the volume function 
for sufficiently large $z$ (cf.\ \cite[Prop.~2.1]{AADT}).
Moreover, the function $z\mapsto c_X(z)$ is piecewise linear when not identically equal to the volume constraint curve, and when its graph has infinitely many nonsmooth points lying above the volume curve, we say that $c_X$ has an \textbf{infinite staircase}: see Definition~\ref{def:stair}.

For $0\leq b<1$, we denote by 
$$
H_b:=\mathbb{C}P^2_1\#\overline{\mathbb{C}P}^2_{b}
$$ 
the one point blow up of $\C P^2$, where the line has symplectic area $1$ and the exceptional divisor has symplectic area $b$.
In light of \cite{ball,AADT}, we know that the ellipsoid embedding functions of $\mathbb{C}P^2$ and of 
$H_{1/3}$ 
have infinite staircases. In this note we report on the search for infinite staircases 
in the ellipsoid embedding functions of the Hirzebruch surfaces in $H_b$.

By \cite[Theorem 1.11]{AADT}, which applies because the moment polytope of  $H_b$ has rational conormals, we know that if $c_{H_b}$ has an infinite staircase, then its accumulation point is at the point 
$z=\acc(b)$, the unique solution $>1$ of the following quadratic equation involving $b$:
\begin{equation}\label{eqn:accb}
z^2-\left( \frac{\left(  3-b\right)^2}{1-b^2}-2 \right)z+1=0
\end{equation}
Furthermore, if an infinite staircase exists, then at $z=\acc(b)$ the ellipsoid embedding function must 
coincide with the volume constraint curve, that is we must have 
\begin{align}\label{eq:cHb}
c_{H_b}(\acc(b))=\sqrt{\frac{\acc(b)}{1-b^2}} = : V_b(\acc(b)).
\end{align}
As shown in Fig.~\ref{fig:101}, the function $b\mapsto \acc(b)$ decreases for $b\in [0,1/3)$, with minimum value $\acc(1/3) = 3+2\sqrt2$, and then increases.  The nature of this function leads to an interesting and as yet not well understood    interplay between the $z$-coordinates that parametrize  the domain\footnote
{
Our convention is that a general point in the domain is denoted by the letter $z$ while special points (such as step corners, or break points of constraints) are  denoted $a$.}
 and the $b$-coordinates parametrizing the target. The point $b=1/3$ is the only known nonzero rational number such that $H_b$ has a staircase, and according to 
 \cite[Conjecture 1.20]{AADT} it is expected to be the only such value. In this paper, we focus on irrational values of $b$.

\begin{center}
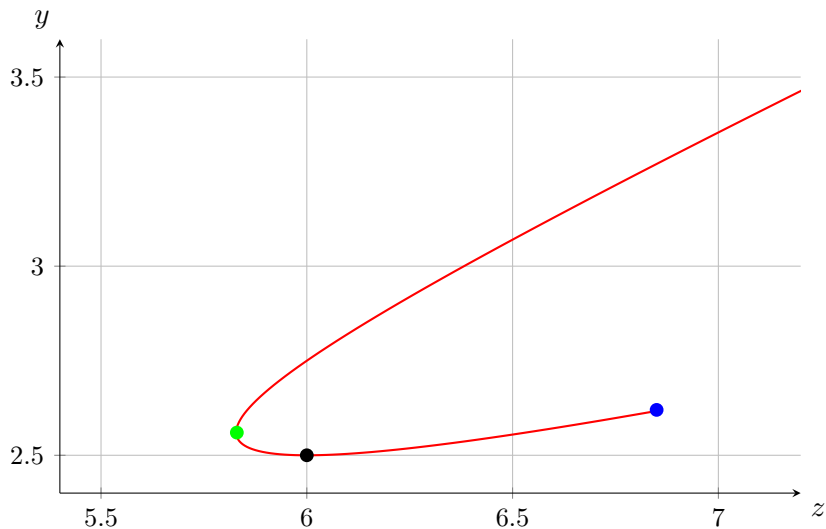
\begin{figure}[H]
\begin{tikzpicture}
\begin{axis}[
	axis lines = middle,
	xtick = {5.5,6,6.5,7},
	ytick = {2.5,3,3.5},
	tick label style = {font=\small},
	xlabel = $z$,
	ylabel = $y$,
	xlabel style = {below right},
	ylabel style = {above left},
	xmin=5.4,
	xmax=7.2,
	ymin=2.4,
	ymax=3.6,
	grid=major,
	width=4.5in,
	height=3in]
\addplot [red, thick,
	domain = 0:0.7,
	samples = 120
]({ (((3-x)*(3-x)/(2*(1-x^2))-1)+sqrt( ((3-x)*(3-x)/(2*(1-x^2))-1)* ((3-x)*(3-x)/(2*(1-x^2))-1) -1 ))},{sqrt(( (((3-x)*(3-x)/(2*(1-x^2))-1)+sqrt( ((3-x)*(3-x)/(2*(1-x^2))-1)* ((3-x)*(3-x)/(2*(1-x^2))-1) -1 )))/(1-x^2))});
\addplot [black, only marks, very thick, mark=*] coordinates{(6,2.5)};
\addplot [blue, only marks, very thick, mark=*] coordinates{(6.85,2.62)};
\addplot [green, only marks, very thick, mark=*] coordinates{(5.83,2.56)};
\end{axis}
\end{tikzpicture}
\caption{ This curve indicates the location of the accumulation point $(z,y)=\bigl(\acc(b), V_b(\acc(b))\bigr)$ for $0\le b< 1$.
The blue point with $b=0$ is at $(\tau^4, \tau^2)$ and is the accumulation point for the Fibonacci stairs.  The green point  
at $(3+2\sqrt{2}  ,1/3)$  
is the accumulation point for the 
stairs in  $H_{1/3}$, 
and is the minimum of the function $z\mapsto\acc(z)$.  
The black point  at $(6, 5/2)$ corresponds to $b=1/5$ and 
is the location where $V_b(\acc(b))$ takes its minimum. 
}\label{fig:101} 
\end{figure}
\end{center}

\MS

The capacity function $c_{H_b}$ may be described either in terms of exceptional divisors $\bE$ as recounted in \S\ref{sec:obstr}
or in terms of  ECH capacities as laid out  in \S\ref{sec:ECH}.  Both approaches have added to our understanding in this project.
ECH capacities gave us a way to understand the broad outlines of the problem and to discover interesting phenomena.
Exceptional divisors
are then well-adapted to establishing  
the intricate details; they also yield intriguing numerical data.  
The most relevant exceptional classes have \lq\lq centers''
at rational numbers that, in the case of a staircase,  form the first coordinate of  the corners of the steps.
We now lay out two  results that we obtained early on.  For notation, see \S\ref{sec:obstr}.
We say that a staircase ascends (resp.\ descends) if the steps occur at increasing (resp.\ decreasing) values of the domain parameter $z$.  

In the first result, we identify intervals of shape parameters $b$  for which $H_b$ cannot admit an infinite staircase.
Exceptional divisors $\bB$ give rise to {\bf obstruction functions} $\mu_{b,\bB}(z)$ that give lower bounds 
$$c_{H_b}(z)\geq \mu_{\bB,b}(z).$$ 
For a fixed class
$\bB$, then, we can let $b$ vary and determine intervals of $b$ for which 
$$\mu_{\bB,b}(\acc(b))>V_b(\acc(b)),$$ 
violating \eqref{eq:cHb} and
thereby {blocking} the existence of a staircase for those $b$ parameters.

\begin{thm}\label{thm:block} For each $n\ge 0$, the exceptional divisor in class $$
\bB_n^U:  =(n+3)L - (n+2)E_0 - \sum_{i=1}^{2n+6} E_i
$$
 blocks the existence of a staircase for $b$ in the interval 
 \begin{align*}
1/3\; < \; \be_{\bB^U_n,\ell}: = \frac{(2n^2+6n+3)-\sqrt{\si_n}}{2n^2+6n+2}<b<\frac{(n+3)\left(3n+7+\sqrt{\si_n}\right)}{5n^2+30n+44} =: \be_{\bB^U_n,u} < 1,
\end{align*}
where $\si_n: = (2n+1)(2n+5)$.  Correspondingly, there is no staircase with accumulation point in the interval 
\begin{align*}
\al_{\bB^U_n,\ell}: =  \frac{\sigma_n + (2n+3)\sqrt{\sigma_n}}{2(2n+1)} \ < z < 6+\frac{\sigma_n + (2n+3)\sqrt{\sigma_n}}{2(2n+5)}  =: \al_{\bB^U_n,u},
\end{align*}
where  $\al_{\bB^U_n,\ell}$ and  $\al_{\bB^U_n,u}$ have continued fraction expansions
\begin{align*}
\al_{\bB^U_n,\ell}
= [\{ 2n+5,2n+1\}^\infty], \qquad \al_{\bB^U_n,u} 
=[2n+7;\{ 2n+5,2n+1\}^\infty]
\end{align*}
\end{thm}

For a proof see the end of \S\ref{ss:thmU}. 
We will say that an interval such as 
$\bigl(\be_{\bB^U_n,\ell},\be_{\bB^U_n,u}\bigr)$
that is blocked by an exceptional divisor $\bB$ is a {\bf $\bB$-blocked $b$-interval},   
and denote it by $J_{\bB}$.  The corresponding interval $\acc(J_\bB)$ on the $z$-axis is denoted 
$I_{\bB} =  (\al_{\bB,\ell}, \al_{\bB,u})$, and, if $J_{\bB}\subset (1/3,1)$ (resp.\ if $J_{\bB}\subset [0,1/3)$), consists of points that cannot be the accumulation point of any staircase  for $b> 1/3$ (resp.\ $b< 1/3$).
The class $\bB$ itself is called a {\bf blocking class}.  

In Theorem~\ref{thm:block},  the blocked 
$z$-intervals  
$(\al_{\bB^U_n,\ell}, \al_{\bB^U_n,u})$ have \lq centers' $2n$ at distance $2$ apart, while their lengths increase with limit $2$. 
{ Correspondingly} most  values of $b$ close to $1$ are blocked; see Fig.~\ref{fig:blockU}.
Further, the above classes $\bB^U_n$ are {\bf center-blocking} in the sense of  Definition~\ref{def:centerbl},
which implies by Proposition~\ref{prop:block} that they also
block  the existence of a descending (resp.\ ascending) staircase 
that accumulates at
the lower endpoint 
$\al_{\bB^U_n,\ell}$ (resp.\ upper endpoint  $\al_{\bB^U_n,u}$) of $I_\bB$. 
\MS

In our second main result,  proved in \S\ref{ss:thmLE}, we find a new infinite family of shape parameters for which
the embedding capacity function does include an infinite staircase.

\begin{thm}\label{thm:stair} There is a  decreasing sequence of parameter values 
$b_n: = b^E_{u,n,\infty} $ in the interval $(1/5,1/3)$, for  $n\ge 0,$   with limit $1/5$, so that each $H_{b_n}$ admits a descending staircase $\Ss_{u,n}^E=\bigl(\bE^E_{u,n,k}\bigr)_{k\ge 0}$ 
whose steps have exterior corners
at the points $z=a^E_{u,n,k}$ with continued fraction expansions 
\[
a^E_{u,n,k}=[5; 1,2n+6,\{2n+5,2n+1\}^k, 2n+4], \quad k\ge 0.
\]
The accumulation point $\acc(b^E_{u,n,\infty})$ is
\begin{align*}
a^E_{u,n,\infty} &
= [5; 1,2n+6,\{2n+5,2n+1\}^\infty].
\end{align*}
These accumulation points  form an increasing sequence that converges to $6$
as $n\to \infty$.
\end{thm}

We found this family of staircases by trial and error, using a combination of the computer programs described in \S\ref{sec:Mathem} and the algebro-geometric methods described in  Remark~\ref{rmk:ECH}.

\begin{rmk}\label{rmk:numbers} \normalfont
(i) 
Note that, in contrast to the $z$-coordinates,  the continued fraction expansions of the shape parameters 
$b^E_{u,n,\infty}$ seem to have no special properties. 
On the other hand, because the value $\si_n$ that occurs in the formula for $b_n$ is so closely related to the iterated pair $2n+5,2n+1$ in the  continued fraction of
the accumulation points $a^E_{u,n,\infty}$, 
one can show that 
the 
$a^E_{u,n,\infty}$   are also numbers in $\Q+ \Q\sqrt{\si_n}$.
From this point of view, the fact that the shape parameters $b^E_{u,n,\infty}$ also involve the surd $\sqrt{\si_n}$ is rather mysterious, but this follows because the new staircases accumulate at the end points of blocked $z$-intervals; see Corollary~\ref{cor:block2}.
\MS

\NI (ii) 
The staircases in Theorem~\ref{thm:stair} actually consist of two intertwining  sequences of classes that are fully 
described in Theorem~\ref{thm:E} below.  In fact, all the new staircases that we find  in this paper have this form. 
In Figure~\ref{fig:515},
one can see the steps with centers 
 $[5;1,5,1,4] =  {199}/{34} \approx 5.853$ and $[5;1,5,1,5,2] =  {521}/{89} \approx 5.8539$,
while in Figure~\ref{fig:751}, one can see the steps with centers $[7;5,2] =  {79}/{11} \approx 7.1818$ and
$[7;5,1,4]= {208}/{29} \approx 7.172$.  The small step with center just $> 7.190$ is not part of the staircase.
\hfill$\er$
\end{rmk}

Rather surprisingly, the two theorems above are related both numerically (the same square roots and continued fractions occur) and geometrically, in ways that we explore in \S\ref{ss:prestair} and \S\ref{ss:Fib}.  
 In general, we make the following conjecture:

\begin{CONJEC}\label{conj:block}
For each center-blocking class $\bB$ with corresponding blocked $b$-interval $J_\bB=(\be_{\bB,\ell}, \be_{\bB,u})$, 
both  $H_{\be_{\bB,\ell}}$ and  $H_{\be_{\bB,u}}$
admit staircases $\Ss_{\be_{\bB,\ell}}, \Ss_{\be_{\bB,u}}$ where $\Ss_{\be_{\bB,\ell}}$ ascends and  $\Ss_{\be_{\bB,u}}$ descends if $J\subset (1/3,1)$ and the reverse holds true if $J\subset [0, 1/3)$.  Further, if $b$ is irrational and  $H_b$  has a staircase, then 
$b$ is the endpoint of some (center-)blocked $b$-interval $J_\bB$.
\end{CONJEC}

\noindent Figures \ref{fig:515} and \ref{fig:751} illustrate Conjecture \ref{conj:block} for the center-blocking class $\bB^U_0$ of 
Theorem \ref{thm:block}.   

\newpage

\begin{figure}[H]
\includegraphics[width=3.5in]{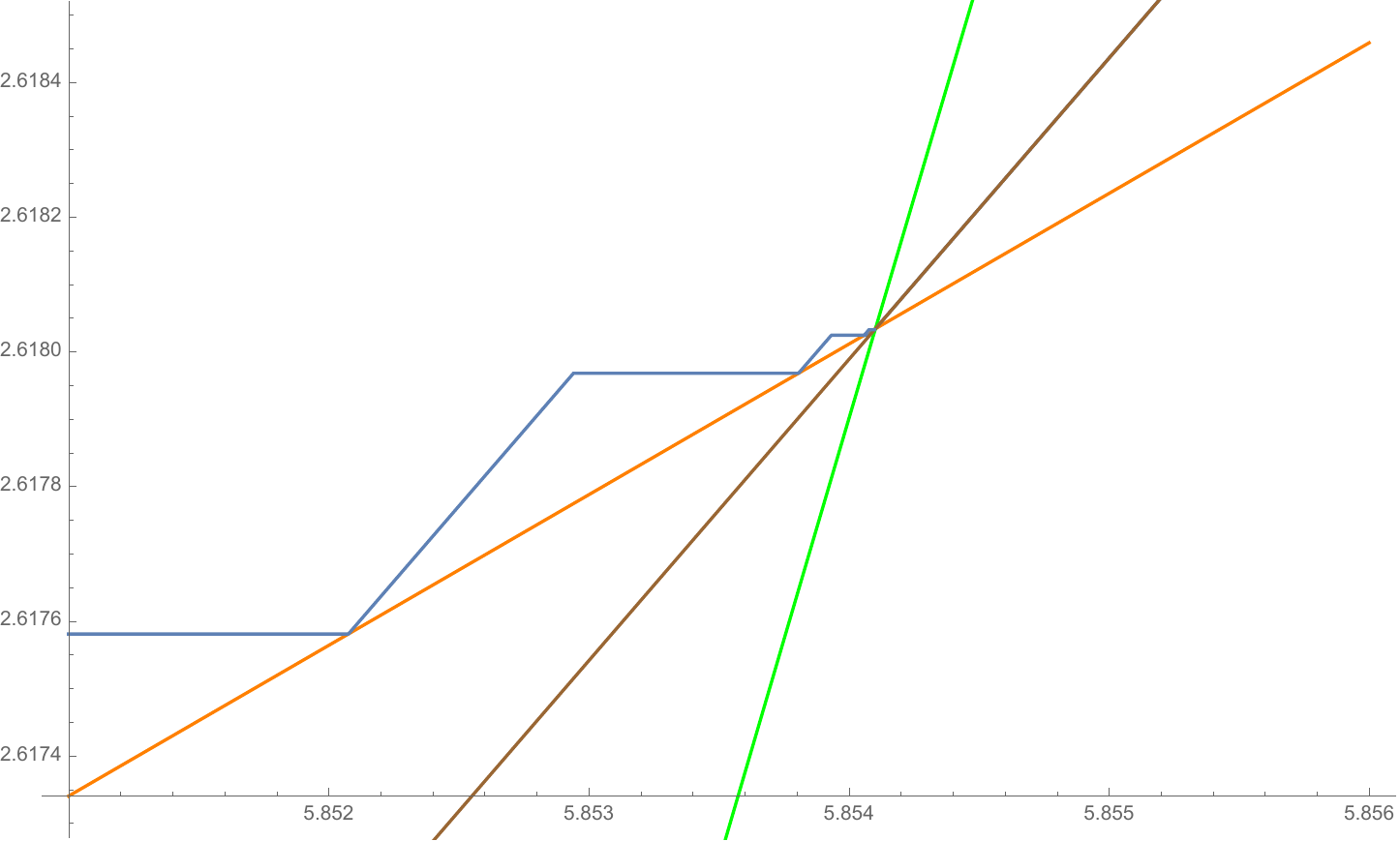}
\caption{In blue is the plot of a lower bound for $c_{H_b}$,  where $b= 
{\beta_{\bB^U_0,\ell}} = \left(3-\sqrt 5\right)/2  > 1/3$, 
which 
contains the infinite staircase with the same numerics as the $n=0$ case of the family $\Ss^U_{\ell, n}$ (Theorem \ref{thm:U}) 
 whose accumulation point has a $z$-value with continued fraction $[\{5,1\}^\infty]$; see Remark~\ref{rmk:SUell0}. 
 The volume constraint $V_b$ is in orange, and the parametrized curve $\bigl(\acc(b),V_b(\acc(b))\bigr)$, which intersects the volume constraint at the 
 accumulation point of $\Ss^U_{\ell, 0}$, is in green. The obstruction $\mu_{\bB^U_0,b}$ (see (\ref{eq:obstr})) from $\bB^U_0$ is in brown. 
 }
 \label{fig:515}
\end{figure}

\begin{figure}[H]
\includegraphics[width=4in]{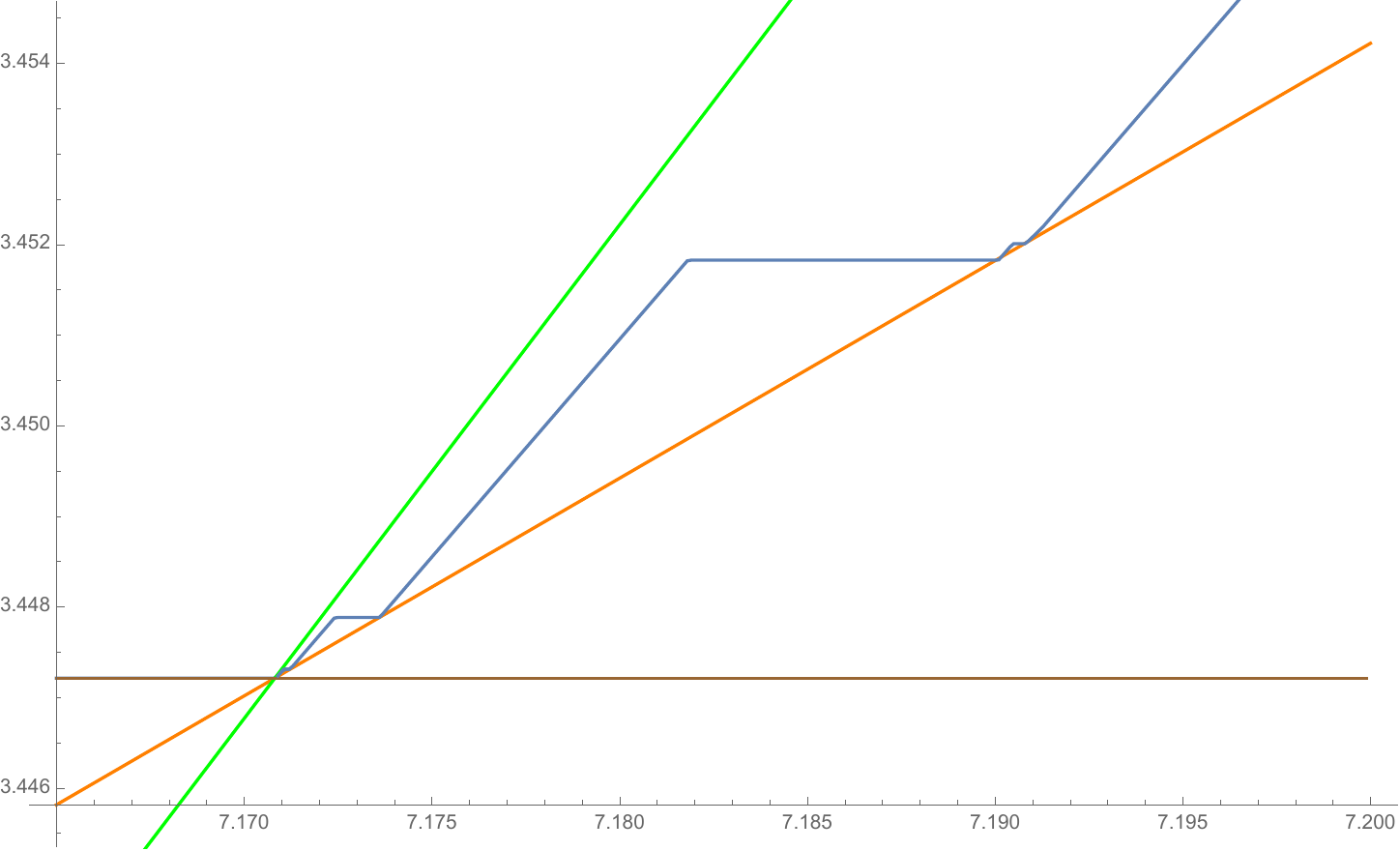}
\caption{In blue is the plot of a lower bound for $c_{H_b}$ 
where $b= \beta_{\bB^U_0,u}$, which 
contains the infinite staircase $\Ss^U_{u,0}$ (Theorem \ref{thm:U})
whose accumulation point has a $z$-value with continued fraction $[7;\{5,1\}^\infty]$. The volume constraint $V_b$ is in orange, and the parameterized curve $\bigl(\acc(b),V_b(\acc(b))\bigr)$, which intersects the volume constraint at the accumulation point of $\Ss^U_{u,0}$, is in green. The obstruction $\mu_{\bB^U_0,b}$ (see (\ref{eq:obstr})) from the blocking class $\bB^U_0$ is in brown.}\label{fig:751}
\end{figure}

As evidence for this conjecture, we  prove the following theorem in \S\ref{ss:thmLE}.

\begin{thm}\label{thm:blockstair}  For each $n\ge 1$ there is a center-blocking  class 
$\bB^E_n$  { blocking $b$-values} $J_{\bB^E_n} \subset (0,1/3) $  such that 
the accumulation point $a^E_{u,n,\infty}$ of the decreasing staircase $\Ss^E_{u,n}$ in Theorem~\ref{thm:stair}  
is  the upper endpoint $\al_{\bB^E_n,u}$ of { the corresponding $z$-interval}   $I_{\bB^E_n}$.
\end{thm}

It turns out that this second set of center-blocking classes $(\bB^E_n)_{n\ge 1}$  is related to the first set $(\bB_n^U)_{n\ge 0}$ by a symmetry (described in Corollary~\ref{cor:symm}) that arises from the properties of the recursive sequence $1,6,35,204,\dots$. 
In fact, in \S\ref{ss:Fib} we define three families of center-blocking classes, the two $(\bB^U_n), (\bB^E_n)$ 
that are mentioned above, together with one more called $(\bB^L_n)$, that are all related by symmetries. We also describe the six associated families of staircases, with full proofs of our claims given in  \S\ref{ss:thmU},
 \S\ref{ss:thmLE} and \S\ref{ss:reduct}.   As explained in  Remark~\ref{rmk:symm}, these symmetries generate yet more staircases, that will be explored more fully in our next paper.

As we show in \S\ref{ss:Fib}, these new staircases are variants of the Fibonacci staircase for the ball, and belong to a class of staircases (called {\bf pre-staircases} and defined in Definition~\ref{def:prestair}) made up of exceptional classes whose entries are determined by the centers of the steps together with a linear relation. As we show in Theorem~\ref{thm:blockrecog},  the   coefficients of this relation are in turn determined by  the parameters of the associated center-blocking class.
Thus the structure of these staircases is very different from that of the staircase at $b=1/3$ (described in Example~\ref{ex:13} on page~\pageref{ex:13}), which is not associated to any blocking class.  
\MS

Our last main result concerns the point $b=1/5$ that has associated accumulation point $\acc(1/5) = 6$.
It is relatively straightforward to show that $H_{1/5}$ is  {\bf unobstructed}, that is
its capacity and volume functions agree at the point $z=\acc(b)$; see Examples~\ref{ex:52},~\ref{ex:seq}~(i).  
As we remarked above, this is a necessary condition for $H_b$ to admit a staircase; however, as the following result shows, it is not sufficient.

\begin{thm}\label{prop:150}    $H_{1/5}$ is unobstructed,  but does not have a staircase.
\end{thm} 

\NI We prove this fact  in \S\ref{ss:15}, using ECH tools developed in \cite{rationalellipsoids}.
\MS

The fundamental question that concerns us here is that of the structure of the sets 
\begin{align*}
{\it Stair}: &= \{b \ | \ H_b \mbox{ has a staircase} \} \subset [0,1) ,\;\;\mbox{ and }\\
{\it Block}: & = \bigcup \bigl\{ J_{\bB} \ \big| \ \bB  \mbox{ is a blocking class} \bigr\} \subset [0,1).
\end{align*}
{\it Block} is a union of open intervals, and as far as we know  it could be dense in $[0,1)$. 
 Indeed, Proposition~\ref{prop:stairblock} shows that most (if not all) of the classes in the new staircases  are themselves blocking classes.   Note also that if Conjecture~\ref{conj:block} is true, then $H_b$ is unobstructed at each end of a center-blocked $b$-interval (provided these are irrational), so that each 
 such class determines a single
 connected  component of {\it Block}.  It is unknown whether there are other components of {\it Block}.
Further,   {\it Stair} is contained in the closed set $[0,1)\less {\it Block}$, but is not equal to it by Theorem~\ref{prop:150}.
This result also shows that the subset {\it Stair} is not closed, since Theorem~\ref{thm:stair} shows that $1/5$ is the limit point of a sequence in {\it Stair} (and also the limit of a sequence of endpoints of  blocked $b$-intervals by Theorem~\ref{thm:blockstair}),  while $1/5\notin {\it Stair}$ by Theorem~\ref{prop:150}. 

We conjecture above that the irrational points in {\it Stair}  are precisely the union of the endpoints of the center-blocked intervals in {\it Block}, while, by \cite[Conj.~1.20]{AADT}, the only rational points in {\it Stair}  are believed to be $0$ and $1/3$.
However, at the moment we cannot say anything more about these sets.  
There are many open questions here. For example, are there any blocking classes that are not center-blocking?   
Do all staircases (except for the one at $b=1/3$) have the special structure of a pre-staircase?
 We will explore such questions  in more depth in our forthcoming manuscript.

\begin{rmk}\normalfont   
All the staircases that we mention here are generated by families 
of perfect exceptional classes in the sense of Definition~\ref{def:perf}.  Hence, the proof of   \cite[Prop.3.6.1]{cghm}
should generalize to show that they all stabilize when we multiply both domain and target by $\C^k$.  \hfill $\er$
 \end{rmk}

\subsection{Organization of the paper}

In \S\ref{sec:obstr}, we first explain how to define the embedding capacity function $c_{H_b}$ from knowledge of exceptional spheres.  We then develop the general properties of exceptional classes, 
contrasting the notion of a  {\bf quasi-perfect } class (which is numerically defined) with that of a {\bf perfect} class (which satisfies an addition geometric condition); see Definition~\ref{def:perf}.  Proposition~\ref{prop:live} shows how to tell when a class is  {\bf live}, i.e. contributes to the capacity function; while   Propositions~\ref{prop:stair0} and~\ref{prop:stair00}
formulate our basic  staircase recognition criteria.  After introducing  the notion 
 of a {\bf blocking class} in \S\ref{ss:block}, we show in Proposition~\ref{prop:block}  that suitable knowledge about  the capacity function forces a quasi-perfect blocking class to be perfect.  We also begin to analyse the relation between blocking classes and staircases.   Finally \S\ref{ss:prestair} introduces the notion of a {\bf pre-staircase }(Definition~\ref{def:prestair}), which is a numerically defined set of classes of special form.  The main results are Theorem~\ref{thm:stairrecog} which gives criteria for a pre-staircase to be an actual staircase, and Theorem~\ref{thm:blockrecog} which explains how the existence of suitable  perfect pre-staircases that converge to the ends of a potential blocking class forces that class to be a perfect blocking class.

The next section describes the three staircase families in Theorems~\ref{thm:L},~\ref{thm:U} and ~\ref{thm:E}.  
This can be read once the basic definitions are understood.  
The rest  of the section develops the techniques necessary to prove that these really are staircases.
There are two parts to this proof. The first part (in \S\ref{ss:thmU}, \S\ref{ss:thmLE}) involves showing that the staircase classes are quasi-perfect and that the estimates in Theorem~\ref{thm:stairrecog} hold.  The second part is a proof that the staircase classes are perfect, in other words that they satisfy the necessary geometric condition (reduction under Cremona moves) to be true (rather than fake) exceptional classes.
Here the proof is greatly eased by the existence of the symmetries described in Corollary~\ref{cor:symm} that transform one staircase into another: see Proposition~\ref{prop:CrEL}.

\S\ref{sec:ECH} describes a different (but equivalent)  way to compute the capacity function, this time using capacities coming from embedded contact homology. These ``ECH capacities'' are the most efficient way to do certain calculations, as is shown in the proof in \S\ref{ss:15} that $c_{H_b}$ has no infinite staircase at $b=1/5$. In \S\ref{subsec:ecECH} we also prove Lemma \ref{lem:ECHcap}, which explains the relationship between the obstructions from perfect classes and those from ECH capacities.
 This allowed us to identify live classes using computer models based on ECH capacities, as explained in \S\ref{subsec:findingk}.

Finally, \S\ref{sec:Mathem} explains the computer programs that we used for our graphical explorations. In \S\ref{subsec:computingECHcap} we explain the combinatorics of lattice paths which allowed us to reduce computation time; our methods should generalize to any convex toric domain with two sides with rational slopes. After discussing our actual Mathematica code in \S\ref{subsec:ECHcapcode}-\ref{subsec:excclasscode}, in \S\ref{subsec:findingk} we explain our strategy for identifying values of $b$ for which $c_{H_b}$ has an infinite staircase. We conclude with plots of $c_{H_b}$ which illustrate several of the major phenomena we discuss in this paper. Remark~\ref{rmk:ECH} describes how the computer-aided approach of \S\ref{sec:Mathem} interacted fruitfully with the methods of \S\ref{sec:obstr} in our search for potential staircases.

\subsection*{Acknowledgements}
We are indebted to the organizers (Bahar Acu,
Catherine Cannizzo,
Dusa McDuff,
Ziva Myer,
Yu Pan, and
Lisa Traynor) of the Women in Symplectic and Contact Geometry and Topology (WiSCon) network
for bringing us together at a research collaboration workshop in July 2019.  ICERM was a gracious host for our week in
 Providence.  We would also like to thank
Dan Cristofaro-Gardiner, Michael Hutchings, and Ben Wormleighton for helpful conversations and lemmata.

\vskip 0.2in

\noindent MB was supported by the grant DFG-CRC/TRR 191 ``Symplectic Structures in Geometry, Algebra and Dynamics''.

\noindent TSH was supported by NSF grant DMS 1711317 and is grateful for the hospitality of Clare Hall during 2019--2020.

\noindent EM was supported by the EPSRC grant Centre for Doctoral Training, {\it London School of Geometry and Number Theory}.

\noindent GTM was supported by Funda\c{c}\~{a}o para a Ci\^{e}ncia e a Tecnologia (FCT) through the grant PD/BD/128420/2017.

\noindent MW was supported by NSF grant RTG--1745670.

\section{Embedding obstructions from exceptional spheres}\label{sec:obstr}
In this section, we explain how to define the embedding capacity function $c_{H_b}$ 
using exceptional spheres. The first subsection explains the general properties of exceptional classes, 
and the second discusses the properties of quasi-perfect and perfect classes, and 
formulates our first  staircase recognition criteria (in Propositions~\ref{prop:stair0} and~\ref{prop:stair00}).
The next subsection explains the properties of blocking classes,  while the last introduces the notion of a pre-staircase and, by combining their properties   with those of blocking classes, gives in Theorem~\ref{thm:stairrecog} a  powerful staircase recognition criterion and in 
Theorem~\ref{thm:blockrecog}  a  blocking class  recognition criterion.

\subsection{The role of exceptional spheres}
{ It was shown in \cite{Mc0} that the existence of a symplectic 
embedding of a rational ellipsoid $E(1,z)$ into either the projective 
plane $\C P^2$ or another rational ellipsoid is equivalent to a 
particular embedding of balls into $\C P^2$.  This result extends 
trivially to embeddings of ellipsoids into any blow up of $\C P^2$;
with some additional argument to  embeddings into rational convex 
toric domains  (see \cite[Theorem~1.6]{cg} and   
\cite[Theorem~1.2]{AADT}); and into certain other toric domains (see 
\cite{LMT}).}
In our framework, when $z$ is rational there is
a symplectic embedding
\begin{equation}
\label{eqn:balls}
E(1,z)\sembeds \lambda H_b \iff \bigsqcup_{j=1}^n B(w_j)\sembeds\lambda  H_b \iff  \bigsqcup_{j=1}^n B(w_j)\bigsqcup B(\lambda b)\sembeds B(\lambda).
\end{equation}
Here, $\bw(z)=(w_1,\dots,w_n)$ is the {\bf  weight expansion} of $z\in\Q$ defined in \eqref{eq:wgt} below.
Thus  the ellipsoid embedding question converts into a question about symplectic forms on 
an $(n+1)$-fold blowup\footnote
{
For discussion of the relationship between ball embeddings and the symplectic  blow up, see for example, \cite[Ch.7.1]{McSal}.} of $\C P^2$, a problem that was solved in the 1990s. 

The natural basis for $H_2\left(\C P^2\# (n+1)\overline{\C P}\,\!^2\right)$ consists of the original generator 
$L$ on $\C P^2$, plus
the classes of the $n+1$ exceptional spheres $E_0,\dots,E_{n}$. It is known that there is
a ball embedding as in \eqref{eqn:balls} if and only if there is a symplectic form $\omega$ on $\C P^2\# (n+1)\overline{\C P}\,\!^2$
satisfying
$$
PD[\omega] = \la L-\la bE_0 -\sum_{i=1}^n w_iE_i
$$
with canonical class $K$ satisfying
$$
PD(K) = -3L + E_0 + \sum_{i=1}^n E_i.
$$
We set $\mathscr{E}_{n+1}$ to be the set of {\bf exceptional classes} in $H_2\left(\C P^2\# (n+1)
\overline{\C P}\,\!^2\right)$.  It
consists of those classes $A$ that have square $-1$, $A\cdot K = -1$, and that can be represented by a smoothly 
embedded, symplectic sphere.   Li and Li have shown that
the requirement that a sphere be symplectic is superfluous \cite{li-li}; thus  $\mathscr{E}_{n+1}$ is independent of  
$\omega$ (though it depends on $K$).  An exceptional class  $$
\bE: = (d,m; \bbm) : = dL - m E_0 - \sum_{i\ge 1} m_i E_i
$$
 must satisfy the following 
Diophantine equations, which correspond to the homological conditions $A^2=-1$ and $A\cdot K=-1$:
\begin{eqnarray}\label{eq:diophantine}
d^2 - m^2 -\sum_{i=1}^n m_i^2 = -1 \mbox{ and }
3d - m -\sum_{i=1}^n m_i = 1. 
\end{eqnarray}
We call classes that satisfy these two equations \eqref{eq:diophantine}  {\bf Diophantine classes}.
As explained in \S\ref{ss:reduct}, Diophantine classes
are exceptional 
if and only if they reduce to
the class $E_1$  under repeated Cremona moves.

Li-Liu~\cite{li-liu} proved that  the symplectic cone of $\C P^2\# (n+1)\overline{\C P}\,\!^2$ consists of classes $[\om]$ such that $[\om]^2>0$ and $PD([\om])\cdot E > 0$ 
for all  $E\in \mathscr{E}_{n+1}$ .
It follows that    
 there is a symplectic embedding \eqref{eqn:balls} if and only if 
\begin{itemize}\item[$\bullet$]  the volume of $E(1,z)$ is less than that of $\la H_b$, that is $z < \la^2(1-b^2)$;
\item[$\bullet$] for each exceptional class
$\bE: = (d,m; \bbm)\in \mathscr{E}_{n+1}$, we have $\omega(E)>0$, that is
$
mb + \sum_{i=1}^n m_i\frac{w_i}{\lambda} <d;
$
or equivalently
$$
\lambda> \frac{\sum_{i=1}^n m_iw_i}{d-mb} = \frac{\bbm\cdot \bw(z)}{d-mb}.
$$
\end{itemize} 
Thus, the exceptional classes provide potential  {obstructions} to the desired embeddings.  
\MS

For any exceptional class $\bE= (d,m; \bbm)$,
we may vary $z$ to get an obstruction function as follows.   Any real number $z$ has (possibly infinitely long) {\bf weight expansion}
$\bw(z) = (w_1,w_2,\dots)$ defined for $z>1$ as follows:
\begin{itemize}\item[$\bullet$]  $w_1 = 1$ and $w_n\ge w_{n+1}>0$
 for all $n$,
\item[$\bullet$] if $w_i> w_{i+1} = \dots = w_n$ (where we set $w_0: = z$), then 
\begin{align}
\label{eq:wgt}
w_{n+1} = 
\begin{cases} 
 w_{i+1} & \mbox{ if }\;\; w_{i+1} + \dots +w_{n+1} = (n-i+1)w_{i+1}  < w_i,\\
w_i - (n-i)w_{i+1} & \mbox{ otherwise}
\end{cases}
\end{align}
\end{itemize}
Thus $\bw(z)$  starts with $\lfloor z\rfloor$ copies of $1$ (where $\lfloor z\rfloor$ is the largest integer $\le z$), is finite for rational $z$, and  satisfies the recursive relation
\begin{align*}
\bw(z) = \Bigl(1^{\times \lfloor z\rfloor}, (z-\lfloor z\rfloor) \bw\bigl(\frac 1{z-\lfloor z\rfloor}\bigr)\Bigr),
\end{align*}
where we take  $\bw\bigl(\frac 1{z-\lfloor z\rfloor}\bigr)$ to be empty if $z= \lfloor z\rfloor$. 
If we think of the weight expansion as an infinitely long sequence $\bw(z) = (w_1,w_2,\dots)$, appending
infinitely many zeros if necessary, then the coordinates of the weight expansion are continuous (but not differentiable!) functions
$w_i(z)$ of $z$.

\begin{EXAMple}\label{ex:ctfr0}   \normalfont
If $z = p/q$ (in lowest terms) and  $k_0<p/q<k_0+1$ , then 
\begin{align}\label{eq:ctfract}
q \bw(p/q) = \Bigl(q^{\times k_0}, (p-k_0q)^{\times k_1}, (q-k_1(p-k_0q))^{\times k_2}, \dots, 1\Bigr)
\end{align}
where $z$ has continued fraction expansion $[k_0;k_1,k_2,\dots]$, i.e. the
 multiplicities of the entries in $\bw(z)$ give the continued fraction expansion of $z$.  
 Notice that all the entries here are integral and that the last entry is $1$.   Conversely, the  
 continued fraction expansion of $z$ determines the  (continuous!) linear functions of $z = p/q$ that give the weight expansion.
For example,
$$
\bw\Bigl(\frac {35}6\Bigr) = \Bigl(1^{\times 5}, \frac 56,\frac 16^{\times 5}\Bigr),\qquad 
\frac {35}6 = [5;1,5] = 5 + \frac 1{1 + \frac15}.
$$
We  call $q \bw(p/q)$ the {\bf integral weight expansion} of $p/q = [\ell_0;\ell_1,\dots,\ell_N]$.
An easy inductive argument on the length $N$ shows that the integers $p,q$ found by this inductive process from a continued 
fraction $[\ell_0;\ell_1,\dots,\ell_N]$ are always relatively prime.
 For more details, see \cite[\S1.2]{ball}, and also Example~\ref{ex:356} below. \hfill$\er$
\end{EXAMple}

As in~\cite[\S1.2]{ball}, it is not hard to check the following identities for the weight expansion
 $\bw(z)$ at $z = p/q$:
  \begin{equation}\label{eq:wtexp}
\sum_i w_i = z + 1 - \frac 1q,\qquad \sum_i w_i^2 = \frac pq. 
\end{equation}

To define the {\bf obstruction function} $z\mapsto \mu_{\bE, b}(z)$ associated to $\bE$, we  first pad $(d,m; \bbm)$ with zeros on the
right to make it infinitely long, and  then set
\begin{equation}\label{eq:obstr}
\mu_{\bE, b}(z) : = \mu_{(d,m; \bbm), b}(z): = \frac{\bbm\cdot \bw(z)}{d-mb}.
\end{equation}
The above discussion implies that these obstruction functions, together with the volume, completely determine the embedding capacity function 
$c_{H_b}$, namely
\begin{equation}\label{eq:obstr1}
 c_{H_b}(z) = \sup_{\bE\in \Ee} \bigl(\mu_{\bE,b}(z), \ V_b(z) \bigr),
\end{equation}
where $V_b(z): = \sqrt{\frac z{1-b^2}}$ denotes the volume obstruction.

{
\begin{REMark}  \normalfont
A key point in our analysis of staircases is  
the result in
\cite[Theorem~1.8]{AADT} that the geometry of the target $H_b$ 
determines the accumulation point $\acc(b)$ of any staircase via 
\eqref{eqn:accb}.  This result
is valid (and the proof is the same) when the target $X$ is any blow 
up of $\CP^2_1$ by balls of sizes $b_1,\dots, b_N$, provided that we 
replace $3-b$ by  $3-\sum_i b_i$ and $1-b^2$ by the volume
$1-\sum_i b_i^2$.
  \hfill$\er$
\end{REMark}
}

\begin{DEFN}\label{def:nontriv}  An obstruction 
$\mu_{\bE,b}$ is said to be {\bf nontrivial at $z$} if $\mu_{\bE,b}(z)> V_b(z)$, and {\bf live at $z$} if it is nontrivial and equals the capacity function $c_{H_b}$ at $z$.
Further, it is said to  {\bf block $b$} if it is nontrivial at the accumulation point $\acc(b)$, that is, if
$$\mu_{\bE,b}(\acc(b))>V_b(\acc(b)).$$

We say that $H_b$ is {\bf unobstructed at $\acc(b)$} (or simply {\bf unobstructed}) if 
$$
c_{H_b}(\acc(b)) = V_b(\acc(b)).
$$  Otherwise, we say that $H_b$ is {\bf obstructed} or, equivalently, {\bf blocked}.  Further, a $b$-interval $J$ is {\bf blocked} if $c_{H_b}(\acc(b))> V_b(\acc(b))$  for all $b\in J$.
\end{DEFN}

Since all functions involved are continuous, the condition that $\mu_{\bE,b}$ is nontrivial at $z$ is continuous in both $z$ and $b$.  Similarly, the condition that $\mu_{\bE,b}$ blocks $b$ is an open condition on $b$, so that the set 
\begin{align}\label{eq:blockset}
{\it Block}: = \{b\in [0,1) \ | \ H_b \mbox{ is blocked}\}
\end{align}
is open.
Thus the set of $b$ that are unobstructed  is closed and contains the set 
$$
{\it Stair} =\{b\in [0,1) \ | \ H_b \mbox{ has a staircase}\}.
$$
Theorem~\ref{prop:150} guarantees that ${\it Stair}$  is a proper subset of the unobstructed values. 

Observe that the condition of being live at $(z_0, b_0)$ need  not be continuous with respect to either variable $z$ or $b$
since   it might happen that two different classes are live at $(z_0,b_0)$, one remaining live as $z$ or $b$ increases and the other remaining live as $z$ or $b$ decreases.  If  $b$ is kept fixed,  there  is no known example of this behavior at an outward corner\footnote
{
i.e. a nonsmooth point where the slope decreases} of the capacity function, though it  does happen at  other  corners provided that these are not on the volume curve. 

\begin{rmk}\label{rmk:Dioph}   \normalfont
Hutchings~\cite{Hut} was the first to observe that  the inequalities defining the capacity function do not require that $\bE\in \Ee$ but only that $\bE$ be Diophantine: 
that its entries satisfy~\eqref{eq:diophantine}. 
However, it is sometimes very important to know that a class does lie in $\Ee$  because any two such classes have nonnegative intersection. See, for example, Proposition~\ref{prop:live}.  
\hfill$\er$
\end{rmk}

\begin{DEFN}\label{def:perf} Let $\bE: = (d,m; \bbm)$ be a  Diophantine class.
If there is a rational number $a= p/q$ such that 
$\bbm = q \bw(a)$, then we say that 
 $\bE$ has {\bf center} $a= p/q$   and (as in \cite[Def.1.4]{usher}) call the class {\bf quasi-perfect}.    If in addition $\bE$ is an exceptional class, then we say that $\bE$ is {\bf perfect}.   A class that is quasi-perfect but not perfect, is called {\bf fake}.
 
 Finally, for a quasi-perfect, resp.\ perfect, class $\bE$, if 
 $|bd - m| < \sqrt{1-b^2}$ then we say that $\bE$ is
 {\bf $b$-quasi-perfect}, resp.\  {\bf $b$-perfect}.
\end{DEFN}

\begin{EXAMple}\label{ex:356} \normalfont
We now show how to compute an obstruction function on a small interval.
Let $a ={35}/6$ with $\bw({35}/6) = (1^{\times 5}, 5/6, (1/6)^{\times 5})$ as in Example~\ref{ex:ctfr0}.  It is not hard to check that there is a corresponding exceptional class $\bE_3 = (15,4; 6\bw( {35}/6))$. 
 If ${35}/6 \le z< 6$ then $\bw(z) =  (1^{\times 5}, z-5, (6-z)^{\times k},\dots)$ for some $k\ge 5$  and we have $\bbm\cdot \bw(z) = 35$. On the other hand, if $z<  {35}/6$, then its 
 weight expansion is $(1^{\times 5}, z-5, (6-z)^{\times 4}, (5z-29)^{\times k},\dots)$ provided that
$4(6-z)<z-5<5(6-z)$, which is equivalent to $ {29}/5 < z <  {35}/6$.  In this case
$$
\bbm\cdot \bw(z) = 30 + 5(z-5) + 4(6-z) + (5z-29) = 6z.
$$
Thus
\begin{align*} \mu_{\bE,b}(z) =\begin{cases} \frac{6z}{15- 4b} & \quad\mbox{ if } \;\; \frac {29}5 < z < \frac {35}6\\
\frac {35}{15- 4b} & \quad\mbox{ if } \;\; \frac{35}6 \le z < 6.
\end{cases}
\end{align*}
For a more general result, see Lemma~\ref{lem:munearc} below.\hfill$\er$
\end{EXAMple}

The next task is to determine whether a particular exceptional class  $\bE$  gives us a useful embedding
obstruction.
We denote by $\Ee: = \bigcup_n \Ee_n$ the set of exceptional classes in arbitrary blow ups of $\C P^2$, and will  define the {\bf length} $\ell(\bbm)$ of $\bbm$ (resp.\ $\ell(z)$ of $z$)  to be the number of nonzero entries in 
the vector $\bbm$ (resp.\ $\bw(z)$).  The following  lemma, taken from
\cite[Lemma~2.21 \&\ Prop.2.23]{AADT},
 extends results in \cite{ball} about embeddings into a ball to the case of a general convex toric target such as $H_b$.

\begin{lemm}\label{lem:0}   Let $\bE=(d,m; \bbm)$ be a Diophantine class with corresponding obstruction function $\mu_{\bE,b}$.

\begin{enumerate}
\item[{\rm(i)}]  The function $\mu_{\bE,b}(z)$ is continuous with respect to both variables 
$z,b$, and  for fixed $b$, is piecewise linear with rational coefficients.

\item[{\rm(ii)}] For any $b$, the obstruction $\mu_{\bE,b}(z)$ 
given by $\bE=(d,m; \bbm)$ is nontrivial at $z$ only if  $\ell(z) \ge \ell(\bbm)$.

\item[{\rm(iii)}] 
Any maximal interval $I$ on which $\mu_{\bE,b}$ is nontrivial  contains a unique point $a$, called the {\bf break point},  with
$\ell(a) = \ell(\bbm)$.   Moreover,
the function $$
I\to (0,\infty),\;\;  z\mapsto \mu_{\bE,b}(z) - V_b(z)
$$
 reaches its maximum at 
the break point $a$.
\end{enumerate}
\end{lemm}
\begin{proof}  Denote by $w_i(z)$ the $i^{\mathrm{th}}$ weight of $z$ considered as a function of $z$. Then, by \eqref{eq:wgt} it is piecewise linear, and is linear on any open interval that does not contain an element $a'$ with $\ell(a')\le i$.   This implies (i).

Property (ii) 
is an easy consequence of the  inequality
\begin{align}\label{eq:CSineq}
\sqrt{d^2-m^2}\ \sqrt{1-b^2}\le d-mb,
\end{align}
which, after squaring both sides simply reduces to $0\leq(db-m)^2$.

For property (iii), we note
that the graph of $V_b(z) = \sqrt{\frac z{1-b^2}}$ is concave down, so the function
 $ \mu_{\bE,b}(z)$ cannot be linear on any maximal interval on which it is nontrivial.  
On the other hand, it follows from \eqref{eq:wgt} that the formula for $w_i(z)$ can change only if it, or one of the 
earlier weights, becomes zero.  From this, it is not hard to see that  $\mu_{\bE,b}$ is linear on every open interval  
on which $\ell(z)> \ell(\bbm)$.  
The above remarks imply that
  $I$ must contain at least one break point $a$, which occurs when $\ell(a)=\ell(\bbm)$.    
     Moreover this point $a$ is unique because  if $\ell(a) = \ell(a')$ for some $a<a'$,
      then it is shown in  \cite[Lem.2.1.3]{ball} that  there is $z \in (a,a')$ with $\ell(z) < \ell(a)=\ell(\bbm)$,
       contradicting property (ii).
\end{proof}

 Note that part (iii) of the following lemma shows that, if $\bE$ is perfect with center $a$, then it is $b$-perfect in the sense of Definition~\ref{def:perf} exactly  if $\mu_{\bE, b}$  is nontrivial at $a$.  
Further, part (ii) establishes that the capacity function  $c_{H_b}(z)$ is the maximum, rather than just the supremum, 
of the obstruction functions $\mu_{\bE,b}(z)$.

\begin{lemm}\label{lem:perfect}
Let $\bE= (d,m; \bbm)$ be a Diophantine class. 
\begin{enumerate}\item[{\rm(i)}]  If  $\mu_{\bE, b}$ is  nontrivial at $z$, we must have 
$|bd-m|<\sqrt{1-b^2}$.  Furthermore,
$$
\mu_{\bE ,b}(z) \le
 V_b(z) \sqrt{ 1+ \frac 1{d^2 - m^2}}.
$$
\item[{\rm(ii)}]  If $c_{H_{b}}(z) > V_{b}(z)$, there is an exceptional class $\bE$ such that
$c_{H_b}(z) = \mu_{\bE,b}(z)$.  Moreover, $c_{H_b}$ is piecewise linear on any interval on which it is bounded away from the volume.
Finally, if $J\subset [0,1)$ is an open subset such that $c_{H_{b}}(z) > V_{b}(z) + \eps$ for all $b\in J$ and some $\eps>0$, then  for all $b$
outside of a finite subset of $J$, we may assume that
$c_{H_{b'}}(z) = \mu_{\bE,b'}(z)$ for all $b'$ in some neighborhood $J_b$ of $b$.

 \item[{\rm(iii)}]  If  $\bE$  has center $a=p/q$, then
$\mu_{\bE, b}(z)$  is nontrivial at $z=a$ if and only if $|bd-m|<\sqrt{1-b^2}$.  Further, when $b_0 = {m}/d$, we have
 \begin{align}\label{eq:upperb}
 \mu_{\bE, b_0}(a) =  \frac{pd}{d^2-m^2}\le \frac dq \Bigl(1 +\frac 2{pq}\Bigr).
 \end{align}
\end{enumerate}
\end{lemm}
\begin{proof}  The obstruction $\mu_{\bE, b}$ satisfies  
\begin{align}\label{eq:estim}
(d - m b)\, \mu_{\bE, b}(z) =  \bbm\cdot \bw(z) \le \|\bbm\|\ \|\bw\| = \sqrt{z} \sqrt {d^2 - m^2 + 1}.
\end{align}
But $\sqrt z \frac{\sqrt {d^2 - m^2 + 1}}{d-m b} > V_b(z)$
 if
$$
\frac{\sqrt {d^2 - m^2 + 1}}{d-m b} > \frac 1{\sqrt{1-b^2}}.
$$
This happens if 
\begin{align*}
&(d^2 - m^2 + 1) (1-b^2)\;\; > \;\; (d-mb)^2, \quad  i.e.\\
&d^2 - m^2 + 1 - b^2d^2 +b^2 m^2 - b^2 \;\; > \;\; d^2 - 2 m db + m^2 b^2,  \quad  i.e.\\
& 1-b^2 \;\; > \;\; d^2 b^2 - 2m db + m^2, \quad  i.e.\\
&1-b^2  \;\; > \;\; (db-m)^2.
\end{align*}
This proves the first claim in (i).

Combining the inequality
   $
d -m b\ge \sqrt{d^2 - m^2}\ \sqrt{1-b^2},
$  from \eqref{eq:CSineq} with \eqref{eq:estim},  we have  %
$$
\mu_{\bE ,b}(z) \le \frac{\sqrt{z}\sqrt {d^2 - m^2 + 1}}{\sqrt{d^2 - m^2}\ \sqrt{1-b^2}}
=  V_b(z) \sqrt{ 1+ \frac 1{d^2 - m^2}}.
$$
This completes the proof of  (i).
\MS

To prove (ii), we first show that if $c_{H_b}(z)=(1+\eps') V_b(z)$ for some $\eps'>0$ and a fixed value of $b$,  then there are only finitely many exceptional  classes $\bE$ such that $\mu_{\bE,b}(z) > (1+\eps'/2) V_b(z)$.  To this end, notice that (i) 
implies that $\sqrt{1 + \frac 1{d^2-m^2}} > 1 + \eps/2 > \sqrt{1+\eps}$, so that  we must have
$\frac 1{d^2-m^2} > \eps$.  Thus there is a constant $C$ so that $d^2-m^2 < C$.
 
 Since  we must also have 
 $|db-m|< 1$ by (i), 
 we are then looking for pairs of non-negative integers $(d,m)$ which lie in a compact region of the plane, so
there is an upper bound for the degrees $d$ of the relevant  classes.  Further, we must have $\ell(\bbm)\le \ell(z)$ by 
 Lemma~\ref{lem:0}.    Therefore there are only a finite number of relevant classes.  Hence, for each $b$,  there is 
 equality for at least one such class. 
 
 Moreover, if we now consider  sufficiently small neighborhoods $I'$ of $z$ and $J'$ of $b$,  then the inequality 
 $\mu_{\bE,b'}(z') > (1+\eps/2) V_{b'}(z')$
 holds for all $z'\in I'$ and $b'\in J'$, 
 so that again there are only finitely many classes $\bE$ such that 
  $c_{H_{b'}}(z') = \mu_{\bE, b'}(z')$ for some $z'\in I'$ and $b'\in J'$.  
  By Lemma~\ref{lem:0}, each function $z\mapsto \mu_{\bE, b'}(z)$ is piecewise 
  linear and hence with a finite number of nonsmooth points; see also Lemma~\ref{lem:munearc}.  
  Therefore the same is true for $z'\mapsto c_{H_{b'}}(z')$  on any interval on which it is bounded 
  away from the volume function.  Thus for fixed $b'$, these functions $c_{H_b'}$ and $\mu_{\bE,b'}$ of $z'$ either 
  agree at a single point or 
  agree on a finite union of  closed $z'$-intervals.  Moreover, as $b'$ changes  these $z'$-intervals change continuously. 
  Hence if we now consider our fixed $z$, all but finitely many $b'$ values lie in the interior of one of these intervals. 
 This completes the proof of item (ii).

The first claim in (iii) follows from the first calculation in (i) since in this case, the inequality in \eqref{eq:estim} is an equality.  
Further,  because $d^2-m^2 = pq-1$ and $\bbm\cdot \bw(a) =  q\bw(a)\cdot \bw(a) = p$, we have
  $$
  \mu_{\bE, m/d}(a) = \frac {pd}{d^2 - m^2} = \frac {pd}{pq-1}  = \frac dq \frac{1}{1-1/pq} \le \frac dq\Bigl(1 + \frac 2{pq}\Bigr),
  $$
  where the last inequality is a consequence of the fact that $\frac{1}{1-1/x}\le 1+\frac2x$ for $x\ge 2$. 
   This completes the proof.
\end{proof}

The calculation in Example~\ref{ex:356} generalizes as follows.

\begin{lemm}\label{lem:munearc}
Let  $\bE= (d,m; \bbm)$ be a quasi-perfect class with center $a = p/q$. 
Then the corresponding capacity function $\mu_{\bE,b}$ 
has the following form near $a$:  there are numbers $z_1 < a < z_2$ such that
\begin{align} \mu_{\bE,b}(z) =\begin{cases} \frac{qz}{d-mb} & \quad\mbox{ if } \;\; z_1< z < a\\
\frac p{d-mb} & \quad\mbox{ if } \;\; a\le z < z_2.
\end{cases}\label{eqn:munearcenter}
\end{align}
In particular, if  $\mu_{\bE,b}(a) > V_b(a)$ then this formula applies throughout the neighborhood of $a$ on which $\mu_{\bE,b}$ is nontrivial.
\end{lemm}
\begin{proof}  Consider the continued fraction expansion  $a = [\ell_0; \ell_1,\dots,\ell_N]$. 
We will carry out the proof when $N$ is odd, since this is the case for the exceptional classes $\bE_{n,k}$ that give our new staircases.  The proof for $N$ even is similar.

When $N$ is odd, 
by \cite[Lem.2.2.1]{ball} there is $\eps>0$ and $h>0$ so that  $z\in (a-\eps,a)$ has weight expansion
\begin{align}\label{eq:wtexph}
\bw(z) = \bigl( 1^{\times \ell_0}, (x_1(z))^{\times \ell_1},\dots, (x_N(z))^{\times \ell_N},
\ (x_{N+1}(z))^{\times h}, 
\dots,\bigr),
\end{align}
where the functions $x_j(z) = \al_j + z\be _j$ are linear functions of $z$, that increase w.r.t. $z$ when $j$ is odd and decrease when $j$ is even.  Denote the $k$th convergent to $a$ by $p_k/q_k$ so that $p_0/q_0 = \lfloor a\rfloor$ and $\p_n/q_n = p/q = a$.  Then 
 \cite[Lem.2.2.3~(ii),~(iii)]{ball} shows that
 $$
| \al_j| = p_{j-1},\qquad |\be_j| = q_{j-1} \qquad \mbox{ for } \;\; 1\le j\le N+1.
 $$
Further,  \cite[Lem.2.2.3~(i)]{ball} shows that the  weight expansion $\bw(a)$ may be written
 $$
 \bigl(1^{\times \ell_0}, (x_1(a))^{\times \ell_1}, \dots, (x_N(a) = 1/q)^{\times \ell_N}\bigr),
 $$ 
 where the weights $x_j(a)$ are determined by the convergents to the mirror continued fraction.
 Because $\bbm = q \bw(a)$,  we have
$$
\bbm\cdot \bw(z) =  q \bw(a)\cdot \bw(z) = \sum_{j=1}^{N} q \ell_j \ x_j(a) \ x_j(z) = qz,\qquad \mbox{ if } z\in (a-\eps,a),
$$
where the last equality holds for odd $N$ by \cite[Cor.2.2.7~(ii)]{ball}.
 This establishes the claimed formula when $z<a$.  The formula for $z>a$ now follows by 
 \cite[Prop.2.3.2]{ball}, since the assumption that $\bbm = q\bw$ implies that the smallest entry in $\bbm$ is $1$.
 
 This proves the first claim.  The second holds because the formula for $\mu_{\bE,b}(z)$ can change only at a point $a'$ with $\ell(a')\le \ell(\bbm)$, while Lemma~\ref{lem:0}~(iii) shows that   $a$ is the unique point with this property in the maximal neighborhood of  $a$ on which $\mu_{\bE,b}$ is nontrivial.
\end{proof}  

\begin{rmk} \label{rmk:munearcenter}\normalfont 
(i)  When $N$ is odd, the first formula in \eqref{eqn:munearcenter} holds for all 
$z$ that have a continued fraction expansion of the form  \eqref{eq:wtexph}.  Hence we may take
our value $z_1 = 
[\ell_0; \ell_1,\dots, \ell_N + 1] < a$.    (Note that, somewhat counterintuitively, when you increase an odd-placed coefficient of the continued fraction the number it represents decreases; see Remark~\ref{rmk:ctfr}). Similarly, 
the second formula holds for all $z< z_2$, where we have the value $z_2 = [\ell_0; \ell_1,\dots, \ell_N - 1]> a$, unless $\ell_N = 2$, in which case we 
have $z_2 = [\ell_0; \ell_1,\dots, \ell_{N - 1}+1]$.
 \MS

\NI  (ii) In the case of a ball (i.e. $H_0$), the paper \cite{ball} considered exceptional classes as above but with $m = 0$.    In this case, every exceptional class that 
has a center  $a$ is perfect, and hence (as in Lemma~\ref{lem:perfect}) is nontrivial at $a$.  Moreover,
 the obstruction function $\mu_{\bE,0}$   is live  throughout the interval on which it is nontrivial.  It also turns out that  $a< \tau^4 = \acc(0)$, and that  the  capacity function $c_{H_0}(z)$ for $z< \tau^4$ is completely determined by these perfect classes.
\MS

\NI(iii)  The claims in (ii) no longer hold when $b>0$.
 Lemma~\ref{lem:perfect}~(iii) shows that every exceptional class $\bE$ with center $a$ is live at $a$ for some $b$, but this $b$ may be  far from the most relevant value of $b$, namely $ \acc^{-1}(a)$.  (Here $b\mapsto \acc(b)$  is as in \eqref{eq:cHb} and we take the appropriate branch of its inverse as in \eqref{eq:acc1}.)   However, as we will see in \S\ref{ss:block},  if $\bE$ is a perfect blocking class with center $a$  then it does 
share some of the characteristics of  perfect classes for the ball. In particular, we show in Proposition~\ref{prop:block} that  
$\mu_{\bE,b}$  is  live at $a$ for all $b\in J_\bE$, and hence in particular at  $\acc^{-1}(b)$.
\hfill$\er$
\end{rmk}

The following remarks about the structure of a general class $\bE$ that gives a nontrivial   $\mu_{\bE,b}$ 
can be disregarded on a first reading.  They  will be useful  when establishing the existence of staircases; 
see  Example~\ref{ex:SsUu0}.   They are slight generalizations of results in \cite[\S2]{ball}.   We always assume that $\bbm= (m_1,\dots,m_n)$ is {\bf ordered}, i.e. $m_1\ge m_2\ge\dots\ge m_n$, since, because weight decompositions are ordered, the obstruction given by an ordered $\bbm$ is at least as large as that given by the any unordered tuple. 
If $a = [\ell_0; \ell_1,\dots,\ell_N]$ has weight decomposition
$\bw(a) = (1^{\times \ell_0}, w_1^{\times \ell_1}, \dots, w_N^{\times\ell_N})$, we call each group $w_i^{\times\ell_i}$ of equal terms a {\bf block} of $\bw(a)$.

\begin{lemm}\label{lem:obstruct}  Let $\bE = (d,m;\bbm)$ be a Diophantine class such  that the function $\mu_{\bE,b}$ is nontrivial  on a maximal interval $I$ with break point $a = [\ell_0; \ell_1,\dots,\ell_N]$.
Suppose further that $\bbm= (m_1,\dots,m_n)$ is ordered and that $z_0\in I$.
Then: 
\begin{itemize}
\item[{\rm(i)}]   there is a constant $ \al(z_0,b)>0$ 
and a vector $\eps$ of length  $\ell(z_0)$ such that
 \begin{align}\label{eq:eps1}
\bbm = \al(z_0,b) \bw(z_0) + \eps, \qquad \|\eps\|^2 = \eps\cdot\eps <1,
\end{align}
 where the constant $\al(z_0,b)$ is such  that the obstruction from $\al(z_0,b)\bw(z_0)$ is the volume constraint $V_b(z_0)$, that is,
 \begin{align}\label{eq:eps2}
\al(z_0,b) = \frac{d -m b}{\sqrt {z_0(1-b^2)}}.
\end{align}
 \item[{\rm(ii)}]  The vector $\bbm$  is constant on all blocks of $\bw(z_0)$ of length two or more, except perhaps for one block on which all the entries but one  are the same.  In the latter case, the entries on this block differ by $1$.
\end{itemize}
\end{lemm}

\begin{proof}  
Since $\mu_{\bE ,b}$ is nontrivial  at $z_0$, it follows from \eqref{eq:eps1},
\eqref{eq:eps2} that  we must have 
\begin{align} \label{eq:bigger}
\bw(z_0)\cdot \eps > 0.
\end{align}
Indeed, because $\bw(z_0)\cdot \bw(z_0)=z_0$, the choice of $\al(z_0,b)$ implies that 
\begin{align*}
V_{b}(z_0) & < \mu_{\bE,b}(z_0)\\
& = \frac{\al(z_0,b) \bw(z_0)\cdot \bw(z_0) + \eps\cdot \bw(z_0)}
{d-mb} = V_{b}(z_0)  +  \frac{ \eps\cdot \bw(z_0)}
{d-mb}.
\end{align*}
Now notice that because $\bE \cdot \bE  = -1$ and $\bw(z_0)\cdot \bw(z_0) = z_0$ by \eqref{eq:wtexp}, we have
\begin{align*}
d^2 - m^2 + 1 &= \bbm \cdot \bbm \\
 &= (\al(z_0,b) \bw(z_0) + \eps)\cdot(\al(z_0,b) \bw(z_0) + \eps)\\
& = \al^2(z_0,b)\  z_0 + 2\al(z_0,b)\ \eps\cdot\bw(z_0) + \eps\cdot\eps 
\\
& =  \frac{(d -m b)^2}{1-b^2} + 2\al(z_0,b)\ \eps\cdot\bw(z_0)+ \eps\cdot\eps \quad \mbox{by } \eqref{eq:eps2}\\
& = d^2 - m^2 +   \frac {(m  - bd )^2}{1-b^2} +2\al(z_0,b)\ \eps\cdot\bw(z_0)+  \eps\cdot\eps.
\end{align*}
  Therefore,  we must have
\begin{align}\label{eq:estim1}
\frac {(m  - b d )^2}{1-b^2} +2\al(z_0,b)\ \eps\cdot\bw(z_0)+  \eps\cdot\eps = 1.
\end{align}
Now observe that $\al(z_0,b)>0$ by \eqref{eq:eps2},  and $ \eps\cdot\bw(z_0)>0$ by \eqref{eq:bigger}. 
Hence we have  $\eps\cdot\eps < 1$ as claimed in (i). 
  
  To prove (ii), 
consider a block of $\bw(z_0)$ of length $s\ge 2$, and suppose that the corresponding  entries of $\bbm $ are the integers $m_{1},\dots, m_s$.  If $\la$ is the size of the corresponding entries in $\al(z_0,b) \bw(z_0)$ we have
$$
\sum_{i=1}^s(m_i-\la)^2 \le \|\eps\|^2<1.
$$
In particular, $|m_i-\la|<1$ for all $i$, which implies that the  integers $m_i$ differ by at most one on this block.
Write these entries as $m^{\times r}, (m-1)^{\times(s-r)}$, and set $\mu: = m-\la$.   Then
\begin{align*}
\sum_{i=1}^s(m_i-\la)^2& = r\mu^2 + (s-r)(1-\mu)^2 = R \mu^2 -2S\mu + T \\
&\qquad \mbox { where } R = s,\ S = (s-r), \ T = s-r
\end{align*} 
and hence has minimum value  $T - {S^2}/{R}$ taken at $\mu =S/R$. 
It remains to check that $$T - \frac{S^2}{R} = \frac{(s-r)r}s< 1$$ only if $r=0,1, s-1,s$. 
\end{proof}

\begin{rmk}\label{rmk:ECH}  \normalfont
{\bf (Relation to ECH obstructions)}   (i)    As explained in \S\ref{sec:ECH} below, there is an alternative approach to these calculations that use the obstructions coming from embedded contact homology (ECH). Thus,  the capacity function $z\mapsto c_{H_b}(z)$ may also be described as the supremum of a ratio of ECH lattice counts as in \eqref{def:csup}.  
If a class $\bE: = \bE(d,m;q\bw(p/q))$ with center $p/q$ is live at $p/q$ for some value of $b$, then $c_{H_b}(z) = \mu_{\bE, b}(z)$ for $z\approx p/q$, and  it follows from Lemma~\ref{lem:ECHcap} that $c_{H_b}(z)$ must equal a ratio of  counts for a particular lattice path.  Thus, there is some interval around $p/q$ on which the function $ z\mapsto\mu_{\bE, b}(z)$ 
agrees with the obstruction given by the $k$th ECH capacity, where,  by \eqref{eqn:Ldm},  
\begin{align}\label{eq:ECHind}
k = \tfrac 12\bigl(d(d+3) - m(m+1)\bigr).
\end{align}
  However, as is shown by Figure~\ref{fig:center17029}, for other values of $z$ these two obstruction functions may be  quite different.  Notice here that 
$\bE' = (73,20; 29\bw(170/29))$ is an exceptional class  with center at $170/29\approx
5.682.$  Further, $\bigl(d(d+3) - m(m+1)\bigr)/2 = 2564$.      \MS

\NI (ii)  If we fix $b$, then we can often figure out from the graph of $z\mapsto c_{H_b}(z)$ exactly which exceptional classes are giving the   obstruction.   
For example, in 
Figure~\ref{fig:b3k125accpt}, we see that when $b = 3/10$, the capacity function $c_{H_b}$ has an outer corner at $a =23/4$.    There is an exceptional class $\bE$ with center $23/4$, namely $\bE = \bigl(10,3; 4\bw(23/4)\bigr)$, such that $\mu_{\bE,3/10}$ is live near $23/4$ by Proposition~\ref{prop:live} below.  Hence we see that $c_{H_b}=\mu_{\bE,3/10}$ is given near $a$  by the $59$th ECH capacity.  As another example, we found that the $125$th capacity has a corner at about $35/6$ and from 
this it was not too hard  to figure out the corresponding  exceptional class.   Indeed, with $p/q =35/6$ we simply need to find suitable
 $d, m$ satisfying $(d-m)(d+m) = pq -1$, which we can do in this case by factoring $pq-1 = 209$. In this way, we are led to the  class $\bE = \bigl(15,4;6\bw(35/6)\bigr)$.
This was one of the first ways we used to identify relevant exceptional classes.  Once the first few such classes were found, we found others by numerical experimentation. One method, using Mathematica, is explained in \S\ref{subsec:findingk}.
\hfill$\er$
\end{rmk}

\begin{rmk}\label{rmk:computer}  \normalfont
The computer calculations described in \S\ref{sec:Mathem} take into account the ECH obstructions 
for  $k\le 25,000$.  The obstruction $\mu_{\bE,b}$ given by a class $\bE = (d,m;\bbm)$  with break point $a$
corresponds to the $k$th ECH obstruction near $z = a$, where $k =\bigl(d(d+3) - m(m+1)\bigr)/2$.  If we restrict the range of $b$, then we can estimate the degree $d$ of $\bE$ as a function of $k$ since we know from Lemma~\ref{lem:perfect} that $|b d - m|<1$.  For example,  if $b \le 1/3$, then  all relevant classes  have $m < bd+1 <  d/3+1$.  Combining that with the inequality
$d(d+3) - m (m+1)  < 50,000$, we find that $\frac 89 d^2< 
50,000$, which roughly corresponds to 
  $d \le 237$. 
   The only proviso here is that the computer works with a fixed value for $b$, which is often a rational approximation to the irrational value $b_\infty$ of interest.  However, it follows from results such as Proposition~\ref{prop:live} below that one can quantify the $b$-interval on which a class is visible.  
We will rigorously prove all our results, rather than relying on computer programs to tell us that certain classes cannot exist, since the relevant numbers grow quickly, escaping the range that is easily accessible to our computers.
\hfill$\er$
 \end{rmk}

\subsection{Characterizing staircases}\label{ss:charstair}

The section  explains how to recognize families of obstructions that give staircases for the Hirzebruch domains  $H_b, 0<b<1$. 
We begin with a useful criterion for a class to be live.  It is an adaptation  of \cite[Lem.4.1]{usher}.

\begin{prop}\label{prop:live} Suppose that  $\bE = (d,m;\bbm)$ is $b$-perfect with  center $a$.  Then: 
\begin{enumerate}
\item[{\rm(i)}]  If $b$ satisfies
\begin{align}\label{eq:bcondit}
\frac{m^2-1}{dm} \le b \le  \frac{1 + m(d-m)}{1+d(d-m)},
\end{align}
 the obstruction $\mu_{\bE, b}$ is live at $a$; that is, $c_b(a) = \mu_{\bE, b}(a)> V_b(a)$.
  \item[{\rm(ii)}] 
  If 
  $b< r/s$ satisfies
 \begin{align}\label{eq:bcondit3}
\frac{m^2-1}{dm} \le b\le  \frac{s + m(rd-sm)}{r+d(rd-sm)},
\end{align}
then $\mu_{\bE, b}(a)\ge \mu_{\bE', b}(a)$ for every exceptional divisor $\bE' = (d',m',\bbm')$ with $m'/d' \le r/s$. 
  \item[{\rm(iii)}]  If $\frac{m}{d}> \frac rs(1 + \frac 1{d^2})$ and $b>  r/s$ satisfies
 \begin{align}\label{eq:bcondit4}
\frac{m(sm-rd) - s}{d(sm-rd)  -r} \le b\le  \frac{m}d
\end{align}
then $\mu_{\bE, b}(a)\ge \mu_{\bE', b}(a)$ for every exceptional divisor $\bE' = (d',m',\bbm')$ with $m'/d' \ge r/s$.
\end{enumerate}
\end{prop}
\begin{proof}  Suppose that $b$ satisfies \eqref{eq:bcondit}, and
 let  $\bE' = (d', m'; \bbm')$ be
any exceptional class. Then because $\bE\cdot \bE' \ge 0$ we have
\begin{align}\label{eq:posit}
dd' - m m' - \bbm\cdot \bbm' \ge 0.
\end{align}
Therefore if $b\le m/d$ we have
\begin{align*}
\mu_{\bE', b}(a) &= \frac{\bbm'\cdot \bw(a)}{d'-m'b}\;\; \le\;\; \frac{\bbm' \cdot \bbm}{q(d'-m'\frac{m}{d})}  \;\;\mbox{ since } b\le \frac{m}d\\
&\le \frac {d (dd'-mm')}{q(dd'-mm')} = \frac dq.
\end{align*}
On the other hand,  because $\bw(a)\cdot \bw(a) = a=p/q$ by \eqref{eq:wtexp}, 
$$
\mu_{\bE, b}(a) = \frac {\bbm\cdot\bw(a)}{d-mb}= \frac p{d-mb}\ge  \frac dq \;\;\mbox{ if } \;\;  pq \ge d^2-dm b.
$$
Since $pq= d^2 - m^2 +1$ by \eqref{eq:diophantine}, this will hold if 
$$
dm b \ge d^2-pq = m^2-1, 
$$
i.e. $  b\ge \left(m^2-1\right)/\left(dm\right).$  This shows that when  $\left(m^2-1\right)/\left(dm\right) \le b \le m/d$  
the obstruction  $\mu_{\bE, b}(a)$ is at least as large as
any other that is defined by an exceptional class.  Finally, we have assumed that $\bE$ is $b$-perfect, so $|bd-m|< \sqrt{1-b^2}$. 
Hence $\mu_{\bE, b}$ is live at $a$ by Lemma~\ref{lem:perfect}~(iii).  

Similarly,  if $b> {m}/d$, we have
\begin{align*}
\mu_{\bE', b}(a) &= \frac{\bbm'\cdot \bw(a)}{d'-m'b} \;\; \le\;\; \frac{dd'-mm'}{q(d'-m'b)}  \;\;\mbox{ by } \eqref{eq:posit}\\
& \le \frac{d-m\frac{m'}{d'}}{q(1-b\frac{m'}{d'})} 
\;\; \le\;\; \frac {d-m}{q(1-b)} 
\end{align*}
where the last inequality uses the fact that because ${m}/d < b< 1$ 
the function $x\mapsto \left(1-\frac{m}d x\right)/\left(1-bx\right)$  increases on the interval 
$(0,1]$.
Hence, as  above, the inequality
$\mu_{\bE, b}(a)\ge \mu_{\bE', b}(a)$ holds whenever
$$
\frac p{d-mb}  \ge \frac {d-m}{q(1-b)}. 
$$
Substituting $pq = d^2-m^2 + 1$, one readily checks that this is equivalent to the upper bound in \eqref{eq:bcondit}.
Since we again have assumed $|bd-m|< \sqrt{1-b^2}$, this completes the proof of (i).

Claim (ii) is proved in exactly the same way  except that if $b>{m}/d$ 
we use the fact that ${m'}/{d'}\le  r/s$ to conclude that $\mu_{\bE', b}(a) \le \frac{sd-rm}{q(s-rb)}$. Hence it suffices that 
$$
\frac {pq}{d-mb}  = \frac {d^2 - m^2 + 1}{d-mb} \ge \frac {sd-rm}{s-rb},
$$
which is equivalent to the condition $b\le \frac{s + m(rd-sm)}{r+d(rd-sm)}$. 

Finally, to prove (iii) we consider the function $x\mapsto \frac{d-mx}{1-bx}$ for $ r/s < b < {m}/d<1$ and $ r/s \le x < 1$.    This is a decreasing function.  Hence, taking $x ={m'}/{d'}$  we have
\begin{align*}
\mu_{\bE', b}(a) &  \le\;\; \frac{d-m\frac{m'}{d'}}{q(1-b\frac{m'}{d'})}  \;\;\mbox{ by } \eqref{eq:posit}\\
& \le \frac{sd-rm}{q(s-rb)}.  
\end{align*}
Hence it suffices that 
$$
\frac {pq}{d-mb}  = \frac {d^2 - m^2 + 1}{d-mb} \ge \frac {sd-rm}{s-rb}, 
$$
which is equivalent to the condition $b\ge \frac{sm^2 - rm d - s}{smd - rd^2 - r}$, since we assumed that  $sm> r(d + 1/d)$. 
\end{proof}

\begin{EXAMple}\label{ex:52}   \normalfont
Consider the exceptional class $\bE = (2,0;1^{\times 5})$.  It has center $a=5$ and is $b$-perfect   for $b$ such that $|2b|< \sqrt{1-b^2}$, i.e. for $b\in [0,1/\sqrt5)$.  By Lemma~\ref{lem:munearc}, $\mu_{\bE,b}(z) = 5/2$ for all $z\ge 5$. Notice that
 $$
5/2=  \mu_{\bE, b}(6) =  \sqrt{\frac 6{1-b^2}} = V_{b}(6),\quad\mbox{ when } b=1/5.
 $$
Further $b$ satisfies condition
 \eqref{eq:bcondit} precisely if $b< 1/5$.  
Hence  $c_{H_{b}}(z)  = 5/2$ for all $b<1/5$ and $z\in [5,6)$, so that by continuity, this  also holds at $z = 6$ and $b = 1/5$.    Therefore $H_{1/5}$ is unobstructed.  Further, there can be no ascending staircase for $b = 1/5$.  However, this argument does not rule out the possible existence of a descending staircase for $b = 1/5$.  For further discussion of the point $b=1/5$, see Example~\ref{ex:19th} and \S\ref{ss:15}.
\hfill$\er$
\end{EXAMple}

\begin{rmk}\label{rmk:fake}    \normalfont
The proof of Lemma~\ref{lem:perfect} only uses the Diophantine conditions~\eqref{eq:diophantine} satisfied by $\bE$, and hence holds for {\bf fake  classes} $\bE = (d,m;\bbm)$ that satisfy these conditions but do not reduce correctly under Cremona moves, and so are not exceptional classes.
 It turns out that 
there are many such fake classes: see Example~\ref{ex:seq} below. 
  By contrast,   the proof of Proposition~\ref{prop:live} is based on  the fact that different classes in $\Ee$ have nonnegative  intersection, which is false for fake classes. In fact, it seems likely that the obstructions given by fake classes are never live, though we do not attempt to give a proof here. In any case, it follows from \eqref{eq:cHb} that if there were a live fake class for some values of $a,b$ there would have to be a class in $\Ee$ giving the same obstruction at $a,b$.
  
  By contrast, results such as Lemma~\ref{lem:block0}  about blocking classes do not require  the class in question to be exceptional, and hence we phrase them in terms of Diophantine classes. It turns out that many quasi-perfect blocking classes are in fact perfect: see Proposition~\ref{prop:block}.
 \hfill$\er$ \end{rmk}

We now give a formal definition of staircase, to clarify the language used below.  In \S\ref{ss:overv}, we said that $H_b$ has an infinite staircase precisely when
the  graph of $c_{H_b}$  has infinitely many nonsmooth points that lie above the volume curve.  
We saw in  Lemma~\ref{lem:perfect}~(ii) that $c_{H_b}$ has finitely many nonsmooth points on any interval on which it is bounded away from the volume.  It follows that it can have only countably many nonsmooth points that lie above the volume curve.  Moreover,
by \cite[Thm. 1.11]{AADT}, the corresponding set of $z$-coordinates has a unique limit point $a=\acc(b)$.  Hence we can always arrange these $z$-coordinates into a convergent sequence $a_k$.   
 By Lemma~\ref{lem:0}~(iii) and  Lemma~\ref{lem:perfect}~(ii), each such  point $a_k$ is the break point   of some obstructive class $\bE_k$  such that $\mu_{\bE_k,b}$ is live at $a_k$ in the sense of   Definition~\ref{def:nontriv}.  Hence $H_b$ has an infinite staircase precisely when there is a sequence of such obstructive classes.

\begin{DEFN}\label{def:stair}  A {\bf staircase} $\Ss$ for $H_b$ is a  sequence of 
distinct Diophantine classes $\bigl(\bE_k\bigr)_{k\ge 0}$ such that each
 $\mu_{\bE_k,b}$ is live at its break point $a_k$ for some sequence of distinct points $a_k$ that converge to $a_\infty: = \acc(b)$.  
 The points $a_k$ are called the {\bf steps} or {\bf exterior corners} of $\Ss$, and 
 $\Ss$ is said to be  {\bf ascending} (respectively {\bf descending}) if the sequence of steps $a_k$ is increasing (resp.\ decreasing). 
We say that an ascending (resp.\ descending) staircase is {\bf complete} if there is a one-sided neighborhood $(a_\infty-\eps, a_\infty]$  (resp.\ $[a_\infty,a_\infty + \eps)$)  of $a_\infty$  such that there is no class $\bE'$ other than the $\bE_k$ with break point $a'$ in this neighborhood  and such that $\mu_{\bE',b}$ is live at $a'$. Finally we say that a stairase $\Ss$ is {\bf perfect} if all its classes are perfect.
\end{DEFN}

\begin{rmk} \normalfont
(i)
It follows from the above remarks that $H_b$ has an infinite staircase  if and only if  there is a sequence of classes 
$\bE_k$ with the above properties.   However notice that even if $\bigl(\bE_k\bigr)$ is complete, with ascending  steps $a_k$, then the capacity function $c_{H_b}$ may not be determined on any interval $(a_\infty - \eps, a_\infty]$ by the classes $\bE_k$, since there might be \lq big' obstructions as in Example~\ref{ex:big}.
\MS

\NI (ii)  The new staircases that we describe in \S\ref{ss:Fib} consist of two interwoven sets of classes.  According to our definition, each of these sets forms a staircase, as does their union.   It is likely (though we do not prove that here) that  their union is complete in the sense  defined above.\hfill$\er$
\end{rmk}

We are now in a position to establish  our first {\bf staircase recognition} criterion. 
We will show in the proof that condition (ii) below cannot be satisfied if $b_\infty = 0$,  though we do not assume that from the beginning.

\begin{prop} \label{prop:stair0} Suppose that $\Bigl(\bE_k = \bigl(d_k, m_{{k}}; q_k\bw(p_k/q_k)\bigr)\Bigr)_{k\ge 1}$ is a sequence of perfect  classes with $m_k \ne 0$ such that
\MS

\begin{itemize}\item[{\rm(i)}]  $m_{k}/d_k\to b_\infty \in [0,1)$ and $p_k/q_k\to a_\infty$;
\MS

\item[{\rm(ii)}]  $ \frac{m_{k}^2-1}{d_km_{k}} < b_\infty<  \frac{1 + m_{k}(d_k-m_{k})}{1+d_k(d_k-m_k)}
 $ for all $k\ge k_0$.
\end{itemize}

\NI
Then  the classes $\bigl(\bE_k\bigr)_{k\ge k_0}$ provide a staircase for $H_{b_\infty}$. 
Moreover, $a_\infty = \acc(b_\infty)$.
\end{prop} 
\begin{proof}    Condition (ii) implies that for sufficiently large $k$ 
there is a  constant $c $ depending on 
$1-b_\infty>0$ such that
$$
\frac{m_k}{d_k}  - \frac 1{m_kd_k} < b_\infty < \frac{m_k}{d_k + \frac1{d_k-m_k}} + 
\frac{1}{d_k(d_k-m_k)} < \frac{m_k}{d_k} + \frac{1}{cd_k^2}.
$$
Rearranging, this gives us $1/m_k>m_k-b_\infty d_k>-1/(cd_k)$. 
Now we must have $d_k \to \infty$, since there are only finitely many exceptional classes of bounded degree. 
Further, since  $m_k \ne 0$,  we also have $m_k \to \infty$, since otherwise $b_\infty = 0$  and we have $m_k<1/m_k$, which is impossible since $m_k\ge0$ is an integer.  This
 implies $|m_k-b_\infty d_k|\to0$. We obtain in particular the bound $|m_k - d_k b_\infty| < \sqrt{1-b_\infty^2}$ for large $k$,  so that the classes
$\bE_k$ are $b_\infty$-perfect for large $k$. 
 Proposition~\ref{prop:live} (i) now implies  that  $E_k$ is live at $b_\infty$ for sufficiently large $k$.  Thus,  we have a sequence of live obstructions $\mu_{\bE_k, b_\infty}$ whose centers converge to $a_\infty$; in other words, a staircase.
Finally,  we may use \cite[Thm.1.11]{AADT} to deduce that $a_\infty = \acc(b_\infty)$, as desired.
\end{proof}

Since condition (ii) above is rather strong, we now establish a weaker version.
We begin by giving conditions that guarantee that $H_{b_\infty}$ is unobstructed.

\begin{lemm}\label{lem:unobstr}
Suppose that $\Bigl(\bE_k = \bigl(d_k, m_k; q_k\bw(p_k/q_k)\bigr)\Bigr)_{k\ge 1}$ is a sequence of distinct quasi-perfect  classes such that $m_k/d_k\to b_\infty \in (0,1)$, that $|d_kb_\infty-m_k|<1$ for all $k$, and that $a_k: = p_k/q_k\to a_\infty$.  Then
\begin{itemize}\item[{\rm(i)}]  $a_\infty = \acc(b_\infty)$.
\item[{\rm(ii)}]  If the $\bE_k$ are perfect, then $H_{b_\infty}$ is unobstructed.
\end{itemize}
 \end{lemm}

\begin{proof}   By Lemma~\ref{lem:perfect}~(iii)   the obstructions $\mu_{\bE_k, b_\infty}(a_k)$ are nontrivial for all $k$.  
The proof of \cite[Thm.1.11]{AADT} then shows that $a_\infty = \acc(b_\infty)$. Indeed, a careful reading of the \cite{AADT} 
argument shows that it never uses the hypothesis that the obstructions $\mu_{\bE_k, b_\infty}$ are live at $a_k$, instead 
using extensions of the arguments in Lemma~\ref{lem:obstruct} above. 
This proves (i).  
\MS

If we now assume that the $\bE_k$ are perfect and set $b_k: = m_k/d_k$, then Proposition~\ref{prop:live} implies that 
 $\mu_{\bE_k, b_k}$ is live at $a_k$.  Thus, by Lemma~\ref{lem:perfect}~(i), we have
   $$
V_{b_k}(a_k) <  c_{H_{b_k}}(a_k) = \mu_{\bE_k, b_k}(a_k) < V_{b_k}(a_k)\sqrt{1 + 1/(d_k^2 - m_k^2)}.
$$
But $d_k^2 - m_k^2 = p_kq_k-1\to \infty$.    Hence by the continuity of the function $(z,b)\mapsto c_{H_b}(z)$,
we must have $c_{H_{b_\infty}}(a_\infty) = V_{b_\infty}(a_\infty)$.    In other words, $H_{b_\infty}$ is unobstructed.
 \end{proof}
 
 \MS

\begin{EXAMple}\label{ex:seq}  \normalfont
The following examples illustrate the need for some of these conditions.\MS

\NI (i) For each $k \ge 1$ there is an exceptional  class  $\bE_k = (5k, k+1; \bbm_k)$ with 
center $a_k = 6- 1/(2k) = \left(12k-1\right)/\left(2k\right)$, where 
$$
\bbm_k = \bigl((2k)^{\times 5}, 2k-1, 1^{\times (2k-1)}\bigr)= 2k\bw\bigl(\frac{12k-1}{2k}\bigr).
$$
These are exceptional  classes because they evidently satisfy \eqref{eq:diophantine} and it is not hard to prove 
by induction that they reduce correctly.
Since ${m_k}/{b_k} \to 1/5$ and $a_k\to 6 = \acc(1/5)$, the first  two conditions in Lemma~\ref{lem:unobstr} hold.
However,  condition (ii) in  Proposition~\ref{prop:stair0}  does not hold: indeed $\bE_k$ is not even $1/5$-perfect.  
Nevertheless, the existence of these classes implies that $H_{1/5}$ is unobstructed.  We prove that it does not admit a staircase in \S\ref{ss:15}.
\MS

\NI (ii)  There is an exceptional class $\bE' = \bigl(73,20; 29\bw({170}/{29}) \bigr)$ with center $a ={170}/{29}$.  Here we have $29\bw({170}/{29})  = (29^{\times 5}, 25, 4^{\times 6}, 1^{\times 4})$, and it not hard to check that $\bE'$ does reduce correctly under Cremona moves.  By Proposition~\ref{prop:live}
this class is live for $b\approx 20/73 = 0.27397...$ while\footnote
{
The function $\acc^{-1}_L$ is a branch of the inverse to $b\mapsto\acc(b)$; see \eqref{eq:acc1}.}  $\acc_L^{-1}({170}/{29}) \approx  .275425...$ is almost the same.   Therefore one can expect that this class is relevant to the study of staircases.  Indeed, it turns out that this class is the entry $\bE^E_{0,0}$  in the staircase $\Ss^E_0$ in Theorem~\ref{thm:stair}.
    \MS

\NI (iii)   There is another  related\footnote
{
In fact there is a family of such classes, whose centers $a_i = 6 - \frac i{(i+1)(i+2) - 1}, i\ge 3,$ converge to $6$; for  $i=2n+4$ 
these are the exceptional classes $\bE_{n,0}$ in $\Ss^E_n$, but they seem to be fake for odd $i$. }
 class $\bE'' = \bigl(48, 14; 19 \bw({111}/{19}) \bigr)$ with center ${111}/{19}$, but one can check that this one is fake,
i.e. quasi-perfect but not perfect:  see Remark~\ref{rmk:fake}. 
As is shown by Figure~\ref{fig:center11119}, when $b$ is chosen so that $\acc(b) = 111/19$ (i.e. $b\approx .296654$)
 the corresponding function $z\mapsto \mu_{\bE'', b}(z)$ agrees with the $k$th ECH capacity
 where $k$ is given by \eqref{eq:ECHind}, and lies between the volume function and the capacity function $c_{H_b}$.    
\hfill$\er$
\end{EXAMple}

Although some of the staircases that we identify do satisfy the criteria in Proposition~\ref{prop:stair0}, 
some of them only satisfy a modification of condition (ii), such as that in \eqref{eq:bcondit3}. 
Thus we have a sequence of classes $\bE_k$ that are nontrivial, though perhaps not live, at $b_\infty$.  In  order for such a  sequence  to form a staircase,  we have to rule out the existence of an overshadowing class $\bE'$, defined as follows.

\begin{DEFN}\label{def:big}  Let $I = (y_0,y_1)$ be a maximal open interval on which the obstruction $\mu_{\bE', b}$ is nontrivial.    Then we say that 
\begin{itemize}\item[$\bullet$] 
$\bE'$ is a {\bf right-overshadowing class at $y_0$} if
\begin{itemize}\item[$\circ$] $\acc(b) = y_0$, and 
\item[$\circ$]
$\mu_{\bE',b}$ is live  on some nonempty subset $(y_0,y_0+\eps)\subset I$;
\end{itemize} 
\item[$\bullet$] $\bE'$ is a {\bf left-overshadowing class at $y_1$}  if
\begin{itemize}\item[$\circ$] $\acc(b) = y_1$, and 
\item[$\circ$] $\mu_{\bE',b}$ is live  on some nonempty subset $(y_1-\eps, y_1)\subset I$.
\end{itemize}
\end{itemize}
If $\Ss$ is a sequence of  quasi-perfect classes whose centers $a_k$ converge to $a_\infty$, we say that $\bE'$ {\bf overshadows} $\Ss$ (or {\bf is an overshadowing class}, if $\Ss$ is understood from the context)  if either
\begin{itemize}\item[$\bullet$]  the $a_k$ descend and $\bE'$ is right-overshadowing (i.e. overshadows to the right of $a_\infty$); or
\item[$\bullet$]  the $a_k$ ascend and $\bE'$ is left-overshadowing at $a_\infty$ (i.e. overshadows to the left of $a_\infty$).
\end{itemize}
\end{DEFN}

Here is a  refined version of Proposition~\ref{prop:stair0}. Notice that setting $r/s = 1$ in (iii) below gives the same as
condition (ii) in Proposition~\ref{prop:stair0}.

\begin{prop}\label{prop:stair00}
Let  $\Bigl(\bE_k = \bigl(d_k, m_k; q_k\bw(p_k/q_k)\bigr)\Bigr)_{k\ge 0}$ be a sequence of perfect  classes  that satisfies the following conditions for some $r/s \in (0,1]$:

\begin{itemize}\item[{\rm(i)}]  $m_k/d_k\to b_\infty \in [0,r/s)$ and the centers $p_k/q_k$ are a monotonic  sequence with limit  $a_\infty$;
\smallskip

\item[{\rm(ii)}]  $ \frac{m_k^2-1}{d_km_k} \le b_\infty\le  \frac{s + m_k(rd_k-sm_k)}{r+d_k(rd_k-sm_k)}
 $ for all $k\ge k_0$.
\smallskip

\item[{\rm(iii)}]  There is no overshadowing class  $\bE'$ at 
$a_\infty$ of degree $d'< s/(r-sb_\infty)$ and with $m'/d'>r/s$
\end{itemize}
\smallskip

\noindent Then there is $k_0$ such that the classes $\bigl(\bE_k\bigr)_{k\ge k_0}$ are live at $b_\infty$ and hence form a staircase for $b_\infty$.
\end{prop}

\begin{proof}     We know from Lemma~\ref{lem:unobstr} that $H_{b_\infty}$ is unobstructed and $\acc(b_\infty) = a_\infty$.
Further, because  condition (ii) holds, it follows from Proposition~\ref{prop:live}~(ii)  that $\mu_{\bE_k,b_\infty}(p_k/q_k) > \mu_{\bE',b_\infty}(p_k/q_k)$  for all exceptional classes $\bE'=(d',m',\bbm')$ with $m'/d'\le r/s$.  Hence if $\bE_k$ is not live at $b_\infty$ there is a class 
$\bE' = (d',m';\bbm')$ with $m'/d'>r/s$ such that $\mu_{\bE',b_\infty}(a_\infty) > 
\mu_{\bE_k,b_\infty}(a_\infty)$. But then Lemma~\ref{lem:perfect} implies that 
$|b_\infty d' - m'|<1$ and hence $|m'/d'-b_\infty|  < 1/d'$.  Further,  $b_\infty < r/s$ by (i), so that
$m'/d'-b_\infty > r/s - b_\infty$.  
 Therefore  we must have $1/d' > r/s- b_\infty>0$, or, equivalently,
 $d' < s/(r-sb_\infty)$.  In particular, there are a finite number of such   classes $\bE'$.  It follows 
that either all but finitely many $\bE_k$ are live at $b_\infty$ or there is a single class $\bE'$ with 
$d'<  s/(r-sb_\infty)$ whose  obstruction for $b = b_\infty$ is live at all but finitely many of the points $a_k$.
But in the first case the $(\bE_k)_{k\ge k_0}$ form a staircase, while in the second, since $H_{b_\infty}$  is unobstructed,
$\bE'$  satisfies the conditions to be an  overshadowing class.
Since the latter is ruled out by condition (iv), the classes $(\bE_k)_{k\ge k_0}$ must form a  
(possibly incomplete) staircase, 
as claimed. 
\end{proof}

\begin{rmk}\label{rmk:stair00} \normalfont
(i) There is an analogous result for  classes  $\bE_k$ such that the ratios $b_k = {m_k}/{d_k}$ decrease with limit $b_\infty>r/s$
in which the inequalities  in (iii) above are replaced by \eqref{eq:bcondit4} from Proposition~\ref{prop:live}~(iii):
 \begin{align*}
\frac{m_k(sm_k-rd_k) - s}{d_k(sm_k-rd_k)  -r} \;\le\;  b  \;\le\;  \frac{m_k}{d_k}, \qquad d'<\frac s{sb_\infty - r}.
\end{align*}
Notice that if $m_k/d_k>r/s$ for all $k$ then,  since the $d_k\to \infty$, the inequality 
$
m_k/d_k>r(1 + 1/{d_k^2})/s
$
required by Proposition~\ref{prop:live}~(iii) does   hold for large $k$.  Further, as
 above, any potentially overshadowing  class would have parameters $d',m'$ where ${m'}/{d'} < r/s$,
and could only be live at $b_\infty$ if 
$b_\infty- r/s< b_\infty-m'/d'  < 1/d'$, which gives the bound $d'<  s/(sb_\infty - r)$.\MS 

\NI (ii)   
Although in principle there can be overshadowing classes for both ascending and descending sequences of obstructions, in practice we only seem to have to worry about them in the   descending case.   See Lemmas~\ref{lem:SsLellbest} and ~\ref{lem:SsLubest}, for example.
\hfill$\er$
\end{rmk}

\begin{EXAMple}  \label{ex:big} \normalfont
The exceptional class $\bE': = \bigl(3,1;2,1^{\times 5}\bigr)$ is a potentially awkward  class that  is a rearranged version of the large nontrivial blocking  class $\bE_0: = \bigl(3,2;1^{\times 6}\bigr)$
in Theorem~\ref{thm:block}.  For many values of $z$ it corresponds to the $8$th ECH capacity, and, as explained in \S\ref{subsec:pictures}, is visible on many of the figures.
Though $\bE'$ has no center and so is not perfect, it has break point at $a=6$, and gives the following obstruction
\begin{equation*}
\mu_{\bE',b}(z) = \left\{\begin{array}{ll} \displaystyle{\frac {1+z}{3-b}} , & 5<z<6,\\  
\phantom{.} & \\
\displaystyle{\frac {7}{3-b}}, & z \ge 6.\end{array}\right.
\end{equation*}
Note that $\bE'$ is nontrivial  at its break point  $z=6$ precisely when
$$ \mu_{\bE',b}(6) = \frac 7{3-b}\; >\; V_b(6) = \sqrt {\frac 6{1-b^2}}\;\;\;\;\;\;\mbox{i.e. if }\;\;
 b\in \left(\frac 15,\frac 5{11}\right).$$
Further, one can check from \eqref{eqn:accb} that the range  of $b$ in which $\acc(b)< 6$ is also
the interval  $1/5< b<5/11 $.  
Therefore, for $b$ in this range  we have
$$ 
\mu_{\bE',b}\bigl(\acc(b)\bigr) =  \frac {1+\acc(b)}{3-b}, 
$$
and this turns out to be   
 precisely the same as the volume obstruction.  Indeed,
if we define $c(b) = \frac{(3-b)^2}{1-b^2} - 2$,
we have
\begin{align}\label{eq:volacc}
&\frac {1+\acc(b)}{3-b}= \sqrt{\frac {\acc(b)}{1-b^2}} \\ \notag
& \qquad\quad  \Longleftrightarrow\;\; \bigl(1+\acc(b)\bigr)^2 = \acc(b)\frac{(3-b)^2}{1-b^2} = \acc(b)\bigl(c(b) + 2\bigr),\\  \notag
& \qquad\quad \Longleftrightarrow\;\;  \acc(b)^2 - c(b) \acc(b) + 1 = 0
\end{align}
which holds by the definition of $\acc(b)$ in \eqref{eqn:accb}.  
Therefore, for 
$b\in (1/5,5/11)$,  this class 
satisfies at least some of  the conditions to be overshadowing, and so potentially obstructs the existence of a descending staircase.
\MS

However, it does not overshadow the staircases $\Ss^E_{u,n}, n\ge 0,$ defined in Theorem~\ref{thm:E}.  These  are descending with limit points at $(a_{n,\infty},\ b_{n,\infty}) \in (5,6)\times (1/5,1/3)$; and we show in Example~\ref{ex:SsEu0} and Lemma~\ref{lem:SsEubest} that they satisfy condition (ii) in Proposition \ref{prop:stair00} (by way of Lemma~\ref{lem:DMineq} (ii)) with $r/s=1/3$. Therefore, because $\bE'$ has $m/d=1/3$ (rather than $>1/3$), it is not overshadowing.  
One can also see  the initial peaks of the steps of $\Ss^E_{u,0}$ 
poking above the line $z\mapsto \mu_{\bE',b} (z)$ in Figure~\ref{fig:b0E}. 
\comment{10/3  rewrote next sentence and also changed end of Example~\ref{ex:SsUu0}, which as you pointed out, Morgan, was wrong before.}
{\tiny 
Similarly, the class $\bE'$ does not overshadow the staircases $\Ss^U_{u,n}, n\ge 0$, defined in Theorem~\ref{thm:U}, as shown in Example~\ref{ex:SsUu0}.}
Note that the class $\bE'$ does not overshadow any staircases (such as the $\Ss^U_{u,n}$) that accumulate at points $> 6$ because the obstruction $z\mapsto \mu_{\bE',b}(z)$ is constant for $z>6$.
 \hfill$\er$
\end{EXAMple}

\begin{rmk} \label{rmk:big} \normalfont
(i)
The observation in \eqref{eq:volacc} about the function $b\mapsto V_b(\acc(b))$
turns out to be an essential ingredient of our proof of the relation between staircases and blocking classes: see 
Lemma~\ref{lem:block2}. 
\MS

\NI (ii) All of the new
 ascending  (resp.\ descending)   staircases   that we have found
accumulate at the initial (resp.\ final) point $\bigl(a_\infty, V(a_\infty)\bigr)$  of a \lq\lq big''  obstruction, namely, an obstruction that goes through  $\bigl(a_\infty, V(a_\infty)\bigr)$ and for the given value of $b$ obstructs above (resp.\ below) $a_\infty$.   
Moreover, in many cases (though not in the case of the Fibonacci staircase at $b=0$) this big obstruction is given by a perfect blocking class as described in \S\ref{ss:block}.
Note that  this  phenomenon is quite different from that of a (potentially) overshadowing class since the latter type of obstruction   overhangs  the staircase, and so is active on the staircase side of the accumulation point.
\hfill$\er$
\end{rmk}

\begin{EXAMple} \label{ex:19th}\normalfont
The degree $5$ class $\bE_5: = (5,1; 2^{\times 6},1)$ also affects the capacity function for some ranges of $b,z$, and for $b\approx1/5$ the obstruction $\mu_{\bE_5,b}$ equals the obstruction from the $19^\text{th}$ ECH capacity for $z\approx7$, a behavior that is discussed in \S\ref{subsec:ecECH}.
One can check that  the corresponding obstruction is
$$
\mu_{\bE_5, b}(z) = \frac{6+z}{5-b}, \quad  11/2< z < 7
$$
which  goes through the accumulation point $(6,5/2)$ when $b= 1/5$, and is tangent to the graph of $z\mapsto V_{1/5}(z)$ at  that point.  We show in \S\ref{ss:15} that $c_{H_{1/5}}(z) = \mu_{\bE_5, b}(z)$ on some interval $[1/5, 1/5 + \eps)$.

We claim  that $b=1/5$ is the unique value where the obstruction given by $\bE_5$ might affect the existence of staircases since we have
\begin{align}\label{eq:deg5}
\mu_{\bE_5, b}(\acc(b)) < V_b(\acc(b)),\quad b\ne 1/5.
\end{align}
To justify this, 
 recall that $
V_b(\acc(b)) = \frac{1 + \acc(b)}{3-b}$ by  \eqref{eq:volacc}.   
Since $\mu_{\bE_5,b}(z)$ is constant  for $z\ge 7$ and the increasing function $z\mapsto \frac{1+z}{3-b}$  is greater than $\mu_{\bE_5,b}(z)$ at $z=7$,  the class $\bE_5$
can only affect staircases with $\acc(b) \le 7$.    Moreover, $\acc(b)\ge 3 + 2\sqrt2 > 11/2$, and in this range 
 we have 
$$
\mu_{\bE_5, b}(z) = \frac{6+z}{5-b} = : f_1(z),
$$
whose graph is a line of  smaller slope than that of  $\frac{1 + z}{3-b} = : f_2(z)$. 
Thus \eqref{eq:deg5} will hold if we show that for each $b\ne 1/5$,  the two lines
$y= f_1(z)$ and $y = f_2(z)$ intersect at a point 
$z_b$  with $f_2(z_b)< V_b(\acc(b))$.  
But $z_b = (13-5b)/2$, and $f_2(z_b) = 5/2 < V_b(\acc(b))$ for all $b\ne 1/5$, since
the function  $b\mapsto V_b(\acc(b)) $ has a minimum value of 
$5/2$, taken at $b= 1/5$; see  Fig.~\ref{fig:101}.
Thus \eqref{eq:deg5} holds. 
\hfill$\er$
\end{EXAMple}

\begin{EXAMple}\label{ex:13} \normalfont
{\bf(The staircase at $b = 1/3$)} \; This staircase 
 is somewhat different from the new staircases that we find in the current paper.  For one thing, it consists of three interwoven  sequences, while the new ones that we discuss here  consist of two such sequences.\footnote
{
See \cite{AADT} for a discussion of such symmetries.}   
For another, it does not satisfy either the condition \eqref{eq:bcondit} or its variant~\eqref{eq:bcondit3}.
Indeed, both these conditions imply that $|b_\infty d_k - m_k| = O(d_k^{-1})$, while the sequences at $b = 1/3$ satisfy $|b_\infty d_k - m_k| = O(1)$.   Further, the ratios $m_k/d_k$  lie on both sides of $1/3$, while in the new staircases these ratios are monotonic.
Finally, as in the case of the Fibonacci staircase
in \eqref{eq:fib}, the numerators and denominators of the center points $a_k: = p_k/q_k$ of each of the three substaircases  fit into a single sequence $g_0,g_1,g_2\dots$ where $a_k = g_k/g_{k-1}$.
Thus  there are  three interwoven  sequences
$$
\bE_{k,i} = \bigl(d_{k,i}, m_{k,i}; g_{k-1,i} \bw(\frac{g_{k,i}}{g_{k-1,i}})\bigr),\quad i = 0,1,2
$$
where for each $i$ the numbers $g_{k,i}$ satisfy the recursion formula
$$
g_{k+1,i} = 6 g_{k,i} - g_{k-1,i}
$$
with seeds (i.e. initial values)
$$
g_{0,0} = 1,\;\;g_{1,0} = 2,\quad g_{0,1} = 1,\;\;g_{1,1} = 4,\quad g_{0,2} = 1,\;\;g_{1,2} = 5.
$$
Thus in all cases, the sequence of centers $g_{k.i}/g_{k-1.i}$ increases,\footnote
{
In fact, one can check that $g_{k+1,i}g_{k-1,i} = g^2_{k, i} + c_i$, where $c_i= 7$ for $i=0,1$ and $c_2 =4$.} 
and can be represented by the continued fractions $[5;\{1,4\}^{k+\eps},\eend_i]$, with $\eend_0 = 2,
\eend_1=(1,3)$ and $\eend_2 =\emptyset$, and $\eps = 0, \pm1$ as appropriate.
(Compare this with the staircase descriptions in Theorems~\ref{thm:L},~\ref{thm:U} and ~\ref{thm:E}.)
The numbers $d_{k,i}, m_{k,i}$ satisfy modified versions of this recursion, and are determined for each $i$ and $k$ by the requirement $3d_{k,i} -m_{k,i}= g_{k,i} + g_{k-1,i}$ and the following additional relation:
$$
d_{k,0} = 3m_{k,0} - (-1)^{k}, \quad d_{k,1} = 3m_{k,1} + (-1)^k, \quad d_{k,2} = 3m_{k,2} - (-2)^k.
$$
For example, the first sequence starts with the classes:
$$
\bE_{1,0} = \bigl(1,0;\bw(2)\bigr), \;\; \bE_{2,0} = \bigl(5,2;2\bw(\frac{11}{2})\bigr), \;\; \bE_{3,0} = \bigl(28,9;11\bw(\frac{64}{11})\bigr).
$$
Thus the linear relations in this staircase are inhomogeneous, and hence different in nature from those in the new staircases, which, as we will see in Theorem~\ref{thm:blockrecog}, are homogeneous and arise because the latter are associated to blocking classes.

The methods developed in \cite{casals-vianna} and  \cite{AADT} involve identifying staircases by constructing tight packings at the inner corners 
using almost toric fibrations (ATFs). It might be rather difficult to use ATF methods for the staircases identified in this manuscript, since the capacity function 
$c_{H_{b_\infty}}$ may well not have well defined inner corners because of the presence of an 
obstruction that nearly overshadows the staircase.\footnote
{
By contrast, Usher conjectures in \cite[Conj.4.23]{usher} that in the cases he considers, 
the capacity function is defined on $(a_\infty, a_1)$ by the union of  the relevant staircase  
classes, which would imply  that the capacity function has well delineated inner corners.}  
This possibility is illustrated in Fig.~\ref{fig:b0E} in the 
case of the descending staircase $\Ss^E_{u,0}$ in Theorems~\ref{thm:stair} and~\ref{thm:E}, which is almost overshadowed by the obstruction from $(3,1;2,1^{\times 5})$.  
Thus, in this case one cannot find tight packings of the fixed target $H_{b_\infty}$  by a sequence of ellipsoids 
$E(1,z_k)$ where $z_k$ increases to $a_\infty$.

Nevertheless, one might be able to use tight packings to show that
there is a decreasing  sequence of unobstructed points $a_k' = \acc(b_k')$  with $b_k'\to b_\infty$.
This would also be enough to show that $\Ss^E_{u,0}$ is  a descending staircase for $H_{b_\infty}$.  
Further, as is shown in Figs.~\ref{fig:515} and~\ref{fig:751}, the staircases  in the other families 
$\Ss^U_{\bullet,\bullet}, \Ss^L_{\bullet,\bullet}$ discussed in \S\ref{ss:Fib} below may not be so nearly obscured, since the problematic class $\bE = (3,1; 2,1^{\times 5})$ in Example~\ref{ex:big}  is not obstructive in the relevant region.
\hfill$\er$ \end{EXAMple}

\subsection{Blocking classes}\label{ss:block}

Let $\bE$ be a Diophantine class in the sense of Definition~\ref{def:perf}, and recall from Definition~\ref{def:nontriv} that an obstruction $\mu_{\bE,b}$  is said to block $b$ if $$
\mu_{\bE,b}(\acc(b)) > V_b(\acc(b)).
$$ 
Further, we call  $\bE = \bigl(d,m;\bbm\bigr)$  a {\bf blocking class} if there is some $b_0$ such that the corresponding obstruction $\mu_{\bE,b_0}$ blocks $b_0$.
We will show  in Remark~\ref{rmk:block0} that $b_0$ need not be $m/d$.
Instead, the next lemma shows that the most relevant $b$-value is $\acc^{-1}(a)$ where $a$ is the break point of $\mu_{\bE, b_0}$ and we choose an appropriate branch  of the inverse.
\MS

As preparation, recall from \S\ref{sec:intro} that the point  $\acc(b)$ is the unique solution  $>1$ of the equation
\begin{align}\label{eq:acc}
z^2 - \left(\frac{(3-b)^2}{1-b^2} - 2\right) z + 1 = 0.
\end{align}
As illustrated in Fig.~\ref{fig:101},  this function is at most two-to-one,  with  minimum  value
$ \acc(1/3)=3 + 2\sqrt2$.
Hence for any  $z\in [1,\infty)$, it has a well defined local inverse 
that may be calculated as follows.
If $z,  1/z$ are the two solutions to \eqref{eq:acc}, write
$
\ell: =  2 + z +1/z,
$
so that $\frac{(3-b)^2}{1-b^2} = \ell$.  Then we have 
\begin{align*}
b  = \frac{ 3 \pm \sqrt{\ell^2 -8\ell}}{\ell+1}, \quad \ell: =  z + \frac 1z + 2.
\end{align*}
With $\ell = \ell(z)$ as above, we will consider the following functions:
\begin{align}\label{eq:acc1} 
&\acc_L^{-1}: \left[3 + 2\sqrt 2, \frac {7+3\sqrt 5}2\right]
\to \left[0,\frac 13\right],\quad z\mapsto \frac{ 3 -\sqrt{\ell^2 -8\ell}}{\ell+1}\\\notag
& \acc_U^{-1}: \left[3 + 2\sqrt 2, \infty\right)\to \left[\frac 13,1\right),\quad z\mapsto \frac{ 3 + \sqrt{\ell^2 -8\ell}}{\ell+1},
 \end{align}
 writing $\acc^{-1}$ to denote one or other of these branches if there is no need to specify the branch any further.

 The following lemma will simplify some calculations below.
 
 \begin{lemm}\label{lem:accV}  Let $\bE = \bigl(d,m;q\bw(p/q)\bigr)$ be a quasi-perfect class,
   and suppose that $\acc (b) = p/q$. Then:
 \begin{itemize} 
 \item[{\rm (i)}]  
 $b$ is given by the formula
  \begin{align}\label{eq:bsi}
 b = \frac{3pq \pm (p+q)\sqrt\si}{p^2+q^2 + 3pq},\quad \mbox{ where } \si: = p^2 + q^2 - 6pq.  
 \end{align}
and we use the sign $+$, resp.\ $-$, if $b>1/3$, resp.\ $b<1/3$. 
 \item[{\rm (ii)}]  $ \mu_{\bE,b}(p/q) = \mu_{\bE,b}(\acc(b)) > V_b(\acc(b))$ if and only if 
 \begin{align}\label{eq:accV0}
& pq \bigl(3(p+q)^2\mp (p+q)\sqrt\si)\bigr) \\ \notag
&\qquad\quad >  (p+q)\Bigl(d(p^2+q^2 + 3pq) -m(3pq \pm (p+q)\sqrt\si)\Bigr)
 \end{align}
 where we use the top, resp.\ bottom, signs if  $b>1/3$, resp.\ $b<1/3$.
\end{itemize}
 \end{lemm}
 \begin{proof}   Since 
 $$
 \ell-8 = \frac pq + \frac qp - 6 = \frac{p^2 + q^2 - 6pq}{pq} = : \frac{\si}{pq},
 $$
we can obtain \eqref{eq:bsi} by substituting for $\ell$ in \eqref{eq:acc1}. 

 To prove (ii), recall from
 \eqref{eq:volacc}  that $ V_b(\acc(b)) = \frac{1 +\acc(b)}{3-b}$. Therefore, we have 
  $ \mu_{\bE,b}(\acc(b)) > V_b(\acc(b))$ if and only if $\frac p{d-mb} > \frac{1 +\acc(b)}{3-b}$.
 Now substitute for $b$ using \eqref{eq:bsi}.
 \end{proof}  
 
\begin{lemm}\label{lem:block0}  Let $\bE$ be a Diophantine class such that $\mu_{\bE,b_0}$ blocks  $b_0$, and let $I_{b_0}$ be the maximal interval containing 
	$\acc(b_0)$ on which $\mu_{\bE,b_0}> V_{b_0}$.  Then,
\smallskip

\NI {\rm (i)} 
If $a\in I_{b_0}$ is the break point of $\mu_{\bE,b_0}$ and $b_1 = \acc^{-1}(a)$ (where $\acc^{-1}$ denotes   the branch\footnote
{
this is well defined since $b_0$ cannot equal the unobstructed point $ 1/3$}
 of the inverse  whose image contains $b_0$), then
 $\mu_{\bE, b_1}$ blocks $b_1$.  
\smallskip

\NI {\rm (ii)}  There is an open interval $J_\bE$ containing $b_0, b_1$ such that
 \begin{itemize}\item[$\bullet$]   $\mu_{\bE,b}$ blocks $b$ for all $b\in J_{\bE}$;
 \item[$\bullet$]   $\mu_{\bE,b}(z) =V_b(z)$ for $z = \acc(b),\ b\in \p J_{\bE}$. 
 \end{itemize}
\end{lemm}
\begin{proof}  Since all functions involved are continuous, the set 
of $b$ such that $\mu_{\bE,b}$  blocks $b$ is open.  Let $J_\bE$ be the connected component of this set that contains $b_0$.  Then, for all $b\in J_\bE$, the interval $I_{b}$ on which  $\mu_{\bE,b}$ is obstructive contains a break point by Lemma~\ref{lem:0}, and because this point is the unique point of shortest length in $I_{b}$ and equals $a$ when $b=b_0$, this point must be independent of $b$ and hence equal to $a$.  Hence $a\in I_{b}$ for all $b\in J_{\bE}$.  To prove (i), we need to show that
$\acc^{-1}(a) \in J_{\bE}$.
To see this, consider a sequence of points $b_n\in J_\bE = (\be_\ell, \be_u)$ that converge to its right endpoint $\be_u$.  Then 
$\acc(b_n)\in I_{b_n}$for all $n$, and, as $n\to \infty$, 
$\acc(b_n)$ must converge to one of the endpoints of $I_{\be_u}$, the right endpoint if $b_n> 1/3$, and the left endpoint if $b_n <1/3$.  For the sake of clarity, let us suppose that $b_n>  1/3$ so that the functions $b\mapsto \acc(b)$ and $z\mapsto \acc_U^{-1}(z)$ preserve orientation.  (The other case is similar, and left to the reader.)  In this case, because $a\in I_{b}$ for all $b\in J_{\bE}$ we may conclude that 
$a < \acc(b_n)$  for all sufficiently large $n$.
Similarly, if $b_m'\in J_{\bE}$ converges to the left endpoint $\be_\ell$,  
$\acc(b_m')$ converges to the left endpoint of  $I_{\be_\ell}$, so that 
$a> \acc(b_m')$ for  sufficiently large $m$.  Therefore $\acc^{-1}(a) \in (b_m', b_n)\subset J_\bE$.  
This completes the proof of (i).

The first claim in (ii) follows immediately from the definition of $J_\bE$, while the second holds by the continuity of 
the  function $z\mapsto \mu_{\bE,b}(z)$ with respect to $z$ and $b$.
 \end{proof}

\begin{cor} \label{cor:block1} Let $\bE$ be a  Diophantine class such that, for some $b$, the obstruction $\mu_{\bE,b}$   is nontrivial on a maximal interval $I$ 
with break point $a> 3 + 2\sqrt 2$. Then $\bE$ is
a blocking class if and only if 
$\mu_{\bE, b_0}(a) > V_{b_0}(a)$, where $b_0$ is one of the solutions to $\acc(b_0) = a$.
\end{cor}

\begin{DEFN}\label{def:centerbl}  Let $\bE = \bigl(d,m; q\bw(p/q)\bigr)$ be a quasi-perfect class with center $a = p/q> 3 + 2\sqrt 2$ and with ${m}/d\ne 1/3$.  Then we say that $\bE$ is {\bf center-blocking} if 
$\mu_{\bE, b_0}(a) > V_{b_0}(a)$, where $b_0$ is  the solution to $\acc(b_0) = a$ that lies in the same component of $[0,1)\less \{1/3\}$ as does ${m}/d.$
\end{DEFN}

\begin{rmk}\label{rmk:block0} \normalfont
(i)  One might expect that a  class $\bE = (d,m; \bbm)$ would block  some $b$  only if it blocked $b = {m}/d$.  However, this is not true even if $\bE$ is perfect.  For example, by Theorem~\ref{thm:block}, the class $\bB^U_0 = (3,2;1^{\times 6})=(3,2;\bw(6))$ blocks the $b$-interval 
$\bigl(\frac{3-\sqrt 5}{2},\  \frac{3(7+\sqrt 5)}{44}\bigr) \approx (0.382, 0.629)$; and this interval does not 
contain $m/d = 2/3$, though it does contain $5/11 =\acc_U^{-1}(6)$.   
For further discussion, see \S\ref{sssec:E0}.

\MS

\NI (ii)  Here are two examples of obstructive exceptional classes that are not blocking classes.
\begin{itemize}
\item[$\bullet$]  The  class $\bE = (2,0; 1^{\times 5})$  is perfect with center of  length $5$. Further
every open $z$-interval $I$ that contains a point $a'$ with $\ell(a') = 5$  must contain numbers $z\le 5$.  It follows that $\bE$ cannot be a blocking class.  For if it were, there would be a $z$-interval $(\al_{\bE,\ell}, \al_{\bE,u})$ with 
lower end point $\al_{\bE,\ell} < 5$ and  in the image of the function $b\mapsto \acc(b)$.  But this is impossible.  

\item[$\bullet$] The class $\bE': = (5,1; 2^{\times 6},1)$ with break point $7$  gives an obstruction $\mu_{\bE',1/5}$ at $b=1/5$ that goes through the accumulation point $(6, 5/2)$.  However, if it were $b$-blocking there would have to be another  point $z> 7$ of the form $z= \acc(b)$ for some $b<1/3$ at which
$\mu_{\bE',b}(z) = V_b(z)$.  But we show in Example~\ref{ex:19th} above that such a point does not exist. 
\end{itemize}

\MS

\NI (iii)
The classes in the (ascending) staircase at $b=1/3$ (see Example~\ref{ex:13}) are certainly not center-blocking since their centers  are less than the accumulation point $ 3+2\sqrt 2$ and so lie outside the range of $b\mapsto \acc(b)$. 
 It is also very likely that they are not  blocking classes.  Indeed, if one of them, say $\bE_n$, were a blocking class with blocked  $b$-interval $J_b$ and blocked $z$-interval $I_b$, then as in (ii) above
the center of $\bE_n$ could not lie in $I_b$.
However, as is shown in Fig.~\ref{fig:center17029}, it is possible for a perfect class to be obstructive in an interval that does not contain its center point.
\hfill$\er$
\end{rmk}

\begin{lemm}\label{lem:Fib}
All the classes (except for the first)  in the Fibonacci staircase  of \eqref{eq:fib}  are 
center-blocking.
\end{lemm}
\begin{proof}   These classes are $\bigl(g_k,0; g_{k-1}\bw(g_{k+1}/g_{k-1})\bigr)$, 
where $g_0=1,g_1=1,g_2=2,g_3=5,\dots.$ Therefore, by well-known identities for Fibonacci numbers
(see \cite[\S3.1]{ball}), we have
  $d = (p+q)/3, m=0$, $d^2 = pq -1$ and $p^2 + q^2 = 7pq -9$.
By \eqref{eq:accV0} we need
$$
3pq(3(p+q)^2 + (p+q)\sqrt \si) > (p+q)^2(p^2 + q^2 + 3pq) = (p+q)^2(10pq -9).
$$
Since $\si = pq-9$, this  holds exactly if
$$
3pq\sqrt\si > (p+q)(pq -9) = 3d \si.
$$
Thus we need $pq > d\sqrt{pq-9}$.  But this holds because $d^2 = pq-1$.
\end{proof}

\MS

\begin{prop}\label{prop:block}  Suppose that $\bE = \bigl(d,m;q\bw(p/q)\bigr)$ is a   center-blocking class with center $a = p/q$ that blocks the $b$-interval $J_\bE = (\be_\ell,\be_u)\subset [0,1/3)\cup (1/3,1)$.

\begin{itemize}\item[{\rm (i)}] if $J_\bE = (\be_\ell,\be_u)\subset (1/3,1)$  then
$H_{\be_u}$ admits no ascending staircase and
$H_{\be_\ell}$ admits no descending staircase.
If in addition  $H_b$ is unobstructed for $b\in \p J_{\bE}$, then 
$\mu_{\bE,b_u}$ is live on $[a, \al_u)$ and $\mu_{\bE,b_\ell}$ is live on $(\al_\ell,a]$.

\item[{\rm (ii)}]  if $J_\bE = (\be_\ell,\be_u)\subset (0, 1/3)$, then
$H_{\be_\ell}$ admits no ascending staircase and
$H_{\be_u}$ admits no descending staircase.
If in addition  $H_b$ is unobstructed for $b\in \p J_{\bE}$, then 
$\mu_{\bE,b_\ell}$ is live on $[a, \al_u)$ and $\mu_{\bE,b_u}$ is live on $(\al_\ell,a]$.

\item[{\rm (iii)}] If  $H_b$ is unobstructed for $b\in \p J_{\bE}$,
then $\mu_{\bE,b}$ is live at  $a$ for all $b\in J_\bE$ and 
$\bE$ is perfect. 
Moreover, $J_{\bB}$ is a connected component of {\it Block}.
\end{itemize}
\end{prop}  

\begin{proof}  In case (i)   the map $b\mapsto \acc(b)$ preserves orientation.  
 Consider the upper endpoint $ \be_u$ of $J_\bE$.  By Lemma~\ref{lem:munearc}, the function $ \mu_{\bE, \be_u}(z)$ is constant for $z\in [a, \al_u: = \acc(\be_u))$.  Also we have
    $ \mu_{\bE, \be_u}(\al_u) = V_{\be_u}(\al_u)$ 
by Lemma~\ref{lem:block0}.   If ${H_{\be_u}}$ did have an ascending  staircase, we saw in \eqref{eq:cHb} that 
we would have to have $c_{H_{\be_u}} = V_{\be_u}(\al_u)$, i.e. ${H_{\be_u}}$  would be unobstructed.  But then, because 
$c_{H_{\be_u}}$ is nondecreasing, $c_{H_{\be_u}}$ would be constant and equal to $\mu_{\bE, \be_u}$ on
the interval $[a, \al_u]$.  Thus there can be no ascending staircase.  (Indeed in this case $\bE$ is 
left-overshadowing in the sense of Definition~\ref{def:big}.)  This proves the claims in (i) that pertain to the end point $\be_u$.

Now consider the lower endpoint $\be_\ell$.  By Lemma~\ref{lem:munearc}
the graph of
$ \mu_{\bE, \be_\ell}(z) = \frac{qz}{d-m \be_\ell}$ for $z\in [\al_\ell,a]$  is a line through the origin that passes through the point $\bigl(\al_\ell, V_{b_\ell}(\al_\ell)\bigr)$.  If $H_{\be_\ell}$ did have a staircase, then again we would have
$V_{b_\ell}(\al_\ell) = c_{H_{b_\ell}}(\al_\ell)$,  i.e. ${H_{\be_\ell}}$  would be unobstructed.  In this case the 
the scaling property \eqref{eq:scale} of capacity functions implies that the graph of $c_{H_{\be_\ell}}$, which as we just saw  goes through $(\acc(\be_\ell), V_{\be_\ell})$, cannot lie above the line $z\mapsto \mu_{\bE, \be_\ell}(z) = \frac{qz}{d-m\be_\ell}$ for $z> \acc(\be_\ell)$. 
Thus we again conclude that $ \mu_{\bE, \be_\ell}$  is live for $z\in (\al_\ell,a]$ (and so is 
right-overshadowing), and that there cannot be a decreasing staircase for $H_{\be_\ell}$.
This completes the proof of (i).

The proof of (ii) is similar, and is left to the reader.
\MS

Now consider (iii). For clarity we will again suppose that $J_\bB\subset (1/3,1)$, leaving the other case to the  reader.   
First note that if $H_b$ is unobstructed then $b\notin {\it Block}$, so that $J_{\bB}$ is a connected component of ${\it Block}$ as claimed.  

We next show that  $\mu_{\bE,b}$  is live at $a$ for all $b\in J_\bE$.   To see this, consider the set  of $b \in [\be_\ell,\be_u]$  such that $\mu_{\bE,b'}$  is live at $a$ for all $b'\in [\be_\ell,b)$.  This is a closed, connected 
 subset of $ [\be_\ell,\be_u]$.    If it is empty, define $b_{\max}: = \be_\ell$, and if it is 
 nonempty but proper, let $b_{\max}< \be_u$ be its maximal element.  Then, 
 by Lemma~\ref{lem:perfect}~(ii), 
  there must be  an exceptional class $\bE' = (d',m_{0}';\bbm')$ and $\eps>0$ such that 
 $$
 \mu_{\bE', b_{\max}} (a) = \frac {\bbm'\cdot\bw(a)}{d'-m_{0}' b_{\max}}  = 
\frac {p}{d -m b_{\max}} =  \mu_{\bE, b_{\max}} (a)
$$
and $ \mu_{\bE_{\max}, b} (a)> \mu_{\bE, b} (a)$ for  $b\in (b_{max}, b_{max}+\eps)$.  
In particular,
\begin{align*} 
\frac{\p}{\p b}\big|_{b=b_{\max}} \mu_{\bE', b} (a)& =
\mu_{\bE', b_{\max}} (a)\  \frac{m'}{d' -m_{0}' b_{\max}}.
\\ & \ge \ \frac{\p}{\p b}\big|_{b=b_{\max}} \mu_{\bE, b} (a)\\
& = \ \mu_{\bE, b_{\max}} (a)\  \frac{m}{d -m b_{\max}}.
\end{align*}
But this implies that $m'\ d \ge m\ d'$.  Thus, because this inequality is independent of
$b_{\max}$, 
it can hold at $b = b_{\max}$ only if it holds for all $b \in [b_{\max}, \be_u]$.  
Moreover, we cannot have  $m'\ d = m\ d'$  since this would imply that
 the two obstruction functions are equal for $b\in [b_{max}, b_{max}+\eps)$.  Hence 
 we must have  $m'\ d > m\ d'$, in which case 
$\mu_{\bE, b}(a)$ could not be live at $b=\be_u$. Therefore  this scenario does not happen, and so $b_{\max} = \be_u$.    This completes the proof that $\mu_{\bE,b}(a)$ is live at $a$ for all $b\in J_\bB$.

 \MS
 
It remains to show that $\bE$ is perfect. By Lemma~\ref{lem:perfect}~(ii), 
there is an exceptional class $\bE' = (d', m';\bbm')$
and an open subset  $J_b$ of $J_\bE$  such that 
\begin{align*}
\mu_{\bE',b}(a) = \mu_{\bE,b}(a), \quad  \forall\ b\in J_b. 
\end{align*}
Then we may write
  $\bbm' = \la \bbm + \bn$ where $\bbm \cdot  \bn=0$.
Since
$$
\frac{\bbm' \cdot \bw(a)}{d' - m'b} = 
\frac{\la p}{d' - m'b}  =
\frac{p}{d - mb} 
 $$
 for all $b\in J_b$, we must have  $d'=\la d,\ m'=\la m$ for some $\la > 0$.
 The identities
\begin{align*}
& \bbm'\cdot \bbm' -1= d'^2 - (m')^2 = \la^2 (d^2-m^2) = \la^2 (pq - 1),\\
& \bbm'\cdot \bbm'  -1 = \la^2 \bbm\cdot\bbm + \|\bn\|^2 - 1 = \la^2 pq + \|\bn\|^2 - 1
\end{align*}
then imply  that $ \|\bn\|^2=1-\la^2$.  Therefore, unless $\bE' = \bE$ we must have  $0<\la < 1$.
Further 
$$
\bE'\cdot \bE = d'd - m' m - \bbm'\cdot \bbm = \la (d^2-m^2 -\bbm\cdot \bbm) =  - \la.
$$
But $\bE'\cdot \bE$ is  an integer.    It follows that $\bE' = \bE$, so that $\bE$ is perfect as claimed.
\end{proof}

\begin{rmk}\label{rmk:block} \normalfont
(i)  
Notice that there may be no $b\in J_\bE$ such that $\mu_{\bE,b}$ 
is live on the whole of the $z$-interval on which it is obstructive,
 since for each such $b$ there may be classes with break points $a'$ outside this interval (and hence with $\ell(a')\le \ell(a)$) that are live near one end or other  of this $z$-interval.  
For further discussion of this point, see  \S\ref{sssec:E0} and the associated figures.
\MS

\NI (ii)   We will see in Proposition~\ref{prop:stairblock} that many of the classes that contribute to a staircase are themselves blocking classes.   
   Fig.~\ref{fig:center17029} shows that the class $\bE'$ with center $170/29$, which is the $k=0$ step with $\operatorname{end}_0=4$ of the staircase  $\Ss^E_{u,0}$, obstructs at  $\acc^{-1}({170}/{29})$.  Hence it is a blocking class by Corollary~\ref{cor:block1}. 
     \hfill$\er$
\end{rmk}

The next lemma  is the key to the proofs of our results about the relation between blocking classes and staircases.

\begin{lemm}\label{lem:block2}   Let $ \bB = \bigl(d,m;q\bw(p/q)\bigr)$ be a 
quasi-perfect class with $m/d \ne 1/3$  and define
$\acc^{-1}$ to be the branch of the inverse whose image contains $m/d$. 
\begin{itemize}\item[{\rm (i)}]   If 
 $ \bB = \bigl(d,m;q\bw(p/q)\bigr)$ is center-blocking with $J_ \bB = (\be_{ \bB,\ell},\be_{ \bB,u})$,  and $I_\bB = (\al_{\bB,\ell}, \al_{\bB,u}) = \acc(J_\bB)$, then 
\begin{align}\label{eq:alellu1}
& \acc^{-1}(\al_{\bB,\ell}) = \frac{(1+\al_{\bB,\ell})d- 3q \al_{\bB,\ell}}{(1+\al_{\bB,\ell})m-q\al_{\bB,\ell}},
\\ \notag
& \acc^{-1}(\al_{\bB,u}): = \frac{(1+\al_{\bB,u})d - 3p}{(1+\al_{\bB,u})m -p}.
\end{align}
\item[{\rm (ii)}]  Conversely, suppose given numbers $z_\ell < p/q < z_u$ such that
$\ell(p/q) < \ell(z)$ for all 
$z\in (z_\ell,z_u)\less \{p/q\}$ and
\begin{align}\label{eq:alellu2}
& \acc^{-1}(z_\ell) = \frac{(1+z_\ell)d- 3q z_\ell}{(1+z_\ell)m-qz_\ell},
\\ \notag
& \acc^{-1}(z_u): = \frac{(1+z_u)d - 3p}{(1+z_u)m -p}.
\end{align}
Then $\bB$ is center-blocking, and we have $I_\bB = (z_\ell, z_u)$.
\end{itemize}
\end{lemm}

\begin{proof}  
Let $c(b): = \frac{(3-b)^2}{1-b^2} - 2$.  Then, because 
$\acc(b)> 0$ and $0\le b<1$, we have 
\begin{align} \label{eq:MDvol}\notag
\frac {1+\acc(b)}{3-b}= \sqrt{\frac {\acc(b)}{1-b^2}} \;\;& \Longleftrightarrow\;\; (1+\acc(b))^2 = \acc(b)\frac{(3-b)^2}{1-b^2} = \acc(b)\bigl(c(b) + 2\bigr),\\
&\Longleftrightarrow\;\;  \acc(b)^2 - c(b) \acc(b) + 1 = 0,
\end{align}
which holds by the definition of $\acc(b)$ in \eqref{eqn:accb}.  
Therefore the function $$
b\mapsto V_b(\acc(b)) = \sqrt{\frac {\acc(b)}{1-b^2}}
$$
 is also given by the formula
$
b\mapsto \frac {1+\acc(b)}{3-b}.
$  

Since $\bB$ is a quasi-perfect class, Lemma~\ref{lem:munearc} shows that the obstruction function $z\mapsto \mu_{\bB,b}(z)$ on $I_\bB$ is given by the formulas $z\mapsto \frac{q z}{d-mb}$ for $z<a$ and 
 $z\mapsto \frac{p}{d-mb}$ for $z> a$. 
 Thus because
 $V_b(\al_{\bB,\ell}) =  \mu_{\bB,b}(\al_{\bB,\ell})$, the point  
 $z=\al_{\bB,\ell}$ satisfies
 $$
 \frac {1+\al_{\bB,\ell}}{3-b} = \frac{q \al_{\bB,\ell}}{d-mb}, \quad \mbox{where }\ b: = \acc^{-1}(\al_{\bB,\ell}).
 $$  
 Similarly, the point  
 $z=\al_{\bB,u}$ satisfies
 the equation $$
 \frac {1+\al_{\bB,u}}{3-b} = \frac{p}{d-mb}, \quad \mbox{where }\ b: = \acc^{-1}(\al_{\bB,u}).
 $$
Now rearrange these identities to obtain  \eqref{eq:alellu1}.
This proves (i).

The identities in~\eqref{eq:MDvol} show that $z = \acc(b)$ if and only if $(1+z)/(3-b) = V_b(z)$.
Because the length $\ell(p/q)$ of the point $p/q$ is minimal  among all points in $(z_\ell,z_u)$, Lemma~\ref{lem:munearc} implies that
$$
\mu_{\bB,b}(z_\ell) =   \frac{qz_\ell}{d-mb},\quad \mu_{\bB,b}(z_u)  = \frac{p}{d-mb}.
$$
We saw in the proof of (i) above that
$$
b=  \frac{(1+ z_\ell)d- 3q z_\ell}{(1+ z_\ell)m-q z_\ell} \Longleftrightarrow 
 \frac{1+z_\ell}{3-b} = \frac{qz_\ell}{d-mb}.
 $$ 
Hence if $b = \acc^{-1}(z_\ell)$ is given by the formula  
$b=  \frac{(1+ z_\ell)d- 3q z_\ell}{(1+ z_\ell)m-q z_\ell}$, then 
$$
V_b(z_\ell) = \frac{1+z_\ell}{3-b} = \mu_{\bB,b}(z_\ell), \quad b: =  \acc^{-1}(z_\ell).
$$
A similar argument with $z_u$ shows that our hypothesis implies
$$
V_b(z_u) = \frac{1+z_u}{3-b} = \mu_{\bB,b}(z_u),  \quad b: =  \acc^{-1}(z_u)
$$
Hence the constraint defined by the class $\bB$ equals the volume obstruction at the two pairs
$(z,b) = (z_\ell,\acc^{-1}(z_\ell)),$   and $(z,b) = (z_u,\acc^{-1}(z_u))$ where $z_\ell < z_u$.  It follows that $\bB$ is a center blocking class that blocks the $z$-interval $(z_\ell,z_u)$, as claimed.  
\end{proof}

\begin{cor}\label{cor:block2}
If $\Ss$ is a staircase in $H_b$  that accumulates at  a point $a_\infty$ that is an endpoint of
the blocked $z$-interval $I_\bB$  for some quasi-perfect blocking class $\bB$, then the two numbers  
$a_\infty, b$ can be expressed in terms of the same quadratic surd, i.e. there is $\si\in \N$ such that $a_\infty, b\in \Q + \Q\sqrt \si$.
\end{cor}
\begin{proof}  This is an immediate consequence of Lemma~\ref{lem:block2}~(i).
\end{proof}

\subsection{Pre-staircases and blocking classes}\label{ss:prestair}

In this section we begin by discussing the structure of the staircases that we encounter, and then relate their properties to those of associated blocking classes. Our aim is to clarify exactly what we need to prove in order to show that a particular sequence of Diophantine classes $(\bE_k)$ does  form a staircase.  

\begin{DEFN}\label{def:prestair}  We will say that  a 
sequence $\Ss = \bigl(\bE_k\bigr)_{k\ge 0}$ of quasi-perfect classes 
$\bE_k: = \bigl(d_{k},m_k;q_{k}\mathbf{w}\left({p_{k}}/{q_{k}}\right)\bigr)$
is a {\bf pre-staircase} if it has the 
 following properties:  
\begin{itemize}\item[$\bullet$]  {\bf (Recursion)} There is an integer $\si \ge 0$  such that
$\si + 4$ is a perfect square, and
each of the sequences $x_k: =  d_k, m_k,p_k,q_k$ satisfies the recursion
\begin{align}\label{eq:recur0}
x_{k+1} = (\si + 2) x_k - x_{k-1}  \quad \mbox{ for all } k\ge 0,
\end{align}
\item[$\bullet$]   {\bf (Relation)} there are  integers $R_0,R_1,R_2$ such that the following linear relation holds
\begin{align}\label{eq:linrel0}
R_0 d_k = R_1 p_k +R_2 q_k  \quad \mbox{ for all } k\ge 0,
\end{align}
\end{itemize}
Moreover, if the classes $\bE_k$ are perfect for all  $k$, then we say that $\Ss$ is a {\bf perfect pre-staircase.} 

\end{DEFN} 
In this situation, the whole sequence of classes is determined by the first two centers
$p_0/q_0,\ p_1/q_1$, since the other centers are then determined  by  the recursion, the $d_k$ are determined by the linear relation \eqref{eq:linrel0}, and then the $m_k$ are determined by the linear Diophantine identity $3d_k = m_k +p_k+q_k$.  Of course, for arbitrary choices of initial data, there is no guarantee that $d_k$, if so defined, is a positive integer, or that the quadratic Diophantine identity holds.

The following lemma explains the importance of the (Recursion) condition.

\begin{lemm} \label{lem:recur}   Let $x_{k}, k\ge 0,$ be a sequence of integers that satisfy the recursion 
$x_{k+1} = (\si + 2)x_{k} - x_{k-1}$, where $\si+4$ is a perfect square, and let $\la\in \Q[\sqrt{\si}]$ be the larger root of the equation
$x^2 - (\si+2)x + 1=0$.  Then there is a number  $X \in \Q[\sqrt{\si}]$  such that
\begin{align}\label{eq:recur}
x_k  = X\la^k + \ov X\, {\ov\la}\,\!^k,
\end{align}
where $\ov{a + b\sqrt\si}: = (a - b\sqrt\si)$, so that $\la \ov\la = 1$.  
\end{lemm}  
 \begin{proof}  
 If the monomials $x_k= c^k$ satisfy the recursion then we must have
$c^2-(\si+2) c+1=0$, so that
  $c  =  \bigl((\si+2) \pm N \sqrt{\si}\bigr)/2$, where $N^2 = \si+4$.      Let $\la$ be the larger solution, so that $\ov \la$ is the smaller one, and we have $\la \ov\la = 1$.
 Since  \eqref{eq:recur}  has a unique solution once given the seeds $x_0, x_1$, 
  it follows that  for each choice of constants $A,B$, the numbers 
  $$
  x_k: = A \la^k +  B\ov\la^k
  $$ 
  form the unique solution with
  $$
  x_0 = A+B,\qquad  x_1 = A \la + B \ov \la.
  $$
 Then $A,B\in  \Q[\sqrt{\si}]$,
 and it is easy to check that $x_0, x_1\in \Q$ only if 
we also have $B: = \ov A$.  This completes the proof.
\end{proof}

\begin{rmk}\label{rmk:recur}  \normalfont
(i)  All the pre-staircases that we consider have $\si = (2n+1)(2n+5)$ for some $n\ge 0$ so that $\si+4 = (2n+3)^2,$ a quantity whose square root happens to equal the constant $R_0$ in (Relation).   Note that, if $X = X' + X''\sqrt\si$ is as in Lemma~\ref{lem:recur} and the initial values are  $x_0, x_1$, we have
\begin{align}\label{eq:recurX}
\la = \frac{\si+2+(2n+3)\sqrt \si}{2},\quad X' = \frac{x_0}{2},\quad X'' = \frac{2x_1 - x_0(\si+2)}{2(2n+3)\si}.
\end{align}
\smallskip

\NI (ii)  It turns out that, as in Lemma~\ref{lem:SsL}~(i) below,  some of our staircases  can be extended by a \lq class' $\bigl(d_{-1}, m_{-1}, q_{-1}\bw(p_{-1}/q_{-1})\bigr)$ defined for $k = -1$ such that the (Recursion) \eqref{eq:recur} and (Relation) \eqref{eq:linrel0} conditions hold for all $k\ge -1$.  The word \lq class' above is in quotes because in some cases the numbers $d_{-1}, m_{-1}$ are negative (though $p_{-1},q_{-1}$ are positive), so that the tuples  have no geometric meaning.  However,  they can still be used for computational purposes.  For example a quantity such as $X$ in \eqref{eq:recur} can be computed from knowledge of the terms $x_{-1},x_0$ instead of from $x_0, x_1$.  Note in particular that if $m_k, d_k$ both satisfy the recursion \eqref{eq:recur0} for $k\ge -1$, then
\begin{align*}
m_0(d_1 + d_{-1}) = m_0 d_0 (\si+2) =  d_0(m_{-1} + m_1),
\end{align*}
which implies that
\begin{align}\label{eq:md00}
m_0d_1 - m_1d_0 & = m_{-1}d_0 - m_0d_{-1}.
\end{align}
This fact will simplify some calculations below.
\hfill$\er$
\end{rmk}

The following result  shows that at least the tail end of a pre-staircase consists of center-blocking classes.

\begin{prop}\label{prop:stairblock}
Suppose that $\Ss = (\bE_k)$ is a pre-staircase as above, let $\la$ be as in Lemma~\ref{lem:recur}, and denote by
 $D,M,P,Q$ the constants $X$ defined by \eqref{eq:recur}, where $x_k = d_k, m_{k}, p_k, q_k$ respectively.
Suppose that $M/D\ne 1/3$ and that the centers $p_k/q_k$ of $\Ss$ are all $> 3+2\sqrt2$.
Then 
\begin{itemize}\item[$\bullet$]   $P/Q = \lim p_k/q_k, \ M/D = \lim m_k/d_k$ and
\begin{align}\label{eq:acctpt} \acc\left(\frac MD\right) = \frac PQ. 
\end{align}
\item[$\bullet$]  $\bE_k$ is a center-blocking class for sufficiently large $k$, and 
\item[$\bullet$]    If in addition $\Ss$ is a perfect pre-staircase then  $b_\infty: =M/D$ is unobstructed.  
 \end{itemize}
\end{prop}
\begin{proof}  The identities 
$$\frac PQ = \lim \frac{p_k}{q_k}, \qquad \frac MD = \lim \frac{m_k}{d_k}$$
follow immediately from Lemma~\ref{lem:recur}.   Since $3d_k = m_k + p_k + q_k$ and  
$d_k^2 - m_k^2 = p_kq_k -1$ we have
$3D = M+P+Q$ and  $D^2 - M^2 = PQ$.  (Note that the $1$ in the second identity disappears in the limit, after we divide by $\la^{2k}$.)
By \eqref{eq:MDvol}, we have $\acc(M/D) = P/Q$ exactly if
$$\frac{ 1 + \frac PQ}{3-\frac MD}  = \sqrt{\frac{\frac PQ}{1-(\frac MD)^2}}$$
or equivalently
$$\frac{P+Q}{\sqrt {PQ}}  = \frac{3D-M}{\sqrt{D^2-M^2}}.$$
But we saw above that the top and bottom entries on both sides are equal.

To prove that $\bE_k$ is  center-blocking,  Corollary~\ref{cor:block1} shows that  it suffices to check that
 $\mu_{E_k,b_k}({p_k}/{q_k})$  is nontrivial, where
 $b_k: = \acc^{-1}({p_k}/{q_k})$ and we choose the inverse so that $M/D$ is in its range.   
 Hence by  Lemma~\ref{lem:perfect} we must check that $|d_kb_k - m_k|< \sqrt{1-b_k^2}$ for sufficiently large $k$.  
 
 To see this, 
suppose that $ M/D < 1/3$.  (For the other case, we simply choose $+$ instead of $-$ in \eqref{eq:acc1}.) Then
\eqref{eq:acc1} implies that
$$ b_k = \frac{3 - \sqrt{\ell_k^2 - 8\ell_k} }{\ell_k + 1},\quad \mbox{ where }\;\;  \ell_k = \frac{p_k}{q_k} +     \frac{q_k}{p_k} + 2. $$
Since  $\acc(M/D) = P/Q$, we know that
$$ \frac MD = \frac{3 - \sqrt{L^2 - 8L} }{L + 1}, \quad L: = \frac PQ +     \frac QP + 2. $$
But our assumptions imply that $\ell_k = L + O(\la^{-2k}).$
Hence 
$$ b_k = \frac MD+ O(\la^{-2k}). $$
Therefore 
\begin{align*}
|d_kb_k - m_k| & = 
|(D\la^k + O(\la^{-k})) (\frac MD+ O(\la^{-2k})) - (M + O(\la^{-2k}))\la^k | = O(\la^{-k}).
\end{align*}
Therefore, because $1-b_k^2 = 1- (M/D)^2 - O(\la^{-2k}) > 0$  for large $k$, the required inequality 
$|d_kb_k - m_k|< \sqrt{1-b_k^2}$ holds for sufficiently large $k$.  
This proves the second point.
\MS

Finally, if the classes $\bE_k$ are perfect then $M/D$ is unobstructed by Lemma~\ref{lem:unobstr}.
\end{proof}

\begin{rmk}\label{rmk:fib} \normalfont
(i)  
 Proposition~\ref{prop:stairblock} does not quite give an independent proof that any staircase in the manifold $H_b$ must accumulate at $\acc(b)$, since it is not true that a staircase must be a pre-staircase.  For example, we saw in Example~\ref{ex:13} that the staircase in $H_{1/3}$ is not a pre-staircase since the classes do not satisfy a homogeneous linear relation.  Further,  a staircase need not be given by quasi-perfect classes, though no such examples are known. 
 \MS

 \NI(ii) It is very likely that, as in Lemma~\ref{lem:Fib},   all the classes $\bE_k$ in the staircases defined in \S\ref{ss:Fib} are center-blocking, provided that their centers are in the range of the function $b\mapsto \acc(b)$.    One could  prove this in any particular case by being more careful with the estimates in the above lemma.
\hfill$\er$
\end{rmk}

Here is our most powerful staircase recognition criterion.
\MS

  \begin{thm}\label{thm:stairrecog}  Let $\Ss = (\bE_k)$ be a perfect pre-staircase with constants $P,Q,D,M$ as in Proposition~\ref{prop:stairblock}.
     Suppose in addition 
that at least one of the following conditions holds:
\MS

 \NI {\rm (i)} There is $r/s>0$ such that $ M/D<r/s$,
 $$
  \frac{m_k^2-1}{d_km_k} <  \frac MD<  \frac{s + m_k(rd_k-sm_k)}{r+d_k(rd_k-sm_k)},\qquad \forall \ 
k\ge k_0,
$$
and there is no 
overshadowing class 
at $(z,b_\infty)=(P/Q,M/D)$ of degree $d'< s/(r-sb_\infty)$ and with $m'/d' >r/s$. 
\MS

 \NI {\rm (ii)}  There is $r/s>0$ such that $M/D >r/s$,
  $$
   \frac{m_k(sm_k - rd_k)-s}{d_k(sm_k-rd_k) - r} <  \frac MD< \frac{m_k}{d_k} \qquad \forall \ 
k\ge k_0,
 $$ 
and there is no 
overshadowing class 
at $(z,b_\infty)=(P/Q,M/D)$ of degree $d'< s/(sb_\infty - r)$ and with $m'/d' <r/s$. 

Then  $\Ss$ is a staircase for $H_{M/D}$ that accumulates at $P/Q$. 
\end{thm}  
\begin{proof}  Proposition~\ref{prop:stairblock} implies  that $\acc(M/D) = P/Q$ and that $H_{M/D}$ is unobstructed.   Further $m_k/d_k \to M/D$ and $p_k/q_k\to P/Q$ by Lemma~\ref{lem:recur}.  Hence it remains to show that the staircase is live.   If (i) holds this follows from Proposition~\ref{prop:stair00}, while if (ii) holds we argue as in Remark~\ref{rmk:stair00}.
 \end{proof}

So far, the linear relation satisfied by a pre-staircase has played no role in our analysis except  to specify the entries $d_k,m_k$ of the classes $\bE_k$. 
Our second main result in this section shows the relevance of this relation, using it to
obtain the following blocking class recognition criterion.

\begin{thm}\label{thm:blockrecog}    Let $\bB = \bigl(d,m,q\bw(p/q)\bigr)$ be a  
quasi-perfect class 
with $d\ne 3m$ 
and suppose given two perfect pre-staircases $\Ss_\ell, \Ss_u$ that satisfy the following conditions:
\begin{itemize} \item[{\rm(i)}] If the constants for $\Ss_\ell$ (resp.\ $\Ss_u$) are denoted
$D_\ell, M_\ell, P_\ell, Q_\ell$ (resp.\ $D_u, M_u, P_u, Q_u$), then
$$
\frac{M_\ell}{D_\ell} \ne \frac 13 \ne \frac {M_u}{D_u},\qquad 
3+2\sqrt2 < \frac{P_\ell}{Q_\ell} < \frac pq < \frac {P_u}{Q_u}.
$$
Moreover $\ell(p/q)< \ell(z)$ for all $z\ne p/q$ in $\bigl({P_\ell}/{Q_\ell} ,  {P_u}/{Q_u}\bigr)$.
\item[{\rm(ii)}]  $\Ss_\ell$ is  ascending  with linear relation 
$R_0 d_k = R_1 p_k + R_2 q_k$ where  \newline  $R_0 = d-3m,\ R_1=q-m, \ R_2 = -m$. 
\item[{\rm(iii)}]  $\Ss_u$ is  descending  with linear relation   $R_0 = d-3m,\ R_1=-m, \ R_2 =p-m$. 
\end{itemize}
Then 
 $\bB$ is a perfect blocking class 
that  blocks the $z$-interval
$I_{\bB} = (\al_\ell, \al_u)$ where
\begin{align}\label{eq:alPQ}
\al_\ell = \frac{P_\ell}{Q_\ell},\ \acc^{-1}(\al_\ell) =  \frac{M_\ell}{D_\ell}, \quad 
\al_u = \frac{P_u}{Q_u},\ \acc^{-1}(\al_u) =  \frac{M_u}{D_u}
\end{align}
\end{thm}
\begin{proof}   Since the hypotheses of Proposition~\ref{prop:stairblock} hold, we know that
$$
\acc( \frac{M_\ell}{D_\ell}) =  \frac{P_\ell}{Q_\ell},\quad \acc(\frac{M_u}{D_u}) = \frac{P_u}{Q_u}.
$$
Moreover, because we assumed that the pre-staircases are perfect, both ${M_\ell}/{D_\ell}$ and ${M_u}/{D_u}$ are unobstructed.

By Lemma~\ref{lem:block2}~(ii), to see that $\bB$ is a class that blocks the given interval $I_\bB$ 
it suffices to check that the equations  in \eqref{eq:alellu2} 
hold when
 $d,p, m$ are defined by $\bB$ and  with $z_\ell = P_\ell/Q_\ell,\ z_u = P_u/Q_u$.  Thus we must check that
\begin{align*}
\frac{M_\ell}{D_\ell}& =\frac{3q P_\ell-(Q_\ell + P_\ell )d}{qP_\ell -(Q_\ell +P_\ell )m }  = \frac{(3q-d)P_\ell -dQ_\ell }{ (q-m)P_\ell-mQ_\ell}
 \;\;\mbox{ and }\\
\frac{M_u}{D_u}&
 = \frac{3pQ_u - (Q_u + P_u )d}{pQ_u -(Q_u +P_u )m} = \frac{(3p-d)Q_u - dP_u}{(p-m)Q_u - mP_u}
\end{align*}
Because $$
(q-m)P_\ell-mQ_\ell  = R_1P_\ell +R_2Q_\ell = (d-3m)D_\ell,
$$ 
the first identity will hold if 
 $  (d-3m)M_\ell = (3q-d)P_\ell -dQ_\ell.$
But  because $M_\ell = 3D_\ell-P_\ell-Q_\ell$ we have
\begin{align*}
  (d-3m)M_\ell & =  3(d-3m) D_\ell-(d-3m) P_\ell-(d-3m) Q_\ell\\
  & = 3\bigl((q-m)P_\ell-mQ_\ell) -(d-3m) P_\ell-(d-3m) Q_\ell\\
  & =(3q-d)P_\ell - d Q_\ell
  \end{align*}
  as required.  
The proof of the second identity is similar; notice that again the coefficients in the denominator equal those in the linear relation for $\Ss_u$.  

This shows that $\bB$ is a blocking class that blocks the interval $(\al_\ell,\al_u)$. 
Since, as we saw in the first paragraph of this proof, 
 $H_{\be_\ell}, H_{\be _u}$ are unobstructed, the class $\bB$
is perfect by Proposition~\ref{prop:block}~(iii).
\end{proof}

\section{The Fibonacci stairs, its cognates, and beyond}\label{sec:manystairs}
The staircases $\Ss^E_n$ in Theorem~\ref{thm:stair} were found by trial and error using methods described in \S\ref{sec:Mathem} and Remark~\ref{rmk:ECH}.
They occur for values of $b\in \bigl(1/5,1/3\bigr)$.
Once we began  looking for staircases in other ranges of $b$, armed  with the numerical knowledge of these stairs, nascent ideas about the importance of blocking classes,   as well as the visual and computational  tools explained in \S\ref{sec:Mathem}, we found many other examples.   
In \S\ref{ss:Fib} we describe three important sets of blocking classes, together with their associated staircases, and explain their relation to the staircases 
in Theorem~\ref{thm:stair}. 
 In \S\ref{ss:thmU}, we explain  our proof strategy. 

We assume that the reader understands the definition of (quasi-)perfect classes (Definition~\ref{def:perf}), 
center-blocking class (Definition~\ref{def:centerbl}) and pre-staircase (Definition~\ref{def:prestair}).  The most important results are Proposition~\ref{prop:stairblock} and Theorem~\ref{thm:stairrecog}.

To put our work in context, recall from \cite{ball} that the
 Fibonacci stairs 
 are given by a family of exceptional divisors 
\begin{align}\label{eq:fib} \bigl(g_k,0; g_{k-1}\bw(\frac{g_{k+1}}{g_{k-1}})\bigr)
\end{align}
 where the $(g_k)_{k\ge 0} $ are the
 odd placed Fibonacci numbers $1,2,5,13,34, \cdots$.
The continued fractions of   the center points ${g_{k+1}}/{g_{k-1}}$ divide naturally into two classes:
the elements in the odd places   have centers $$
5,\ [6;1,4],\ [6;1,5,1,4] ,\ \dots, [6;1,\{5,1\}^{k},4],\dots
$$ while those in the even places  have centers $$
[6;2], \ [6;1,5,2],\ [6;1,5, 1,5, 2],\ \dots, [6;1,\{5,1\}^{k}, 5,2],\dots.
$$
Moreover, each of these classes form a pre-staircase in the sense of Definition~\ref{def:prestair}, with
recursion $x_{k+1} = 7x_k - x_{k-1}$ and linear relation $3d_k = p_k + q_k$.

The staircases that we describe below all have similar numerics, and limit at points $a_\infty = \lim {p_k}/{q_k}$ with $2$-periodic continued fractions.
There are other pre-staircases with continued fractions of higher periods; these will be discussed in our next paper.

  \subsection{The main theorems}\label{ss:Fib}

 We now describe three families of perfect center-blocking classes, $(\bB^U_n)$, $(\bB^L_n)$, and $(\bB^E_n)$ together with their associated  staircases.  
The classes $(\bB^U_n), (\bB^E_n)$ are those in Theorems~\ref{thm:block} and
~\ref{thm:blockstair}.   We will begin with the classes $(\bB^L_n)$,
 since, as explained in Remark~\ref{rmk:Fib0},
we can consider the Fibonacci staircase  to be the 
 initial
 (slightly anomalous) member of
the corresponding family of ascending staircases $(\Ss^{L}_{\ell,n})$.
We will say that an ascending (respectively descending) staircase $\Ss_{\ell}$ (resp.\ $\Ss_u$) is associated to
 the blocking class $\bB$ if it accumulates at the point  $\al_{\bB,\ell}$ (resp.\ $\al_{\bB,u}$).
 Thus staircases labelled $\ell$ always ascend, while those labelled $u$ always descend, regardless of whether the associated value of $b$ is in $(0,1/3)$ or $(1/3,1)$.

\begin{rmk} \label{rmk:evenodd} \normalfont
Each of the staircases below consists of two intertwining 
sequences of classes as follows
\begin{itemize}\item[$\bullet$]  the first has centers ${p_{n,k}}/{q_{n,k}}$  for $k\ge 0$, with $\eend_n = 2n+4$, and 
\item[$\bullet$]  the second 
has centers ${p'_{n,k}}/{q'_{n,k}}$ for $k\ge 0$, with $\eend_n = (2n+5, 2n+2).$
\end{itemize}
Each such sequence is a pre-staircase in the sense of Definition~\ref{def:prestair}.
  When the first of the end entries occurs in an even place (as with the ascending staircase $\Ss^L_{\ell,n}$) then the fact that $2n+4 < 2n+5$ means that 
${p_{n,k}}/{q_{n,k}}<{p'_{n,k}}/{q'_{n,k}}$.\footnote
 {Remark~\ref{rmk:ctfr}  explains the order properties of points that are specified in terms of their continued fraction expansions.}  Similarly, 
 when the first of the end entries occurs in an odd place  we have 
${p_{n,k}}/{q_{n,k}}>{p'_{n,k}}/{q'_{n,k}}$.

Recall also from Definition~\ref{def:stair} that we use the word \lq staircase' rather loosely; thus it could refer to just one of these sequences of classes, or both, depending on context.
\hfill$\er$
\end{rmk}
 \MS
 
 \begin{thm}\label{thm:L}  The classes $\bB^L_n = \bigl(5n,n-1; 2n\bw(\left(12n+1\right)/(2n))\bigr), n\ge 1,$ with decreasing centers, 
 are  perfect and center-blocking, and have the following associated staircases  $\Ss^{L}_{\ell,n}, \Ss^{L}_{u,n}$ for $n\ge 1$:
 \begin{itemize}  \item[$\bullet$]   $\Ss^{L}_{\ell,n}$ is ascending, with limit point $a^{L}_{\ell,n, \infty}= 
 [6;2n+1, \{2n+5,2n+1\}^\infty]$, and has
\begin{equation}
\begin{array}{ll}
\mbox{\rm (Centers)} &\quad 
 [6;2n+1, \{2n+5,2n+1\}^k,\eend_n],  \\ \notag
& \qquad \quad \eend_n = 2n+4\;\;\mbox{ or } (2n+5,2n+2), \ k\ge 0;\\  \notag
\mbox{\rm (Recursion)} &\quad  x_{n,k+1} = (\si_n + 2)x_{n,k} - x_{n,k-1}, \;\; \si_n: = (2n+1)(2n+5)\\ \notag
\mbox{\rm (Relation)} &\quad (2n+3)d_{n,k} =   (n+1) p_{n,k} - (n-1) q_{n,k}.
\end{array}
\end{equation}
 \item[$\bullet$]  $\Ss^{L}_{u,n}$ is descending, with limit point $a^{L}_{u,n, \infty}= 
 [6;2n-1,2n+1, \{2n+5,2n+1\}^\infty]$
 and has 
 \begin{equation}
 \begin{array}{ll}
\mbox{\rm (Centers)} &\quad  [6;2n-1,2n+1,\{2n+5,2n+1\}^k,\eend_n], \\ \notag
\mbox{\rm (Recursion)} &\quad x_{n,k+1} = (\si_n + 2)x_{n,k} - x_{n,k-1},
\\ \notag
\mbox{\rm (Relation)} &\quad (2n+3)d_{n,k} = -(n-1)p_{n,k} + (11n+2)q_{n,k}
\end{array}
\end{equation}
with the same possibilities for $\eend_n$ and the same $\si_n$. 
 \end{itemize}
  The  limit points $a^L_{\bullet,n,\infty}$ (with $\bullet = \ell$ or $u$)   form a decreasing sequence in $(6,7)$ with limit $6$, while the corresponding $b$-values lie in $(0,1/5)$ and increase with limit $1/5$.
\end{thm}

\begin{rmk}\label{rmk:Fib0}  \normalfont
The Fibonacci stairs has the same numerics as the case $n=0$ of  $\Ss^{L}_{\ell,n}$.  However, there cannot be an associated blocking class since this would have to block a $z$-interval with lower endpoint $\al^F_\ell$ equal to the limit point $\tau^4$.  But then $\acc_L^{-1}(\al^F_\ell) = 0$ would be  the upper endpoint of the corresponding $b$-interval, which is clearly impossible.  
Nevertheless, the obstruction given by the class  $\bE_0: = (3,0;2,1^{\times 6})$ with break point $a=7$
does go through the accumulation point $\bigl(\tau^4, V_0(\tau^4) = \tau^2\bigr)$, and  this class 
can be considered as a substitute for the blocking class.  
In this paper, we will often ignore this distinction and will use the notation $\Ss^{L}_{\ell,0}$ to refer to  the Fibonacci stairs.
  \hfill$\er$
\end{rmk}

Here is the analogous result for the blocking classes in Theorem~\ref{thm:block}.

 \begin{thm}\label{thm:U}  The classes $\bB^U_n = \bigl(n+3,n+2; \bw(2n+6)\bigr), n\ge 0,$
 with increasing centers, 
 are perfect and  center-blocking, with the following associated staircases  $\Ss^{U}_{\ell,n}, \Ss^{U}_{u,n}$, where  $\si_n$ and $\eend_n$ are as in Theorem~\ref{thm:L}.
 \begin{itemize}  \item[$\bullet$]  for each $n\ge 1$, $\Ss^{U}_{\ell,n}$ has limit point $a^{U}_{\ell,n, \infty}= 
 [\{2n+5,2n+1\}^\infty]$, and has
\begin{equation}
\begin{array}{ll}
\mbox{\rm (Centers)}&\quad  [\{2n+5,2n+1\}^k,\eend_n], 
\\  \notag
\mbox{\rm (Recursion)}&\quad x_{n,k+1} = (\si_n + 2)x_{n,k} - x_{n,k-1},\\ \notag
\mbox{\rm (Relation)}&\quad  (2n+3)d_{n,k} = (n+1) p_{n,k} + (n+2) q_{n,k}.  
\end{array}
\end{equation}
 \item[$\bullet$]  for each $n\ge 0$, $\Ss^{U}_{u,n}$ has limit point 
 $a^{U}_{u,n,\infty}= 
 [2n+7; \{2n+5,2n+1\}^\infty]$, and
 has 
 \begin{equation}
 \begin{array}{ll}
\mbox{\rm (Centers)} &\quad  [2n+7;\{2n+5,2n+1\}^k,\eend_n], \\ \notag
\mbox{\rm (Recursion)} &\quad x_{n,k+1} = (\si_n + 2)x_{n,k} - x_{n,k-1},
\\ \notag
\mbox{\rm (Relation)} &\quad (2n+3)d_{n,k} = (n+2)p_{n,k} - (n+4) q_{n,k}
\end{array}
\end{equation}
 \end{itemize}
   The  limit points $a^U_{\bullet,n, \infty}$  form  increasing unbounded sequences in $(6,\infty)$, while the corresponding $b$-values lie in $ (5/{11}, 1\bigr)$, where $5/{11} = \acc_U^{-1}(6)$.
\end{thm}

Figure \ref{fig:blockU} depicts the image under the parameterization $b\mapsto\bigl(\acc(b),V_b(\acc(b))\bigr)$ of the intervals $J_{B^U_n}$ for $n=0,\dots,3$ and part of the interval $J_{B^U_4}$.

\begin{figure}[H]
\includegraphics[width=\textwidth]{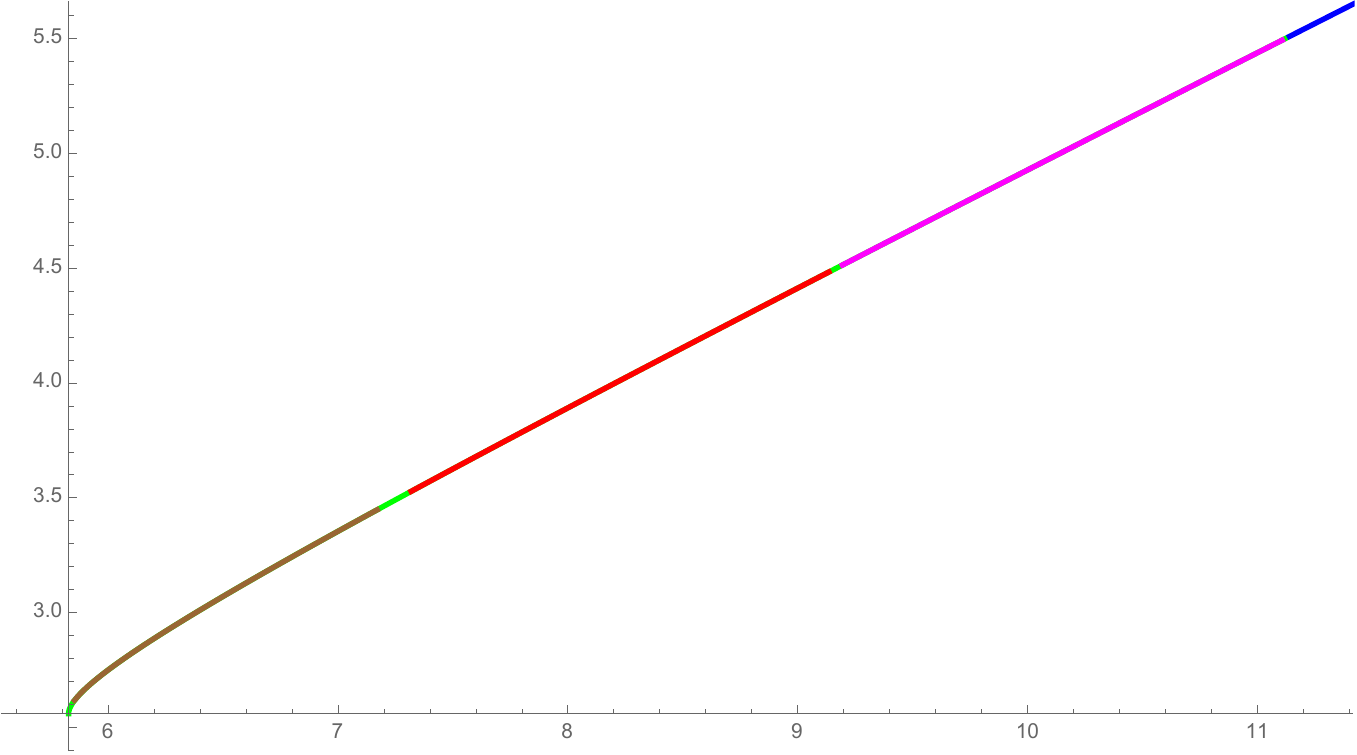}
\caption{Depicted are the intervals $J_{B^U_n}$ for $n=0,\dots,2$ and part of the interval $J_{B^U_3}$, in brown, red, pink, and blue, respectively, for $1/3<b<0.795$, transplanted onto the accumulation point curve $b\mapsto(\acc(b),V_b(\acc(b)))$. The infinite staircases $\Ss^U_{\ell,n}$ accumulate at the left endpoints of the intervals and the infinite staircases $\Ss^U_{u,n}$ accumulate at their right endpoints. Note that $Stair$ is contained in the short 
green intervals between the  intervals blocked by the $J_{B^U_n}$.}\label{fig:blockU}
\end{figure}

\begin{rmk}\label{rmk:SUell0}  \normalfont
Theorem~\ref{thm:U} lists descending staircases $\Ss^{U}_{u,n}$ for all $n\ge 0$, but ascending staircases $\Ss^{U}_{\ell,n}$ only for $n\ge 1$.   There is an ascending staircase for $n=0$ with numerics obtained by setting $n=0$ in the sequence $\Ss^U_{\ell,n}$ above, and depicted in Figure \ref{fig:515}, but because the class $\bB^U_0$ has center $6$, it has limit point $a^U_{\ell,0,\infty} < 6$, and therefore, as we will explain in Remark~\ref{rmk:symm}~(iii) below, it is better to consider this staircase as part of a different family. 
Indeed. as pointed out in Corollary~\ref{cor:symm}~(ii), the accumulation points $[6;1,\{5,1\}^\infty]$ for $\Ss^U_{\ell,0}$ and $[7;\{5,1\}^\infty]$ for $\Ss^U_{u,0}$
are mutual images under the reflection $\Phi$ of  Lemma~\ref{lem:symm}; and we explain in Remark~\ref{rmk:symm} how we expect that the associated families of staircases are organized.
 \hfill$\er$
\end{rmk}

Finally, here is a sharper version of  
Theorems~\ref{thm:stair} and~\ref{thm:blockstair}.  Notice that the centers $6-\frac1{2n+6}$ of the  classes $\bB^E_n$  increase with $n$.

\begin{thm}\label{thm:E}  The classes $\bB^E_n = \bigl(5(n+3), n+4; (2n+6) \bw(\frac{12n+35}{2n+6})\bigr), n\ge 0,$ with increasing centers,
 are perfect and center-blocking, and have the following associated staircases  $\Ss^{E}_{\ell,n}, \Ss^{E}_{u,n}$,  where  $\si_n$ and $\eend_n$ are as in Theorem~\ref{thm:L}.
 \begin{itemize}  \item[$\bullet$]  for each $n\ge 1$, $\Ss^{E}_{\ell,n}$ is ascending, with limit point $a^{E}_{\ell,n,\infty}= 
 [5;1,2n+4, 2n+1, \{2n+5,2n+1\}^\infty]$,
 and has
\begin{equation}
\begin{array}{ll}
\mbox{\rm (Centers)}&\quad  [5;1,2n+4, 2n+1, \{2n+5,2n+1\}^k,\eend_n], 
\\  \notag
\mbox{\rm (Recursion)}&\quad x_{n,k+1} = (\si_n + 2)x_{n,k} - x_{n,k-1},\\ \notag
\mbox{\rm (Relation)}&\quad  (2n+3)d_{n,k} =  (n+2)  p_{n,k} -(n+4)q_{n,k}  
\end{array}
\end{equation}
 \item[$\bullet$]  for each $n\ge 0$, $\Ss^{E}_{u,n}$ is descending, with limit point 
 $a^{E}_{u,n,\infty}= 
 [5;1, 2n+6, \{2n+5,2n+1\}^\infty]$, and
 has 
 \begin{equation}
 \begin{array}{ll}
\mbox{\rm (Centers)} &\quad  [5;1,2n+6,\{2n+5,2n+1\}^k,\eend_n], \\ \notag
\mbox{\rm (Recursion)} &\quad x_{n,k+1} = (\si_n + 2)x_{n,k} - x_{n,k-1}, 
\\ \notag
\mbox{\rm (Relation)} &\quad (2n+3)d_{n,k} = - (n+4) p_{n,k} + (11n +31)q_{n,k}.
\end{array}
\end{equation}
 \end{itemize}
  The limit points $a^E_{\bullet,n,\infty}$  form an increasing sequence in the interval $\bigl({35}/6 = [6;1,5], 6\bigr)$ while the corresponding $b$-values lie in $\bigl(1/5, {19}/{61}\bigr)$ where $
{19}/{61}= \acc_L^{-1}({35}/6)< 1/3$.
\end{thm}

 The numerics of these three families are very similar, since they all contain repetitions of the same $2$-periodic pair $(2n+5, 2n+1)$, which as is shown in Lemma~\ref{lem:Cr0}, implies that the recursion is the same.  
 We now observe that there are symmetries of the $z$-axis that relate the centers of the different families of blocking classes, as well as those of the corresponding staircase classes, as follows.

In the following we denote by $CF(x)$ the continued fraction expansion of a rational number $x\ge 1$,
and, by abuse of language, we sometimes identify $z$ with $CF(z)$.   For example,
every\footnote
{
This is true even for numbers such as $13/2 = [6;2]$ since in this case we can also write $z=[6;1,1]$.}
 rational $z \in (6,7)$ has continued fraction of the form $z=CF(z)=[6;k,CF(x)]$ for some integer $k\ge 0$
 and rational number $x\ge 1$.

\begin{lemm}\label{lem:symm}  {\rm (i)}
The fractional linear transformation $$
\Psi: (6, \infty) \to (6, \infty):  \quad w\mapsto \frac{6 w - 35}{w-6}
$$ 
reverses orientation and has the following properties:
\begin{itemize}\item[{\rm (a)}] 
$
\Psi\circ\Psi = 1,\qquad \Psi(7) = 7;
$
\item[{\rm (b)}] 
 $\Psi([6;k, CF(x)])=[6+k; CF(x)]\in (7,\infty)$ for all  $z=[6;k, CF(x)]\in (6,7)$;
\end{itemize}

\NI {\rm (ii)} The fractional linear transformation $$
\Phi: (\frac {35}6, \infty) \to (\frac {35}6, \infty):  \quad w\mapsto \frac{35 w - 204}{6w-35}
$$ 
reverses orientation and has the following properties:
\begin{itemize}\item[{\rm (a)}] 
$
\Phi\circ\Phi = 1,\qquad \Phi(6) = 6$ and
\begin{align*} \Phi(7) = \frac{41}7
 = [5;1,6], \qquad \Phi(8) = \frac{76}{13} = [5;1,5,2].
\end{align*}
\item[{\rm (b)}] every  $z\in   \bigl( \frac{76}{13},  \frac{41}7
\bigr)$ has $CF(z) = [5;1,5,1, CF(x)]$ for $x>1$
and $$
\Phi([5;1,5,1, CF(x)]) = [7; CF(x)] \in (7,8),\quad \mbox{ for all } x>1.
$$
\end{itemize}

\smallskip

\NI {\rm (iii)}  The bijective map $$
Sh: (1,\infty)\to (5,6), \quad z\mapsto \frac{6z-1}{z}
$$
preserves orientation, and for each rational $x>1$ and $k\ge1$ we have
$$
Sh(z) = Sh([k+5; CF(x)]) = [5;1,k+4, CF(x)],\quad \mbox{ if }\ z = [k+5; CF(x)]\in (6,\infty).
$$
Further, $\Phi\circ \Psi = Sh: (6,\infty)\to (35/6,6)$.
\end{lemm}

\begin{proof}  We prove (i).   (a) follows immediately from the definition.  
To prove (b), observe that 
\[
w=[6;k+1,CF(x)]=6+\frac{1}{k+1+\frac{1}{x}}=\frac{(6k+7)x+6}{(k+1)x+1}
\]
which gives
\[
x=\frac{w-6}{(6k+7)-(k+1)w}
\]
Therefore we have
\begin{align*}
z&=[k+7;CF(x)]=k+7+\frac{1}{x}
\\&=k+7+\frac{(6k+7)-(k+1)w}{w-6}
\\&=\frac{6w-35}{w-6}=\Psi(w) 
\end{align*}
as claimed.  
 This proves (i).
\MS

The proofs of (ii), (iii) are very similar and are left to the reader.
\end{proof}

\begin{cor}\label{cor:symm} {\rm (i)} For each $n\ge 1$, the involution $\Psi$ takes the center $2n+6$ of the blocking class $\bB^U_n$
to the center $[6;2n]$ of the blocking class $\bB^L_n$. It also takes the centers of the associated staircases $\Ss^U_{\ell,n}, \Ss^U_{u,n}$ to those of  $\Ss^L_{u,n}, \Ss^L_{\ell,n}$, interchanging increasing and decreasing staircases.  Further, $\Psi$ also takes the centers of the decreasing staircase
$\Ss^U_{u,0}$  to the centers of the Fibonacci stairs. \MS

\NI {\rm (ii)}  The involution $\Phi$ fixes the center $6$ of the blocking class $\bB^U_0$ and interchanges the steps of the two associated staircases.  
\MS

\NI {\rm (iii)}  For each $n\ge 1$, the shift $Sh$ takes the center $2n+6$ of the blocking class $\bB^U_n$
to the center $[5;1,2n+5]$ of the blocking class $\bB^E_n$.
It also takes the centers of the associated staircases $\Ss^U_{\ell,n}, \Ss^U_{u,n}$ to those of  $\Ss^E_{\ell,n}, \Ss^E_{u,n}$, preserving the direction of the staircases.  
\end{cor}
\begin{proof}   This is an immediate consequence of Lemma~\ref{lem:symm}.
Note that by Theorem~\ref{thm:block} the  ascending staircase associated to $\bB^U_0$ has steps at
$[5; 1,\{5,1\}^{k},\eend_0]$ while the descending staircase has steps at $[7; \{5,1\}^{k-1},\eend_0]$.
Further, for these values of $z$, $\Phi$  moves the first entry in $\eend_0$ from an even place to an odd place, and hence, in accordance with  Remark~\ref{rmk:evenodd}, takes an ascending stair to a descending one. 
\end{proof}

\begin{rmk}  \label{rmk:symm} \normalfont
(i) Though they might seem similar, the two involutions 
 $\Phi$ and $\Psi$  have very different effects.  As shown by part (ii) of Corollary~\ref{cor:symm},  $\Phi$ fixes the center of the  blocking class $\bB^U_0$ and  interchanges its two staircases,  both of which exist for values of $b> 1/3$.  On the other hand,  the fixed point of $\Psi$ 
is not the center of a blocking class, and it takes staircases for $b>1/3$ to staircases for $b< 1/3$.
Similarly, $Sh$ takes staircases for $b>1/3$ with centers in $(6,\infty)$ to staircases for $b<1/3$ with centers in $({35}/6, 6)$.    However, the transformation $\Sh^2: = Sh\circ Sh$ should take staircases for $b>1/3$ (resp.\ $b<1/3$) to staircases for $b$ in the same interval.

\MS

\NI  (ii)    The maps $\Phi, \Psi, Sh = \Phi\circ \Psi$ generate a set of fractional linear transformations that act on the centers of the blocking classes and of the staircase steps.   
 Corollary~\ref{cor:symm} shows that there has to be some associated action on the other two entries $d,m$ of a perfect blocking class.  However, as yet this is not fully understood.
\MS

\NI (iii) We conjecture that every  $2$-periodic staircase for $b\in (1/3,1)$ is one of the following:
\begin{itemize}\item[$\bullet$]   the staircases $\Ss^U_{\bullet,n}$ in Theorem~\ref{thm:U} with centers in $(6,\infty)$;
\item[$\bullet$]  their images under $\Phi$ with centers in $(  35/6,6)$;
\item[$\bullet$]  the images of these two families under $Sh^{2i}, i=1,2\dots$, with centers in the disjoint intervals
$Sh^2\bigl(( {35}/6,6)\bigr) =({1189}/{204}, {35}/6)$,
$Sh^4\bigl(( {35}/6,6)\bigr)$, and so on.
\end{itemize}
Similarly, the $2$-periodic staircase for $b\in [0,1/3)$ should be formed from the staircases $\Psi\bigr(\Ss^U_{\bullet,n}\bigr)$ with centers in $(6,7)$ by involutions and shifts.  
 This idea, together with many other related results, will be developed in our next paper. 
 We have also found many potential staircases of period  $2m, m>1,$ and have conjectural description of them as well.   In these staircases the repeated portion of the continued fraction expansion of the centers $p_{k}/q_k$ has length $2m$, and turns out to consist of a sequence of integers of the form $(2n+1,2n+3,\dots, 2n + 4m - 1)$ in some order. Thus these staircases are still in some way derived from the Fibonacci staircases.  
 \hfill$\er$
\end{rmk}

\subsection{Proof of Theorems~\ref{thm:U} and~\ref{thm:block}}\label{ss:thmU}

Here is the strategy for proving Theorems~\ref{thm:L},~\ref{thm:U} and ~\ref{thm:E}.

\begin{itemize}\item{\bf Step 1:} Check that the classes $\bB_n$ are quasi-perfect.
\item {\bf Step 2:}   Prove that the classes in the two associated families $\Ss_\ell, \Ss_u$ are perfect.
\item {\bf Step 3:}   Check that the conditions in  Theorem~\ref{thm:blockrecog}  hold, and conclude that each $\bB_n$ is a perfect blocking class. 
\item {\bf Step 4:} Check that each pre-staircase $\Ss_\ell, \Ss_u$ satisfies one of the conditions in Theorem~\ref{thm:stairrecog}, 
and conclude that both are staircases with the given properties.
\end{itemize}

We will begin by proving Theorem~\ref{thm:U}, since this is the most straightforward.
As we will see, in the cases at hand only Steps 2 and 4 require significant argument.
The most computational part of the argument is the proof that the pre-staircase classes are perfect, rather than just quasi-perfect, which involves showing that they reduce correctly under Cremona moves.  
 These proofs are all given  in \S\ref{ss:reduct}.  As we will see there,  the proofs for the different pre-staircases are closely related. 

\begin{rmk}\label{rmk:polyn} \normalfont
The  arguments given below reduce many of the proofs to checking some numerical identities that are polynomial in $n$ and of small degree $d$ (less than $10$ or so).  Therefore
if one verifies these identities by computer for $d+1$ values of  $n$,
 then they will hold for all $n$.
 \hfill$\er$
\end{rmk}

\NI {\bf Step 1:}  For $n\ge 0$, the class $$
\bB_n = \bigl(n+3,n+2; 1^{\times(2n+6)}\bigr)
$$ has $d=n+3, m = n+2, p = 2n+6, q=1$ and  evidently satisfies the Diophantine conditions $3d=m + p+ q$, $d^2 - m^2 = pq-1$.  Hence it is quasi-perfect with center  $2n+6$.  

\MS

\NI {\bf Step 2:}  
 We begin by establishing the following lemma.

\begin{lemm}\label{lem:Cr0} {\rm(i)}  For $k\ge 1$, let $ [2n+5;2n+1,\{2n+5,2n+1\}^{k-1},2n+4]$ have (integral) weight expansion 
$$ \bigl(a_{n,k}^{\times (2n+5)}, b_{n,k}^{\times (2n+1)}, (a_{n,k}-(2n+1)b_{n,k})^{\times (2n+5)},\dots, 1^{\times (2n+4)}\bigr)
$$
and define \begin{align}\label{eq:recur2}
a_{n,0} = 1,\quad b_{n,0} = -1, \quad a_{n,-1} = 2n+2.
\end{align}
Then
\begin{itemize}\item[{\rm(a)}]  $b_{n,1} = 2n+4$, $a_{n,1} = (2n+1)(2n+4) + 1$.
\item[{\rm(b)}]   we have
\begin{align}\label{eq:recur5}
b_{n,k+1} = (2n+5) a_{n,k} + b_{n,k},\quad 
a_{n,k} = (2n+1)b_{n,k} + a_{n,k-1}  \qquad \mbox{ if } k\ge 0.
\end{align}

\item[{\rm(c)}]   For each $n$, the quantities $a_{n,\bullet},b_{n,\bullet}$
satisfy the recursion relation 
\begin{align}\label{eq:recur1}
x_{n,k+1} = (4n^2 + 12 n + 7) x_{n,k} - x_{n,k-1}, \quad k\ge 1.
\end{align}
\end{itemize}

\NI {\rm (ii)}   For $k\ge 1$, let $[2n+5;2n+1,\{2n+5,2n+1\}^{k-1},2n+5, 2n+2]$ have (integral) weight expansion $$
\bigl({a'_{n,k}}\!\!\!^{\times (2n+5)},{b'_{n,k}}\!\!\!^{\times (2n+1)}, (a'_{n,k}-(2n+1)b'_{n,k})^{\times (2n+5)},\dots, 
(2n+2)^{\times (2n+5)},1^{\times (2n+2)}\bigr).
$$ 
and define
\begin{align}\label{eq:recur2'}
a'_{n,0} =  2n+2,\quad b'_{n,0} = 1,\quad a'_{n,-1} = 1 
\end{align}
Then
$$b'_{n,1} = (2n+2)(2n+5) + 1,\qquad a'_{n,1} = (2n+1)\bigl((2n+2)(2n+5) + 1\bigr) + (2n+2);
$$
and \eqref{eq:recur5}, \eqref{eq:recur1} hold for $a'_{n,\bullet}, b'_{n,\bullet}$ as in {\rm (i)}.
\end{lemm}  
\begin{proof}    First we consider case (i).  
As explained in Example~\ref{ex:ctfr0},  the integral weight expansion
$(x_0^{\times \ell_0}, x_1^{\times \ell_1}, \dots, x_N^{\times \ell_N})$  of $p/q = [\ell_0;\ell_1,\dots,\ell_N]$ 
satisfies the identities
$x_{k+1} = x_{k-1} - \ell_k x_k$.    Assuming that one knows the multiplicities $[\ell_0;\ell_1,\dots,\ell_N]$,
one can read this identity  in either direction.  If one knows $ p,q$ then $x_0=q$ so that $x_1 = p-\ell_0 q$, and then one can determine $x_2, x_3, $ and so on.    On the other hand, since we are dealing with the integral expansion the last weight is $1$, and hence we have $x_N = 1, x_{N-1}  = \ell_N$ which determines $x_{N-2}, x_{N-3}$ and so on.
The latter is the relevant procedure here since going from $a_{n,k}, b_{n,k}$ to $a_{n,k+1}, b_{n,k+1}$ means adding an extra pair $(2n+5,2n+1)$ to the continued fraction.  
For example, the integral weight expansion of $[2n+5; 2n+1, 2n+4]$ is $$
(a_{n,1}^{\times (2n+5)}, b_{n,1}^{\times (2n+1)}, 1^{\times (2n+4)}).
$$
This proves (i:a), and (i:b) holds by similar reasoning.
One can check that  the relation also holds when $k=0$ with the given definitions of
$a_{n,0},b_{n,0}$ and $a_{n,-1}$. 
\MS

Finally, we prove by induction on $k\ge 1$ that the numbers $a_{n,k}, b_{n,k+1}$ satisfy \eqref{eq:recur1}.  The base case for $a_{n,1}$ follows from the definitions of $a_{n,0}, a_{n,-1}$ in \eqref{eq:recur0}, while the claim for $b_{n,2}$ holds because
\begin{align*}
b_{n,2} &= (2n+5)\bigl((2n+1)(2n+4) +1\bigr) +  (2n+4)\quad\mbox{ by  (b)}\\
& = (2n+4)(4n^2 +12 n + 6) + (2n+5)\\
& = (4n^2 +12 n + 7) (2n+4) + 1 = (4n^2 +12 n + 7) b_{n,1} - b_{n,0}.
\end{align*}
The inductive step then follows from the linear relations $b_{n,k+1} = (2n+5) a_{n,k} + b_{n,k}$ and
 $a_{n,k+1} = (2n+1)b_{n,k+1} + a_{n,k}$ which were proved in (i)(b).
 This proves (i).
 
 The proof of (ii) is similar, and is left to the reader.  \end{proof}

\MS

\begin{lemm}\label{lem:otherend}  Let $\Ss$ 
 be a pre-staircase  with   linear relation  $(2n+3)d_k = R_1 p_k + R_2 q_k$, where $R_1,R_2 \in \Z[n]$, the ring of polynomials in $n$ with integer coefficients,  and with steps 
at the points 
\begin{align*}
p_k/q_k & = [\ell_0;\ell_1,\dots,\ell_r, \{2n+5, 2n+1\}^k, \eend_n], \quad \eend_n: = 2n+4, \ell_r\in  \Z[n],
\end{align*}
and  other coefficients   $d_k,m_k$.  Further,  let $\Ss'$ be the corresponding pre-staircase   with the same linear relation  $(2n+3)d'_k = R_1 p'_k + R_2 q'_k$ and steps $p_k'/q_k'$ defined as above but with 
$\eend_n = (2n+5,2n+2)$.  Then
\begin{itemize} \item[{\rm (i)}]  $p_k,q_k\in  \Z[n]$, and 
$(2n+3)$ divides $R_1 p_k + R_2 q_k$  if and only if  it divides
$R_1 p'_k + R_2 q'_k$.  Moreover, this holds for all $k$ if and only if it holds for $k=0$.
\item[{\rm (ii)}]  $m_1 d_0 - m_0 d_1 = m'_1 d'_0 - m'_0 d'_1.$
\end{itemize}
\end{lemm}
\begin{proof}  
By Lemma~\ref{lem:Cr0}~(i) 
$$
q_k\bw(p_k/q_k) = \bigl(\dots,\ (\ell_r b_{n,k+1} + a_{n,k}^{\times \ell_{r-1}},\ b_{n,k+1}^{\times \ell_r}, a_{n,k}^{\times 2n+5}, \dots \bigr).
$$

Therefore, by induction on $r$, one can see that there are polynomials  $C_{1}, \dots, C_{4}$  that depend only  $\ell_0,\dots, \ell_r$ such that
$$
q_k = C_{1} b_{n,k+1} + C_{2} a_{n,k},\quad p_k =C_{3} b_{n,k+1} + C_{4} a_{n,k}.
$$
Similarly, we have
$$
q'_k = C_1 b'_{n,k+1} + C_2 a'_{n,k},\quad p'_k =C_3 b'_{n,k+1} + C_4 a'_{n,k},
$$
with the same  $C_i$.   Therefore, there are elements  $N_5, N_6\in \Z[n]$ such that
$$
R_1 p_k + R_2 q_k = N_5 b_{n,k+1} + N_6 a_{n,k},\quad R_1 p'_k + R_2 q'_k = N_5 b'_{n,k+1} + N_6 a'_{n,k}
$$
Next note that to prove (i), it suffices to  check that it holds for $k=0,1$ since the general case then follows by the recursion.  
Hence to prove (i) we must check that $(2n+3)$ divides $N_5 b_{n,k+1} + N_6 a_{n,k}$ for $k=0,1$  if and only if it also divides $N_5 b'_{n,k+1} + N_6 a'_{n,k}$ for $k=0,1$.

But by Lemma~\ref{lem:Cr0},  modulo $(2n+3)$ we have
$$
a_{n,0} \equiv 1, b_{n,1}\equiv 1,\;\;  a_{n,1} \equiv  - 1,  b_{n,2}\equiv -1,\qquad   
a'_{n,0} \equiv -1, \; b'_{n,1}\equiv -1,\;  a'_{n,1} \equiv   1, b'_{n,2}\equiv 1.  
$$
Hence in both cases we need precisely that $N_5+N_6$  is divisible by $2n+3$.   Since this identity follows already from the case $k=0$, this proves (i).
\MS

To prove (ii),  notice that given $p_k,q_k$ we define $d_k$ by the staircase relation.  Therefore we may write $d_k = C_7 a_{n,k} + C_8 b_{n,k+1}$, where $C_7,C_8$ are (possibly rational, but in practise polynomial) functions of $n$ that depend also on $\ell_0,\dots,\ell_r$.  When $k=0,1$, \eqref{eq:recur5}  gives an explicit linear expression for  
$b_{n,k+1}$ in terms of $a_n, b_n$, with coefficients that also depend on $n$ (but are independent of $k$).  Therefore, for $k=0,1$  we can write  $d_k = N_1 a_{n,k} + N_2 b_{n,k}$ where $N_1,N_2$ are functions of $n,
\ell_0,\dots,\ell_r$.   Since  $m_k: = 3d_k-p_k-q_k$, a similar argument shows that there are functions $N_1,\dots,N_4$ of $n, \ell_0,\dots,\ell_r$ such that the first two lines below hold.  
\begin{align*}
&d_0 = N_1 a_{n,0} + N_2 b_{n,0}, \quad m_0 =  N_3 a_{n,0} + N_4 b_{n,0}\\
&d_1 = N_1 a_{n,1} + N_2 b_{n,1}, \quad  m_1 = N_3 a_{n,1} + N_4 b_{n,1}\\
&d'_0 =N_1 a'_{n,0} + N_2 b'_{n,0}, \quad m'_0 =   N_3 a'_{n,0} + N_4 b'_{n,0}\\
&d'_1 = N_1 a'_{n,1} + N_2 b'_{n,1}, \quad m'_1 =   N_3 a'_{n,1} + N_4 b'_{n,1}.
\end{align*}
The second two lines then hold with the same constants since \eqref{eq:recur5}  also holds in this case.
Hence 
\begin{align*}
& m_1d_0 - m_0 d_1 = (N_2N_3 - N_1 N_4)( a_{n,1}b_{n,0} - a_{n,0}b_{n,1})\\
&m'_1d'_0 - m'_0 d'_1 = (N_2N_3 - N_1 N_4)( a'_{n,1}b'_{n,0} - a'_{n,0}b'_{n,1}).
\end{align*}
Therefore it suffices to check that 
$$
a_{n,1}b_{n,0} - a_{n,0}b_{n,1}=a'_{n,1}b'_{n,0} - a'_{n,0}b'_{n,1}.
$$
But by Lemma~\ref{lem:Cr0} both sides equal $-(4n^2 + 12n + 9)$.
\end{proof}

\begin{lemm}\label{lem:dquad}  Suppose that the positive integers $d_k, m_k, p_k,q_k$ are defined for $k\ge 0$ and satisfy the recursion 
$$
x_{k+1} = R x_k - x_{k-1},\quad k\ge 1.
$$
Then the identity
$d_k^2 - m_{k}^2 = p_kq_k - 1$ holds for all $k$, if and only if it holds for $k = 0,1$, and also
\begin{align}\label{eq:newid}
2(d_0d_1 - m_0m_1) = (p_1q_0 + p_0q_1) - R.
\end{align}
Hence it holds for all $k$ if and only if it holds for $k = 0,1,2$.
\end{lemm}
\begin{proof}  Abbreviate $x: = x_k, x': = x_{k-1}$ so that $x_{k+1} = Rx-x'$.  Then observe that
\begin{align}\label{eq:recur3}
&\Bigl(d^2 - m^2 = pq - 1, \;\;  (d')^2 - (m')^2 = p'q' - 1
\Bigr) \Longleftrightarrow \\ \notag
&\qquad\qquad \Bigl( (Rd-d')^2 - (Rm-m')^2 = (Rp-p')(Rq-q') - 1\Bigr),
\end{align}
exactly if
$$
2(dd' - mm') = (pq' + p'q) -R,
$$
But if the latter equation holds, we have
\begin{align*}
2\bigl((Rd-d') d - (Rm-m')m\bigr)& =  2R(d^2-m^2) - 2(dd' - mm') \\
& = 2R (pq - 1) + R -  (pq' + p'q)\\
& = - R + (Rp-p')q + (Rq-q')p.
\end{align*}
Hence, arguing inductively, we see that  \eqref{eq:newid} implies 
$$
2(d_{k+1}d_k - m_{k+1}m_k) = -R + (p_kq_{k+1} + p_{k+1}q_k),\quad \mbox{ for all } k\ge 0.
$$
Thus, the calculation in \eqref{eq:recur3} implies that $d_k^2 - m_{k}^2 = p_kq_k - 1$ holds for all $k
\ge 0$.
\end{proof}

\begin{lemm}\label{lem:SsUellu} \begin{itemize}
\item[{\rm (i)}]
	For each $n \ge 1$, the classes  in the ascending pre-staircase $\Ss^U_{\ell,n}$ are 
	integral and satisfy the linear Diophantine identity.
\item[{\rm (ii)}]
For each $n \ge 0$, the classes  in the descending pre-staircase $\Ss^U_{u,n}$ are 
	integral and satisfy the linear Diophantine identity.
\end{itemize}	
\end{lemm}
\begin{proof}  First consider  case (i) with ending $\eend_n = 2n+4$. 
Then $\bE_{\ell,n,0}^U$ has center $2n+4$, and one can calculate the corresponding degree $d$ by
using the staircase relation $$
(2n+3)d_{\ell,n,k}^U = (n+1) p_{\ell,n,k}^U + (n+2)q_{\ell,n,k}^U.
$$
When $k=1$ we may use Lemma~\ref{lem:Cr0}~(i) to obtain $$
q_{\ell,n,1}^U = a_{n,1},\;\;\quad  p_{\ell,n,1}^U = (2n+5)a_{n,1} + b_{n,1}.
$$ 
We compute $d^U_{\ell,n,1}$ using the relation, and then define $m^U_{\ell,n,1}$ by the linear Diophantine identity.
This gives
 the following:
\begin{align}\label{eq:SsUell}
&\bigl(d^U_{\ell,n,0}, m^U_{\ell,n,0}, p^U_{\ell,n,0},q^U_{\ell,n,0}\bigr) =\bigl(n+2,n+1, 2n+4,1\bigr)
\\ \notag
&d^U_{\ell,n,1} = 4n^3 + 20 n^2 + 30 n + 13,\quad m^U_{\ell, n,1} = 4n^3 + 16 n^2 + 18 n + 5.\\ \notag
& p^U_{\ell,n,1} = 8n^3 + 40 n^2 + 62 n + 29,\quad q^U_{\ell,n,1} =   4n^2+10n + 5,
\end{align}
The  tuples for $k\ge 2$ are then defined using the recursion.
Since $d^U_{\ell,n,k}$ is an integer for $k=0,1$, the recursion implies that it is an integer for all $k$.  
The tuple satisfies the linear Diophantine relation by definition.

The claims when 
 $\eend_n = (2n+5,2n+2)$ follow by arguing similarly, using Lemma~\ref{lem:otherend}~(i).  This proves (i).
 \MS
 
The steps of $\Ss^U_{u,n}$ have centers $[2n+7; \{2n+5,2n+1\}^k,\eend_n]$ for $k\ge 0$, and 
the relation is
$$
(2n+3)d^U_{u,n,k} = (n+2)p^U_{u,n,k} - (n+4)q^U_{u,n,k}.
$$
Since the staircase $\Ss^U_{\ell,n}$ has steps with centers $$
[\{2n+5,2n+1\}^k,\eend_n] = q^U_{\ell,n,k}\ \bw\Bigl(\frac{p^U_{\ell,n,k}}{q^U_{\ell,n,k}}\Bigr),
$$
it follows  that\footnote
{ 
In general, if $[\ell_0; \ell_1,\dots, ] =  p/q$ then   $[s; \ell_0, \ell_1,\dots,] = (sp+q)/p$.}
$$
q^U_{u,n,k}  = p^U_{\ell,n,k} , \quad p^U_{u,n,k} = (2n+7)p^U_{\ell,n,k} + q^U_{\ell,n,k}.
$$
Therefore
\begin{align*}
R_1 p^U_{u,n,k} + R_2 q^U_{u,n,k} &=  (n+2)\bigl((2n+7)p^U_{\ell,n,k} + q^U_{\ell,n,k}\bigr) - (n+4)p^U_{\ell,n,k}\\
&\equiv (n+1) p^U_{\ell,n,k} + (n+2)q^U_{\ell,n,k} \pmod{(2n+3)},
\end{align*}
and hence is divisible by $2n+3$ by (i).   Again, the linear Diophantine relation holds by the definition of $m^U_{u,n,k}$.
 For later use, we note that 
\begin{align}\label{eq:SsUu}
\bigl(d^U_{u,n,0}, m^U_{u, n,0}\bigr)& =\bigl(2n^2 + 11n +14,\;\; 2n^2 + 9n + 9\bigr)
\\ \notag
d^U_{u,n,1} &= 8n^4 + 68n^3 + 202 n^2 + 245n + 100,\\ \notag
 m^U_{u, n,1} &= 8n^4 + 60 n^3 + 158 n^2 + 171 n + 63.
 \end{align}
This completes the proof.
\end{proof}

To complete  Step 2 we must show that the classes are perfect.  This proof is deferred to \S\ref{ss:reduct}; see Proposition~\ref{prop:reduct}.
The quadratic Diophantine equality then follows (though one could also prove it by an inductive argument based on Lemma~\ref{lem:dquad}.)  
 \MS

\NI {\bf Step 3:}   We must check that the conditions in Theorem~\ref{thm:blockrecog} hold.  
The fact that $3+2\sqrt 2< P_\ell/Q_\ell < p/q < P_u/Q_u$ follows immediately from the continued fraction expansions of $P_\ell/Q_\ell , P_u/Q_u$.    See Remark~\ref{rmk:ctfr}, which explains the somewhat subtle ordering of numbers that are given by continued fractions, and Remark~\ref{rmk:evenodd} which explains how to check whether a staircase ascends or descends.  The claim that $\ell(p/q) < \ell(z)$ for all $z\in \bigl(P_\ell/Q_\ell , P_u/Q_u\bigr)$ is  very straightforward in the case of the staircases $\Ss^U_{\bullet,n}$, since  $p/q$ is an integer and $\ell(p/q) = p$.  
We also must check that $M_\bullet/D_\bullet \ne 1/3$ for $\bullet = \ell,u$.
Since $M_\bullet, D_\bullet$ are both numbers of the form $a + b\sqrt{\si}$ where $a,b\in \Q$ it suffices to check that rational part $a$ of $3M_\bullet$ is not equal to that of $D_\bullet$.  Therefore, by \eqref{eq:recurX} it suffices to check that in each staircase $3m_0 \ne d_0$, a fact that is immediate from the formulas in   \eqref{eq:SsUell} and \eqref{eq:SsUu}. 
Thus condition (i) in Theorem~\ref{thm:blockrecog} holds,
and  conditions (ii) and (iii)
can be verified by comparing the given linear relations with the parameters of $\bB_n^U$.
\MS

\NI {\bf Step 4:}  By Step 2
 and Section \ref{ss:reduct}  we know that for each $n$ the pre-staircase classes are perfect for large $k$, but it remains to check that they are live at the appropriate limiting value for $b$.
This  requires considerable work, and
again we begin with a general result.
\MS

\begin{lemm}\label{lem:DMineq} 
	 Consider a pre-staircase with classes $\bigl(d_k,m_k;q_k\bw(p_k/q_k)\bigr)$, where the ratios $b_k: =m_k/d_k$ 
have  limit $b_\infty$, and let the constants $D,D',D'',M,M',M'', \si: = \si_n$ be as in \eqref{eq:recurX}, with $x_k = d_k, m_k$ respectively. 
\begin{itemize}\item[{\rm (i)}]  Suppose that $M\ov D - \ov M D\ne 0$.  Then the $b_k$ are strictly increasing iff
$$
M\ov D - \ov M D = 2\sqrt{\si_n}(M''D'-M'D'')> 0,
$$
and otherwise they are strictly decreasing.
\item[{\rm (ii)}]  if $M''D'-M'D'' > 0$, $b_\infty<r/s \le 1$,  and
\begin{align}\label{eq:DMineq}
|M''D' - M'D''| \le \frac{sD-rM}{2\sqrt\si\, |rD-sM|},
\end{align}
then
 there is $k_0$ such that
$$
\frac{m_k}{d_k} \le b_\infty = \frac MD  \le 
\frac{s+ m_k(rd_k-sm_k)}{r+d_{k}(rd_k-sm_k)},\quad \mbox { for } \;\; k\ge k_0.
$$
\item[{\rm (iii)}]  
if $M''D'-M'D'' < 0$, $b_\infty>r/s > 0$, and   \eqref{eq:DMineq} holds,
then  there is $k_0$ such that
$$
b_k: = \frac{m_k(sm_k-rd_k)-s}{d_k(sm_k-rd_k) -r}\le  b_\infty = \frac MD  \le \frac{m_k}{d_k}, \quad  \mbox { for } \;\; k\ge k_0, 
$$
\end{itemize}
 \end{lemm}
\begin{proof}  Write $$
m_k = M\la^k + \ov M\, \ov \la^k, \quad d_{k} = D\la^k + \ov D\, \ov \la^k, 
$$
as in Lemma~\ref{lem:recur}.  
Then
\begin{align*}
b_k: = \frac{m_k} {d_k} & = \frac {M\la^k + \ov M\, \ov \la^k}{D\la^k + \ov D\, \ov \la^k} < 
\frac {M\la^{k+1} + \ov M\, \ov \la^{k+1}}{D\la^{k+1} + \ov D\, \ov \la^{k+1}} = \frac{m_{k+1}} {d_{k+1}}
\end{align*}
if and only if
\begin{align*}
\bigl(M\la^k + \ov M\, \ov \la^k\bigr)\bigl(D\la^{k+1} + \ov D\, \ov \la^{k+1}\bigr) < \bigl(M\la^{k+1} + \ov M\, \ov \la^{k+1}\bigr)\bigl(D\la^k + \ov D\, \ov \la^k\bigr).
\end{align*}
This is equivalent to
\begin{align*}
\ov M  D \la + M \ov D\, \ov \la\; <\; \ov M D \ov \la  + M \ov D \la,
\end{align*}
and hence to the condition
\begin{align*}
(M \ov D - \ov M D)\la \;>\; (M \ov D - \ov M D)\ov\la.
\end{align*}
Since $\la > \ov \la > 0$, this holds exactly if $M \ov D - \ov M D> 0$.
\MS

 Now write $M = M' + M''\sqrt\si, D = D'+D''\sqrt\si$ as in Remark~\ref{rmk:recur}.  Then
 \begin{align*}
 M\ov D - \ov M D & = (M' + M''\sqrt\si)(D' - D''\sqrt\si) - (M' - M''\sqrt\si)(D' + D''\sqrt\si)\\
 &
 = 2\sqrt\si(M''D' - M'D'').
\end{align*}
 This proves (i). \MS

To prove (ii)  we must check that
\begin{align*}
& rM + M(D\la^k+\ov D \, \ov \la^k)\bigl(r(D\la^k+\ov D \, \ov \la^k)-s(M\la^k+\ov M \, \ov \la^k)\bigr) \\
 &\qquad\qquad \le 
 sD + D(M\la^k+\ov M \ \ov \la^k)\bigl(r(D\la^k+\ov D \, \ov \la^k)-s(M\la^k+\ov M \, \ov \la^k)\bigr).
\end{align*}
The coefficients of $\la^{2k}$ are the same on both sides, while, after a little manipulation, the constant term  gives the inequality
 $$
 (M\ov D - \ov M D)(rD-sM) \le {sD-rM}.
 $$
 Since $b_\infty: =  M/D< r/s \le 1$ by assumption,  we have
 $rD-sM >0$, and $sD-rM>0$ since $D>M, s\ge r$.
 The further assumption  $$
  M\ov D - \ov M D = 2\sqrt\si(M''D'-M'D'') > 0,
  $$
shows that the above inequality is equivalent to requiring that 
 $|M''D' - M'D''| \le  \frac{sD-rM}{2\sqrt\si (sM-rD)}$, as claimed.

 This completes the proof of (ii).  The proof of (iii) is almost identical, and is left to the reader.
\end{proof}

\begin{rmk}\label{rmk:DMineq}   \normalfont
(i)  The inequalities in Lemma~\ref{lem:DMineq} can be simplified because, in the notation of \eqref{eq:recurX} we have
$
X = X'+X''\sqrt{\si_n}$ 
 where 
\begin{align}\label{eq:recurX1}
X' = \frac {x_0}2, \;\; X'' = \frac{2x_{1} - x_{0}(\si_n+2)}{2(2n+3)\si_n},\quad \si_n = (2n+1)(2n+5).
\end{align}
Hence if the values of $d_{n,k}, m_{n,k}$ for $k=0,1$ are denoted
$d_k, m_k$, we have $$
D' = \frac {d_0}2, \;\; D'' = \frac{2d_{1} - d_{0}(\si_n+2)}{2(2n+3)\si_n},\;\;
M' = \frac {m_0}2, \;\; M'' = \frac{2m_{1} - m_{0}(\si_n+2)}{2(2n+3)\si_n}
$$
so that
\begin{align}\label{eq:DMin}
M''D' - M'D'' = \frac 1{2(2n+3)\si_n} \bigl(m_1d_0 - m_0 d_1\bigr).
\end{align}
Therefore, we just need to check the sign of $m_1d_0 - m_0d_1$, and 
the following reformulation of \eqref{eq:DMineq}:
\begin{align}\label{eq:DMin1} 
\frac{|m_1d_0 - m_0d_1|}
{2n+3} \le  \frac{\sqrt{\si_n}(sD-rM)}{|sM-rD|}. 
\end{align}
Note that this inequality only involves the entries in the first two terms of the pre-staircase.\MS
\MS

\NI (ii)  By Lemma~\ref{lem:otherend}~(ii), the quantity $m_1d_0 - m_0d_1$ does not depend on the chosen end $\eend_n$.  Further, the ratio $M/D$ is also independent of the end, since $P/Q$ is by definition, and $M/D = \acc^{-1}(P/Q)$.  
Therefore the right hand side of the inequality \eqref{eq:DMin} is also  independent of the choice of end. Hence our arguments below apply with both choices.
\MS

\NI (iii)  
We observe  that in all cases  encountered below 
 the expression $m_1d_0 - m_0 d_1$ is a multiple of
$2n+3$. Further, the calculations can sometimes be simplified by noting that the sign of $m_1d_0 - m_0 d_1$ determines whether $m_0/d_0$ is greater or less than $ m_1/d_1$.  Since the function   $x\mapsto  \frac{s-rx}{|sx-r|}$ 
is monotonic, it is sometimes enough to calculate 
the simpler expression $\frac{sd_0-rm_0}{|sm_0-rd_0|}$ instead of $\frac{sD-rM}{|sM-rD|}$.

\NI (iv) 
If $\Ss = (\bE_k)$ is a pre-staircase whose steps have centers $p_k/q_k$, then we saw in Proposition~\ref{prop:stairblock} that the classes $\bE_k$ are perfect blocking classes for large $k$.  The corresponding family of blocked $b$-intervals $J_{\bE_k}$ 
cannot include the limit $b_\infty$
and if the $b$ values are  $> 1/3$ (resp.\ $< 1/3$) these intervals ascend if and only if the centers ascend  (resp.\ descend).  Although, as we saw in Remark~\ref{rmk:block0}~(i), the values $b_k = m_k/d_k$ do not have to lie in these blocked   $b$-intervals,  they cannot be too far away.  Hence one would expect them to increase for ascending (resp.\ descending) staircases with $b$-values  $> 1/3$ (resp.\ $< 1/3$.)  
However, this does not seem to be true.  For example the ascending  staircases $\Ss^U_{\ell,n}$ considered in 
Lemma~\ref{lem:SsUellbest} have $b$-values in $(1/3,1)$ but yet the sequence  $b^U_{\ell,n,k}$ decreases.
Since  this implies that $b^U_{\ell,n,k}> b^U_{\ell,n,\infty}$, this implies that $b^U_{\ell,n,k}\notin J_{\bE^U_{n,k}}$.
A similar phenomenon happens with the ascending staircases  $\Ss^L_{\ell,n}, \Ss^E_{\ell,n} $ with $b$-values in $(0,1/3)$ 
and yet increasing $b^L_{\ell,n,k}, b^E_{\ell,n,k}$.  On the other hand, the descending staircases do not seem to exhibit this behavior, though in these cases it can require more work to eliminate potential  overshadowing classes. 
\hfill$\er$
\end{rmk}

\begin{lemm}\label{lem:SsUellbest}  The pre-staircase $\Ss^U_{\ell,n}$ satisfies condition (ii) in Theorem~\ref{thm:stairrecog} for all $n\ge 1$.
\end{lemm}

\begin{proof} By Remark~\ref{rmk:DMineq}~(ii) we may suppose that $\eend_n = 2n+4$.  
Then, using the notation in that remark,
 we have from  Lemma~\ref{lem:SsUellu} that
$d_0: = d^U_{\ell,n,0} = n+2, m_0: = m^U_{\ell,n,0} = n+1$ 
and
$$
d_1 = 4n^3 + 20 n^2 + 30 n + 13,\quad m_1 = 4n^3 + 16 n^2 + 18 n + 5.
$$
Hence 
\begin{align*}
m_1d_0 - m_0 d_1&=(4n^3 + 16 n^2 + 18 n + 5)(n+2) -  (4n^3 + 20 n^2 + 30 n + 13)(n+1)\\
& =-(2n+3).
\end{align*}

Therefore, by Lemma~\ref{lem:DMineq}~(iii), we must check that
\eqref{eq:DMin1} holds for suitable $r,s$, i.e. for each $n$ we need
$$
\frac{|m_1d_0 - m_0 m_1|}{2n+3} = 1 < \frac{\sqrt{\si_n} (sD-rM)}{sM-rD}.
$$
But the function $x\mapsto \frac{s-rx}{sx-r}$ decreases for $x\in (r/s,1]$ with minimum value $1$.
    Hence for each $n$, this inequality holds for all $r/s <M/D = b^U_{\ell,n,\infty} $.
    \MS
    
 Finally, we show that there is no overshadowing class of degree $d' < sD/(sM - rD)$.
    Fix $n$ and abbreviate $b_n: = b^U_{\ell,n,\infty}$.    
    We first claim that $b_n>1/2$ for all $n\ge 1$. 
    Indeed, $ a^U_{\ell,1,\infty} = [7;3;\{7,3\}^\infty]>7$ while one can calculate using \eqref{eq:acc1} that
     $\acc_U^{-1}(7) \approx 0.596$.
   Hence, for each $n$ we may  choose $r/s>0$ so  that 
    $s/(sb_n- r) < 2$, or equivalently so that $0<2r< s(2b_n - 1)$.  
    Thus any overshadowing class would have to have degree $1$, and hence does not exist.
  \end{proof}

It remains to consider the descending pre-staircases $\Ss^U_{u,n}$.  
Again, it suffices by Remark~\ref{rmk:DMineq}~(ii) to treat the case $\eend_n = 2n+4$. 
It is convenient to deal first with the case 
$n=0$.
\MS

\begin{EXAMple}\label{ex:SsUu0}\normalfont
When $n=0$ the classes $\bE^U_{u,0,k}$ for $k=0,1$ have centers $[7;4]$ and $[7;5,1,4]$, and hence have parameters 
$$
(d,m, p,q) = (14,9,29,4),\quad (100,63, 208,29)\quad\mbox{ with }\ \si = 5.
$$
Since $$
X: = X' + X''\sqrt\si: =\frac {x_0}2 + \frac{2x_1 - x_0(\si+2)}{2(2n+3)\si} \sqrt{\si},
$$
 we have
\begin{align*}
& D = 7 + \frac{17}5 \sqrt 5 \approx 14.6,\qquad M = \frac 92 + \frac{21}{10}\sqrt 5 \approx 9.2, 
\end{align*}
Hence $b_\infty =  M/D 
\approx 0.6297 $, and 
$$m_1d_0 - m_0 d_1 = 63\cdot 14 - 9\cdot 100 = -18,$$
so that the sequence $m^U_{u,0,k}/d^U_{u,0,k}$ decreases as expected. 
Since $b_\infty >  1/2$ we may take   $r=1,s=2$, and one can check that this choice
gives the estimate $\frac{\sqrt 5 (2D-M)}{(2M-D)}> 6$ required by \eqref{eq:DMin1}. 

Further, since $b_\infty>5/8$ we have $2/(2b_\infty - 1) < 2/0.25 < 8$, so that  we only need rule out 
overshadowing classes of degree $d'< 8$ with $m' < d'/2$.
But  if $d' = 7$ this means that $m'\le 3$, 
so that  $d'b_\infty - m' > 1$.
 Hence, by Lemma~\ref{lem:perfect}~(i),  such a class cannot be obstructive for  $b$ in a nonempty interval of the form $ (b_\infty, b_\infty + \eps)$, and so is not overshadowing. 
  A similar argument rules out the cases $d' = 2,4,5,6$, while if $d'=3$ we must have $m' \le  1$ so that $\bE' = (3,1; 2,1^{\times 5})$ or  $(3,0; 2,1^{\times 6})$.
 But the centers of both these classes are less than the accumulation point $a_\infty$, so that they cannot overshadow.\footnote
 {
Notice that when $z> 6$ the obstruction from   $(3,1; 2,1^{\times 5})$ is no longer given by $z\mapsto \frac{1+z}{3-b}$.}
 Thus no overshadowing classes satisfy the given restrictions.  It follows that $\Ss^U_{u,0}$ is a staircase, by Theorem~\ref{thm:stairrecog}~(ii).
\hfill$\er$
\end{EXAMple}

\begin{lemm}\label{lem:SsUubest}  For each $n\ge 0$, the  pre-staircase $\Ss^U_{u,n}$  satisfies
condition (ii) in Theorem~\ref{thm:stairrecog}.
\end{lemm}

\begin{proof}  Again we carry out the argument for $\eend_n = 2n+4$, leaving the other case to the reader.    Since the case $n=0$ was treated in
 Example~\ref{ex:SsUu0}, we will suppose that $n\ge 1$.
  By \eqref{eq:SsUu} we have
\begin{align*}
\bigl(d_0, m_0\bigr)& =\bigl(2n^2 + 11n +14,\;\; 2n^2 + 9n + 9\bigr) = \bigl((2n+7)(n+2),\;\; (2n+3)(n+3)\bigr)
\\ \notag
d_1 &= 8n^4 + 68n^3 + 202 n^2 + 245n + 100 \\ \notag
& = (4n^2 + 12n +7)d_0 + 2 = (\si_n + 2)d_0 + 2,\\ \notag
m_1 &= 8n^4 + 60 n^3 + 158 n^2 + 171 n + 63 = (\si_n + 2)m_k.
\end{align*}
Therefore
$$
m_1d_0-m_0d_1 = -2m_0  < 0,
$$
so that  by \eqref{eq:DMin} the $b_k$ decrease, and we are case (iii) of Lemma~\ref{lem:DMineq}.
Further,
\begin{align*}
\frac M D \le \frac{m_0}{d_0} = \frac{2n^2 + 9n+9}{2n^2 + 11n + 14}= \frac 1{1+y}, \quad\mbox{where }\  y: = \frac{2n+5}{2n^2 + 9n + 9}.
\end{align*}
Therefore, because the function $x\mapsto \frac{2-x}{2x-1}$ is decreasing for $1/2 < x\le 1$, we have 
\begin{align*} \frac{2D-M}{2M-D} & > \frac{2d_0-m_0}{2m_0-d_0} = \frac{1+2y}{ 1-y} \\
&  \ge 1+3y  = \frac {2n^2 + 15n+24}   {(2n+3)(n+3)} > \frac{2(n+4)}{2n+3}.
\end{align*}
Therefore,  because  $2n+2< \sqrt {\si_n} < 2n+3$, we have 
$$ \frac{|m_1d_0-m_0d_1|}{2n+3} = 2(n+3) \le 2(n+1)\frac{2(n+4)}{2n+3} \le  \sqrt {\si_n}\frac{2D-M}{2M-D} $$
where the first inequality uses  $ (n+3)(2n+3) \le 2(n+1)(n+4)$ for all $n\ge 1$.

This shows that condition \eqref{eq:DMin1}, or equivalently \eqref{eq:DMineq}, holds with $r/s = 1/2$.  

When $n=1$ we can use \eqref{eq:recurX1} and the above formulas for $(d_i,m_i)$ to calculate $b_\infty\approx 0.738> 0.7$.  Hence $2/(2b_\infty - 1)<  5$.  Therefore it remains to show there are no overshadowing classes of degree $d'\le 4$ and $m'/d' < 1/2$.  But the arguments in Example~\ref{ex:SsUu0}
still apply when $n>0$.  This completes the proof.  
\end{proof}
\MS

 \begin{proof}[Proof of Theorem~\ref{thm:U}]  This follows from Lemmas~\ref{lem:SsUellu},~\ref{lem:SsUellbest}, and~\ref{lem:SsUubest},
 and Corollary~\ref{cor:CrU}, using the reasoning explained at the beginning of \S\ref{ss:thmU}.\end{proof}

\MS

 \begin{proof}[Proof of Theorem~\ref{thm:block}]  Most of the claims in this theorem are restated in Theorem~\ref{thm:U}, and hence are proved above.  It remains to check the  explicit formulas for the endpoints of $J_{\bB^U_n}$ and $I_{\bB^U_n}$:
\begin{align*}
\be_\ell(n) = \frac{(2n^2+6n+3)-\sqrt{\si_n}}{2n^2+6n+2},\quad \frac{(n+3)\left(3n+7+\sqrt{\si_n}\right)}{5n^2+30n+44} = \be_u(n),\\
\al_\ell(n) =  \frac{\sigma_n + (2n+3)\sqrt{\sigma_n}}{2(2n+1)}, \quad  6+\frac{\sigma_n + (2n+3)\sqrt{\sigma_n}}{2(2n+5)}  = \al_u(n).
\end{align*}
By Theorem~\ref{thm:blockrecog},  $\be_\ell(n) = M/D$ where $M,D$ are calculated for $\Ss^U_{\ell,n}$, while
$\be_u(n) = M/D$ where $M,D$ are calculated for $\Ss^U_{u,n}$.
Thus by Lemma~\ref{lem:SsUellbest}, and taking $\eend_n = 2n+4$ we have
  \begin{align*} 
   \be_\ell(n) = \frac
   {(n+1) (2n+3)\sqrt{\si_n} + (4n^3 + 16n^2 + 17n + 3)}
   {(n+2) (2n+3)\sqrt{\si_n} + (4n^3 + 20n^2 + 29n + 12)},
  \end{align*}
  and one can check that this does simplify to the desired expression.
Similarly, Lemma~\ref{lem:SsUubest} shows that
  \begin{align*} 
   \be_u(n) = \frac
   {(n+3)(2n+3) \bigl((2n+3)\sqrt{\si_n} + (\si_n + 2)\bigr)}
   {(n+2)(2n+7)) \bigl( (2n+3)\sqrt{\si_n} + (\si_n + 2)\bigr) + 4},
  \end{align*}
and one can again check that this does simplify to the desired expression.

The closed formulas for $\al_\ell(n), \al_u(n)$ can be obtained in a similar way from the ratios $P/Q$, or by simply calculating $\acc(\be_\ell(n)), \acc(\be_u(n))$.  Further details are left to the interested reader. 
\end{proof}

\subsection{Proof of Theorems~\ref{thm:L},~\ref{thm:E},~\ref{thm:stair} and~\ref{thm:blockstair}}\label{ss:thmLE} 
Since the statement of Theorem~\ref{thm:E} includes those of Theorems~\ref{thm:stair} and~\ref{thm:blockstair}, it suffices to prove the first two theorems.
\MS

We will first consider the staircases associated to the classes $$
\bB_n^L = \Bigl(5n,n-1; 2n\bw(\frac{12n+1}{2n})\Bigl), n\ge 1.
$$
  (Notice that the case $n=0$ makes no sense here.)
We follow the procedure laid out at the beginning of \S\ref{ss:thmU}.
Although we could make the arguments completely afresh, ignoring the connection between the classes $\bB^U_n$ and $\bB^L_n$, we will use the fact from
Corollary~\ref{cor:symm} 
 that the reflection $\Psi: w\mapsto \frac{6w-35}{w-6}$ takes the centers of the classes in the staircases  $\Ss_{\bullet,n}^U$ to those 
of  $\Ss_{\bullet',n}^L$, where $\ell':=u, u': = \ell$, i.e. it interchanges the ascending and descending staircases.

We saw in Example~\ref{ex:ctfr0} that the integers $p,q$ calculated from a continued fraction expansion $[\ell_0;\dots,\ell_N]$ are always coprime.  Further,  if $\gcd (p,q) = 1$ then $
\gcd(6p-35q,p-6q) = 1$
 since any common divisor of $6p-35q,p-6q$ also divides $6p-35q - 6(p-6q) = q$ and hence also divides $p$.
Since $\Psi^{-1} = \Psi$, this proves the following:
\begin{align}\label{eq:PsiPQ}
\Bigl(\gcd (p,q) = 1,\ \Psi(\frac pq)=\frac {\widehat p}{\widehat q} \Bigr) \Longleftrightarrow \ \Bigl( \gcd ({\widehat p},{\widehat q})=1,\  {\widehat p}=6p-35q, {\widehat q} = p-6q\Bigr).
\end{align}

\MS

  To prove Step 1 we simply need to check that the classes $\bB^L_n$ satisfy the Diophantine  equations ~\eqref{eq:diophantine}, which can be done by an easy computation.
  
We move on to  Step 2, which claims that 
the staircase classes are perfect.   The next lemma shows that they are well defined  (i.e. that the degree $d$ is always an integer), and satisfy the linear Diophantine identity.   The proof that they are perfect is given in Proposition~\ref{prop:CrEL}.  One could also use Lemma~\ref{lem:dquad} to check  the quadratic Diophantine identity; however,  since this follows from
 Proposition~\ref{prop:CrEL} we do not do that here.
 Note that there is an analog of Lemma~\ref{lem:SsL} when $\eend_n = (2n+5,2n+2)$, but in view of Lemma~\ref{lem:otherend} we do not need it.

\begin{lemm}\label{lem:SsL} \begin{itemize}\item[{\rm (i)}]
For $n\ge 2$ and $\eend_n = 2n+4$, the pre-staircase $\Ss^L_{u,n}$ can be extended by the class
 for $k = -1$ with coefficients
\begin{align} \label{eq:SsLu-1}
& \bigl(5(n-1),(n-2); (2n-2)\bw(\frac{12n-11}{2n-2})\bigr).
\end{align}
\item[{\rm (ii)}] The classes in  $\Ss_{\ell,n}^L$ and  $\Ss_{u,n}^L$ are well defined and satisfy the linear Diophantine identity for
$n\ge 1$ and all $k$.
\end{itemize}\end{lemm}
\begin{proof}   
The classes in  $\Ss_{u,n}^L$ have  centers
\begin{align*}
a^L_{u,n,k} &= [6;2n-1, 2n+1,\{2n+5,2n+1\}^k,\eend_n],\qquad n\ge 1, k\ge 0, \\
& = \Psi([2n+5; 2n+1, \{2n+5,2n+1\}^k,\eend_n] = \Psi(a^U_{\ell,n,k+1}),
\end{align*}
where the second identity holds by Lemma~\ref{lem:symm}~(i:b).  Notice here that the term in $k$ for $\Ss^L_{u,n}$ corresponds to the term in $k+1$ for $ \Ss^U_{\ell,n}$.  Therefore, there should be a term in
$\Ss^L_{u,n}$ with center $\Psi(\eend_n)$. When $\eend_n = 2n+4$, the center is 
$$
\Psi(2n+4) = \frac{6(2n+4) - 35}{2n+4-6} = \frac{12n-11}{2n-2}.
$$
  Therefore,
because the relation in $
\Ss^L_{u,n}$ is $$
(2n+3) d_{n,k} = -(n-1) p_{n,k} + (11n+2)q_{n,k},
$$
 this entry is
$\bigl(5(n-1),n-2; (2n-2)\bw(\frac{12n-11}{2n-2})\bigr)$.
This is a quasi-perfect class, as can be verified by direct calculation. 
Since the (Relation) holds by definition, to
complete the proof that this class can be added to $\Ss^L_{u,n}$ as in Remark~\ref{rmk:recur}~(ii), we must check that the (Recursion) holds for the triple $k = -1,0,1$, i.e. that
$$
x_1 = (\si+2) x_0 - x_{-1},\qquad x = d,p,q,m.
$$  
Since $d_\bullet, m_\bullet$ are linear combinations of $p_\bullet, q_\bullet$, it suffices to prove this for $p_\bullet, q_\bullet$.  But by definition $p_k, q_k$ are linear functions of
$p'_k,q'_k$, where $p'_k/q'_k = [\{2n+5,2n+1\}^{k+1}, 2n+4]$, 
 and $p_k',q_k'$ satisfy the recursion by Lemma~\ref{lem:Cr0}~(i).    This justifies our extending $\Ss^L_{u,n}$ by this class, and hence proves (i).

Since the linear Diophantine equality is always satisfied by definition of $m$, to prove (ii) we must check that $d$ as defined by the staircase relation is integral for each $n$.  By Lemma~\ref{lem:otherend}~(i) it suffices to do this in the case $\eend_n = 2n+4$, and for only two adjacent values of $k$ (either $k= 0,1$ or $k=-1,0$).
In  the case of $\Ss^L_{u,n}$  with $n\ge 2$, we have the integral vectors
\begin{align}\label{eq:SsLu0}
(d^L_{u,n,-1},m^L_{u,n,-1}, p^L_{u,n,-1},q^L_{u,n,-1})& = \bigl(5(n-1),n-2, 12n-11, 2n-2\bigr); \\ \notag
(d^L_{u,n,0},m^L_{u,n,0}, p^L_{u,n,0},q^L_{u,n,0}) & =  \bigl( 20n^3 + 40n^2 + 6n -1, 4n^3 + 4n^2 - 6n-1. \\ \notag
&\qquad  48n^3+ 100n^2 +22n -1,8n^3+ 16n^2 + 2n -1 \bigr).
\end{align}
Further, when $n=1$, although the tuple
$(0,-1;1,0)$ obtained by setting $n=1$ in \eqref{eq:SsLu-1} makes no sense geometrically, it does make numerical sense, and so $d^L_{u,1,k}$ is integral for each $k$ as well.

As for $\Ss^L_{\ell,n}$ with $\eend_n = 2n+4$, one can check that its initial classes  are
\begin{align}\label{eq:SsLell} 
&(d^L_{\ell,n,0},m^L_{\ell,n,0}) = \bigl(10n^2 + 25n + 13 ,\ 2n^2 + 3n\bigr)\\ \notag
&  (d^L_{\ell,n,1},m^L_{\ell,n,1})  = \bigl(40n^4 + 220 n^3 + 422 n^2 + 331n + 89,\ 8n^4 +36n^3 + 50n^2 + 21n\bigr).
\end{align}
Therefore $d^L_{\ell,n,k}$ is integral for all $k, n$ as claimed.
\end{proof}

Step 3 again follows immediately from the definitions.  (For more details about how to check this, see this step in \S\ref{ss:thmU}; it occurs just after Lemma~\ref{lem:SsUellu}.) 
Hence it remains to prove the analog of Lemmas~\ref{lem:SsUellbest} and~\ref{lem:SsUubest}.
Again it suffices to carry out the argument only for the case $\eend_n = 2n+4$ since the other case will follow by Remark~\ref{rmk:DMineq}~(ii).

\begin{lemm}\label{lem:SsLellbest} For all $n\ge 0$, the  pre-staircase $\Ss^L_{\ell,n}$  satisfies
condition (i) in Theorem~\ref{thm:stairrecog} with appropriate $r/s$.
\end{lemm}
\begin{proof} We first check the sign of $m_1d_0-m_0d_1$.
By \eqref{eq:SsLell} we have
\begin{align*}
m_1d_0-m_0d_1& = \bigl(8n^4 +36n^3 + 50n^2 + 21n\bigr)\bigl(10n^2 + 25n + 13\bigr)\\
& \qquad\quad  -
\bigl(40n^4 + 220 n^3 + 422 n^2 + 331n + 89\bigr)\bigl(2n^2 + 3n\bigr)\\
& = 4n^2+6n = (2n+3)2n> 0 \qquad \mbox { since } n>0,
\end{align*} 
Therefore we are in case (ii) of Lemma~\ref{lem:DMineq}, and, by \eqref{eq:DMin1}, need to find $r/s> M/D$
so that 
\begin{equation}\label{eq:rsineq}
\frac{\sqrt{\si} (sD-rM)}{rD-sM} > 2n, 
\end{equation}
where $M: = M_\ell(n), D: = D_\ell(n)$ are the values appropriate for  $\Ss^L_{\ell,n}$, and $\si: = \si_n$.
Since $(2n+2)^2< \si = (2n+1)(2n+5) <(2n+3)^2$, it suffices to have $\frac{sD-rM}{rD-sM} > 1$, or equivalently
$(s-r)D > (s-r)M$.    But this holds for all $0<r<s$, showing that  (\ref{eq:rsineq}) holds for all such $r, s$.
 Therefore, for each $n$ we are free to choose $1>r/s > M/D$ to minimize the degree $d'< sD/(rD-sM)$ of a potentially  overshadowing class. 
 
 We now claim that $M/D < 1/5$.  One could simply prove this by direct calculation.  Alternatively, recall that we have already proved that this staircase is associated to the blocking class $\bB^L_n$ so that $
 M/D$ is one of the endpoints (in fact the upper one) of the blocked interval $J_{\bB^L_n}$, which as one can readily check is a subset of $(0,1/5)$.    
  Hence we can choose any $r/s >1/5$.  
 In particular, if we take $r/s= 7/10$ then 
 $$
 \frac{ sD}{rD-sM} =  \frac{ s}{r-sM/D} < \frac{ s}{r-s/5} = 2.
 $$
 But the only exceptional classes of degree $d=1$ are $L-E_i-E_j$, which do not overshadow.  
 Hence there are no overshadowing classes, and the proof is complete.
\end{proof}

\begin{lemm}\label{lem:SsLubest} For all $n\ge 1$, the  pre-staircase $\Ss^L_{u,n}$  satisfies
condition (i) in Theorem~\ref{thm:stairrecog}.
\end{lemm}
\begin{proof}    
As we pointed out in \eqref{eq:md00} in cases when we can extend the staircase by a term with $k = -1$ we have
$m_0d_1 - m_1d_0  = m_{-1}d_0 - m_0d_{-1}$.  Therefore we can apply \eqref{eq:DMin} using the terms with $k =-1,0$.
By \eqref{eq:SsLu-1}  and \eqref{eq:SsLu0} we have
\begin{align*}
m_0d_{-1}-m_{-1}d_0 & = (4n^3 + 4n^2 - 6n-1)5(n-1) - (n-2)(20n^3 + 40n^2 + 6n -1)\\
& = 24n^2 + 38 n + 3 = (2n+3)(12n+1)> 0.
\end{align*}
Hence, by Lemma \ref{lem:DMineq} (i) and \eqref{eq:DMin}, for each $n$ the sequence $(b_{n,k})_{k\ge -1}$  increases,   which is what we would expect from a descending staircase with $b$-values $< 1/3$ (but see Remark~\ref{rmk:DMineq}~(iv).)

Next, we must check to see if \eqref{eq:DMin1}  holds with appropriate $r/s$.   Thus we need
\begin{equation}\label{eq:rsbounds}
\frac{m_0d_{-1}-m_{-1}d_0}{\sqrt{\si_n}(2n+3)} = \frac{12n+1}{\sqrt{\si_n}} < \frac {s-r\frac MD}{r-s\frac MD}.
\end{equation}
Since $2n+2< \sqrt{\si_n}$, we have $\frac{12n+1}{\sqrt{\si_n}} < 6$ for all $n\ge 1$.  However there are better upper bounds for small $n$:
\begin{align}
\frac{13}{\sqrt{\si_1}}< 2.84,\;\; \frac{25}{\sqrt{\si_2}} <3.73,\;\;  \frac{37}{\sqrt{\si_3}} <4.22\label{eq:ubs}
\end{align}
Notice also that the values $b_n: = b^L_{u,n,\infty} =M/D$ increase with $n$
because $\acc(b_n) = a^L_{u,n,\infty}$ decreases with $n$ (as one can see because the centers of the associated blocking classes $\bB^L_n$ decrease).  Further, for each $n$  we have  
$$\acc^{-1}_L(\frac{12(n-1)+1}{2(n-1)})< b_n< \acc_L^{-1}(\frac{12n +1}{2n}),$$
 since $\frac{12n +1}{2n}$ is the center point of the blocking class $\bB^L_n$ and $\acc_L^{-1}$ reverses orientation.
Therefore
\begin{equation}\label{eq:bbounds}
0< b_1 < \acc_L^{-1}(\frac{13}{2}) < b_2  <  \acc_L^{-1}(\frac{25}{4})  < b_3 < \acc_L^{-1}(\frac{37}{6})  <  b_4, 
\end{equation}
where, using \eqref{eq:acc1} to compute $ \acc_L^{-1}(\frac{12n +1}{2n})$ for $n = 1,2,3$, we find they are approximately $0.064, 0.121$, and $0.144$.

It turns out that if we choose $r/s =3/10$, we can satisfy the inequality
$\frac{12n+1}{\sqrt{\si_n}} < \frac {s-r b_n}{r-sb_n}$ for all $n$, while the requirement in Lemma~\ref{lem:perfect}~(i)
 that an
obstructive  class $\bE' = (d',m';\bbm')$  at $b_n$ have  $|b_nd'-m'|< 1$ has no suitable solutions if
 $m'/d' >3/10$ and $d' \le \frac{10}{3-10 b_n}$.  Thus there is no overshadowing class.

Here are the details.  Notice that when $0\le x< r/s < 1$ the function
$x\mapsto \frac {s-r x}{r-sx}$ increases.
Therefore $ \frac {10- 3 b}{3-10b} > 3$ for all $b$, which 
in view of \eqref{eq:ubs} gives the  bound in \eqref{eq:rsbounds} for $n=1$.
In fact, one can check that  \eqref{eq:rsbounds} holds for all $n$ because, using the  bounds \eqref{eq:bbounds} for $b_i$ and the  bounds in \eqref{eq:ubs} for 
$\frac{12n+1}{\sqrt{\sigma_n}}\le 6$, we have
\begin{align*}
 \frac {10- 3 b_2}{3-10b_2}> 4,\qquad   \frac {10- 3 b_3}{3-10b_3}>5,\qquad  \frac {10- 3 b_n}{3-10b_n}>6,\  \forall\;\; n\ge 4. 
\end{align*}
Next observe that the maximum degree of $d'$ of a potentially overshadowing class is $<\frac{10}{3-10 b_n} < 10$ 
because $b_n=M/D<1/5<3/10$.
If $d' = 7,8,9$ the inequality $m'/d'>3/10$ implies that $m'\ge 3$, so that $$
|bd' - m'| = m'-d'b \ge  3- 9/5 > 1.
$$
Therefore there are no solutions of these degrees.  
Similarly, one can check that there are no solutions if $d'=4,5$, so that $m'\ge2$, or if $d'=3, m'=2$. If $d'\le3$ with $m'=1$ then the only possible solution 
is $\bE'=(3,1;2,1^{\times5})$, but as pointed out in Example~\ref{ex:big} this class is not even obstructive at its break point $z=6$ when $b< 1/5$.

However, if
$d' = 6$, then the inequality  $m'/d' > 3/10$ is satisfied with  $m' = 2,3$, and if $m'=2$ there are $b<1/5$ such that
$|6b - 2|<\sqrt {1-b^2}$.   We rule this case out as follows.
The break point $a'$ of any overshadowing class $\bE'$  must be $> 6$, and hence have $\ell(a')\ge 7$. 
Further, because  $\bw (a')$ starts with a block of
length at least $6$, it follows  from Lemma~\ref{lem:obstruct}~(ii) that 
  the first $6$ entries of $\bbm'$ are $k^{\times 5}, k'$ where $|k'-k|\le 1$.  
  The identity $\sum m_i=3d'-m'-1=15$
    shows that $k=1$ or $2$, while the quadratic identity shows that $k\ne 1$.  
  Therefore we would have to have $\bE' = 
  (6,2;3, 2^{\times 6})$, with break point $7$; to see this notice that $7$ is the unique point $a>6$ with $\ell(a) = 7$.  Arguing as in the proof of Lemma~\ref{lem:munearc}, we find that
  $$
  \mu_{\bE',b}(z) = \frac{1 + 2z}{6-2b} <  \frac{1 + z}{3-b},\quad \mbox{ for } 6<z<7.
  $$
  But by \eqref{eq:MDvol}  when $z = \acc(b)\in (6,7)$ we have $V_b(z) =  \frac{1 + z}{3-b}$.  Hence the constraint $  \mu_{\bE',b}$ never meets the volume constraint at the accumulation point $z = \acc(b)$ and so cannot be an overshadowing class.
  This completes the proof.
\end{proof} 

\MS

\begin{proof}[Proof of Theorem~\ref{thm:L}]  This holds by 
Lemmas~\ref{lem:SsL},~\ref{lem:SsLellbest} and~\ref{lem:SsLubest} 
together with Proposition~\ref{prop:CrEL}, using the proof scheme explained at the beginning of \S\ref{ss:thmU}.
\end{proof}
\MS

Now consider the staircases $\Ss^E_{\bullet, n}$  with blocking classes 
$$
\bB_n^E =\bigl(5(n+3), n+4; (2n+6)\bw(\frac{12n+35}{2n+6})\bigr).
$$
Since we are given simple formulas for the components of $\bB_n^E$, one can readily check that these are Diophantine classes, and hence quasi-perfect.  

For Step 2, we must check that the staircase classes are quasi-perfect.   The argument in Lemma~\ref{lem:SsL} adapts in a straightforward way.  In particular, as we show in Lemma~\ref{lem:SsEellbest} below, the staircase $\Ss^E_{\ell,n}$ can again be extended by a class defined for $k=-1$, which simplifies the calculations. We will leave it to the reader to check the further details of this step.
 
 Step 3  involves checking that the relations in the pre-staircases are those associated to $\bB_n^E$
 as described in Theorem~\ref{thm:blockrecog}, and so follows by inspection.
However, the last step, which establishes that the staircase classes are live at the limiting $b$-value,  is somewhat more tricky, specially in the case $\Ss^E_{u,n}$.  First consider the ascending pre-staircases $\Ss^E_{\ell, n}$.  
By Corollary~\ref{cor:symm} they are the image by the shift $Sh$ of the corresponding staircases $\Ss^U_{\ell,n}$, and have $b$-values in the range $(1/5,1/3)$.

\begin{lemm}\label{lem:SsEellbest}  The   pre-staircases $\Ss^E_{\ell, n}, n\ge 1,$ satisfy  condition (i) in Theorem~\ref{thm:stairrecog}.
\end{lemm}
\begin{proof} 
We first claim that,   as with  the staircase $\Ss^L_{u, n}$ considered in Lemma~\ref{lem:SsL}, when $n\ge 1$ we can extend each $\Ss^E_{\ell, n}$ by a class $\bE^E_{\ell,n,-1}$ with center $Sh(2n+4) = \frac{12n+23}{2n+4}$ and hence with parameters 
$$  (5n+10, n+3; 12n+23, 2n+4).  $$
  To see this, note that
 by transforming  the formulas for $p/q$ in \eqref{eq:SsUell} by $Sh$ and then using the linear relation
 (or by direct computation), the parameters of this staircase are
\begin{align*} (d_{-1},m_{-1}) &= (5n+10, n+3)\\
 (d_0,m_0) &= (20n^3+ 100n^2 + 156 n + 74, 4n^3 + 24n^2 + 44n + 24).
 \end{align*}
 As in Lemma~\ref{lem:SsL} it is straightforward to check that entries $p_k,q_k$ of these staircases satisfy the recursion for $k = -1,0,1$, so that $d_k, m_k$ do as well.
This  justifies 
our extension of the staircase by a term with $k=-1$.

We may now calculate 
$$\frac{m_0 d_{-1}- m_{-1}d_0 }{2n+3}=   2n+6> 0.$$
Therefore, by \eqref{eq:md00}, and \eqref{eq:DMin},  we are in case (ii) of Lemma~\ref{lem:DMineq}.  
Next notice that, taking  $r=1, s=2$, we have
$$\frac{m_0 d_{-1}- m_{-1}d_0 }{\sqrt{\si_n}(2n+3)} \le \frac{2n+6}{2n+2} \le 2 < \frac{2D-M}{D-2M}, \qquad \forall\ n\ge 1, \ 0\le \frac MD < \frac 12.$$
Therefore, because in fact $M/D< 1/3$ it remains to show there is no overshadowing class with $d'<2/(1-2/3) = 6$ 
and $m'/d'> 1/2$.  Thus we must have
$(d',m') = (3,2), (4,3), (5,3)$ or $(5,4)$.  But no such classes can be live at $b<1/3$ since in none of these cases is
$|d'b - m'|< 1$, which is a necessary condition for a class to be obstructive by Lemma~\ref{lem:perfect}~(i).
This completes the proof.
\end{proof}

It requires somewhat more work to rule out overshadowing classes for the staircases $\Ss^E_{u,n}$.
The following remark will be helpful in understanding the possible break points of such a class since it explains how the order on $\Q$ interacts with continued fraction expansions.

\begin{rmk}\label{rmk:ctfr}    \normalfont
Each rational number $a>1$ has a continued fraction representation $[\ell_0;\ell_1,\dots,\ell_N]$, where  $\ell_i$ is a positive integer, and $N>0$ unless $a\in \Z$. 
By convention, the last entry in a continued fraction is always taken to be $> 1$, since $[\ell_0;\ell_1,\dots,\ell_N, 1] = [\ell_0;\ell_1,\dots,\ell_N+ 1]$; for example
$$[1;3,1] = 1 + \frac 1{3 + \frac 11} =  1 + \frac 1{4} = [1;4].$$
If $N$ is odd, then $[\ell_0;\ell_1,\dots,\ell_N]> [\ell_0;\ell_1,\dots,\ell_N + 1]$, in other words, if the last place is odd,  increasing this  entry decreases the number represented.
For example, because $2/13< 3/19$, we have
\begin{align*}
[1;3,6,2] = 1 + \frac 1{3 + \frac 1{ 6 + \frac 12}} = 1 + \frac 1{3 + \frac 2{13}}  >
[1;3,6,3] = 1 + \frac 1{3 + \frac 1{ 6 + \frac 13}} = 1 + \frac 1{3 + \frac 3 {19}},
\end{align*}
Further, although increasing an even place usually  increases the number represented, this rule does not hold if one increases an even place\footnote{This place would necessarily be the last; and notice that the initial place is labelled $0$ and hence is considered to be even.} from $0$ to a positive number. For example 
\begin{align}\label{eq:ctfr4}
[1;4] = 1+\frac 14 > [1;4,2] = 1+\frac 1{4 + \frac 12} = 1 + \frac 29 > [1;4,1]= [1;5] = 1 + \frac 15.
\end{align}
Similarly, if one increases an odd place from $0$, the corresponding number increases.  In particular, 
\begin{align}\label{eq:odd}
[\ell_0;\ell_1,\dots,\ell_N]> [\ell_0;\ell_1,\dots,\ell_N + 1,\dots] \quad\mbox { if $N$ is odd},\  \ell_N\ge 2 
\end{align} 
whatever the last entries are.\hfill$\er$
\end{rmk}

\begin{EXAMple}\label{ex:SsEu0}  \normalfont
Consider the pre-staircase $\Ss^E_{u,0}$.  
Its first two steps  (with $\eend_0=4$) are
$$\bigl(73,20; 29\bw(\frac{170}{29}) \bigr),\quad \bigl(524, 145;  208\bw(\frac{1219}{208})\bigr).$$
Using \eqref{eq:recurX}, one can calculate
\begin{align*}
&d_0 = 73,\quad m_0 = 20,\quad d_1 = 524,\quad m_1
= 145,\\
&D = \frac{73}2 + \frac{179}{10}\sqrt 5,\qquad M = 10 + 5\sqrt 5,\qquad \frac MD = b_{\infty} \approx 0.27677.
\end{align*}
Therefore 
$$\frac{m_1 d_0 - m_0 d_1}{3}  =  35< \sqrt 5 \frac{3D-M}{D-3M}.$$
Thus we are in case (ii) of Lemma~\ref{lem:DMineq} with $r=1, s=3$,
and must check that there are no overshadowing classes
with 
$$d'< \frac {3}{1-3b_{\infty}} < 18, \quad m'/d'  > 1/3, \;\; \mbox{ and } \;\; |m' - d'b_\infty| < 1.$$ 
 It is straightforward to check that the only possibilities for $(d',m')$ are
$$(d',m') = (11,4),\quad (8,3),\quad (5,2), \quad (4,2), \quad (2,1), \quad (1,1). $$
For example, the pairs $(17,6), (14,5)$  satisfy the first two requirements, but not the third.

If $a'$ is the break point of the overshadowing class $\bE'$, then  $a'> a_{\infty}$ because $\bE'$ is 
right-overshadowing, in other words it is obstructive for $z> a_\infty$ to the right of $a_\infty$; see Definition~\ref{def:big}.  Further,
we must have $\ell(a') < \ell(z)$ for all $z\in [a_{\infty},a')$ where $a_\infty\approx 5.86\dots $ is the accumulation point of the staircase.    Therefore we cannot have $a'>6$.
Since $$
[5;2]< [5;1,6] < a_\infty = [5;1,6,\{5,1\}^\infty] < a',
$$
 it follows from Remark~\ref{rmk:ctfr} that the possibilities for $a'$ in decreasing order are:
$$
6 >  [5;1,7] > [5;1,6,4] > [5;1,6,5]> [5;1,6, 5, 2]>\dots 
$$
Notice for example that  Lemma~\ref{lem:obstruct} implies that we cannot have $a'= [5;1,6,3]$ or $[5;1,6,2]$
because, with $z = [5;1,6,5]$,  $\bbm'$   would differ from $\bw(z)$ in at least two 
places of the last block.  
\MS

Suppose first that $(d',m')=(11,4)$.
If $a'=6$ then, by Lemmas~\ref{lem:0}~(iii) and \ref{lem:obstruct},  $\bbm' = (k^{\times 5}, k + \eps)$ where $\eps \in \{-1,0,1\}$.  Therefore 
the linear Diophantine equation gives   $$
6k + \eps = 3d'-m'-1 = 28,
$$
 which has no solution.  On the other hand, 
if $a' = [5;1,7]$ the tuple $\bbm' = (4^{\times 5}, 1^{\times 8})$ does satisfy this equation.  However, it does not satisfy the quadratic identity  $(d')^2 - (m')^2 = \sum m_i^2 -1$.
A good way to think of this quadratic identity is as follows.  
An exceptional curve in class $\bE'$ is  obtained by blowing up the singular points of a rational curve $C$ of degree $d'$ with branch points of orders $m', m_1',\dots,m_n'$, where $\bbm' = (m_1',\dots,m_n')$ and the order of a branch point is the number of branches through that point.   Thus a generic perturbation of a branch point of order $k$ has $\frac {k(k-1)}2$ nearby simple double points.   The quadratic identity can be  reformulated into the statement that, when $C$ is perturbed  so that it only has simple double points, then the total number of these double points is $\frac{(d'-1)(d'-2)}2$.  Since a curve of degree $11$ has $45$ double points, and the coefficient $m'= 4$ accounts for $6$ of them, the tuple   $\bbm'$ must  therefore account for $39$ double points.  But the tuple 
 $(4^{\times 5}, 1^{\times 8})$ only accounts for $5\times 6 = 30$ double points.   
 Further,  if $\ell(a')$ were longer any possible choices for $\bbm'$ would have even fewer double points.  
 Indeed, when $\bbm'$ is subject to the linear Diophantine constraint, $\sum (m'_i)^2$ is at most 
 $\frac{(3d'-m'-1)^2}{\ell(\bbm')} < (d')^2 - (m')^2 + 1$ whenever $\ell(\bbm')\ge8$. 
 Therefore there are no solutions for $(d',m') = (11,4)$.

 Similar arguments rule out  the cases $(8,3)$ and $(4,2)$, while the degree of $(2,1)$ is simply too small for the break point to be $a'$.
If $(d',m') = (5,2)$, there is an exceptional divisor $(5,2; 2^{\times 5},1^{\times 2})$, but its  break point is $[5;2]< a_\infty$ which  is too small. 
Hence there are no overshadowing classes, so that
 the pre-staircase $\Ss^E_{u,0}$ satisfies condition (i) in Theorem~\ref{thm:stairrecog}.  Hence it is a staircase as claimed in Theorem~\ref{thm:E}.\hfill$\er$
 \end{EXAMple}

\begin{lemm}\label{lem:SsEubest}  The pre-staircases $\Ss^E_{u, n}, n\ge 0,$ satisfy  condition (i) in Theorem~\ref{thm:stairrecog}.
\end{lemm}
\begin{proof}  Since the case $n=0$ is treated in Example~\ref{ex:SsEu0}, we will suppose that $n\ge 1$.
One can calculate that
\begin{align*} (d_{0},m_{0}) &= \bigl(10n^2 + 55n + 73,\ 2n^2+ 13n+ 20 \bigr)\\
 (d_1,m_1) &= \bigl(40n^4+340n^3+1022n^2+1261n+524,8n^4+76n^3+250n^2+331n+145\bigr)
 \end{align*}
Hence 
\begin{align*}
m_1d_0 - m_0d_1 = 24n^2 + 106 n + 105 = (2n+3)(12n + 35) > 0.
\end{align*}
Therefore, because  $\sqrt{\si_n}= 2n+c$ where $2<c<3$, we have 
\begin{align*}
 \frac{m_1d_0 - m_0d_1}{\sqrt{\si_n} (2n+3)}
& = \frac{12n + 35}{\sqrt{\si_n}} = 6 + \frac{c'}{2n+c} \quad\mbox{ for some $c'>0$}.
\end{align*}
Thus this function decreases, with limit $6$.   Further, it is approximately $10.25$ when $n=1$, is $< 8.8$ when $n=2$, $ < 8.1$ when $n=3$, and $< 7.2$ when $n=6$, and $<7$ when $n\ge 8$ .

One can calculate that with $r/s = 1/3$ we have
\begin{align*}
&\frac{3d_0- m_0}{d_0-3m_0}  =  \frac{28n^2 + 152n +199}{4n^2 +16 n + 13 }
= 7 + \frac{40n+108}{4n^2 + 16n + 13}
\end{align*}
Therefore, because the $b_k$ increase,  
it follows from Remark~\ref{rmk:DMineq}~(ii) that 
condition (ii) in Lemma~\ref{lem:DMineq} will hold with $r/s = 1/3$ provided that
\begin{align*}
\frac{12n + 35}{\sqrt{\si_n}} < 7 + \frac{40n+108}{4n^2 + 16n + 13},\quad \mbox{ for all } n\ge 1.
\end{align*}
But this is easy to check from the estimates given above for $(12n + 35)/\sqrt{\si_n}$.
Indeed, this is immediate for $n\ge 8$  since we saw above that 
$\frac{12n + 35}{\sqrt{\si_n}}< 7$ in this case.
 If $n=1$ then we saw that $RHS \approx  10.25$ while
$LHS = 7 + 148/33 >11$.     The other cases follow by a similar  computation.

Hence it remains to check that there are no overshadowing classes with  $m'/d'>1/3$ and suitably small degree.
When $n=1$, we have $b_{1,\infty}\approx 0.25416$, so that $3/(1- 3b_{1,\infty}) < 13$.  Therefore we must rule out overshadowing classes of degrees $d'\le 12$, $m'/d'>1/3$, and $|m'-b_{1,\infty}d'|<1$.
Thus, the  possibilities for $(d',m')$ are:
\begin{align}\label{eq:dm0}
(d',m') &= (8,3),\;\; (5,2),\;\; (4,2), \;\; (2,1), \;\; (1,1).
\end{align}
Because the steps of $\Ss^E_{u,1}$ have centers $[5;1,8,\{7,3\}^k,6]$, it follows 
as in Example~\ref{ex:SsEu0} that the possibilities for the break point $a'$ of $\bE'$ are
$$
6> [5;1,9] >[5;1,8,6] >  [5;1,8,7,4] > \dots.
$$
But it is straightforward to check that, just as in Example~\ref{ex:SsEu0}, there are no suitable classes.  

When $n>1$ the proof is very similar.   Since $b_{n,\infty} \le b_{1,\infty}$, the only possible $(d',m')$ are among those listed in \eqref{eq:dm0}, while the possible values of $a' $ are $6> [5;1,7+2n]> \dots$ and so are longer than those for $n=1$. 
 Therefore as $n$ increases the argument gets easier, because in the case of $(d',m')=(8,3), (5,2)$, or $(4,2)$, we have $\frac{(3d'-m'-1)^2}{\ell(\bbm')}<(d')^2-(m')^2+1$ whenever $\ell(\bbm')\ge7$. Because $\ell( [5;1,7+2n])=12 + 3n$, 
the break point of $\bE'$ has to be $6$, and we showed above that there is no corresponding overshadowing class. As before, we cannot have $(d',m')=(2,1)$ or $(1,1)$.
Therefore, there are no overshadowing classes for any $n$, and the proof is complete.
\end{proof}

\begin{proof}[Proof of Theorem~\ref{thm:E}]  This holds by 
Lemmas~\ref{lem:SsEellbest},~\ref{lem:SsEubest} and Proposition~\ref{prop:CrEL}.
\end{proof}

\subsection{Cremona reduction}\label{ss:reduct}

This subsection is devoted to the proof of the following result.

\begin{prop}\label{prop:reduct} Each pre-staircase
$$
\Ss^L_{\ell,n},\Ss^U_{u,n},\Ss^E_{u,n},\quad n\ge  0, \qquad \Ss^L_{u,n}, \Ss^U_{\ell,n}, \Ss^E_{\ell,n},\quad n\ge  1
$$
mentioned in \S\ref{ss:Fib} is perfect, that is, it consists of exceptional classes.
\end{prop}
\begin{proof}  This is contained in Remark~\ref{rmk:notation}~(ii) (for the case $n=0$), Corollary~\ref{cor:CrU} for the 
pre-staircases labelled $U$, and Proposition~\ref{prop:CrEL} for the other cases.\end{proof}

We will use the 
 following recognition principle, which is explained in \cite[Prop.1.2.12]{ball},  for example.   

\begin{lemm}\label{lem:Cr1}  An integral class $\bE: = dL-\sum_{i=1}^N n_i E_i$ in the $N$ fold blowup $\CP^2\# N \ov{\C P}^2$ represents an exceptional divisor if and only if it may be reduced to $E_1$ by repeated application of  Cremona transformations.  \end{lemm}

 Here, a {\bf Cremona transformation} is the composition of the transformation
$$
\bigl(d, n_1,n_2,n_3,\dots\bigr)\mapsto \bigl(2d-(n_1+n_2+n_3), d-(n_2+n_3), d-(n_1+n_3), d-(n_1+n_2), \dots\bigr)
$$
with a reordering operation that is usually taken to put  the entries in nonincreasing order (but does not have to be this.)   If this reordering does put the entries in decreasing order, we call it a {\bf standard Cremona move}.

Note that because Cremona  transformations preserve the first Chern class and the self-intersection number, once we have proved that $\bE_k$ reduces to $E_1$ there is no need to make an independent check that  $\bE_k\cdot \bE_k = -1$.   Further, 
because Cremona moves are reversible (indeed if one omits the reordering step  they have order $2$),  they generate an equivalence relation on the set of solutions to the Diophantine equations. Hence any two classes with a common reduction are themselves equivalent.

\begin{rmk} \normalfont
(i)  {\bf (On notation)}  \label{rmk:notation}There are two useful notations for an exceptional class, namely  $\bE: = (d,m ;\bbm)$, 
and $(d, m_1,\dots,m_N)$ with $m_1\ge m_2\ge\dots\ge m_N$.
We use the former (with the semicolon) when we are interested in the geometry of the class as in \S\ref{sec:obstr}, while we use the latter (with no semicolon)  when, as in Example~\ref{ex:redUell1}, we are  interested only in its numerics and therefore include $m$ among all the other coefficients. \MS

\NI (ii) {\bf (The case $n=0$)} \  The three pre-staircases 
$\Ss^L_{\ell,0},\Ss^U_{u,0}, \Ss^E_{u,0}$ are rather special since they have no twin.  Indeed  $\Ss^L_{\ell,0}$ is not even defined in Theorem~\ref{thm:L}, although as we saw in Remark~\ref{rmk:Fib0} it can be identified with 
the Fibonacci staircase.
We will prove in Proposition~\ref{prop:CrEL} that the classes in these three staircases are all equivalent.    Therefore, because the Fibonacci staircase is known  to consist of exceptional classes by \cite{ball},  these three pre-staircases are all perfect.  Hence 
the other arguments below assume that $n\ge 1$.
\MS

\NI
 (iii)  As we will see, the reduction processes for the different pre-staircases are closely related.  This appears most clearly in Proposition~\ref{prop:CrEL}.   However, we do not attempt here to provide a unified proof valid for all pre-staircases, but instead will consider each pre-staircase family separately.  This point will be discussed more fully in our next paper.   
 Note also that, because our arguments concern the linear relations among the coefficients at the beginning of the tuple, for  all the arguments  except those in Corollary~\ref{cor:CrU} the choice of $\eend_n$ is irrelevant.
\hfill$\er$
\end{rmk}

We will first discuss the classes $\bE^U_{\ell,n,k}, \bE^U_{u,n,k},$ in the pre-staircases $\Ss^U_{\ell,n}, \Ss^U_{u,n}$ with centers 
\begin{align*}
\frac{p_{n,k}}{q_{n,k}} & = \begin{cases}  [\{2n+5,2n+1\}^k, \eend_n], &\mbox{ for } \Ss^U_{\ell,n}, n\ge 1, k\ge 0\\
[2n+7; \{2n+5,2n+1\}^k, \eend_n], &\mbox{ for } \Ss^U_{u,n}, n\ge 1, k\ge 0,
\end{cases}
\end{align*}
and linear relations 
\begin{align}\label{eq:Ulinrel}
& (2n+3)d_{n,k} = (n+1)p_{n,k} +(n+2)q_{n,k} \quad \mbox{ for }\;   \Ss^U_{\ell,n}\\ \notag
&(2n+3)d_{n,k} = (n+2)p_{n,k} -(n+4)q_{n,k} \quad \mbox{ for }\;   \Ss^U_{u,n}.
\end{align}
As in Lemma~\ref{lem:Cr0}~(i),  the
(integral) weight expansion of $[\{2n+5,2n+1\}^k, 2n+4]$ for $k\ge 1$ is denoted 
\begin{align}\label{eq:Crab}
\bigl(a_{n,k}^{\times (2n+5)}, b_{n,k}^{\times (2n+1)}, (a_{n,k}-(2n+1)b_{n,k})^{\times (2n+5)},\dots, 1^{\times (2n+4)}\bigr).
\end{align}  
while that for 
$[\{2n+5,2n+1\}^k,2n+5,2n+2]$ for $k\ge 1$ is denoted 
\begin{align}\label{eq:Crab'}
\bigl({a'_{n,k}}\!\!^{\times (2n+5)}, {b'_{n,k}}\!\!^{\times (2n+1)}, (a'_{n,k}-(2n+1)b'_{n,k})^{\times (2n+5)},\dots, 1^{\times (2n+4)}\bigr).
\end{align}  
Further, in the pre-staircase $\Ss^U_{\ell,n,k}$ we have 
\begin{align}\label{eq:Crp}
p_{n,k-1} = \begin{cases} b_{n,k},& \quad \mbox{ if } \eend_n = 2n+4, \\
b'_{n,k},& \quad \mbox{ if } \eend_n = 2n+5,2n+2.
\end{cases}
\end{align}

Here are our two main results about these pre-staircases.

\begin{lemm} \label{lem:CrUuell}  For each $n\ge 1, k\ge 0$ there is a tuple $\bE'_{n,k}$ 
that reduces to $\bE^U_{\ell,n,k}$  in  $n+2$ standard Cremona moves and is such that 
 $\bE^U_{u,n,k}$ reduces to it in $n+3$ standard Cremona moves. 
\end{lemm}

\begin{lemm} \label{lem:CrUell}  For each $n\ge 1, k\ge 0$  the class $\bE^U_{\ell,n,k+1}$ reduces to $\bE'_{k,n}$
in $3n+2$ moves all of which are standard except for the last one.
\end{lemm}

\begin{cor}\label{cor:CrU}  The classes  $\bE^U_{\ell,n,k}$ and  $\bE^U_{u,n,k}$ are perfect for all $n\ge 1, k\ge 0$ and both 
choices of $\eend_n$.
\end{cor}
\begin{proof} By Lemmas~\ref{lem:CrUuell} and~\ref{lem:CrUell} it suffices to check that the classes
$
\bE^U_{\ell,n,0}, n\ge 1,$  are perfect.  But the center of this class is  either $2n+4$ or $[2n+5;2n+2]$ and the relation is
$(2n+3) d = (n+1)p+ (n+2)q$.  Therefore with $\eend_n = 2n+4$, the class is $(n+2,n+1; \bw(2n+4)) = B^U_{n-1}$, which is easily seen to be perfect.  Further, with $\eend_n = (2n+5,2n+2)$  the class is
\begin{align*}
&\bigl((n+1)(2n+5), (n+1)(2n+3)-1; (2n+2)^{\times (2n+5)}, 1^{\times (2n+2)}\bigr) \\ &\quad = 
 \bigl( (n+2)(2n+1) + (2n+3), (n+2)(2n+1), (2n+2)^{\times (2n+5)}, 1^{\times (2n+2)}\bigr) \\
 &\quad = \bigl((n+3)s + 2, (n+2)s, (s+1)^{\times (2n+5)},1^{\times (2n+2)}\bigr),\quad s: = 2n+1.
\end{align*}
Each standard Cremona move takes $ \bigl(d, m, (s+1)^{\times i},1^{\times j}\bigr)$ to
$ \bigl(d-s, m-s, (s+1)^{\times(i-2)},1^{\times (j+2)}\bigr)$. Hence after $(n+2)$ such moves, we obtain the 
tuple $\bigl(2n+3, 0, 2n+2, 1^{\times(4n+6)}\bigr)$ that corresponds to the perfect class
$$
\bigl(2n+3, 2n+2; 1^{\times(4n+6)}\bigr) = \bB^U_{2n}.
$$
This completes the proof.
\end{proof}

\begin{proof} [Proof of Lemma~\ref{lem:CrUuell}]
The class $\bE'_{n,k}$ in Lemma~\ref{lem:CrUuell}  has the following form:
\begin{align}\label{eq:E'nk}
&\Bigl(\ell, s_0, s_1^{\times (2n+4)},  q^{\times (2n+5)}, r^{\times (2n+1)} \dots\Bigr), \quad \mbox{ where } s_0 + s_1 = \ell,\quad s_1: = \frac{\ell-q}{2n+3}\\ \notag
&  \qquad  \qquad \mbox{ and }  (\ell,q,r) =  \begin{cases}   (b_{n,k+1}, a_{n,k}, b_{n,k}) & \mbox { if } \eend_n = (2n+4)\\ 
  (b'_{n,k+1}, a'_{n,k},b'_{n,k}) &\mbox { if } \eend_n = (2n+5,2n+2), 
\end{cases}
\end{align}
and where\footnote
{
We assume here that $k>0$ -- the case $k=0$ can safely be left to the reader.} $(q^{\times (2n+5)}, r^{\times (2n+1)} \dots)$ is the integral weight expansion of the center  $p^U_{\ell,n,k}/q^U_{\ell,n,k}$
 of $\bE^U_{\ell,n,k}$.  
 
 We first prove that $\bE'_{n,k}$ reduces to $\bE^U_{\ell,n,k}$.  To this end, note that
one standard Cremona move gives $$
(\ell, s_0,s_1^{\times (2n+4)}, q^{\times (2n+5)}, \dots) \to (\ell-s_1, \ell-2s_1,s_1^{\times (2n+2)}, q^{\times (2n+5)},\dots),
$$
 which still has the property that the first entry equals the sum of the second two.  Hence 
a total of $(n+2)$ such moves gives
$$
\bigl(\ell-(n+2)s_1, \ell-(n+3)s_1,  q^{\times (2n+5)}, \dots\bigr).
$$
To finish the proof that  $\bE'_{n,k}$ reduces to $\bE^U_{\ell,n,k}$ as claimed, 
we must show that the above tuple is precisely $\bE^U_{\ell,n,k}$.  

But by \eqref{eq:E'nk},  the center of $\bE^U_{\ell,n,k}$ is precisely $\ell/q = ((2n+5)q + r)/q$.
Then, by  \eqref{eq:E'nk} we have $s_1 = q + \frac{q+r}{2n+3}$ and we must show that
\begin{align*}
d: = &d^U_{\ell, n,k} =  \ell - (n+2)\bigl(q + \frac{q+r}{2n+3}\bigr), \\
m: = &m^U_{\ell, n,k} = \ell - (n+3)\bigl(q + \frac{q+r}{2n+3}\bigr).
\end{align*}
But the linear relation in $\Ss^U_{\ell,n,k}$ gives  
$$
(2n+3) d = (n+1) \ell + (n+2) q  = (n+1)\bigl((2n+5)q+ r\bigr) + (n+2)q.
$$
It  is now easy to check that the two expressions for $d$ 
are the same.   
\MS

It remains to prove that the class $\bE^U_{u,n,k}$ reduces to $\bE'_{n,k}$.  
As noted in the proof of Lemma~\ref{lem:SsUellu}, we now have
$$
q^U_{u,n,k} = \ell, \quad p^U_{u,n,k} = (2n+7)\ell + q,
$$
where $\ell,q$ are as above, and $d,m$ are redefined  as follows:
\begin{align}\label{eq:dmU} \notag
(2n+3)d: &= (2n+3)\,d^U_{u,n,k}  = (n+2)\bigl((2n+7)\ell + q\bigr) - (n+4)\ell \\ 
& = (2n^2 + 10 n + 10)\ell + (n+2)q,\\ \notag
(2n+3)m: &= (2n+3)\,m^U_{u,n,k} = (2n^2 + 8n+6)\ell + (n+3)q.
\end{align}
Thus  
$$
\bE^U_{u,n,k}\ = \
\bigl(d,m, \ell^{\times (2n+7)}, q^{\times (2n+5)}, \dots\bigr),
$$
 and we aim to reduce it to
$$
\bE'_{n,k} = 
\bigl( \ell, s_0, s_1^{\times (2n+4)},  q^{\times (2n+5)},\dots\bigr), \quad s_0 + s_1 = \ell,\;\; s_1: = \frac{\ell-q}{2n+3}
$$
We will do this by $n+3$ standard Cremona moves, where all the moves until the last replace $2$ copies of $\ell$ by 
$2$ copies of $s_1$,  and the last removes the final $3$ copies of $\ell$ and shrinks length by $3$.
The first move uses   the identity
$$
d - m - \ell = s_1 = \frac{\ell-q}{2n+3},
$$
which follows easily from \eqref{eq:dmU}.
These moves have the following form (where $\dots$ now includes the terms $q^{\times (2n+5)}$)
\begin{align*}
& \bigl(d,m,\ell^{\times (2n+7)},\dots\bigr)  \\
&\quad \to\;\;  \bigl(2d-m-2\ell, \ d-2\ell,\ \ell^{\times (2n+5)}, \dots , (d-m-\ell)^{\times 2} = s_1^{\times 2}, \dots \bigr) \quad \mbox{after $1$ move}\\
&\quad\to\;\;  \bigl(3d-2m-4\ell, \ 2d-m-4\ell ,\ \ell ^{\times (2n+3)}, s_1^{\times 4}, \dots \bigr)\quad \mbox{after $2$ moves}\\
&\quad \to\;\;  \bigl((n+3)d-(n+2)m-2(n+2)\ell  =2\ell , \ \ell ^{\times 3}, \\
&\qquad\qquad\qquad\qquad (n+2)d-(n+1)m-2(n+2)\ell  = s_0, 
  s_1^{\times (2n+4)}, \dots \bigr)\quad \mbox{after $n+2$ moves}\\
  &\quad \to\;\; \bigl( \ell , s_0, s_1^{\times (2n+4)}, \dots\bigr),
\end{align*}
where we have used the  identity
\begin{align}\label{eq:qs}
(n+3)d-(n+2)m-2(n+2)\ell  = 2\ell ,  
\end{align}
(which follows from \eqref{eq:dmU}) and its corollary $$
(n+2)d-(n+1)m-2(n+2)\ell  = 2\ell  -d + m=  \ell -s_1=: s_0.
$$
This completes the proof.
\end{proof}

\begin{EXAMple}\label{ex:redUell1}  \normalfont
Since the proof of Lemma~\ref{lem:CrUell} is somewhat more complicated,
we begin with an example.
Since $\bE^U_{\ell,1,2}=\bigl(1538, 987; \bw(3191/436)\bigr)$, we have
$$
\bE^U_{\ell,1,2} = \bigl(1538, 987, 436^{\times 7}, 139^{\times 3}, 19^{\times 7},6^{\times 3}, 1^{\times 6}\bigr)
$$
which reduces to $\bE'_{1,1}$ as follows:
\begin{align*}
\begin{array}{c||c|c|c|c|c|c|c|c|}
\hline
0&1538 & 987 & 436^{\times 7} & 139^{\times 3}& 19^{\times 7}&\dots&&\\ \hline
1&1217&666&436^{\times 5}&139^{\times 3} &115^{\times 2}&19^{\times 7}&\dots&\\ \hline
2&896&436^{\times 3}&345&139^{\times 3}  &115^{\times 4}&19^{\times 7}&\dots&\\ \hline
3&484&345&139^{\times 3}  &115^{\times 4}& 24^{\times 3}&19^{\times 7}&\dots&  \\ \hline
4&345&206&139 &115^{\times 4}&24^{\times 3}&19^{\times 7}&\dots&-2 \\ \hline
5&230&115^{\times 3}&91& 24^{\times 4}&19^{\times 7}&\dots&&-1  \\\hline
6&139 &115 & 24^{\times 6}& 19^{\times 7} &\dots& &&-1 \\\hline
\end{array}
\end{align*}
All the above Cremona moves are standard, except for the last one where we reduce on the terms $115^{\times 2}, 91$ instead of the  three largest.  Further, the last column indicates the number of terms that have been reduced to zero; thus altogether the length of the class is reduced by $4$.
Notice also that  these transformations only affect the beginning terms: nothing happens to any term 
of size $19$ or less. Finally, one can check that the last row is precisely $\bE'_{1,1}.$

This reduction has several stages.  In the notation used in Lemma~\ref{lem:CrUell} below, the first three moves reduce the multiplicity of $436=: q$ from  $7$ to $0$; the next move reduces the multiplicity  of $139=:r$ from $3$ to its final value $1$,  and finally we reduce the multiplicity of the new term $115= s_0$ from $4$ to $1$, whilst increasing that of $24=:s_1$.   

Finally notice that although the last move was nonstandard, if we replace it by a standard move we would get
$(115, 91, 24^{\times 4} ,\dots)$ which is precisely what we get from $\bE'_{1,1}$ by doing one standard move.  
Hence one can compare the reductions we do in this paper, with those provided by a computer program.  We chose here to reduce to $\bE'_{n,k}$ because it has a nice formula in terms of the basic parameters of the problem.
\hfill$\er$
\end{EXAMple}

\begin{proof} [Proof of Lemma~\ref{lem:CrUell}]
We must show that for each $n\ge 1, k\ge 1$ the class $\bE^U_{\ell,n,k+1}$ may be reduced to $\bE'_{n,k}$ by   Cremona moves.
The class  $\bE^U_{\ell,n,k+1}$ for $k>0$ starts with the entries\footnote
{
Note that  the values of  $d,m,q,r$ here are different from those in Lemma~\ref{lem:CrUuell}; for example now $q = a_{n,k+1}$ rather than $b_{n,k+1}$.  However, their geometric meaning is the same.} 
\begin{align}\label{eq:EUelln}
\bigr(d_{n,k+1}, m_{0,n,k+1}, a_{n,k+1}^{\times(2n+5)}, b_{n,k+1}^{\times(2n+1)},\dots\bigr) =: \bigl(d,m,q^{\times(2n+5)}, 
r^{\times(2n+1)},\dots \bigr),
\end{align}
and we aim to reduce it to
\begin{align}\label{eq:CrE'}
&\bE_{n,k}' : = \Bigl(r, r-s_1, s_1^{\times (2n+4)}, v \bw\bigl({r}/{v}\bigr)\Bigr),
\quad \ \ v: = a_{n,k} = q-(2n+1)r, \ \ s_1: = \frac{r-v}{2n+3}.
\end{align}

The first part of the reduction process is formally the same as that in the reduction of $\bE^U_{u,n,k}$ to $\bE'_{n,k}$ in Lemma~\ref{lem:CrUuell}:
we reduce the multiplicity of the term in $q$ from $2n+5$ to $3$ by  
$(n+1)$ standard moves, and then reduce it to zero by one further standard move: 
\begin{align*}
& \bigl(d,m,q^{\times (2n+5)},\dots\bigr)  \\
&\quad \to\;\;  \bigl((n+2)d-(n+1)m-2(n+1)q, \ q^{\times 3}, (n+1)d-nm-2(n+1)q, \dots \\
&\qquad\qquad\qquad\qquad  (d-m-q)^{\times (2n+2)}, \dots \bigr)\quad \mbox{after $n+1$ moves}\\
&\quad \to\;\;  \bigl(2(n+2)d-2(n+1)m-(4n+5)q,\ (n+1)d-nm-2(n+1)q, 
 \dots\\
 &\qquad\qquad  (d-m-q)^{\times (2n+2)}, 
((n+2)d-(n+1)m-2(n+2)q)^{\times 3} \dots \bigr)\quad \mbox{after $n+2$ moves.}
\end{align*}
We claim that $$
s_0: = d-m-q, \quad s_1: = (n+2)d-(n+1)m-2(n+2)q
$$ satisfy the identities in \eqref{eq:CrE'}, namely
$$
 s_0 + s_1 =  r, \qquad  (2n+3) s_1 = r-v = (2n+2)r-q.
$$
These are two linear identities in the quantities $d,m,q,r$  
 that determine $\bE^U_{\ell,n,k}$.
The easiest way to  check them is to verify that they hold when $k=0,1$, and then use the recursion. 
The values for $d,m,q,r$ at $k=0$ can be obtained from Lemma~\ref{lem:Cr0}; note in particular
 that when $k=0$ and $\eend_n = 2n+4$ we have $(d,m,q,r) = (n+2, n+1, 1, -1)$. The values for $k=1$ and $\eend_n = 2n+4$ are in \eqref{eq:SsUell}.  The details of the other case are left to the reader.
 
 We can write the other entries in the final tuple above as
 $$
 2(n+2)d-2(n+1)m-(4n+5)q = 2s_1 + q,\qquad (n+1)d-nm-2(n+1)q = q + s_1 - s_0.
 $$
Thus, after $n+2$ moves, we have obtained the tuple
\begin{align*}
& \bigl(2s_1+q,\ q+s_1-s_0, r^{\times(2n+1)},
 s_0^{\times (2n+2)}, 
s_1^{\times 3}, \dots \bigr),
\end{align*}
and the next task is to perform $n$ moves on the first three terms to reduce the multiplicity of the term $ r^{\times(2n+1)}$ to $1$. Because $2s_1+q - (q+s_1-s_0) - r = s_0+s_1 - r = 0$, at each step we reduce the length of the tuple by $2$.  Moreover, $q+s_1-s_0 = 2s_1+q -r$ so that under these moves the first two terms 
change as follows:
$$
 \bigl(2s_1+q,\ 2s_1 + q-r\bigr) \to  \bigl(2s_1+q-r,\ 2s_1 + q-2r\bigr) \to \dots 
\to \bigl(2s_1+q- nr,\ 2s_1 + q-(n+1)r\bigr).
$$
Thus we have
$$
\bigl(2s_1+q- nr,\ 2s_1 + q-(n+1)r, r, s_0^{\times (2n+2)},  s_1^{\times 3}, \dots \bigr).
$$
The next reduction step gives
\begin{align*}
&\bigl(2s_1+q- nr - s_0,\ 2s_1 + q-(n+1)r-s_0, r-s_0, s_0^{\times (2n+1)},  s_1^{\times 3}, \dots \bigr)\\
& \quad =\bigl((n+1)s_0 - (n-1)s_1,\ ns_0 - ns_1,  s_0^{\times (2n+1)},  s_1^{\times 4}, \dots \bigr)\\
\end{align*} 
Finally we do $n$ moves that use two copies of $s_0$ at each step  to
 reduce the multiplicity of $s_0$ to $1$, while building up the multplicity of $s_1$ to $2n+4$ and reducing each of the first two terms  by $s_0-s_1$ at each step.   Note that throughout this reduction the terms designated by $\dots$ in \eqref{eq:EUelln} remain unchanged, and hence start with $v^{\times (2n+5)}.$  Thus finally we have
 \begin{align*}
\bigl( s_0+s_1 = r,  s_0,  s_1^{\times (2n+4)}, v^{\times (2n+5)}, \dots \bigr) = : \bE'_{n,k}.
\end{align*}
This completes the proof of Lemma~\ref{lem:CrUell}.
\end{proof}

By Corollary~\ref{cor:CrU}, we have now completed the proof
that  the pre-staircases $\Ss^U_{\bullet,\bullet}$ are perfect.
It remains to consider the other pre-staircase families.  
For ease of reference, here are the steps and relations for the relevant classes:
\begin{align*}
\begin{array}{ll}
\bE^U_{\ell,n,k+1} &\quad   {p'}/{q'}: = [2n+5; 2n+1,\{2n+5,2n+1\}^k, \eend_n], \\
& \qquad\qquad\quad (2n+3)d' =  (n+1)p'+(n+2)q',\;\; \; n\ge 1,\\
\bE^L_{u,n,k}&\quad   p/q: = [6;2n-1, 2n+1, \{2n+5,2n+1\}^k, \eend_n], \\
&\qquad\qquad (2n+3)d=- (n-1)p + (11n+2)q, \;\; \; n\ge 1,\\
\bE^E_{\ell,n,k}&\quad   p/q: = [5;1, 2n+4, 2n+1, \{2n+5,2n+1\}^k, \eend_n], \\
&\qquad\qquad\quad (2n+3)d= (n+2)p - (n+4)q,\;\; \; n\ge 1,\\
\bE^U_{u,n,k} &\quad   {p'}/{q'}: = [2n+7; \{2n+5,2n+1\}^k, \eend_n], \\
&\qquad\qquad\quad (2n+3)d' =  (n+2)p'-(n+4)q'. \;\; \; n\ge 0,\\
\bE^L_{\ell,n,k}&\quad   p/q: = [6;2n+1,  \{2n+5,2n+1\}^k, \eend_n], \\
&\qquad\qquad\quad (2n+3)d=- (n+1)p + (n+2)q,\;\; \; n\ge 0,\\
\bE^E_{u,n,k}&\quad   p/q: = [5;1,2n+6, \{2n+5,2n+1\}^k, \eend_n], \\
&\qquad\qquad\quad (2n+3)d=- (n+4)p + (11n+31)q, \;\; \; n\ge 0.
\end{array}
\end{align*}
Note that we now use $'$ to distinguish the entries in $\Ss^U_{\bullet,n}$ from those in the other pre-staircases.

\begin{prop}\label{prop:CrEL}   
\begin{itemlist}\item[{\rm(i)}] For $n\ge 1$, the tuple obtained from $\bE^L_{u,n,k}$ by $2$ standard moves 
 that get rid of the term $q^{\times 6}$   equals the one obtained from
$\bE^U_{\ell,n,k+1}$ by $3$ moves that reduce the multiplicity of $q'$ from $2n+5$ to $2n-1$, where the first two are standard, each creating two copies of a new term $c'$, and the third   uses two copies of $q'$ and one of $c'$.
\item[{\rm(ii)}] For $n\ge 0$, the tuple obtained from $\bE^L_{\ell,n,k}$ by $3$ standard moves, 
 that first get rid of the term $q^{\times 6}$ and then get rid of the new term $2d-2q$ of multiplicity $3$,  equals  that  obtained from
$\bE^U_{u,n,k}$ by $3$ standard moves that reduce the multiplicity of $q'$ from $2n+7$ to $2n+1$.  
\item[{\rm(iii)}] For $n\ge 1$, the tuple obtained from $\bE^E_{\ell,n,k}$   by  $5$ standard moves 
equals that  obtained from
$\bE^U_{\ell,n,k+1}$  by $1$ standard move.
\item[{\rm(iv)}] For $n\ge 0$, the tuple obtained from $\bE^E_{u,n,k}$  by  $5$ standard moves 
equals that  obtained from
$\bE^U_{u,n,k}$ by $1$ standard move.
\end{itemlist}
In particular, the classes in all the above pre-staircases are perfect.
\end{prop}
\begin{proof}   We prove (i).
We have
\begin{align*}
&\bE^L_{u,n,k} = \bigl(d, m; q^{\times 6}, r^{\times(2n-1)},  v^{\times (2n+1)}, \dots\bigr), \qquad\mbox{where} 
\\ & \qquad 
(2n+3)d=(5n+8)q-(n-1)r, \
(2n+3) m = (n+3)q -   5nr.
\end{align*}
Thus $d = 2q+ c$, where $c: = \frac{(n+2)q-(n-1)r}{2n+3}$; and, if we reorder $\bE^L_{u,n,k}$ so that $m$ comes after $q^{\times 6}$ and then
 do two moves to get rid of the terms $q^{\times 6}$, we obtain the tuple
\begin{align}\label{eq:dcq}
 \bigl(4c-q,  c^{\times 3}, (2c-q)^{\times 3},  r^{\times(2n-1)}, m, \dots \bigr).
\end{align}
On the other hand, 
\begin{align*}
&\bE^U_{\ell,n,k} = \bigl(d', m', (q')^{\times (2n+5)}, (r')^{\times(2n+1)},  \dots\bigr), \qquad\mbox{where} \quad q'=r, r' = v \\
& \quad d'=(n+2)q'+\frac{n+1}{2n+3}(q'+r'),  \quad m' = (n+1)q' + \frac{n}{2n+3}(q'+r').
\end{align*}
Hence $d'-m'-q' = \frac{q'+r'}{2n+3}=: c'. $
Now reduce three times  (using one copy of $c'$ and two of $q'$ at the third move) to obtain
\begin{align*}
&\bigl(d', m', (q')^{\times (2n+5)}, \dots\bigl)\to  \bigl(2d'-m'-2q',  d-2q',  (q')^{\times (2n+3)},(c')^{\times 2}, \dots\bigl)\\
&\qquad \to   \bigl(3d'-2m'-4q',  2d'-m'-4q', (q')^{\times (2n+1)}, (c')^{\times 4}, \dots\bigl)\\
&\qquad \to     \bigl(6d'-4m'-10q' - c', (3d'-2m'-5q'-c')^{\times 2}, 3d'-2m'-6q', \\
&\qquad\qquad \qquad \qquad  (q')^{\times (2n-1)}, 2d'-m'-4q', (c')^{\times 3}, \dots\bigl)
\end{align*}
We now claim that this tuple is precisely the same as that in \eqref{eq:dcq}.  To see this,  express
all the quantities in terms of the  variables $q',r', $  recalling that
$$
r=q',  \quad v = r', \quad  q = (2n-1)r + v = (2n-1)q'+r'.
$$
Further, by definition of $c'$ we have $3d'-2m'-5q'-c'=2d'-m'-4q'$  so that this term occurs with multiplicity $3$,
 and one can check that $c= 2d'-m'-4q'$. 
One can then check that the other terms agree:
 $$
  4c-q = 6d'-4m'-10q' - c',  \quad 2c-q = c', \quad  m = 3d'-2m'-6q'.
   $$
This completes the proof of (i).  

The proof of (ii) is very similar, and can be done as above. Details are left to the reader.   Alternatively, since reduction is a linear transformation 
it suffices to prove this for  $k=0,1$ and arbitrary $n$; the  result for $k>1$ then follows by recursion.  One can either do this by calculating the relevant transformations explicitly, or by noting that 
when $k = 0,1,$ the entries $d',m', \dots$ in the tuples $\bE^L_{\ell,n,k}, \bE^U_{u,n,k}$ are 
 polynomials in $n$ of degree at most $5$ (the explicit formulas for $\eend_n=2n+4$ are given in \eqref{eq:SsUu} and \eqref{eq:SsLell}).
 Hence we need to prove some polynomial identities  of degree at most $5$.  But these hold for all $n$ if and only if they hold for $5$ different values of $n$.  Thus, without even explicitly calculating these polynomials, we know that the claim must hold for all $n,k$ provided that it holds for $k= 0,1$ and $0\le n \le 5$, 
something that is very  easy to check by computer.
\MS

We illustrate (iii), (iv) by an example.  We have 
$$
\bE^U_{\ell,1,1} = \bigl(67,43, 19^{\times 7},\dots\bigr) \to \bigl(53,29,19^{\times 5},\dots,5^{\times 2}\dots\bigr).
$$
On the other hand, $5$ standard moves on $\bE^E_{\ell,1,0} $ give: 
\begin{align*}
\bE^E_{\ell,1,0} & = \bigl(350,139^{\times 5}, 120, 96,  19^{\times 6},\dots\bigr)
 \to \bigl(283,139^{\times 2}, 120, 96, 72^{\times 3},19^{\times 6},\dots\bigr)\\ & \to
\bigl(168, 96,72^{\times 3},  24^{\times 2},19^{\times 6},\dots,5,\dots\bigr)
\stackrel{-2}\to \bigl(96, 72, 24^{\times 3}, 19^{\times 6}, \dots, 5, \dots\bigr) \\
& \stackrel{-2}\to 
\bigl(72, 48, 24, 19^{\times 6}, \dots, 5, \dots\bigr)
\stackrel{-1}\to (53, 29, 19^{\times 5},\dots, 5^{\times 2}, \dots\bigr),
\end{align*}
where the superscripts over the arrows denote the number of entries that go to zero at that step.
Thus overall $5$ terms go to zero, which is the difference in length between the center $[5;1, 2n+4, 2n+1, \{2n+5,2n+1\}^k, \eend_n]$ of $\bE^E_{\ell,n,k}$ and the corresponding center $[2n+5; 2n+1, \{2n+5,2n+1\}^k, \eend_n]$ of $\bE^U_{\ell,n,k+1}$.
A formal proof of the general case can be written out by one of the methods explained above, and is left to the interested reader.

The final claim that all the pre-staircases are perfect holds when $n=0$ because the Fibonacci staircase is perfect (by \cite{ball}) and holds when $n>0$ by Corollary~\ref{cor:CrU}.
\end{proof}

\section{Obstructions from ECH capacities}\label{sec:ECH}

An alternative way to characterize the ellipsoid embedding function for the target $X$ involves the ECH capacities of the ellipsoid and of $X$. When we use the computer to plot an approximation of the graph of $c_X$, we use that alternative characterization of $c_X$:

\begin{equation}\label{def:csup}
c_X(z)=\sup_k\frac{c_k(E(1,z))}{c_k(X)},
\end{equation}
so long as $X$ is ``convex'' (see Definition \ref{def:cvx}); the equality (\ref{def:csup}) is a consequence of \cite[Theorem 1.2]{cg}. First defined in \cite{Hut}, ECH capacities are invariants of a symplectic 4-manifold $(X,\omega)$ that obstruct symplectic embeddings in the following sense:
\begin{equation}\label{ECHimplies}
(X,\omega)\hookrightarrow(Y,\omega') \implies  \forall_{k\in\mathbb{N}_0}\,\, c_k(X,\omega)\leq c_k(Y,\omega'),
\end{equation}
where the $c_k$ form a non-decreasing sequence 
$$0=c_0(X,\omega)<c_1(X,\omega)\leq c_2(X,\omega)\leq \ldots\leq \infty.$$

\subsection{Toric domains}

To compute the ECH capacities of the Hirzebruch surfaces $H_b$, we use the more general theory of ECH capacities of toric domains.

\begin{DEFN}\label{def:cvx} A \textbf{toric domain} is the region $X_\Omega$ in $\C^2$ determined by
\[
X_\Omega:=\left\{(z_1,z_2)\in\C^2\middle|\left(\pi|z_1|^2,\pi|z_2|^2\right)\in\Omega\right\}
\]
where $\Omega\subset\R^2$. A toric domain is \textbf{concave} if $\Omega$ is a closed region in the first quadrant under the graph of a convex function with both axis intercepts nonnegative. Following \cite{hbeyond}, a toric domain is \textbf{convex} if $\Omega$ is a closed region in the first quadrant under the graph of a nonincreasing concave function with both axis intercepts nonnegative.
\end{DEFN}

When $X$ is a toric domain we use the standard symplectic form $\omega=\sum_{i=1}^2dx_i\wedge dy_i$, where $z_i=x_i+\sqrt{-1}y_i$, and we drop $\omega$ from the notation for the ECH capacities. Note that ellipsoids are both concave and convex toric domains.

Let $X_b$ denote the convex toric domain $X_{\Omega_b}$, where $\Omega_b$ is the Delzant polytope of $H_b$. By \cite[Theorem 1.2]{AADT}, embeddings of ellipsoids into $H_b$ and $X_b$ are equivalent, and therefore $c_{X_b}=c_{H_b}$.

\begin{REMark}\label{ECHiff2} \normalfont
When $X$ is a concave toric domain and $Y$ is a convex toric domain, \eqref{ECHimplies} is in fact an equivalence (\cite[Theorem 1.2]{cg}). This is the case when we have an ellipsoid embedding into $X_b$, which justifies the alternative definition \eqref{def:csup}.
\hfill$\er$
\end{REMark}

The way that we use \eqref{def:csup} to plot approximations to the graph is by taking the maximum over a large but finite number of ratios of ECH capacities, rather than the supremum. In order to do this, we need to compute the ECH capacities of an ellipsoid and of $X_b$.
 
\begin{REMark}\label{rem:computeECH}\normalfont
The sequence of ECH capacities for an ellipsoid $E(a,b)$ is the sequence $N(a,b)$, where for $k\geq 0$, the term $N(a,b)_k$ is the $(k+1)^\text{st}$ smallest entry in the array $(am+bn)_{m,n\in\mathbb{N}_0}$, counted with repetitions \cite[\S1]{m2}.

To compute the ECH sequence of $X_b$ we use a different method, based on \cite[Appendix A]{cg}. The definitions and theorem below are based on \cite[Definitions A.6, A.7, A.8]{cg} and can be found there in more detail.\hfill$\er$
\end{REMark}

\begin{DEFN}\label{def:latticepathdefs}
A \textbf{convex lattice path} is a piecewise linear path  $\Lambda:[0,c]\to\mathbb{R}^2$ such that all its vertices, including the first $(0,x(\Lambda))$ and last $(y(\Lambda),0)$, are lattice points and the region enclosed by $\Lambda$ and the axes is convex. An \textbf{edge} of $\Lambda$ is a vector $\nu$ from one vertex of $\Lambda$ to the next.
The \textbf{lattice point counting function} $\mathcal{L}(\Lambda)$ counts the number of lattice points in the region bounded by a convex lattice path $\Lambda$ and the axes, including those on the boundary.

Let $\Omega\subset\mathbb{R}_{\geq0}^2$ be a convex region in the first quadrant. The $\Omega$\textbf{-length} of a convex lattice path $\Lambda$ is defined as
\begin{equation}\label{eqn:Omega-length}
\ell_\Omega(\Lambda)=\sum_{\nu \in\text{Edges}(\Lambda)}\det \left[ \nu\,\, p_{\Omega,\nu}\right]
\end{equation}
where for each edge $\nu$ we pick an auxiliary point $p_{\Omega,\nu}$ on the boundary of $\Omega$ such that $\Omega$ lies entirely ``to the right'' of the line through $p_{\Omega,\nu}$ and direction $\nu$.
\end{DEFN}

\begin{thm}\cite[Corollary A.5]{cg}\label{thm:lattice path}
Let $X$ be the toric domain corresponding to the region $\Omega$. Then its $k^\text{th}$ ECH capacity is given by:
\begin{equation}\label{cbylatticepaths}
c_k(X)=min\left\{ \ell_\Omega(\Lambda): \Lambda  \text{ is a convex lattice path with } \mathcal{L}(\Lambda)=k+1\right\}.
\end{equation}
\end{thm}

In fact, the lattice path $\Lambda$ that realizes the minimum in \eqref{cbylatticepaths} is shaped roughly like the boundary of the region $\Omega$, see \cite[Ex. 4.16(a)]{hECHlec}. In particular, the slopes of the edges of the minimizing lattice path must also be slopes of edges of $\Omega$, and we use this fact to compute the capacities of $X_b$: see Lemma \ref{lem:restrictedlps}.

\subsection{ECH capacities and exceptional classes}\label{subsec:ecECH}

In this subsection we prove a relationship between the obstructions from exceptional classes of \S3 and those from ECH capacities. This relationship underlies our use of ECH capacities to identify live perfect classes contributing to $c_{H_b}$, as explained in \S\ref{subsec:findingk}. Let $\Lambda_{d,m}$ denote the lattice path from $(0,d-m)$ to $(m,d-m)$ to $(d,0)$, which encloses $\mathcal{L}(\Lambda_{d,m})$ lattice points as in Definition \ref{def:latticepathdefs}.

\begin{lemm} \label{lem:ECHcap}  If $\mathbf{E}=(d,m;\mathbf{m})$ is an exceptional class with $\mathbf{m}=q\mathbf{w}\left(p/q\right)$, then there is an interval containing the center $p/q$ of $\mathbf{E}$ on which the obstructions from $\mathbf{E}$ and $c_{\mathcal{L}(\Lambda_{d,m})-1}(X_b)$ satisfy
\begin{equation}\label{eqn:classgencorresp}
\mu_{\mathbf{E},b}(z)\leq\frac{c_{\mathcal{L}(\Lambda_{d,m})-1}(E(1,z))}{c_{\mathcal{L}(\Lambda_{d,m})-1}(X_b)}
\end{equation}
\end{lemm}

\begin{proof} We will show that $c_{\mathcal{L}(\Lambda_{d,m})-1}(X_b)\leq d-mb$ and that $c_{\mathcal{L}(\Lambda_{d,m})-1}(E(1,z))=\mathbf{m}\cdot\mathbf{w}(z)$ for $z$ sufficiently close to $p/q$. For the former, we have by \eqref{cbylatticepaths} and \eqref{eqn:Omega-length}:
\[
c_{\mathcal{L}(\Lambda_{d,m})-1}(X_b)\leq\ell_{\Omega_b}(\Lambda_{d,m})=d-bm.
\] 
For the latter, we first analyze $c_{\mathcal{L}(\Lambda_{d,m})-1}(E(1,z))$. Note that
\begin{equation}\label{eqn:multiplybyq}
qN\left(1,\frac{p}{q}\right)_k=N(q,p)_k.
\end{equation}
Let $K(p,q)=\frac{(p+1)(q+1)}{2}-1$. We claim that
\[
N(q,p)_{K(p,q)}=N(q,p)_{K(p,q)+1}=pq
\]
and these are the only two values of $k$ for which $N(q,p)_k=pq$. If the point in $\Z^2_{\geq0}$ with coordinates $(n_1,n_2)$ is labeled $n_1q+n_2p$ then $N(q,p)$ is the list of such labels ordered as follows. Take the line of slope $-p/q$ and move it from left to right across the plane. Then $N(q,p)$ is the $k^\text{th}$ label whose lattice point this line crosses, starting from zero. If points are colinear then their indices start from the number of lattice points between the line of slope $-p/q$ through $n_1q+n_2p$ and the axes and increase from there. Because $p$ and $q$ are coprime,
\[
n_1q+n_2p=pq\Rightarrow n_2p=(p-n_1)q
\]
implies that both $n_2p$ and $(p-n_1)q$ must equal an integer multiple of $pq$, and we conclude that either $n_1=0$ or $n_1=p$. Therefore there are only two values of $k$ for which $N(q,p)_k=pq$, and the values of $k$ start from the number of lattice points between the line of slope $-p/q$ through $(pq,0)$ and $(0,pq)$ and the axes. There are exactly $K(p,q)$ such lattice points.

Next we show that $\mathcal{L}(\Lambda_{d,m})=K(p,q)+1$. Firstly, by counting the lattice points in the triangle between $(0,0), (0,d)$, and $(d,0)$, then removing those in the triangle between $(0,d-m), (0,d)$, and $(m,d-m)$, we obtain
\begin{align}
\mathcal{L}(\Lambda_{d,m})&=\frac{1}{2}(d(d+3)-m(m+1))+1\label{eqn:Ldm}
\\&=\frac{1}{2}(d^2-m^2+1+3d-m+1)\nonumber
\\&=\frac{1}{2}(pq+p+q+1)\nonumber&\mbox { \footnotesize{by  \eqref{eq:diophantine} and \eqref{eq:wtexp}} } 
\\ \notag &=K(p,q)+1
\end{align}
By \eqref{eqn:Ldm} and \eqref{eqn:multiplybyq}, we have $c_{\mathcal{L}(\Lambda_{d,m})-1}\left(E\left(1,p/q\right)\right)= N(1,p/q)_{K(p,q)}=p$, and by (\ref{eq:wtexp}), we have $\mathbf{m}\cdot\mathbf{w}\left(p/q\right)=p$.

We will show that there is an interval $I$ containing $p/q$ for which, when $z\in I, z\geq p/q$,
\begin{equation}\label{eqn:ageqpq}
N(1,z)_{K(p,q)}=p=\mathbf{m}\cdot\mathbf{w}(z)
\end{equation}
and when $z\in I, z\leq p/q$,
\begin{equation}\label{eqn:aleqpq}
N(1,z)_{K(p,q)}=z=\mathbf{m}\cdot\mathbf{w}(z)
\end{equation}

The first equalities in (\ref{eqn:ageqpq}) and (\ref{eqn:aleqpq}) both follow from considering the lattice $\Z^2_{\geq0}$ labeled with $n_1q+n_2zq$ defining $N(q,zq)_k$. As $z$ increases from $p/q$, the slope of the line defining $N(q,zq)_k$ decreases from $-p/q$ to $-z$. When moving this line from left to right, it will now reach $(pq,0)$ slightly before it reaches $(0,zq)$, but if $z$ is not increased too much, the triangle between the line through $(pq,0)$ and the axes will still contain $K(p,q)$ points, and there will be no points the line passes strictly between $(pq,0)$ and $(0,zq)$, so $N(q,zq)_{K(p,q)}=pq$, showing the first equality of \eqref{eqn:ageqpq}.

Similarly, as $z$ decreases, the slope of the line incresases from $-p/q$ to $-z$. Now the triangle between the line through $(0,zq)$ of slope $-z$ and the axes will contain $K(p,q)$ points. There will be no points between the line through $(0,zq)$ of slope $-z$ and the line through $(pq,0)$ of slope $-z$ if $z$ is close enough to $p/q$, so $N(q,zq)_{K(p,q)}=z$, showing the first equality of \eqref{eqn:aleqpq}.

The second equalities in both (\ref{eqn:ageqpq}) and (\ref{eqn:aleqpq}) follow from 
Lemma~\ref{lem:munearc}.
\end{proof}

\begin{rmk}\label{rmk:lpec}\normalfont 
In all cases we have checked, the inequality (\ref{eqn:classgencorresp}) is an equality. Note that if $\mathbf{E}$ is live, then we must have an equality. However, it seems to be true more generally that there are pairs of lattice paths and exceptional classes for which (\ref{eqn:classgencorresp}) is an equality on an interval, even when $\mathbf{E}$ is not live (see Figure \ref{fig:center11119}), or is not quasi-perfect (see Figure \ref{fig:k8}). In each of Figures \ref{fig:center11119} and \ref{fig:k8}, there is an interval on which the obstruction coming from the ECH capacity (in bright blue) coincides with the obstruction coming from the exceptional class (in brown), and this interval includes the break point of the exceptional class obstruction. 
Both figures were plotted using 
 {\fontfamily{lmtt}\selectfont\textbf{PlotSingleCapacityObstructionFxn}} (see \S\ref{subsec:ECHcapcode})
 for the bright blue curve and  {\fontfamily{lmtt}\selectfont\textbf{PlotMu}} (see \S\ref{subsec:excclasscode}) for the brown curve.

Specifically, every constraint $\mu_{\bE, b}$ has a break point $a'$ in the sense of Lemma \ref{lem:0} even if $\bE$ is not perfect, and formulas analogous to those in Lemma~\ref{lem:munearc} do always hold.  Exactly what these formulas are depends on the relation between $\bw(a')$ and $\bbm$. One might be able construct a general proof of this fact using the approach to ECH capacities taken in \cite{cgh,cghm} where  these capacities are interpreted in terms of the properties of the ECH cobordism map given by a symplectic embedding $E(1,a)\sembeds H_b$; see in particular the discussion in \cite[\S3.1]{cghm}.    However, to prove this here would take us too far afield.

Instead, we will occasionally discuss $\bE$ which ``correspond'' to an ECH capacity $c_k$ in the sense that \eqref{eqn:classgencorresp} is an equality on an interval of interest on which $\mu_{\bE,b}$ is nontrivial (if $\bE$ is quasi-perfect, this interval will always contain its center).
\hfill$\er$
\end{rmk}

\begin{figure}[H]
	\includegraphics[width=.9\linewidth]{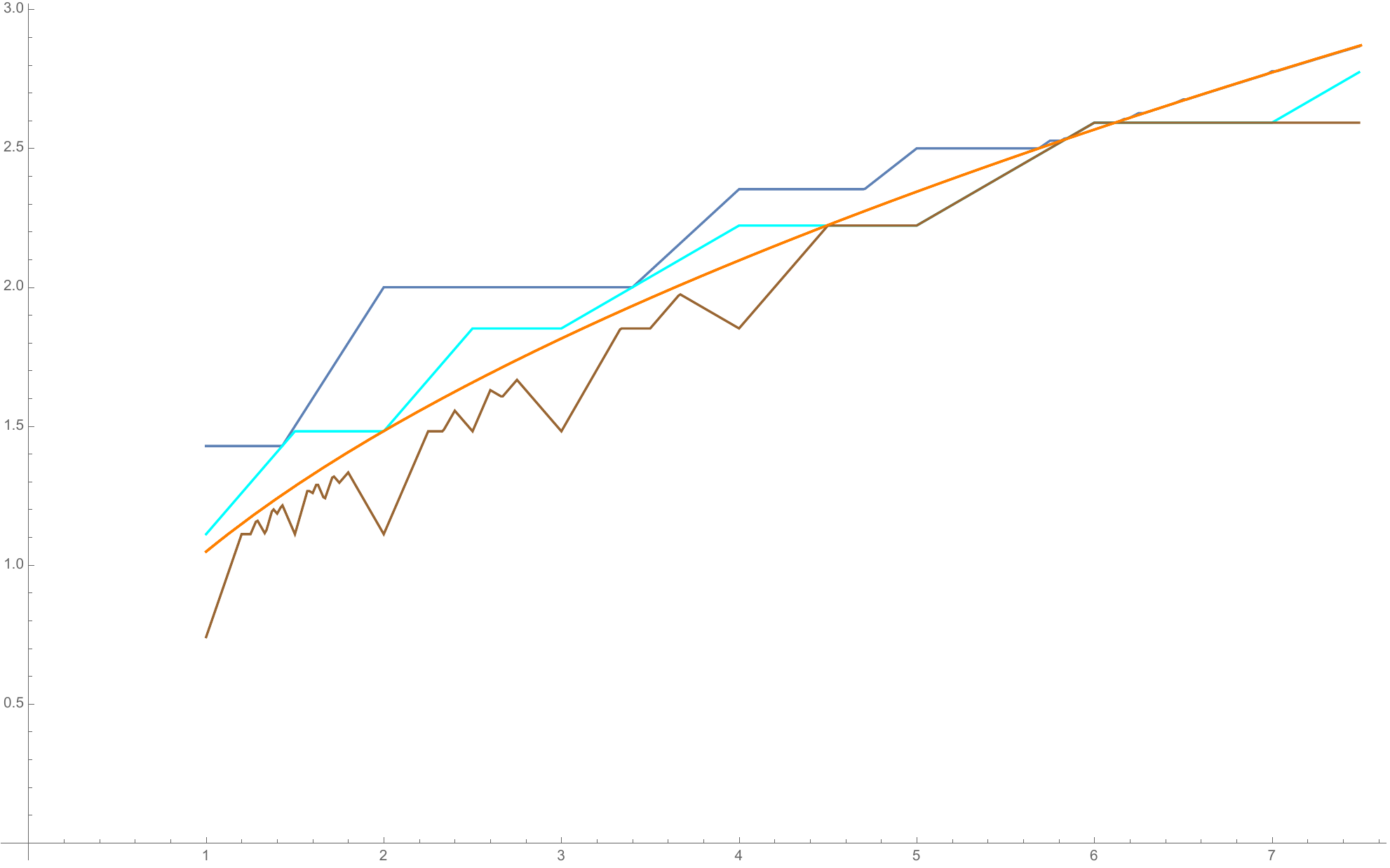}
	\caption{Comparison between the ECH and exceptional class obstructions for $b = 0.3$.  Here a lower bound $c_{H_{0.3}}^\le$ for $c_{H_{0.3}}$ is in dark blue (see \S\ref{sec:Mathem}), the volume obstruction is in orange, the obstruction from the $8^\text{th}$ capacity $c_8(X_{0.3})$ is in bright blue, and the obstruction from the class $\mathbf{E}=\left(3,1;2,1^{\times5}\right)$ is in brown. The plot of $\mu_{\bE,0.3}$ is atop that of the obstruction from $c_8(X_{0.3})$ where they agree.}
	\label{fig:k8}
\end{figure}

\subsection{There is no infinite staircase for $b=1/5$}\label{ss:15}

Let 
$$X=5 H_\frac15=\mathbb{C}P^2_5\#\overline{\mathbb{C}P}^2_{1}.$$ 
The embedding capacity function of $X$ is a scaling of that of $H_\frac15$, with $c_X(z)=\frac15 c_\frac{1}{5}(z)$. Thus, if $X$ has an infinite staircase then it accumulates at $z=\acc(1/5)=6$.

\begin{thm}\label{prop:15}
Let $X=5 H_\frac15$. 
For sufficiently small $\varepsilon$ the ellipsoid embedding function of $X$ is given by
$$c_X(z)= \begin{cases} \frac12 &\mbox{for }z\in(6-\varepsilon,6] \\
\frac{z+6}{24} & \mbox{for } z\in[6,6+\varepsilon) \end{cases}.$$
Thus, there is no infinite staircase for $X$, or equivalently for the Hirzebruch surface $H_\frac15$.
\end{thm}

The proof follows closely the proof that the ellipsoid $E(3,4)$ does not have an infinite staircase, as presented in \cite[Section 2.5]{rationalellipsoids}.

\begin{proof}
The main part of the proof is showing the following claim, whose proof we postpone by a few paragraphs.

\begin{CLAim}\label{claim15}\normalfont
For $z\in[6,6+\varepsilon)$, we have $c_X(z)\leq (z+6)/24$.
\end{CLAim}

Assuming this claim, the rest of the proof goes as follows: the claim gives us an upper bound for $c_X(z)$ for 
$z\in[6,6+\varepsilon)$.
To get a lower bound, either see Example~\ref{ex:19th} or use ECH capacities and \eqref{def:csup}: the 19th ECH capacity of $X$ is 24 (see how to compute it in \S\ref{subsec:computingECHcap}), whereas the 19th ECH capacity of $E(1,z)$ for this range of $z$ is $z+6$, see Remark \ref{rem:computeECH}. Then we have $c_X(z)\geq (z+6)/24$, and in fact $c_X(z)=(z+6)/24$.

For the range $z\in(6-\varepsilon,6]$, we obtain a lower bound either via Example~\ref{ex:52} or again using ECH capacities and \eqref{def:csup}: the 5th ECH capacity of $X$ is 10 (see how to compute it in \S\ref{subsec:computingECHcap}), whereas the 5th ECH capacity of $E(1,z)$ for this range of $z$ is 5, see Remark \ref{rem:computeECH}. Therefore $c_X(z)\geq 1/2$. To get the upper bound $c_X(z)\leq1/2$ we first observe that using Claim \ref{claim15} and the volume constraint at $z=6$, we obtain $c_X(6)=1/2$. Since $c_X(z)$ is nondecreasing, we have $c_X(z)\leq c_X(6)=1/2$ for $z\leq 6$, and therefore in fact $c_X(z)=1/2$ for $z\in(6-\varepsilon,6]$.

\MS

We now proceed to proving Claim \ref{claim15}. 
Note that since $c_X(z)$ is continuous we can assume that $z$ is irrational; this will be convenient in some of the arguments that follow. 
We begin by observing that
\begin{align*} 
c_X(z)\leq\frac{z+6}{24} &\iff  E(1,z)\hookrightarrow\frac{z+6}{24}\,X \\ 
& \iff E(\frac{24}{z+6},\frac{24z}{z+6})\hookrightarrow X\\
&\iff \text{ehr}_{\Delta_{\frac{z+6}{24},\frac{z+6}{24z}}}(t)\geq \text{cap}_X(t), \qquad\forall t\in\mathbb{Z}_{\geq 0},
\end{align*}
where the cap function counts the number of ECH capacities up to $t$
$$\text{cap}_X(t)=\#\left\{k| c_k(X)\leq t  \right\}$$
and the Ehrhart function, which is the cap function for an ellipsoid, counts the number of lattice points inside a scaling of a triangle
$$\text{ehr}_{\Delta_{u,v}}(t)=\#\left\{\mathbb{Z}^2\cap t\Delta_{u,v} \right\},$$ 
with $\Delta_{u,v}$ the triangle with vertices $(0,0)$, $(u,0)$ and $(0,v)$. 

The last equivalence then follows from \eqref{ECHimplies}, Remark \ref{ECHiff2}, and the fact that the ECH capacities of $X$ are all integers because the corresponding region $\Omega\subset\mathbb{R}^2$ is a lattice polygon, see \eqref{eqn:Omega-length} and \eqref{cbylatticepaths}.

Because it is hard to count lattice points in a generic triangle, we compare $\text{ehr}_{\Delta_{\frac{z+6}{24},\frac{z+6}{24z}}}(t)$ with $\text{ehr}_{\Delta_{\frac{z+6}{24},\frac{z+6}{24z}}}(t)|_{z=6}=\text{ehr}_{\Delta_{\frac{1}{2},\frac{1}{12}}}(t)$:
\begin{equation}\label{eq:triangledecomp}
\text{ehr}_{\Delta_{\frac{z+6}{24},\frac{z+6}{24z}}}(t)=\text{ehr}_{\Delta_{\frac{1}{2},\frac{1}{12}}}(t)+D-U-d
\end{equation}
where $D=D(t)$ and $U=U(t)$ are respectively the number of lattice points in the closed regions $R_U$ and $R_D$ as in Figure \ref{fig:triangleUD}, and $d=d(t)$ is the number of lattice points on the segment that is the left boundary of $R_D$ (but excluding the potential lattice point $(t/4,t/24)$).

\begin{figure}[H] 
\centering
  \includegraphics[width=8cm]{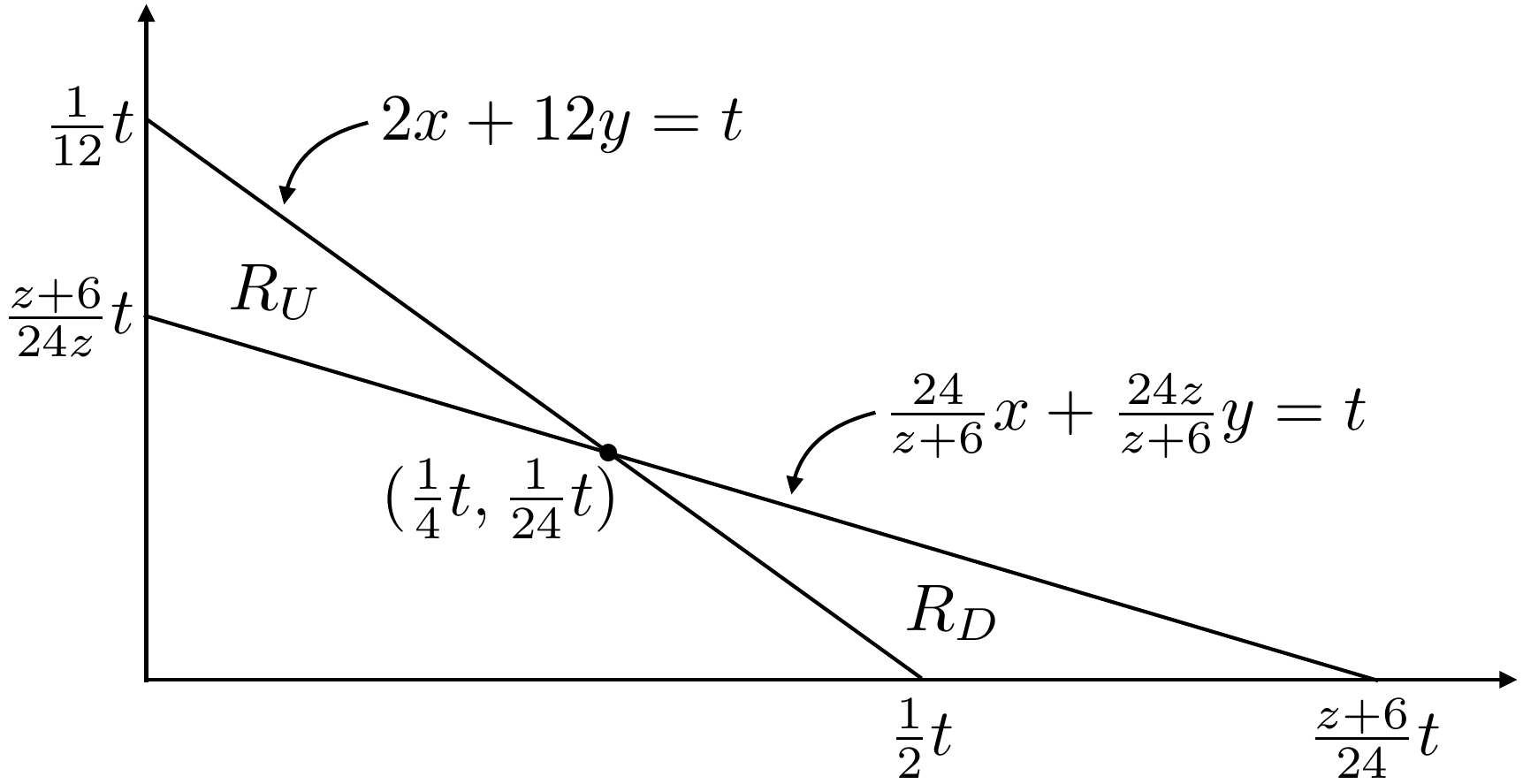}
  \caption{The triangles $\Delta_{\frac{z+6}{24}t,\frac{z+6}{24z}t}$ and $\Delta_{\frac{1}{2}t,\frac{1}{12}t}$ delimit the regions $R_U$ and $R_D$. The corresponding lattice point count is given in \eqref{eq:triangledecomp}.}
  \label{fig:triangleUD}
\end{figure}

We will then also compare the cap function of $X$ with the Ehrhart function of the same triangle:
$$\text{cap}_X(t)=\text{ehr}_{\Delta_{\frac{1}{2},\frac{1}{12}}}(t)-\tilde{d},$$
where $\tilde{d}=\tilde{d}(t):=\text{ehr}_{\Delta_{\frac{1}{2},\frac{1}{12}}}(t)-\text{cap}_X(t).$

The proof of Claim \ref{claim15} will then be complete if we show the following: 

\begin{CLAim}\label{claimUD} \normalfont
For all $t$ we have $U\leq D$, and furthermore for $t\equiv 10\,\, (\mbox{mod } 24)$ we have $U\leq D-1$ .
\end{CLAim}

\begin{CLAim}\label{claimdtilde} \normalfont
For $t\not\equiv 10\,\, (\mbox{mod } 24)$ we have $\tilde{d}\geq d$ and for $t\equiv 10\,\, (\mbox{mod } 24)$ we have $\tilde{d}= d-1$.
\end{CLAim}

\begin{proof}[Proof of Claim \ref{claimUD}]
Recall that we are assuming that $z\not\in\mathbb{Q}$.

We begin by proving the more general inequality $U\leq D$, in the following way: for each $y_0\in[t/24,t/12]\cap \mathbb{Z}$ we define a corresponding $y_1:=\floor{{t}/{12}}-y_0\in[0,{t}/{24}]\cap \mathbb{Z}$.
The values of $y_0$ range through all possible heights of lattice points in $R_U$, and for each height we will show that the number of lattice points on that slice of $R_U$ is no greater than the number of lattice points in the slice of $R_D$ at height $y_1$, which implies that $U\leq D$:
\begin{equation}\label{eq:slice}\#\left\{R_U\cap\left\{y=y_0\right\} \right\}\leq\#\left\{R_D\cap\left\{y=y_1\right\} \right\}.\end{equation}
We can rewrite \eqref{eq:slice} as
\begin{equation}\label{eq:slice2}\floor{x_2}-\ceil{\max\left\{0,x_1 \right\}}+1\leq \floor{x_4}-\ceil{x_3}+1,\end{equation}
where 
\begin{align*}
x_1&=\frac{t(z+6)-24z y_0}{24}           &  x_2&=\frac{t-12y_0}{2}            \\
x_3&=\frac{t-12y_1}{2}                &  x_4& =\frac{t(z+6)-24z y_1}{24}.        
\end{align*}
Now, since $\floor{x_2}-\ceil{\max\left\{0,x_1 \right\}}+1\leq\floor{x_2}-\floor{x_1}$, we have that \eqref{eq:slice2} follows from proving that
\begin{align}
&\floor{x_2}-\floor{x_1}\leq \floor{x_4}-\floor{x_3}   & \mbox{ whenever }  x_3\not\in\mathbb{Z}  \label{eq:integerornot1}\\
&\floor{x_2}-\floor{x_1}\leq \floor{x_4}-\floor{x_3}+1   & \mbox{ whenever }  x_3\in\mathbb{Z}.   \label{eq:integerornot2}
\end{align}
We first write
\begin{align*}
\floor{x_2}-\floor{x_1}&=\frac{t}{2}-\left\{\frac{t}{2}\right\}-6y_0-\frac{t(z+6)}{24}+zy_0+\left\{\frac{t(z+6)}{24}-zy_0\right\}\\
\floor{x_4}-\floor{x_3}&=\frac{t(z+6)}{24}-zy_1-\left\{\frac{t(z+6)}{24}-zy_1\right\}-\frac{t}{2}+\left\{\frac{t}{2}\right\}+6y_1,
\end{align*}
where $\left\{x\right\}=x-\floor{x}$ is the fractional part of $x$.
We start with the case when $x_3\not\in\mathbb{Z}$, or equivalenty, when $t$ is odd. The inequality \eqref{eq:integerornot1} can be rewritten as
\begin{equation}\label{eq:firstequiv}
-2\left\{\frac{t}{2} \right\}-(z-6)\left\{\frac{t}{12} \right\}+\delta\leq 0
\end{equation}
where 
$$\delta=\left\{\frac{t(z+6)}{24}-zy_0\right\}+\left\{\frac{t(z+6)}{24}+zy_0-z\floor{\frac{t}{12}}\right\}.$$
Since $\left\{m\right\}+\left\{n\right\}\leq\left\{m+n\right\}+1$, and furthermore we can assume for this simplification that $t\in\left\{0,1,2,\ldots,11 \right\}$, we have
$$\delta\leq \left\{ \frac{t}{2}+z\left\{\frac{t}{12} \right\} \right\}+1\leq(z-6)\left\{\frac{t}{12} \right\}+1.$$
Thus we see that \eqref{eq:firstequiv} holds when $t$ is odd, since $-2\left\{ {t}/{2}\right\}+1=0$.

We now move on to the case when $x_3\in\mathbb{Z}$, or equivalenty, when $t$ is even. The argument follows exactly as above, except that the right hand side of \eqref{eq:firstequiv} is 1 and  $-2\left\{ {t}/{2}\right\}+1=1$.

This concludes the proof that $U\leq D$ for general $t$, we now focus on the case when $t\equiv 10\,\, (\mbox{mod } 24)$ and will show that $U\leq D-1$.

We show that in this case, as $y_0$ ranges over the integers $[{t}/{24},{t}/{12}]\cap \mathbb{Z}$, the corresponding $y_1$ is never equal to the height of the lattice point 
$$\left(\frac{t-12 y'}{2},y'\right)\in \mathbb{Z}^2\cap R_D,$$ 
where $y':=\floor{{t}/{24}}$. This means that that lattice point is not accounted for in the proof given above for $U\leq D$, and implies that in fact $U\leq D-1$.
Indeed, $y_1$ is maximized for $y_0=\ceil{{t}/{24}}$ which makes $y_1=\floor{{t}/{12}}-\ceil{{t}/{24}}$. When $t\equiv 10\,\, (\mbox{mod } 24)$, this becomes $y_1=\floor{{t}/{24}}-1<y'$. 
This concludes the proof of Claim \ref{claimUD}.
\end{proof}

\begin{proof}[Proof of Claim \ref{claimdtilde}]
To prove Claim \ref{claimdtilde}, we must compute $d$ and 
$$\tilde{d}=\text{ehr}_{\Delta_{\frac{1}{2},\frac{1}{12}}}(t)-\text{cap}_X(t).$$

Recall that $d=d(t)$ is the number of lattice points on the segment that is the left boundary of $R_D$ (but excluding the potential lattice point $({t}/{4},{t}/{24})$). 
Any such lattice point $(m,n)$ must satisfy $2m+12n=t$, which implies that $t$ is even. Therefore, $d(t)=0$ for $t$ odd.
Conversely, if $t$ is even and $(x,y)$ is on the left boundary of $R_D$ then $x=\frac{t-12y}{2}$ and $y<{t}/{24}$. If furthermore $y\in\mathbb{Z}$ then also $x\in\mathbb{Z}$. It follows that $d(t)=\ceil{{t}/{24}}$ for $t$ even.

To compute the function $\text{ehr}_{\Delta_{\frac{1}{2},\frac{1}{12}}}(t)$ we first use \cite[Exercise 2.34]{beckrobins} to obtain the quadratic and linear terms. To obtain the constant terms, we then use the fact that by \cite[Theorem 3.23]{beckrobins} the function is quasipolynomial with period at most $lcm(2,12)=12$ and compute the necessary ECH capacities of the ellipsoid $E(2,12)$ using the method described in Remark \ref{rem:computeECH}:
$$\text{ehr}_{\Delta_{\frac{1}{2},\frac{1}{12}}}(t)=\frac{t^2}{48}+
\begin{cases} \frac{t}{3} &\mbox{$t$ even}  \\
 \frac{7 t}{24} &\mbox{$t$ odd}  \end{cases}
+
\begin{cases}
1 &t\equiv 0,8\,\,\qquad \frac{4}{3} \quad  t\equiv 4 \\
\frac{11}{16} &t\equiv 1,9\,\, \qquad \frac{49}{48} \quad  t\equiv 5 \\ 
\frac{5}{4} &t\equiv 2,6\,\, \qquad \frac{7}{12}\quad  t\equiv 10\\ 
\frac{15}{16} &t\equiv 3,7\,\, \qquad \frac{13}{48}\quad  t\equiv 11
\end{cases}
\,\,(\mbox{mod } 12)
$$ 
Alternately, it is possible to compute $\text{ehr}_{\Delta_{\frac{1}{2},\frac{1}{12}}}(t)$ using Pick's theorem.

To compute the function  $\text{cap}_X(t)$ we use \cite[Thm 5.11]{worm}, with $\Omega$ being the region in $\mathbb{R}^2$ corresponding to $X$, and therefore with area and $\Omega$-perimeter both equal to 24 and affine perimeter equal to 14. We then have that there exists $t_0\in\mathbb{N}$ such that for $t>t_0$ this function is a quasipolynomial of the form 
$$\text{cap}_X(t)=\text{ehr}_\Omega(\floor{t/48})+\gamma_i,$$
with period 24 and $\gamma_i\in\mathbb{Q}$ for $i=0,\ldots,23$. We find that $t_0=42$ by checking the conditions in \cite[Thm 5.11]{worm} and computing enough values of $\text{cap}_X(t)$ using the methods in \S\ref{subsec:computingECHcap}:
\begin{enumerate}[label=(\roman*)]
\item $t_0\geq 2\times 24 -14=34$;
\item $\text{cap}_X(t_0-1)<\text{cap}_X(t_0)<\text{cap}_X(t_0+1)<\ldots<\text{cap}_X(t_0+47)$.
\end{enumerate}
Computing further values of $\text{cap}_X(t)$ we obtain the general formula for $t>42$:
$$
\text{cap}_X(t)=\frac{t^2}{48}+\frac{7t}{24}+\begin{cases}
\begin{array}{rlrl}
1 & t\equiv 0,10\,\,  & \frac{13}{48} & t\equiv 11,23\\
\frac{11}{16} &  t\equiv 1,9\,\, &\quad -\frac{3}{2} & t\equiv 12,22\\ 
-\frac{2}{3} & t\equiv 2,8\,\,&\quad -\frac{5}{16} & t\equiv 13,21 \\
-\frac{17}{16} & t\equiv 3,7\,\,&\quad \frac{5}{6} & t\equiv 14,20\\ 
\frac{1}{2} & t\equiv 4,6\,\,&\quad \frac{15}{16} & t\equiv 15,19\\ 
\frac{49}{48} & t\equiv 5\,\, &\quad 0 & t\equiv 16,18\\  
&&-\frac{95}{48} & t\equiv 17.
\end{array}
\end{cases}
\,\,(\mbox{mod } 24).$$
Alternatively, we could compute the function $\text{cap}_X(t)$ by using the fact that the sequence of ECH capacities of X is given by the sequence subtraction $N(5)-N(1)$, where $N(j)$ is the sequence of ECH capacities of the ball $E(j,j)$, see \cite[Definition 2.4 and equation (2.6)]{AADT}. 

Still for $t\geq42$ we now compute $\tilde{d}=\text{ehr}_{\Delta_{\frac{1}{2},\frac{1}{12}}}(t)-\text{cap}_X(t)$ and compare it with $d$:
$$\tilde{d}=\begin{cases}
d&t\equiv 0,1,4,5,6,9,11,14,15,19,20,23\\ 
d+1&t\equiv 2,8,13,16,18,21\\ 
d+2&t\equiv 3,7,12,22\\ 
d+3&t\equiv 17\\ 
d-1&t\equiv 10. 
\end{cases}
\,\,(\mbox{mod } 24)$$
For $t\leq 42$ we instead use the computer (see \S\ref{subsec:computingECHcap}) to obtain $\text{cap}_X(t)$ and conclude that Claim \ref{claimdtilde} holds also in this range.
\end{proof}

Since Claims~\ref{claimUD} and ~\ref{claimdtilde} are now proven, so is Claim~\ref{claim15}, and thus the proof of Theorem~\ref{prop:15} is complete.
\end{proof}

\section{Mathematica code}\label{sec:Mathem}

In this section we explain how we experimentally identified the infinite staircases of Theorems \ref{thm:stair}, \ref{thm:L}, and \ref{thm:U}, and the properties of the blocking classes discussed in \S\ref{ss:block}, particularly Proposition \ref{prop:block}.

We first had to produce code which approximated $c_{H_b}$ in a reasonable amount of time and identified the obstructions from different exceptional classes. In \S\ref{subsec:computingECHcap}-\S\ref{subsec:excclasscode} we explain how to plot a lower bound $c_{H_b}^\le$ to $c_{H_b}$, the obstruction from a single ECH capacity $c_k(X_b)$, and the obstruction $\mu_{\bE,b}$ from a single exceptional class $\mathbf{E}$ using Mathematica. In \S\ref{subsec:computingECHcap} we explain the theoretical framework of our code, while \S\ref{subsec:ECHcapcode}-\ref{subsec:excclasscode} primarily contain the code itself. Our methods allow us to use $25{,}000$ ECH capacities to plot $c_{H_b}^\le$ in a reasonable amount of time (e.g., the algorithm producing the $k^\text{th}$ ECH capacity is $O(k)$). They also allow us to visually compare the obstructions arising from ECH capacities with those from exceptional classes, as explained in \S\ref{subsec:ecECH}.

Very similar methods allow for the computation of the capacities and ellipsoid embedding functions for polydisks, with small changes only to the formulas in the functions {\fontfamily{lmtt}\selectfont\textbf{LatticePts}} and {\fontfamily{lmtt}\selectfont\textbf{Action}}. These changes are due to the fact that the paths $\Lambda$ and the region $\Omega$ determining the polydisk as a toric domain will have sides of different slopes than those of the analogous paths and regions in the case of $H_b$. 
The obstructions from the exceptional classes of polydisks require changes in the formulas for exceptional classes and obstructions: see \cite[\S1.1]{usher} for the case of polydisks and \cite[\S2.2]{AADT} for the case of general rational toric domains. Note that our code is optimized for cases where $\Omega$ has two sides in addition to the sides on the axes, as the Diophantine equations solved by the function {\fontfamily{lmtt}\selectfont\textbf{Index}} become much more computationally expensive as the number of sides of $\Omega$ increases. This makes the code from \cite[Appendix C]{AADT} more useful when $\Omega$ has more than two sides in addition to the sides on the axes.

In \S\ref{subsec:findingk} we explain our strategy for identifying blocking classes, approximating $Block$, and identifying elements of $Stair$ experimentally. While some of our methods are generalizations of the analysis of blocking classes from \S\ref{ss:block}, we explain how we apply these ideas using a computer. Finally, \S\ref{subsec:pictures} contains illustrations of examples discussed throughout the paper.

\subsection{Computing many ECH capacities of $X_b$ quickly}\label{subsec:computingECHcap}

Implementing Theorem \ref{thm:lattice path} directly in practice is rather slow. The number of lattice points $\mathcal{L}(\Lambda)$ bounded by the lattice path $\Lambda$ is a quadratic polynomial in the number of sides of $\Lambda$.  Enumerating all convex lattice paths $\Lambda$ with $\mathcal{L}(\Lambda)=k+1$ requires solving the equation $\mathcal{L}(\Lambda)=k+1$ with the path $\Lambda$ having a number of sides that ranges between one and some upper bound depending on $k$. Once $k$ is large enough that $\Lambda$ could have three or more sides, this procedure slows dramatically.

In order to shorten our search, we constrain the slopes of the edges of the lattice paths to lie within a certain range. We then use Lemma \ref{lem:restrictedlps}, inspired by \cite[Exercise 4.16(a)]{hECHlec}, to restrict the set of lattice paths over which the minimum of (\ref{cbylatticepaths}) is taken to those with sides parallel to the vectors $(1,0)$ and $(1,-1)$, at the expense of allowing a broader range of values of $\mathcal{L}(\Lambda)$.

\begin{DEFN} For any nonzero vector $v\in\R^2$, define $\theta(v)\in(-2\pi,0]$ to be the angle for which $v$ is a positive multiple of $(\cos\theta(v),\sin\theta(v))$.
\end{DEFN}

Note that this is not the same as the $\theta$ of in \cite[Definition A.6]{cg}.

The proof of the following Lemma was pointed out to us by Michael Hutchings.

\begin{lemm}\label{lem:restrictedslopes} The minimum in (\ref{cbylatticepaths}) can be taken over convex lattice paths $\Lambda$ with $\mathcal{L}(\Lambda)=k+1$ and $\theta(\Lambda'(t))\in\left[-{\pi}/{4},0\right]$ for all $t$ for which $\Lambda(t)$ is not a vertex of $\Lambda$.
\end{lemm}
\begin{proof} The argument is a repeat of the first proof of Proposition 5.6 in \cite{hbeyond}.
\end{proof}

\begin{lemm}\label{lem:restrictedlps} When $\Omega=\Omega_b$, the Delzant polytope of $H_b$, the minimum on the right hand side of (\ref{cbylatticepaths}) is the same if it is taken over convex lattice paths $\Lambda$ whose edges are parallel to the edges of $\Omega$ and for which
\begin{equation}\label{eqn:rangeoflp}
k+1\leq\mathcal{L}(\Lambda)\leq 2k+1.
\end{equation}
\end{lemm}

\begin{REMark} \normalfont
The first conclusion of Lemma \ref{lem:restrictedlps} (i.e., without the bound \eqref{eqn:rangeoflp}) is inspired by \cite[Exercise 4.16(a)]{hECHlec}, which holds for convex domains in $T^*T^2$.
\end{REMark}

\begin{proof} By Lemma \ref{lem:restrictedlps}, if $\Lambda(t)$ is not a vertex and $\theta(\Lambda'(t))\not\in\left\{0,-{\pi}/{4}\right\}$ then $\theta(\Lambda'(t))\in\left(0,-{\pi}/{4}\right)$. Let $(t_0,t_1)$ denote the interval on which $\theta(\Lambda'(t))\in\left(0,-{\pi}/{4}\right)$, assuming it is nonempty. Let $\Lambda(t_0)=(x_0,y_0)$ and $\Lambda(t_1)=(x_1,y_1)$.

Denote by $\Lambda_\Omega$ the new convex lattice path obtained by replacing $\Lambda|_{[t_0,t_1]}$ with the convex path consisting of two edges, one parallel to $(1,0)$ followed by one parallel to $(1,-1)$. This means the only vertex of $\Lambda_\Omega$ with both coordinates greater than zero is the point $(x_1-(y_0-y_1),y_0)$; it has positive coordinates because $\theta(\Lambda')|_{(t_0,t_1)}\in\left(0,-{\pi}/{4}\right)$, therefore $x_1-x_0>y_0-y_1$ and $y_0>y_1\geq0$.

The new path $\Lambda_\Omega$ is a lattice path, because $x_1-x_0$ and $y_0-y_1$ are integers. While the argument until this point could apply to more general $\Omega$, it would not always guarantee that $\Lambda_\Omega$ has vertices at lattice points.

The number of lattice points between $\Lambda_\Omega$ and the axes is at least $\mathcal{L}(\Lambda)$. It remains to show that it is at most $2\mathcal{L}(\Lambda)-1$. It suffices to consider the case $x_0=0, y_1=0$. We have
\[
\mathcal{L}(\Lambda_\Omega)=(x_1-y_0+1)(y_0+1)+\frac{y_0(y_0+1)}{2}
\]
while for $\mathcal{L}(\Lambda)$, we have is at least the number of lattice points between the line from $(0,y_0)$ to $(x_1,0)$ and the axes. That is,
\begin{align*}
\mathcal{L}(\Lambda)&\geq\#\{\text{lattice points between the line from $(0,y_0)$ to $(x_1,0)$ and the axes}\}
\\&\geq\frac{(x_1+1)(y_0+1)}{2}
\end{align*}
Therefore if $y_0\geq1$,
\[
\mathcal{L}(\Lambda_\Omega)\leq(x_1+1)(y_0+1)-1\leq2\mathcal{L}(\Lambda)-1
\]
If $y_0=0$, then $\Lambda_\Omega=\Lambda$ because $\Lambda$ consists of a single edge from $(0,0)$ to $(x_1,0)$.
\end{proof}

\subsection{Obstructions from single ECH capacities and a lower bound for $c_{H_b}$}\label{subsec:ECHcapcode}

Recall that
\begin{align*}
{\it Stair}: &= \{b \ | \ H_b \mbox{ has a staircase} \} \subset [0,1) ,\;\;\mbox{ and }\\
{\it Block}: & = \bigcup \bigl\{ J_{\bB} \ \big| \ \bB  \mbox{ is a blocking class} \bigr\} \subset [0,1).
\end{align*}

It its not possible for a computer to compute the right hand side of (\ref{def:csup}) for general $b$. However, understanding a close lower bound for $c_{H_b}$ allowed us to approximate the set $Block$ and to identify many of the blocking classes contributing to $Block$. Define
\[
c_{H_b}^\le(z):=\max_{k=1,\dots,25000}\frac{c_k(E(1,z))}{c_k(X_b)}
\]
Then we have
\[
Block^\le:=\left\{b\in[0,1)\middle|c_{H_b}^\le(\acc(b))>V_b(\acc(b))\right\}\subset Block
\]
because for all $z$,
\[
c_{H_b}^\le(z)\leq c_{H_b}(z)
\]
In this subsection we explain our code to plot $c_{H_b}^\le$ and the obstructions $\frac{c_k(E(1,z))}{c_k(X_b)}$ from a single ECH capacity.

Let $\Lambda$ be a convex lattice path with edges parallel to the vectors $(1,0)$ and $(1,-1)$, i.e., the edges of $\Omega_b$. The set of all such lattice paths is in bijection with the first quadrant, inclusive of the axes, via the Cartesian coordinates $(x,y)$ of the vertex where the edge parallel to $(1,0)$ ends and the edge parallel to $(1,-1)$ begins. In these coordinates, if $\Lambda$ corresponds to $(x,y)$, we compute $\mathcal{L}(\Lambda)$ using

\vspace{3pt}
\noindent
{\fontfamily{lmtt}\selectfont\textbf{LatticePts}[x\textunderscore,y\textunderscore]:=(x+1)(y+1)+(y(y+1))/2}
\vspace{3pt}

We compute the set of lattice paths $\Lambda$ with $\mathcal{L}(\Lambda)=k+1$ using

\vspace{3pt}
\noindent
{\fontfamily{lmtt}\selectfont\textbf{Index}[k\textunderscore]:=\textbf{Values}[\textbf{Solve}[\textbf{LatticePts}[x,y] == k+1 \&\& 0$\leq$x \&\& 0$\leq$y,$\{$x, y$\}$,

Integers]]}
\vspace{3pt}

We compute the set of lattice paths $\Lambda$ with sides parallel to $(1,0)$ and $(1,-1)$ and $\mathcal{L}(\Lambda)$ at most $50{,}001$, and give it the name {\fontfamily{lmtt}\selectfont\textbf{LatticePaths50000}}. By Lemma \ref{lem:restrictedlps}, this will allow us to compute the first $25{,}000$ ECH capacities of $H_b$.

\vspace{3pt}
\noindent
{\fontfamily{lmtt}\selectfont\textbf{LatticePaths50000}=\textbf{Table}[$\{$k,\textbf{Index}[k]$\}$,$\{$k, 0, 50000$\}$]}
\vspace{3pt}

Running {\fontfamily{lmtt}\selectfont\textbf{LatticePaths50000}} takes approximately twenty minutes on a personal laptop. However, it's a one-time cost: plotting each $c_{H_b}$ takes far less time.

To compute the action of a lattice path $\Lambda$ corresponding to $(x,y)$, we compute its $\Omega_b$-length. When $\Omega=\Omega_b$, the definition (\ref{eqn:Omega-length}) is equivalent to
\[
\ell_{\Omega_b}(x,y)=x(1-b)+y
\]
therefore to compute the action of $\Lambda$ we use

\vspace{3pt}
\noindent
{\fontfamily{lmtt}\selectfont\textbf{Action}[b\textunderscore]:=\textbf{Function}[$\{$x,y$\}$,x(1-b)+y]}
\vspace{3pt}

Instead of computing the actions of all the generators from {\fontfamily{lmtt}\selectfont\textbf{LatticePaths50000}} and minimizing in one step, it is much faster to break the process up into the following three functions:

\vspace{3pt}
\noindent
{\fontfamily{lmtt}\selectfont\textbf{ActionList50000}[b\textunderscore]:=\textbf{Table}[\textbf{Table}[\textbf{Action}[b]

@@\textbf{LatticePaths50000}[[k+1]][[2]][[i]],

$\{$i,\textbf{Length}[\textbf{LatticePaths50000}[[k+1]][[2]]]$\}$],$\{$k,0,50000$\}$]}

\vspace{3pt}
\noindent
{\fontfamily{lmtt}\selectfont\textbf{MinActionList50000}[ALb\textunderscore]:=\textbf{Array}[\textbf{Min}[ALb[[\#]]]\&,50001]}

\vspace{3pt}
\noindent
{\fontfamily{lmtt}\selectfont\textbf{CapacitiesList}[MALb\textunderscore]:=\textbf{Join}[$\{$0.$\}$,\textbf{Table}[\textbf{Min}[\textbf{Array}[MALb[[\#]]\&,k,k+1]],

$\{$k,1,25000$\}$]]}
\vspace{3pt}

Because the index origin is set to be $k+1$, {\fontfamily{lmtt}\selectfont\textbf{CapacitiesList}} minimizes over the entries $k+1,\dots,2k+1=2(k+1)-1$ of {\fontfamily{lmtt}\selectfont MALb}, which Lemma \ref{lem:restrictedlps} guarantees are all we need to consider. In order to compute the capacities of $X_b$, it is fastest to run each command in succession, applying {\fontfamily{lmtt}\selectfont\textbf{MinActionList50000}} to the output of {\fontfamily{lmtt}\selectfont\textbf{ActionList50000}}, then {\fontfamily{lmtt}\selectfont\textbf{CapacitiesList}} to the output of {\fontfamily{lmtt}\selectfont\textbf{MinActionList50000}}, instead of running their composition. It is also much faster to treat $b$ as a decimal number rather than an exact number. Note that this is the point in the process where it is possible to introduce imprecision. If the plot of $c_{H_b}^\le$ looks close to but cannot for some reason be less than or equal to the actual plot of $c_{H_b}$, the rounding inherent in treating $b$ as a decimal is probably to blame. This does happen for known infinite staircases at extremely small scales.

To compute the capacities of the ellipsoid $E(1,z)$, we use

\vspace{3pt}
\noindent
{\fontfamily{lmtt}\selectfont\textbf{EllipsoidCap}:=\textbf{Compile}[$\{\{$z,\textunderscore Real$\}$,$\{$K,\textunderscore Integer$\}\}$,\textbf{Take}[\textbf{Sort}[\textbf{Flatten}[\textbf{Table}[

x+z*y,$\{$x,0,\textbf{Ceiling}[z((4/z)(-z-1/2+\textbf{Sqrt}[(z+1/2)$^2$-z+z*K/2]))]$\}$,

$\{$y,0,\textbf{Ceiling}[(4/z)(-z-1/2+\textbf{Sqrt}[(z+1/2)$^2$-z+z*K/2])]$\}$]],\textbf{Less}],K+1]]}
\vspace{3pt}

The bounds on $x$ and $y$ in {\fontfamily{lmtt}\selectfont\textbf{EllipsoidCap}} come from the relationship between the sequence $N(1,z)$ and triangles in the lattice.

In order to compute the obstruction $\frac{c_k(E(1,z))}{c_k(X_b)}$ from the $k^\text{th}$ ECH capacity, use the function 
{\fontfamily{lmtt}\selectfont\textbf{SingleCapacityObstructionFxn}} with {\fontfamily{lmtt}\selectfont\textbf{EllipsoidCap}[z,K]} (for $K\geq k+1$) as {\fontfamily{lmtt}\selectfont ECzk}, and the list of $25{,}000$ capacities previously computed for $X_b$ as {\fontfamily{lmtt}\selectfont CLbK}:

\vspace{3pt}
\noindent
{\fontfamily{lmtt}\selectfont\textbf{SingleCapacityObstructionFxn}[ECzK\textunderscore,CLbK\textunderscore,k\textunderscore]:=ECzK[[k+1]]/CLbK[[k+1]]}
\vspace{3pt}

To plot {\fontfamily{lmtt}\selectfont\textbf{SingleCapacityObstructionFxn}} over an interval from {\fontfamily{lmtt}\selectfont zmin} to {\fontfamily{lmtt}\selectfont zmax} as a line plot with step {\fontfamily{lmtt}\selectfont zstep}, plug the list of capacities for $X_b$ in as {\fontfamily{lmtt}\selectfont CLbK} with {\fontfamily{lmtt}\selectfont k} at most the length of {\fontfamily{lmtt}\selectfont CLbK} into

\vspace{3pt}
\noindent
{\fontfamily{lmtt}\selectfont\textbf{PlotSingleCapacityObstructionFxn}[zmin\textunderscore,zmax\textunderscore,zstep\textunderscore,CLbK\textunderscore,k\textunderscore,b\textunderscore]:=\textbf{Show}[

\textbf{ListLinePlot}[\textbf{Array}[$\{$(\#-1)*zstep+zmin,\textbf{SingleCapacityObstructionFxn}[

\textbf{EllipsoidCap}[(\#-1)*zstep+zmin,k+5],CLbK,k]$\}$\&,\textbf{Floor}[(zmax-zmin)/zstep]],

\textbf{PlotStyle}$\shortrightarrow$\textbf{Cyan}]]}
\vspace{3pt}

To compute $c_{H_b}^\le(z)$, find the maximum over the first $25{,}000$ capacities of $\frac{c_k(E(1,z))}{h}$, where $h$ is the $k+1^\text{st}$ element of the output of {\fontfamily{lmtt}\selectfont\textbf{CapacitiesList}} (remember the first element will be $c_0(X_b)$). This is done by plugging {\fontfamily{lmtt}\selectfont\textbf{EllipsoidCap}[z,25,000]} as {\fontfamily{lmtt}\selectfont ECzk}, the list of $25{,}000$ capacities previously computed for $X_b$ as {\fontfamily{lmtt}\selectfont CLbK}, and $K=25{,}000$ into

\vspace{3pt}
\noindent
{\fontfamily{lmtt}\selectfont\textbf{CapacityFxn}[ECzK\textunderscore,CLbK\textunderscore,K\textunderscore]:=\textbf{Max}[\textbf{Array}[ECzK[[\#+1]]/CLbK[[\#+1]]\&,K]]}
\vspace{3pt}

In order to plot {\fontfamily{lmtt}\selectfont\textbf{CapacityFxn}} on $[\text{{\fontfamily{lmtt}\selectfont zmin},{\fontfamily{lmtt}\selectfont zmax}}]$ as a line plot with step {\fontfamily{lmtt}\selectfont zstep}, plug your list of capacities for $X_b$ as {\fontfamily{lmtt}\selectfont CLbK} with $K=25{,}000$ into

\vspace{3pt}
\noindent
{\fontfamily{lmtt}\selectfont\textbf{PlotCapacityFxn}[zmin\textunderscore,zmax\textunderscore,zstep\textunderscore,CLbK\textunderscore,K\textunderscore,b\textunderscore]:=

\textbf{Show}[\textbf{Plot}[\textbf{Sqrt}[(z/2)/(.5(1-b$^2$))],

$\{$z,zmin,zmax$\}$,\textbf{PlotStyle}$\shortrightarrow$\textbf{Orange}],\textbf{ListLinePlot}[

\textbf{Array}[$\{$(\#-1)*zstep+zmin,\textbf{CapacityFxn}[\textbf{EllipsoidCap}[(\#-1)*zstep+zmin,K],

CLbK,K]$\}$\&,\textbf{Floor}[(zmax-zmin)/zstep]]]]}
\vspace{3pt}

Starting with {\fontfamily{lmtt}\selectfont zstep} at $0.01$ is a good choice if the interval is $[1,10]$; when analyzing the behavior near the accumulation point, with the interval much smaller, it is possible to make {\fontfamily{lmtt}\selectfont zstep} much smaller. We frequently use a {\fontfamily{lmtt}\selectfont zstep} of $0.000001$.

\subsection{Obstructions from exceptional classes}\label{subsec:excclasscode}

We now discuss our code computing $\mu_{\bE,b}$ for exceptional classes $\bE$.

It is often useful to understand obstructions which do not correspond in the sense of Remark \ref{rmk:lpec} to any of the first $25{,}000$ ECH capacities. As in the case of identifying lattice paths, finding all relevant exceptional classes takes a lot of time and effort, whether done by a computer or a human. But once $\bE$ is identified, computing $\mu_{\mathbf{E},b}(z)$ can be done quickly. We explain our code for this here.

Plotting the obstruction $\mu_{\mathbf{E},b}$ from a single exceptional class $\mathbf{E}$ first requires computing the weight expansion of a number $a$, using

\vspace{3pt}
\noindent
{\fontfamily{lmtt}\selectfont\textbf{WeightDecomp}[p\textunderscore,q\textunderscore]:=\textbf{Module}[$\{$list=$\{\}$,aux,bux,k$\}$,\textbf{If}[p>q,$\{$aux,bux$\}$=$\{$p,q$\}$,

$\{$aux,bux$\}$=$\{$q,p$\}$];\textbf{While}[\textbf{Chop}[bux]$\neq$0,k=\textbf{Floor}[aux/bux];\textbf{AppendTo}[list,

\textbf{Table}[bux,k]];$\{$aux,bux$\}$=$\{$bux,aux-k*bux$\}$];\textbf{Flatten}[list]]}
\vspace{3pt}

If $z={p}/{q}$ is rational, then {\fontfamily{lmtt}\selectfont\textbf{WeightDecomp}[p,q]} computes $q\mathbf{w}(z)$, while {\fontfamily{lmtt}\selectfont\textbf{WeightDecomp}[z,1]} computes $\bw(z)$.

We take the dot product of vectors of unequal length using

\vspace{3pt}
\noindent
{\fontfamily{lmtt}\selectfont\textbf{ExtendedDot}[M\textunderscore,w\textunderscore]:=\textbf{Take}[M,\textbf{Min}[\textbf{Length}[M],\textbf{Length}[w]]].\textbf{Take}[w,\textbf{Min}[

\textbf{Length}[M],\textbf{Length}[w]]]}
\vspace{3pt}

For $\mathbf{E}=(d,m,\mathbf{m})$ we compute $\mu_{\mathbf{E},b}(z)$ by setting {\fontfamily{lmtt}\selectfont d }$=d$, {\fontfamily{lmtt}\selectfont m }$=m$, {\fontfamily{lmtt}\selectfont M }$=\mathbf{m}$, and {\fontfamily{lmtt}\selectfont z }$=z$ in

\vspace{3pt}
\noindent
{\fontfamily{lmtt}\selectfont\textbf{Mu}[b\textunderscore,d\textunderscore,m\textunderscore,M\textunderscore,z\textunderscore]:=\textbf{ExtendedDot}[M,\textbf{WeightDecomp}[a,1]]/(d-b*m)}
\vspace{3pt}

We plot {\fontfamily{lmtt}\selectfont\textbf{Mu}} for $z\in[\text{{\fontfamily{lmtt}\selectfont zmin}},\text{{\fontfamily{lmtt}\selectfont zmax}}]$ with step {\fontfamily{lmtt}\selectfont zstep} using

\vspace{3pt}
\noindent
{\fontfamily{lmtt}\selectfont\textbf{PlotMu}[zmin\textunderscore,zmax\textunderscore,zstep\textunderscore,d\textunderscore,m\textunderscore,M\textunderscore,b\textunderscore]:=\textbf{Show}[\textbf{ListLinePlot}[\textbf{Array}[

$\{$(\#-1)*zstep+zmin,\textbf{Mu}[b,d,m,M,(\#-1)*zstep+zmin]$\}$\&,\textbf{Floor}[(zmax-zmin)/

zstep]],\textbf{PlotStyle}$\shortrightarrow$\textbf{Brown}]]}
\vspace{3pt}


\begin{figure}[H]
	\includegraphics[width=.9\linewidth]{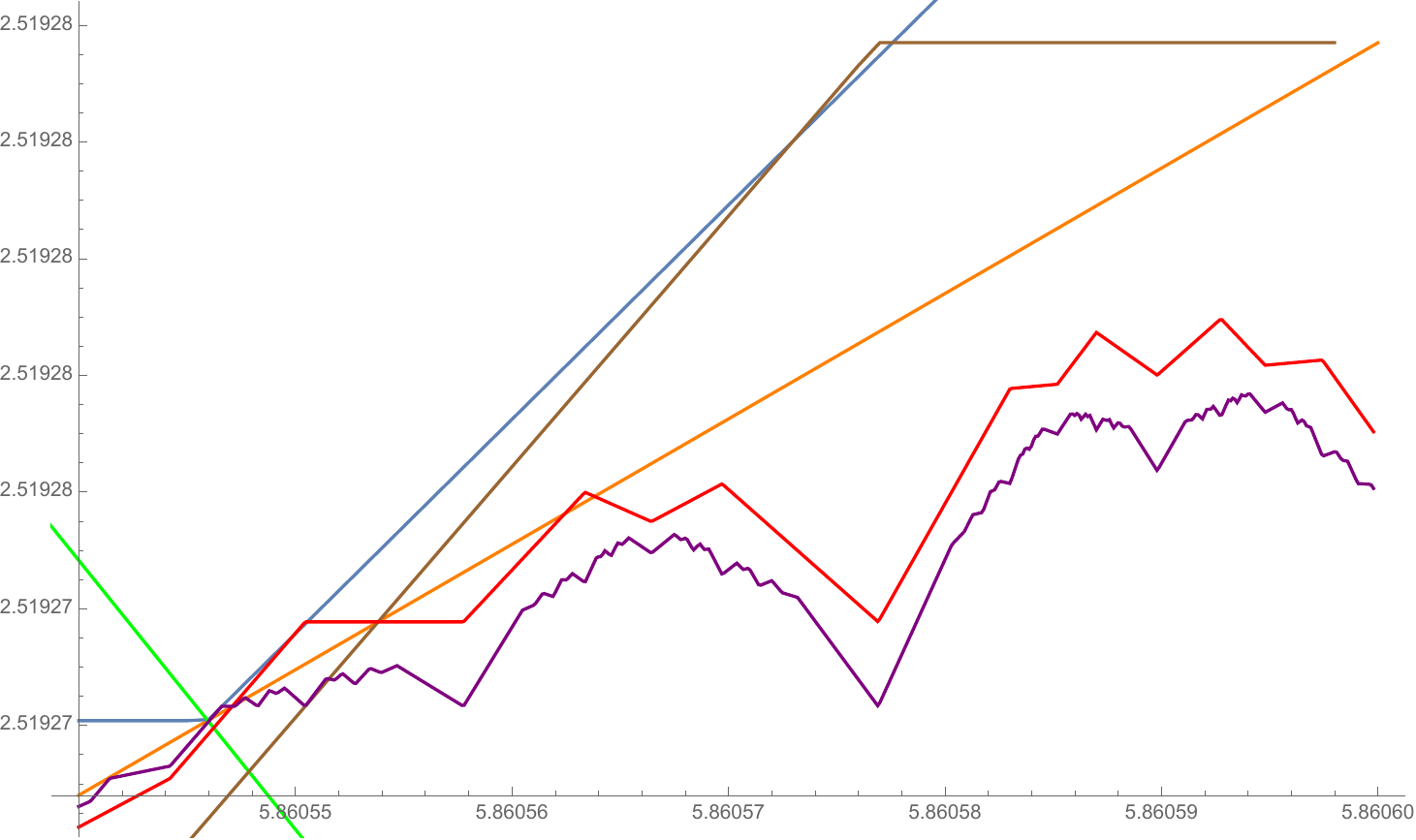}
	\caption{This is a plot of the capacity and volume functions for $b=b_0^E=\acc^{-1}_L([5;1,6,\{5,1\}^\infty])$. The function $c_{H_{b_0^E}}^\le(z)$ is in dark blue, and in this case it is not a good approximation of the true $c_{H_{b_0^E}}(z)$. The volume obstruction is in orange. The green curve is parameterized by $b\mapsto\left(\acc(b),\sqrt{\frac{\acc(b)}{1-b^2}}\right)$, therefore the green curve intersects the orange when $z=\acc(b_0^E)$. 
Using $\bE_{u,0,k}$ to denote the exceptional class with center $[5;1,6,\{5,1\}^k,4]$ and $\bE_{u,0,k}'$ to denote the exceptional class with center $[5;1,6,\{5,1\}^k,5,2]$, the obstruction $\mu_{\mathbf{E}_{u,0,1},b_0^E}(z)$ is in brown, the obstruction $\mu_{\mathbf{E}_{u,0,1}',b_0^E}(z)$ is in red, and the obstruction $\mu_{\mathbf{E}_{u,0,2},b_0^E}(z)$ is in purple. 
Because these correspond, in the sense of Remark \ref{rmk:lpec}, to the obstructions from the $127{,}489^\text{th}, 872{,}234^\text{th}$, and $5{,}971{,}895^\text{th}$ ECH capacities, respectively, they are not captured by our program computing $c_{H_{b_0^E}}^\le$: in fact, on this interval $c_{H_{b_0^E}}^\le$ is given instead by the obstruction from $c_8(H_{b_0^E})$.}
	\label{fig:b0E}
\end{figure}

In Figure \ref{fig:b0E} we have used {\fontfamily{lmtt}\selectfont\textbf{PlotMu}} to depict the obstructions from the exceptional classes whose obstructions underlay the steps
 centered at ${1219}/{208}, {3194}/{545}$, and ${8363}/{1427}$ in the infinite staircase
  for $b=b_0^E$ of Theorem \ref{thm:E}. Using the procedure {\fontfamily{lmtt}\selectfont\textbf{PlotMu}} instead of the procedure 
  {\fontfamily{lmtt}\selectfont\textbf{PlotSingleCapacityObstructionFxn}} allows us to plot these stairs, as the corresponding ECH capacities are the $127{,}489^\text{th}, 872{,}234^\text{th}$, and $5{,}971{,}895^\text{th}$, respectively, which do not contribute to $c_{H_{b_0^E}}^\le(z)$. In order not to overload the figure, we have not drawn the obstruction coming from the classes centered at ${170}/{29}, {463}/{79}$, whose centers are greater than the values of $z$ plotted, or the obstructions from the class centered at ${21895}/{3736}$, or any obstructions from Theorem \ref{thm:E} with $k\geq3$, which Mathematica cannot plot at any scale with any definition.
  
 In general, once we have an exceptional class in mind, it is computationally faster to run {\fontfamily{lmtt}\selectfont\textbf{PlotMu}} instead of {\fontfamily{lmtt}\selectfont\textbf{PlotSingleCapacityObstructionFxn}}, because in order to run the latter we first need to compute all the ECH capacities up to $c_k$, where $k=\frac{1}{1}(d(d+3)-m(m+1))$. If we are interested in an exact irrational value of $b$, there is a great difference, although we almost never do this. However, if we want to overlay our plots on a graph of $c_{H_b}^\le$, there is no discernible difference, because once we have computed the first 25{,}000 capacities of $X_b$ necessary to graph $c_{H_b}^\le$, any difference in speed between {\fontfamily{lmtt}\selectfont\textbf{PlotMu}} and {\fontfamily{lmtt}\selectfont\textbf{PlotSingleCapacityObstructionFxn}} is not noticeable.

\subsection{Strategy for finding staircases}\label{subsec:findingk}

In this section we explain our experimental strategy for identifying infinite staircases. We first explain our method for approximating $Block^\le$, then our two methods for finding staircases once we understood much of $Block^\le$.

We investigate $Block^\le$ by identifying the exceptional class realizing the values of $k$ for which $c_k(X_b)$ determines $c_{H_b}^\le(\acc(b))$. At first, we simply chose values of $b$ at random and identified the smallest $k$ for which the plot from running {\fontfamily{lmtt}\selectfont\textbf{PlotSingleCapacityObstructionFxn}} agreed with the plot from {\fontfamily{lmtt}\selectfont\textbf{PlotCapacityFxn}} for values of $z$ near $\acc(b)$. We encode $\acc(b)$ using

\vspace{3pt}
\noindent
{\fontfamily{lmtt}\selectfont\textbf{AccPt}[b\textunderscore]:=(1/2)((3-b)$^2$/(1-b$^2$)-2+\textbf{Sqrt}[((3-b)$^2$/(1-b$^2$)-2)$^2$-4])}
\vspace{3pt}

We identified the appropriate values of $k$ by visually inspecting the plots of the obstruction from $c_k$ overlaid upon the plots of $c_{H_b}^\le$. See Figure \ref{fig:b3k125accpt}.

\newpage

\begin{figure}[H]
	\includegraphics[width=.6\linewidth]{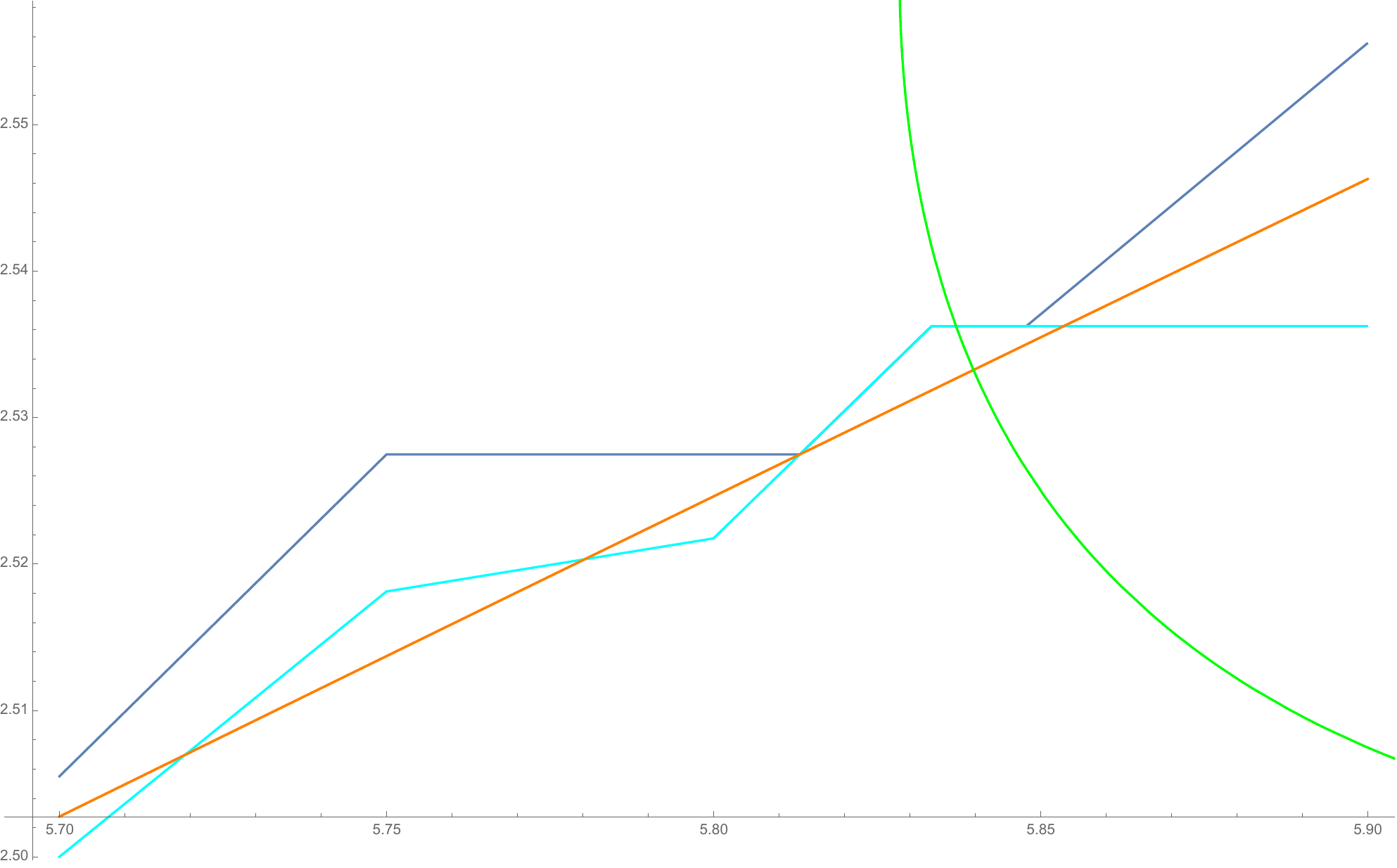}
	\caption{	
	Here, $b = 0.3$.  The function $c_{H_{0.3}}^\le(z)$ is in dark blue, the volume obstruction is in orange, and the obstruction from the $125^\text{th}$ ECH capacity is in bright blue. The green curve is parameterized by the potential accumulation point, therefore the orange and green curves intersect when $z = \acc(0.3)$, and we can see that the dark blue and the light blue plots coincide in a neighbourhood of that point. Indeed, 125 is the smallest $k$ for which $c_k$ obstructs an infinite staircase at $b=0.3$. Note that $c_{H_{0.3}}^\le(z)$ has a corner point at $z=35/6$, and near this value is given by the class 
$\bE=\bigl(15,4;6\bw(35/6)\bigr)$ as described in Remark~\ref{rmk:ECH}~(ii).
	}
	\label{fig:b3k125accpt}
\end{figure}

Inspired by our computations, when $c_{H_b}(\acc(b))>\sqrt{\frac{\acc(b)}{1-b^2}}$ and $k$ satisfies
\[
\frac{c_k(E(1,\acc(b)))}{c_k(X_b)}=c_{H_b}(\acc(b))
\]
we say \textbf{$c_k$ obstructs an infinite staircase at $b$}. 

For values of $b$ very close to one for which $c_{H_b}$ has an infinite staircase, the smallest $k$ for which $c_k$ obstructs an infinite staircase at $b$ becomes very high, so it grows infeasible to check every possible value of $k$ by inspection of plots such as the one depicted in Figure \ref{fig:b3k125accpt}. We introduced a new function, {\fontfamily{lmtt}\selectfont\textbf{IndexAtAccPt}[b\textunderscore,CLbK\textunderscore,K\textunderscore]}. Again, the list {\fontfamily{lmtt}\selectfont CLbK} is the list of the first $25{,}000$ capacities of $H_b$. {\fontfamily{lmtt}\selectfont\textbf{IndexAtAccPt}} outputs the values of $k\in\{0,\dots,K\}$ for which $\frac{c_k(E(1,\acc(b)))}{c_k(X_b)}$ is maximized. The reason we want to be able to use values of $K<25{,}000$ is because when $K=25{,}000$, {\fontfamily{lmtt}\selectfont\textbf{IndexAtAccPt}} is fairly slow.

\vspace{3pt}
\noindent
{\fontfamily{lmtt}\selectfont\textbf{IndexAtAccPt}[b\textunderscore,CLbK\textunderscore,K\textunderscore]:=\textbf{Block}[$\{$list=\textbf{Array}[\textbf{EllipsoidCap}[

\textbf{AccPt}[b],K][[\#+1]]/CLbK[[\#+1]]\&,K]$\}$,\textbf{Position}[list,\textbf{Max}[list]]]}

\vspace{3pt}

Given $k$ and $b$ such that $c_k$ obstructs an infinite staircase at $b$, we also want to identify any quasi-perfect exceptional classes giving the same obstruction, that is, those ``corresponding" exceptional classes in the sense of Remark \ref{rmk:lpec}. For low values of $k$, this can be done by hand, see Remark \ref{rmk:ECH} (ii). For large values of $k$, we do this by applying the function {\fontfamily{lmtt}\selectfont\textbf{FindPQFromDM}} to the values of the output of {\fontfamily{lmtt}\selectfont\textbf{FindDMFromK}}, where

\vspace{3pt}
\noindent
{\fontfamily{lmtt}\selectfont\textbf{FindDMFromK}[k\textunderscore]:=\textbf{Solve}[.5(d(d+3)-m(m+1))==k\&\&d$\geq$0\&\&m$\geq$0,$\{$d,m$\}$,Integers]}

\vspace{3pt}
\noindent
{\fontfamily{lmtt}\selectfont\textbf{FindPQFromDM}[d\textunderscore,m\textunderscore]:=\textbf{Solve}[d$^2$+1-m$^2$==p*q\&\&3d-m==p+q\&\&p$>$q\&\&q$\ge$0\&\&

\textbf{GCD}[p,q]==1,$\{$p,q$\}$,Integers]}
\vspace{3pt}

Through the application of these two functions we have occasionally found multiple candidate quasi-perfect classes, but it has always turned out that if there are any solutions, then exactly one makes sense for the value of $b$ under consideration.

We expect that for many $k$, the ECH capacity $c_k$ obstructs an infinite staircase for an interval of values of $b$. In our computations, for each $k$ where $c_k$ obstructs an infinite staircase, we have been able to find a perfect center-blocking class $\bB$ corresponding to $c_k$ as in Remark \ref{rmk:lpec}. Therefore, for every pair $b_1\leq b_2$ we have found whose infinite staircases are obstructed by the same $c_k$, we know by Lemma \ref{lem:block0} that
\[
[b_1,b_2]\subset J_\bB,
\]
and therefore
\[
[b_1,b_2]\subset Block^\le\subset Block.
\]

By identifying many such pairs $b_1$ and $b_2$, we built up a good approximation to $Block^\le$. Next we searched for values of $b$ whose embedding function $c_{H_b}$ contained infinite staircases outside our approximation to $Block^\le$. We explain our two methods for doing so.

\subsubsection{Checking rational points in $[0,1)\setminus Block^\le$}\label{sssec:firstmethod}
 
Suppose we know that $c_k$ and $c_{k'}$ obstruct infinite staircases at $b,b'$ respectively, with $b<b'$, and that we have checked enough values $b''\in(b,b')$ at small enough resolution to convince ourselves that there is no $k\in\{1,\dots,25{,}000\}$ for which $c_k$ obstructs an infinite staircase at any $b''\in(b,b')$. In terms of our computations, this last condition means that for all $b''\in(b,b')$ we have checked, we have $c_{H_{b''}}^\le(\acc(b''))\le V_{b''}(\acc(b''))$.

We would like to search through all rational numbers ${p}/{q}\in(\acc(b),\acc(b'))$ (if $b,b'\ge1/3$; if $b,b'\le1/3$ then we check the interval $(\acc(b'),\acc(b))$) to see if there are any $d,m$ for which $\left(d,m;q\bw\left({p}/{q}\right)\right)$ solves the Diophantine equations (\ref{eq:diophantine}) and whose continued fractions contain repeated patterns. We do this using the function {\fontfamily{lmtt}\selectfont\textbf{ActualClassesWithCF}} applied to $(\acc(b),\acc(b'),q1,q2)$, which searches through all such rational numbers with $q1\leq q\leq q2$ (when $b,b'\le1/3$, we switch the places of $\acc(b)$ and $\acc(b')$ accordingly). To define this function we need the preliminary function {\fontfamily{lmtt}\selectfont\textbf{NonemptyClassesWithCF}}:

\vspace{3pt}
\noindent
{\fontfamily{lmtt}\selectfont\textbf{NonemptyClassesWithCF}[LB\textunderscore,UB\textunderscore,q1\textunderscore,q2\textunderscore]:=\textbf{Select}[Table[Table[$\{$\textbf{Flatten}[$\{$

\textbf{Values}[\textbf{Solve}[d$^2$+1-m$^2$==pq\&\&3d-m==p+q\&\&d$\geq$0\&\&m$\geq$0,$\{$d,m$\}$,Integers]],p,q$\}$],

\textbf{ContinuedFraction}[p/q]$\}$,$\{$p,\textbf{Ceiling}[LBq],\textbf{Floor}[UBq]$\}$],$\{$q,q1,q2$\}$],\#$\neq\{\}$\&]}

\vspace{3pt}
\noindent
{\fontfamily{lmtt}\selectfont\textbf{ActualClassesWithCF}[LB\textunderscore,UB\textunderscore,q1\textunderscore,q2\textunderscore]:=\textbf{Select}[\textbf{Flatten}[

\textbf{NonemptyClassesWithCF}[LB,UB,q1,q2],1],\textbf{Length}[\#[[1]]]==4\&]}
\vspace{3pt}

{\fontfamily{lmtt}\selectfont\textbf{ActualClassesWithCF}} returns lists with terms of the form $\left((d,m,p,q),CF\left({p}/{q}\right)\right)$ where $\left(d,m,q\bw\left({p}/{q}\right)\right)$ solves the Diophantine equations (\ref{eq:diophantine}), the rational ${p}/{q}\in(\acc(LB),\acc(UB))$ has $q1\leq q\leq q2$, and $CF\left({p}/{q}\right)$ is the continued fraction of $p/q$. If we see any terms whose $CF\left({p}/{q}\right)$ part contains a periodic piece, we check whether there are any solutions to (\ref{eq:diophantine}) for the obvious extension of the periodic pattern, and if so, whether the associated exceptional classes reduce properly under Cremona transformations.

This method works well when the periodic pieces are short. To handle the more complex continued fractions investigated 
in our next paper, 
 we developed our second method, explained below, which also illuminates the relationship between infinite 
 staircases and blocking classes.

\subsubsection{Finding the staircases from a blocking class $\bB$}\label{sssec:secondmethod}

Our second method provides insight into the relevance of periodic continued fractions. Often the denominators of the rationals $p/q$ that we need to check in order to find a pattern are extremely large, making our first method quite slow. However,
 the first method has the advantage that we also learn the ends $\text{end}_n$ differentiating between the continued fractions of the centers of the classes contributing to the individual stairs and the continued fraction of the accumulation point. In contrast, with our second method we can solve for the accumulation point precisely but need to search or guess to find the individual stairs.

Assume $\bB=\left(d,m;q\bw\left({p}/{q}\right)\right)$ blocks an interval $J_\bB$ with $I_\bB=(\alpha_{\bB,\ell},\alpha_{\bB,u})$. First we identify $\alpha_{\bB,\ell}$ and $\alpha_{\bB,u}$ -- these will be the accumulation points of the two infinite staircases determined by $\bB$ as in Conjecture \ref{conj:block}. Our procedure relies on the computer so that it will work for complicated $\bB$; note however that for simpler $\bB$, it may be possible to find $J_\bB$ by hand using Lemma \ref{lem:block2}.


To determine $\alpha_{\bB,\ell}$, let $a_\ell(b)$ be the nonzero solution in $z$ to
\begin{equation}\label{eqn:aell}
\frac{qz}{d-bm}=\sqrt{\frac{z}{1-b^2}}=V_b(z).
\end{equation}
Note that the left hand side of (\ref{eqn:aell}) is the value of $\mu_{\bB,b}(z)$ for $z$ close to and smaller than ${p}/{q}$, by Lemma \ref{lem:munearc}. Then let $b_\ell$ be the solution in $b$ to
\[
\acc(b)=a_\ell(b)
\]
with $0\leq b<1$ and in the same component of $[0,1)-\left\{{1}/{3}\right\}$ as $J_\bB$. Then $\alpha_{\bB,\ell}=\acc(b_\ell)$. Because $\alpha_{\bB,\ell}$ is a quadratic irrationality, its continued fraction is periodic. 

With $\alpha_{\bB,\ell}$ in hand, we can now attempt to guess the endings of the continued fractions of the centers of the classes contributing to the staircase $\Ss_{\beta_{\bB,\ell}}$ (if $J_\bB\subset\left({1}/{3},1\right)$) or $\Ss_{\beta_{\bB,u}}$ (if $J_\bB\subset\left[0,{1}/{3}\right)$). We do this using the function {\fontfamily{lmtt}\selectfont\textbf{SearchingForEnding}}, which uses the function {\fontfamily{lmtt}\selectfont\textbf{FindDMFromCF}}:

\vspace{3pt}
\noindent
{\fontfamily{lmtt}\selectfont\textbf{FindDMFromCF}[CF\textunderscore]:=\textbf{Solve}[d$^2$-m$^2$+1==\textbf{Numerator}[\textbf{FromContinuedFraction}[CF]]*

\textbf{Denominator}[\textbf{FromContinuedFraction}[CF]]\&\&3d-m==\textbf{Numerator}[

\textbf{FromContinuedFraction}[CF]]+\textbf{Denominator}[\textbf{FromContinuedFraction}[CF]]\&\&

d$\geq$0\&\&m$\geq$0,$\{$d,m$\}$,\textbf{Integers}]}

\vspace{3pt}
\noindent
{\fontfamily{lmtt}\selectfont\textbf{SearchingForEnding}[CF\textunderscore,end\textunderscore]:=\textbf{FindDMFromCF}[\textbf{Join}[CF,end]]}
\vspace{3pt}

By constructing tables consisting of {\fontfamily{lmtt}\selectfont\textbf{SearchingForEnding}[CF,end]} with {\fontfamily{lmtt}\selectfont end} varying over judiciously chosen tuples of length up to four, we were able to identify the necessary ends. This method always contains some element of guess-and-check, but our general rule of thumb is that if $CF(\alpha_{\bB,\ell})=[\mathbf{X},\{\mathbf{Y}\}^\infty]$ then there are ways to write $\mathbf{Y}=(\mathbf{y_1},\mathbf{y_2})$ and $\mathbf{Y}=(\mathbf{y_1'},\mathbf{y_2'})$ so that {\fontfamily{lmtt}\selectfont\textbf{SearchingForEnding}[CF,end]} returns nonempty pairs $(d,m)$ for
\[
\text{{\fontfamily{lmtt}\selectfont CF}}=\{\mathbf{X},\{\mathbf{Y}\}^k\},\qquad\text{{\fontfamily{lmtt}\selectfont end}}=\text{exactly one of }\{\mathbf{y_1},2n+2\}, \{\mathbf{y_1},2n+4\}, \{\mathbf{y_1},2n+6\}
\]
and for
\[
\text{{\fontfamily{lmtt}\selectfont CF}}=\{\mathbf{X},\{\mathbf{Y}\}^k\},\qquad\text{{\fontfamily{lmtt}\selectfont end}}=\text{exactly one of }\{\mathbf{y_1'},2n+2\}, \{\mathbf{y_1'},2n+4\}, \{\mathbf{y_1'},2n+6\}.
\]

Identifying the ``opposite'' staircase $\Ss_{\beta_{\bB,\ell}}$ (if $J_\bB\subset\left[0,{1}/{3}\right)$ or $\Ss_{\beta_{\bB,u}}$ (if $J_\bB\subset\left({1}/{3},1\right)$) is done in a similar manner, but using $\alpha_{\bB,u}$, which is identified as follows. Let $a_u(b)$ be the solution in $z$ to
\begin{equation}\label{eqn:au}
\frac{p}{d-bm}=\sqrt{\frac{z}{1-b^2}}=V_b(z).
\end{equation}
Note that the left hand side of (\ref{eqn:au}) is the value of $\mu_{\bB,b}(z)$ for $z$ close to and larger than ${p}/{q}$. Then let $b_u$ be the solution in $b$ to
\[
\acc(b)=a_u(b)
\]
which is in the same component of $[0,1)-\{{1}/{3}\}$ as $J_\bB$. Then $\alpha_{\bB,u}=\acc(b_u)$ and is again a quadratic irrationality, with a periodic continued fraction. We identify endings as for $\alpha_{\bB,\ell}$ above.

With either method we could only identify staircases for single values of $n$, but 
after seeing enough staircases we could usually guess fairly easily how they ought to generalize to higher $n$. We were able to discover all the two-periodic staircases discussed in this paper using the methods of \S\ref{sssec:firstmethod}. The methods of \S\ref{sssec:secondmethod} became necessary for all but the simplest $2k$-periodic staircases of our next paper 
with $k>1$.

\subsection{Plots of ellipsoid embedding functions}\label{subsec:pictures}

In this section we give figures illustrating some of the phenomena mentioned throughout the paper. In the following, all plots show $\mu_{\mathbf{E},b}$ overlaid atop $\frac{c_k(E(1,z))}{c_k(X_b)}$ on the interval where they agree.

\subsubsection{Perfect versus quasi-perfect}

Figures \ref{fig:center17029} and \ref{fig:center11119} illustrate Lemma \ref{lem:perfect}. Figure \ref{fig:center17029} shows that the perfect class 
$\mathbf{E'}=\left(73,20;29\textbf{w}\left({170}/{29}\right)\right)$ of Example \ref{ex:seq}(iii) is live at ${170}/{29}$, as guaranteed by Proposition~\ref{prop:live}~(i).

\newpage

\begin{figure}[H]
	\includegraphics[width=.9\linewidth]{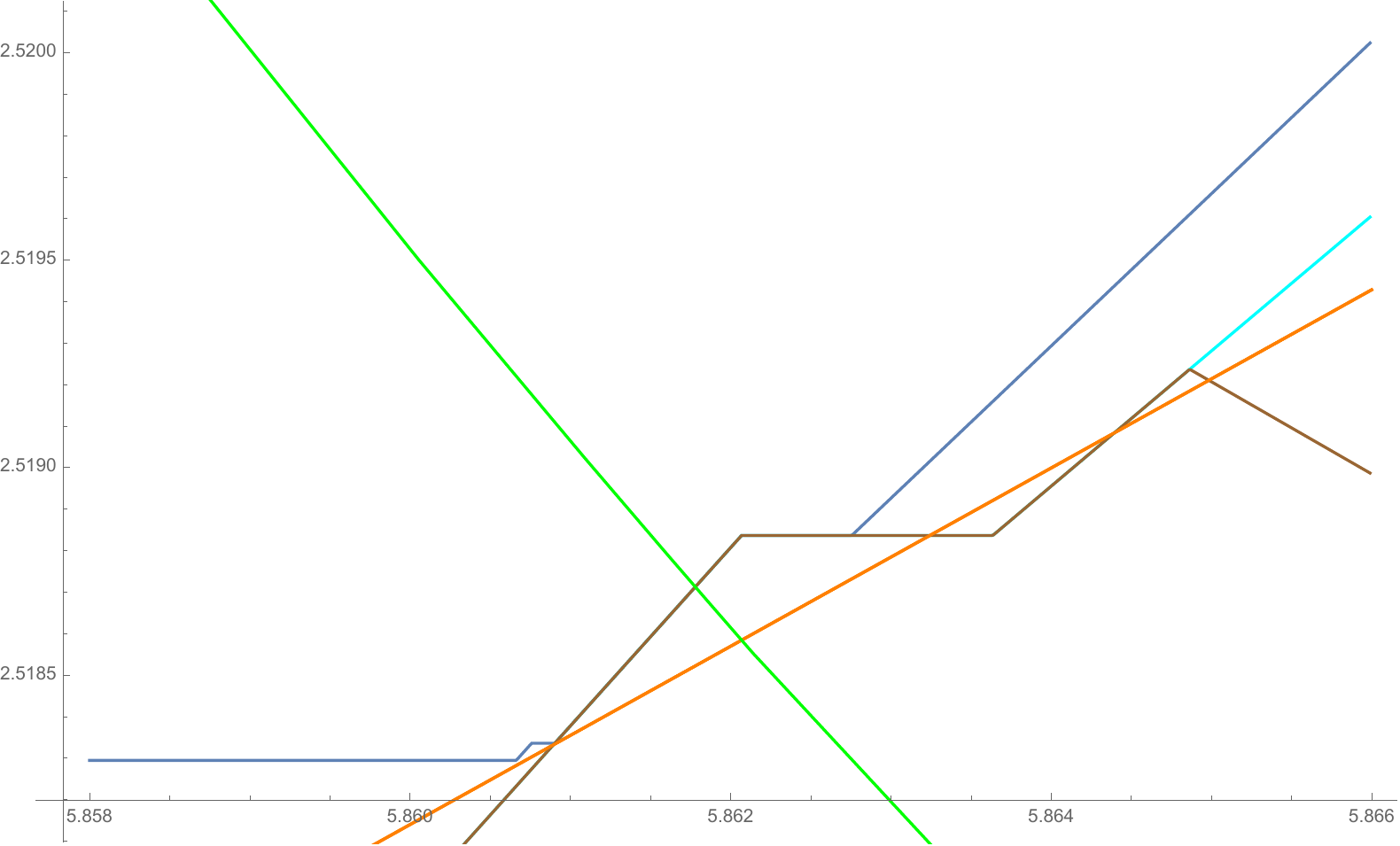}
	\caption{Here $b= \acc_L^{-1}({170}/{29})\approx 5.862$. The function
	$c_{H_b}^\le$ is in dark blue, the volume obstruction is in orange, the obstruction from the $2564^\text{th}$ ECH capacity is in bright blue, and the obstruction from the perfect class $\mathbf{E'}=\left(73,20;29\mathbf{w}\left({170}/{29}\right)\right)$ of Example \ref{ex:seq}(ii) is in brown; notice these last two agree on an interval including the center ${170}/{29}\approx5.86207$ of $\bE'$. 
	The green curve is as before, crossing the orange curve with $z$-coordinate ${170}/{29}$. This diagram also exhibits $\bE'$ (which is a step in the staircase $\Ss^L_{u,0}$) as a blocking class, illustrating Proposition \ref{prop:stairblock}.}
	\label{fig:center17029}
\end{figure}

In Figure \ref{fig:center11119},  
the point $b=\acc_L^{-1}\left({111}/{19}\right)$ is obstructed by the $125^\text{th}$ capacity and the perfect class $\mathbf{E'}=\left(15,4;6\bw\left({35}/{6}\right)\right)=\bB^E_0$ of Example \ref{ex:seq}(ii). The $1119^\text{th}$ capacity and the  quasi-perfect class $\mathbf{E''}=\left(48,14;19\mathbf{w}\left({111}/{19}\right)\right)$ of Example \ref{ex:seq}(ii), while still providing an obstruction, is strictly less than $c_{H_b}^\le\leq c_{H_b}$. Therefore Figure \ref{fig:center11119} illustrates the conclusion of 
Lemma \ref{lem:perfect}~(iii) that, for this $b$, the obstruction 
$\mu_{\mathbf{E''},b}$ is nontrivial at ${111}/{19}$.     However, there is no reason why 
$\mu_{\mathbf{E''},b}$ should be live at ${111}/{19}$ since  $\bE''$  does not satisfy the reduction criterion in Lemma~\ref{lem:Cr0} and so is not perfect.  Thus  Proposition~\ref{prop:live} does not apply;
see Remark~\ref{rmk:fake}.

\newpage

\begin{figure}[H]
	\includegraphics[width=.9\linewidth]{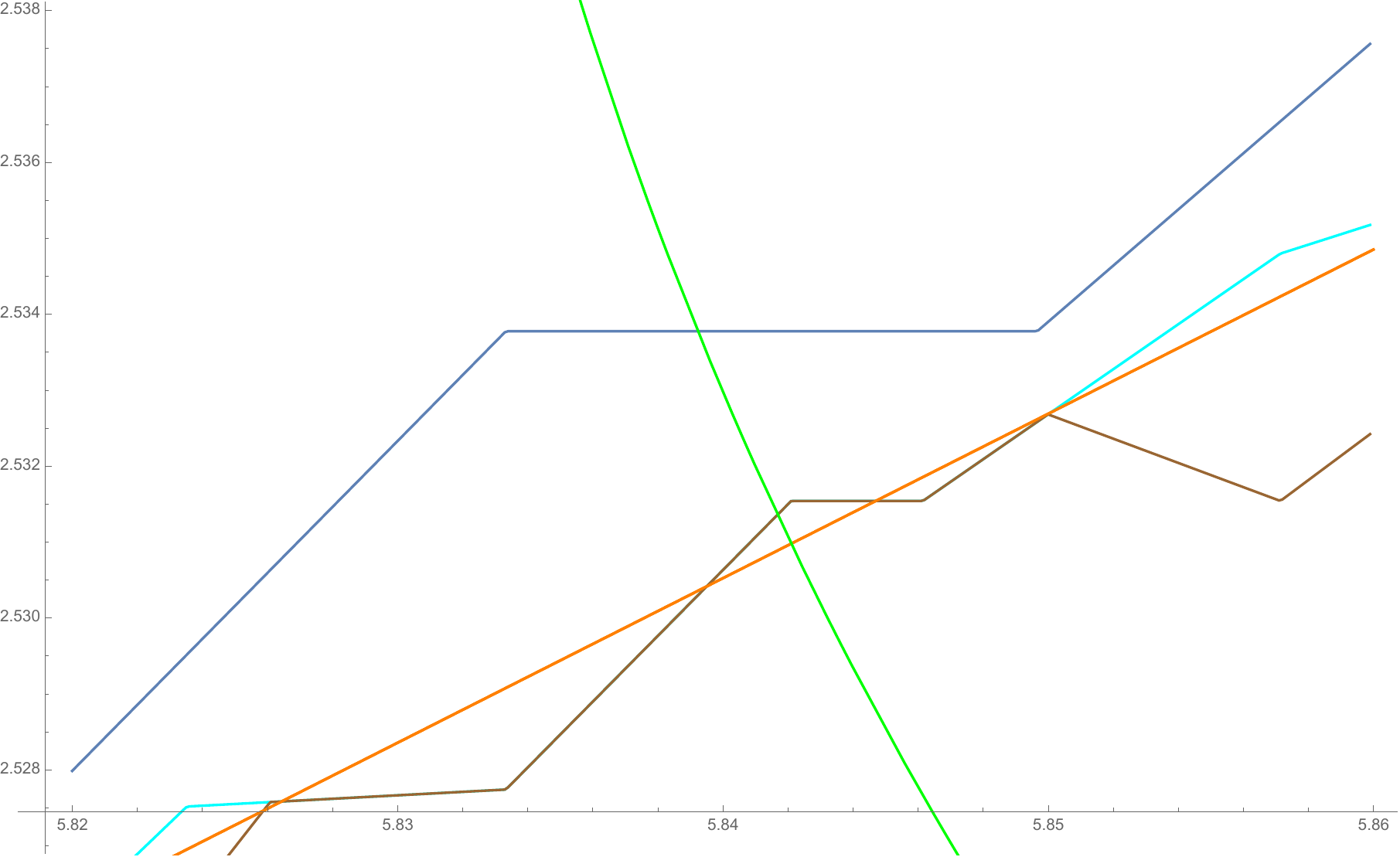}
\caption{Here $b = \acc^{-1}_L\left({111}/{19}\right)$. The function $c_{H_b}^\le$ is in dark blue, the volume obstruction is in orange, the obstruction from the $1119^\text{th}$ ECH capacity is in bright blue, and the obstruction from the quasi-perfect class $\mathbf{E''}=\left(48,14;19\mathbf{w}\left({111}/{19}\right)\right)$ of Example \ref{ex:seq}(ii) is in brown; notice these last two agree on an interval including the center ${111}/{19}$ of $\bE''$. The green curve is as before, crossing the orange curve with $z$-coordinate ${111}/{19}$. For this value of $b$ the obstruction from the $125^\text{th}$ ECH capacity equals that from the perfect class $\bE'=\left(15,4;6\bw\left({35}/{6}\right)\right)$, and both obstructions are stronger than those from the $1119^\text{th}$ ECH capacity and $\bE''$ near $\frac{111}{19}$.}
	\label{fig:center11119}
\end{figure}

\subsubsection{Typical behavior of a blocking class}\label{sssec:E0}

In this subsection we analyze the effect of the blocking class $\bB^U_0=(3,2;\bw(6))$ on $c_{H_b}$ as $b$ varies. Firstly, by the method outlined in \S\ref{sssec:secondmethod}, we find
\[
J_{\bB^U_0}=(\beta_{\bB^U_0,\ell},\beta_{\bB^U_0,u})=\left(\frac{3-\sqrt{5}}{2},\frac{3(7+\sqrt{5})}{44}\right)\approx(0.381966,0.629732),
\]
and by applying $\acc$, we find
\[
I_{\bB^U_0}=(\alpha_{\bB^U_0,\ell},\alpha_{\bB^U_0,u})\approx(5.8541,7.17082).
\]

Because $\bw(6)=(1^{\times6})$, the obstruction $\mu_{\bB^U_0,b}$ 
is nontrivial at $z\geq6$ if
\[
\frac{6}{3-2b}>\sqrt{\frac{z}{1-b^2}}\Leftrightarrow6\leq z<\frac{36(1-b^2)}{(3-2b)^2},
\]
and $\mu_{\bB^U_0,b}$ is nontrivial at $5\leq z<6$ if
\[
\frac{5+(z-5)}{3-2b}>\sqrt{\frac{z}{1-b^2}}\Leftrightarrow\frac{(3-2b)^2}{1-b^2}<z<6.
\]
While it is possible for $\mu_{\bB^U_0,b}$
to be nontrivial for $z<5$, we have $\mu_{\bB^U_0,b}(5)\leq V_b(5)$ for all $b$, meaning that the maximal $z$-interval containing the center $z=6$ of $\bB^U_0$ on which 
$\mu_{\bB^U_0,b}$
is nontrivial is $\left(\frac{(3-2b)^2}{1-b^2},\frac{36(1-b^2)}{(3-2b)^2}\right)$.\footnote{Note also that by Lemma \ref{lem:0} (ii), if $\bB^U_0$ is nontrivial on an interval containing points less than $5$, then that interval cannot include $6$, since $\ell(5)<\ell(6)$, but $6$ would have to be the break point.} This interval is nonempty so long as
\[
b\in\left(\frac{6-\sqrt{6}}{10},\frac{6+\sqrt{6}}{10}\right)\approx(0.355051,0.844949),
\]
and if nonempty it contains $z=6$. Thus when $b=\frac{6-\sqrt{6}}{10},\frac{6+\sqrt{6}}{10}$ we have
\[
\frac{(3-2b)^2}{1-b^2}=\frac{36(1-b^2)}{(3-2b)^2}=6.
\]
The interval $\left(\frac{(3-2b)^2}{1-b^2},\frac{36(1-b^2)}{(3-2b)^2}\right)$ is longest when $b={2}/{3}={m}/{d}$, when it is the interval $\left(5,{36}/{5}\right)$ and $\frac{(3-2b)^2}{1-b^2}$ reaches a minimum when $\frac{36(1-b^2)}{(3-2b)^2}$ reaches a maximum; notice that $I_{\bB^U_0}\subset\left(5,{36}/{5}\right)$. However, $\bB^U_0$ 
does not block $b={m}/{d}$, because $\acc\left({2}/{3}\right)\approx7.66962>{36}/{5}$, as explained in Remark \ref{rmk:block0}.

Figure \ref{fig:E0J} depicts $c_{H_b}$ for 
$$
b= \frac{3-\sqrt{5}}{2} = \be_{\bB^U_{\ell,0}},\quad b={5}/{11}=\acc_U^{-1}(6), \quad b=\frac{3(7+\sqrt{5})}{44} =
\be_{\bB^U_{u,0}}
$$
 highlighting the obstruction $\mu_{\bB^U_0,b}$.  Notice that in (b),(c) there is a new obstruction coming in from the left that dominates $\mu_{\bB^U_0,b}$ near the left end point of the interval where $\mu_{\bB^U_0,b}$ is 
 obstructive.  This illustrates the point that  for general $b$ it is often the case that 
$\mu_{\bB^U_0,b}$ 
is not live over the entire $z$-interval on which it is nontrivial.

\newpage

\begin{figure}[H]
\centering
\subfigure[$b=\beta_{\bB^U_0,\ell}=\frac{3-\sqrt{5}}{2}$]{
\includegraphics[width=0.6\textwidth]{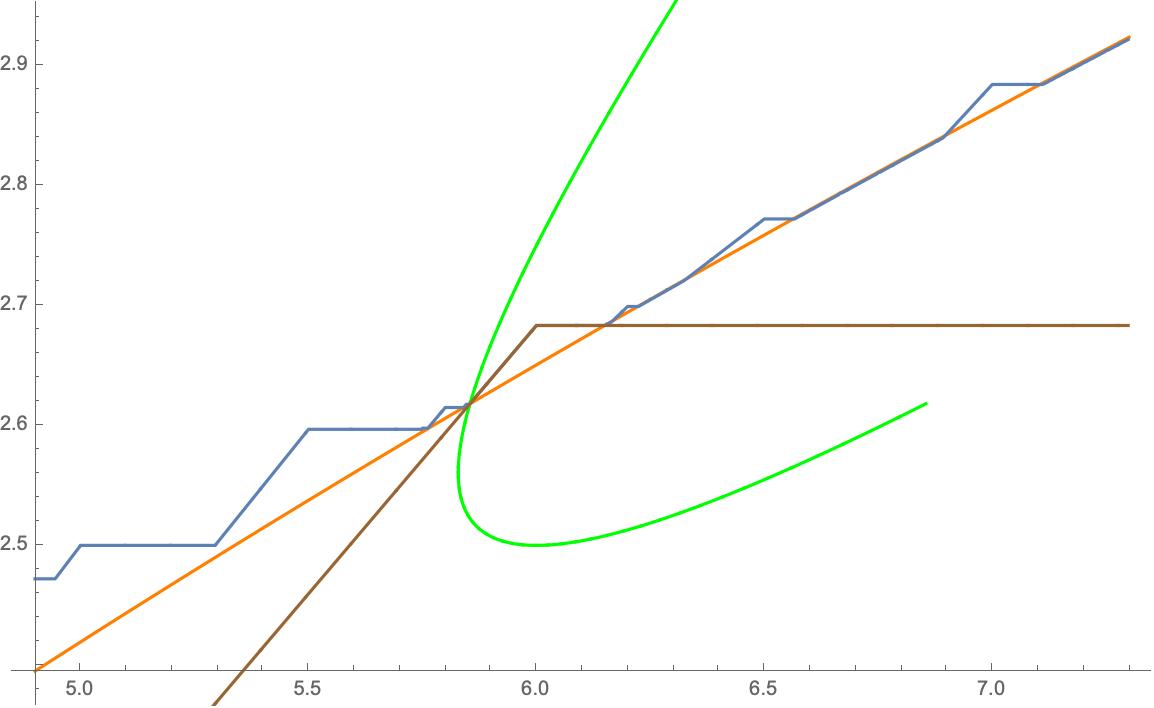}
}\hfil
\subfigure[$b=\acc_U^{-1}(6)=\frac{5}{11}$]{
\includegraphics[width=0.6\textwidth]{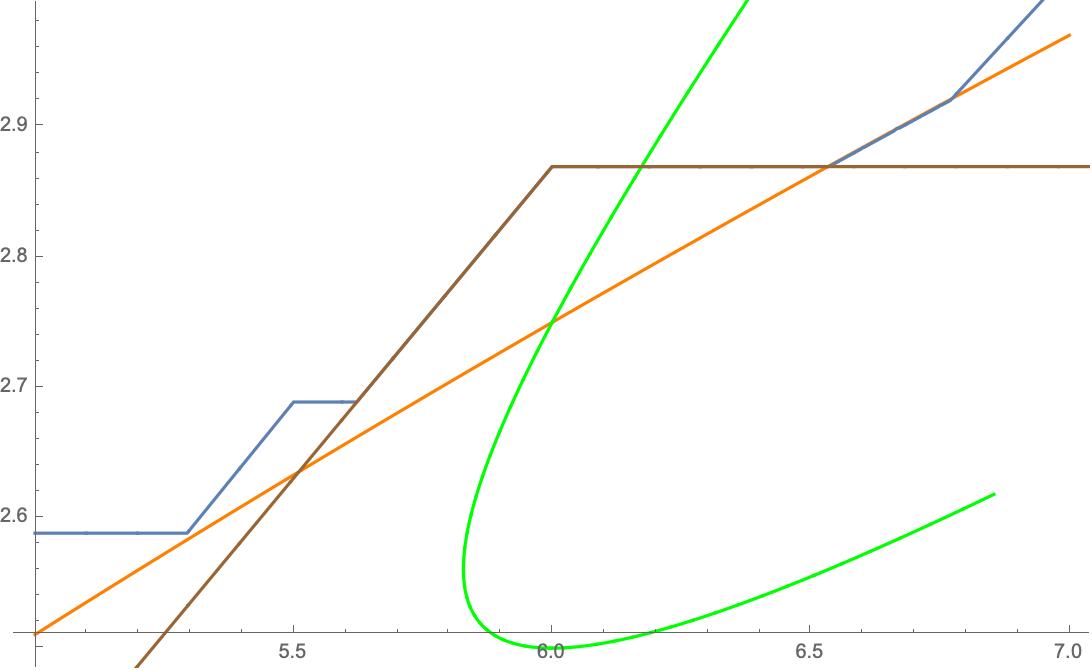}
}\hfil
\subfigure[$b=\beta_{\bB^U_0,u}=\frac{3(7+\sqrt{5})}{44}$]{
\includegraphics[width=0.6\textwidth]{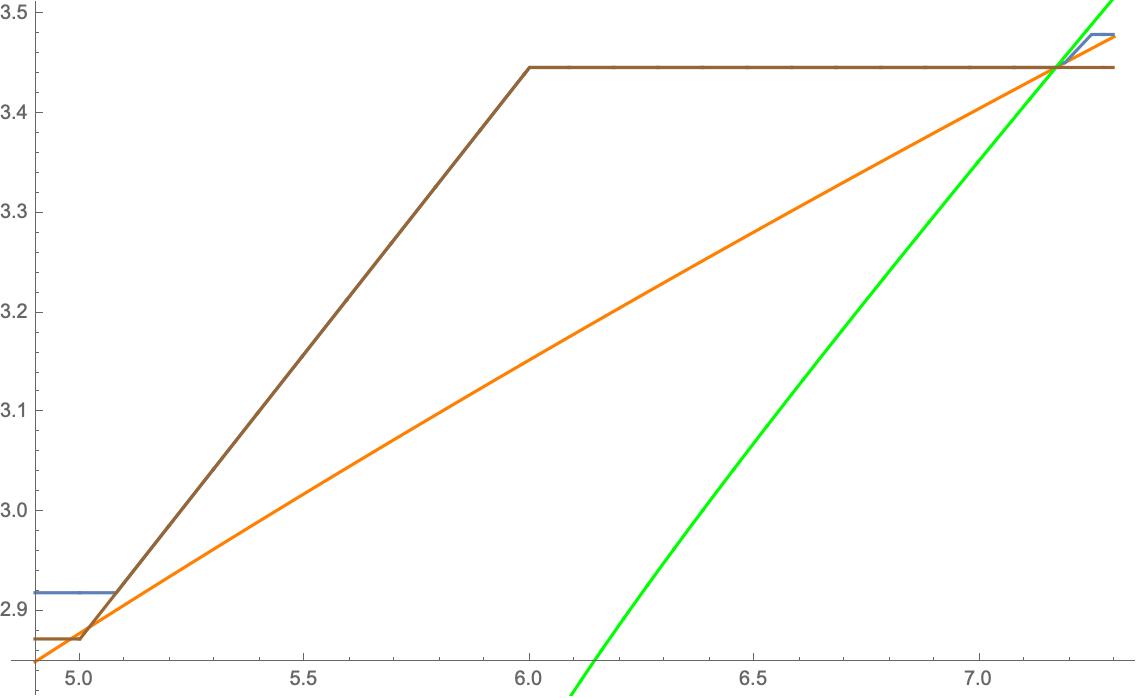}
}
\caption{As $b$ increases from $\beta_{\bB^U_0,\ell}$ to the center of $\bB^U_0$ to $\beta_{\bB^U_0,u}$, the plot of $c_{H_b}$ varies from the infinite staircase $\Ss^U_{\ell,0}$ depicted in (a), through blocked staircases such as that of $c_{H_{{5}/{11}}}$ depicted in (b), to the infinite staircase $\Ss^U_{0,u}$ depicted in (c). In all plots the orange curve indicates the volume obstruction, the function $c_{H_b}^\le$ is in dark blue, the obstruction $\mu_{\bB^U_0,b}$ is in brown, and the curve $(\acc(b),V_b(\acc(b)))$ is in bright green.}
\label{fig:E0J}
\end{figure}

When $b\in\left(\frac{6-\sqrt{6}}{10},\frac{6+\sqrt{6}}{10}\right)-\left(\frac{3-\sqrt{5}}{2},\frac{3(7+\sqrt{5})}{44}\right)$ --- that is, when $\mu_{\bB^U_0,b}$
 is nontrivial but $\bB^U_0$ is not blocking --- this phenomenon persists. Figure \ref{fig:bgeJE0} illustrates the situation when $b\in\left[\frac{3(7+\sqrt{5})}{44},\frac{6+\sqrt{6}}{10}\right)\approx(0.629732,0.844949)$, showing the interaction between  $\bB^U_0$ and $\bB^U_1$.
When $b=0.7$, the obstruction 
$\mu_{\bB^U_0,b}$ 
(in brown), though nontrivial, is not live on the interval 
$\left(7.125,7.171875\right)$; instead 
$\mu_{\bB^U_1,b}$, where $\bB^U_1=(4,3;\bw(8))$,  (in red) is live. As $b$ decreases towards 
$\frac{3(7+\sqrt{5})}{44}$, 
the intervals on which $\mu_{\bB^U_0,b}$
 (in brown) and 
 $\mu_{\bB^U_1,b}$
  (in red) are nontrivial become disjoint, as predicted by 
Proposition \ref{prop:block}, and one begins to see some staircase classes.
Indeed, by Theorem~\ref{thm:U}, when $b=\be_{\bB^U_0,u}$, $H_b$ admits a descending staircase $\Ss^U_{u,0}$ whose $k=0$ step with $\operatorname{end}_0=4$ has center ${29}/{4}=[7;4]$, the largest center amongst steps in $\Ss^U_{u,0}$. Meanwhile, the step with the second-lowest center in the ascending staircase $\Ss^U_{\ell,1}$ (the $k=0$ step with $\operatorname{end}_1=(7,4)$) has the same center; it is visible as the large dark blue step in Figure~\ref{fig:bgeJE0} (b).

Figure \ref{fig:bleJE0} illustrates the situation when $b\in\left(\frac{6-\sqrt{6}}{10},\frac{3-\sqrt{5}}{2}\right]$. When $b\in\left[\frac{3}{8},\frac{3-\sqrt{5}}{2}\right)$, as in (c), there is still a $z$-interval on which $\mu_{\bB^U_0,b}$ (in brown) is live, but when $b\in\left(\frac{6-\sqrt{6}}{10},\frac{3}{8}\right)$, as in (a), the obstruction from the class $(3,1;2,1^{\times5})$ (in red) overwhelms $\mu_{\bB^U_0,b}$. However, note that $(3,1;2,1^{\times5})$ is never blocking in this interval, because the value of its obstruction function at $\acc(b)<6$  is always $V_b(\acc(b))$, as we saw in \eqref{eq:volacc}.  In particular, in each of (a), (b), (c) the orange, green and red curves have a common point of intersection.

\newpage

\begin{figure}[H]
\centering
\subfigure[$b=0.7$]{
\includegraphics[width=0.8\textwidth]{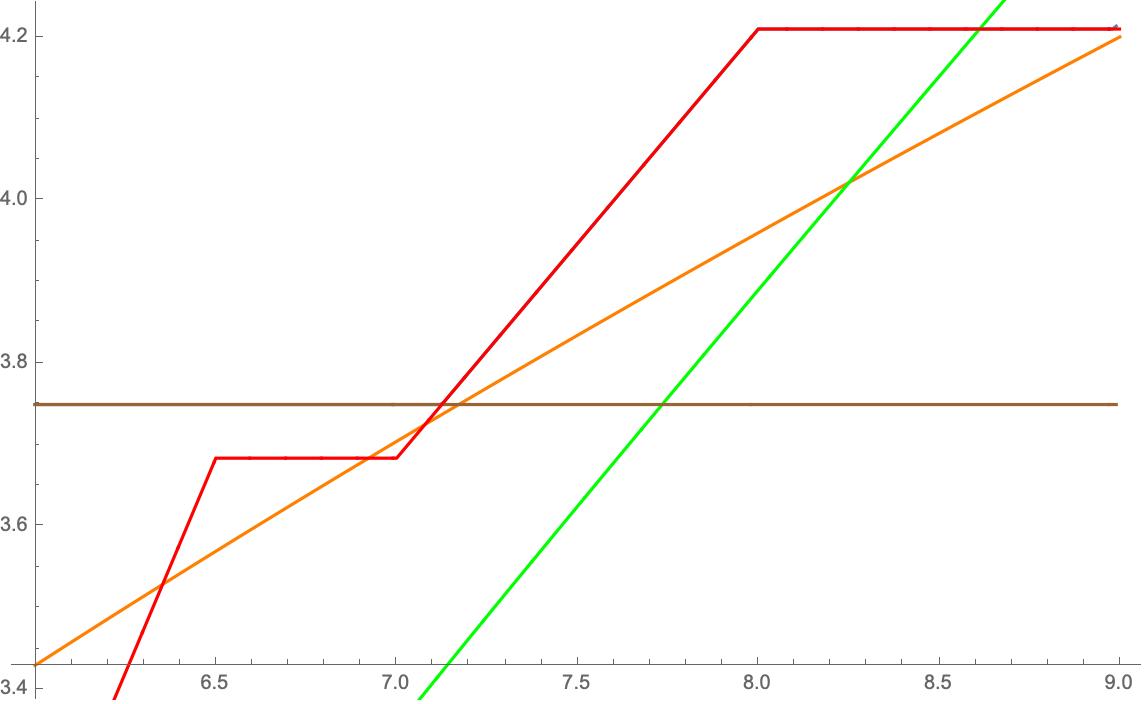}
}\hfil
\subfigure[$b=0.64$]{
\includegraphics[width=0.8\textwidth]{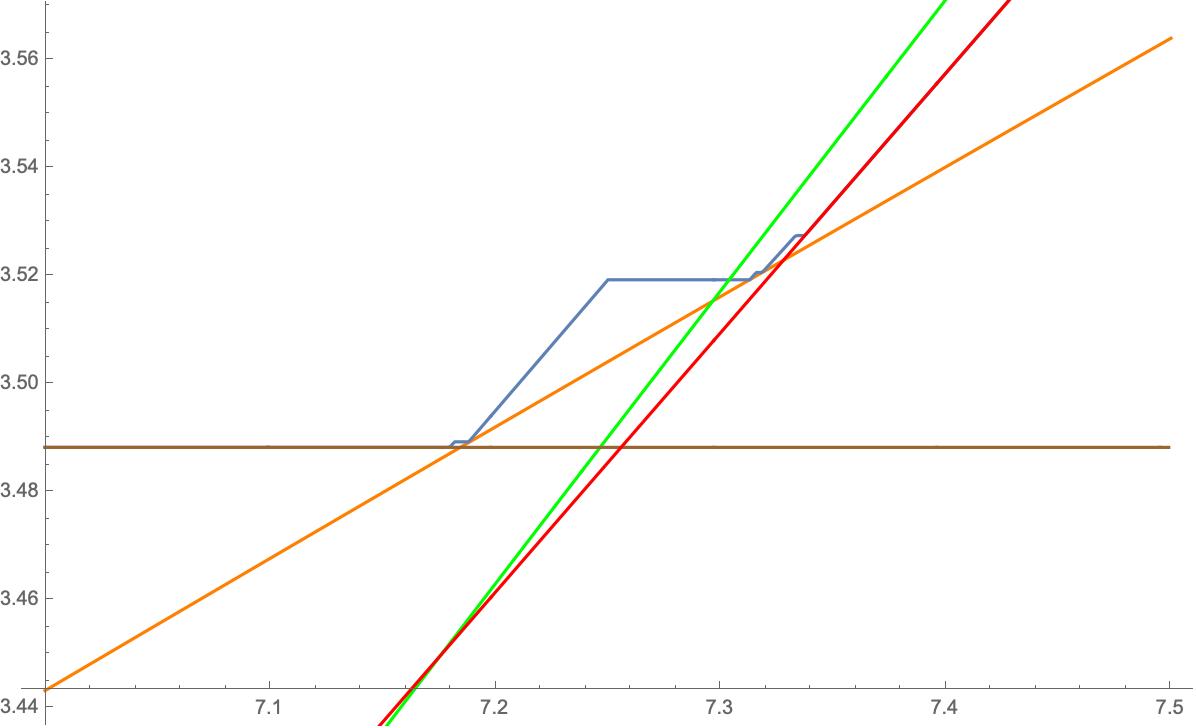}
}
\caption{In both plots the orange curve indicates the volume obstruction, the function $c_{H_b}^\le$ is in dark blue when visible, the obstruction $\mu_{\bB^U_0,b}$ is in brown, the obstruction $\mu_{\bB^U_1,b}$ is in red, and the curve $(\acc(b),V_b(\acc(b)))$ is in bright green. Note the differing scales; the break point $8$ of the red curve is visible in (a) but not in (b).
}
\label{fig:bgeJE0}
\end{figure}
%
%
%
%

\begin{figure}[H]
\centering
\subfigure[$b=0.37$]{
\includegraphics[width=0.6\textwidth]{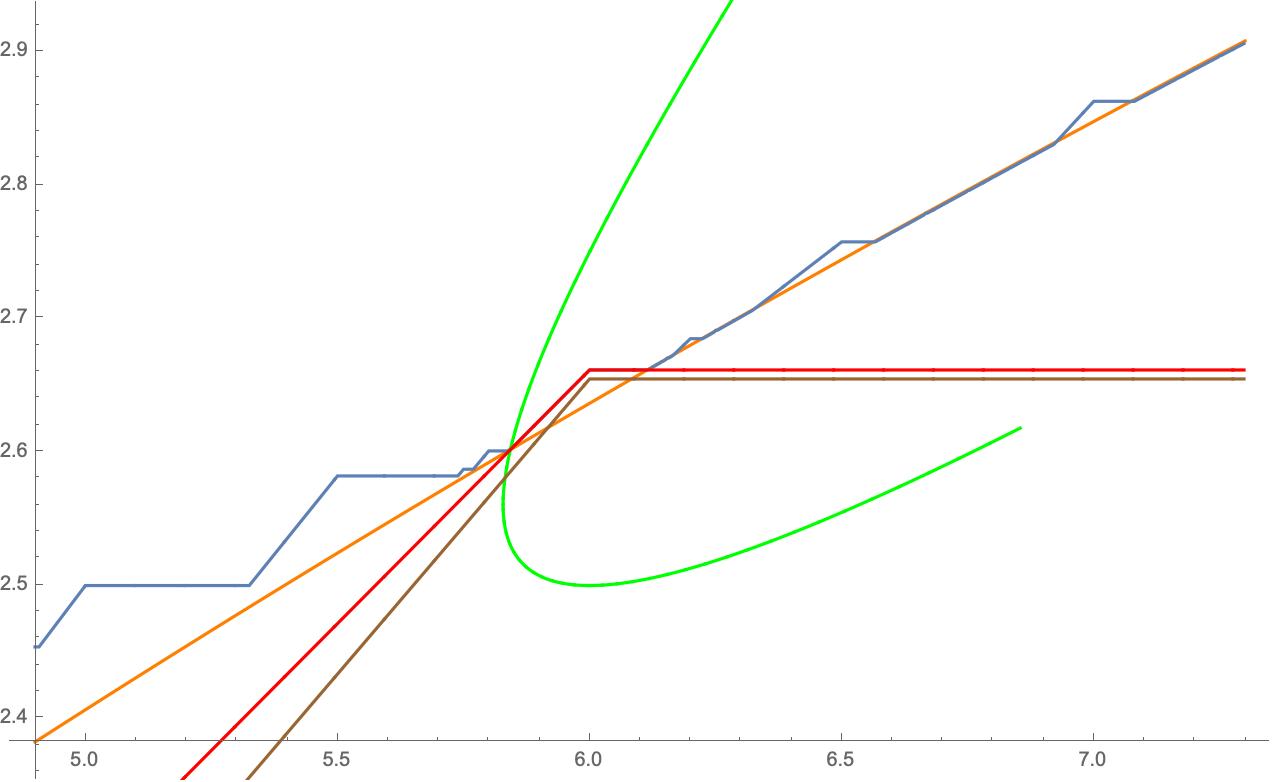}
}\hfil
\subfigure[$b=\frac{3}{8}=0.375$]{
\includegraphics[width=0.6\textwidth]{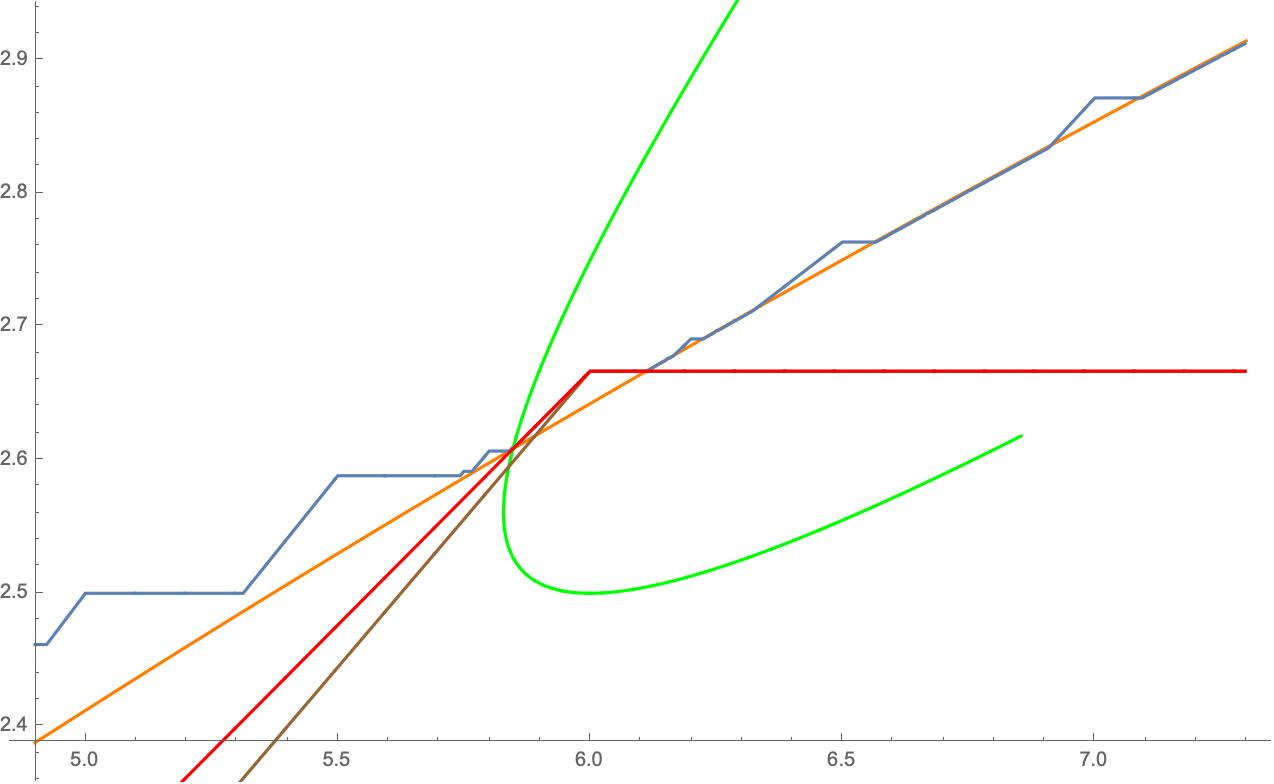}
}\hfil
\subfigure[$b=0.38$]{
\includegraphics[width=0.6\textwidth]{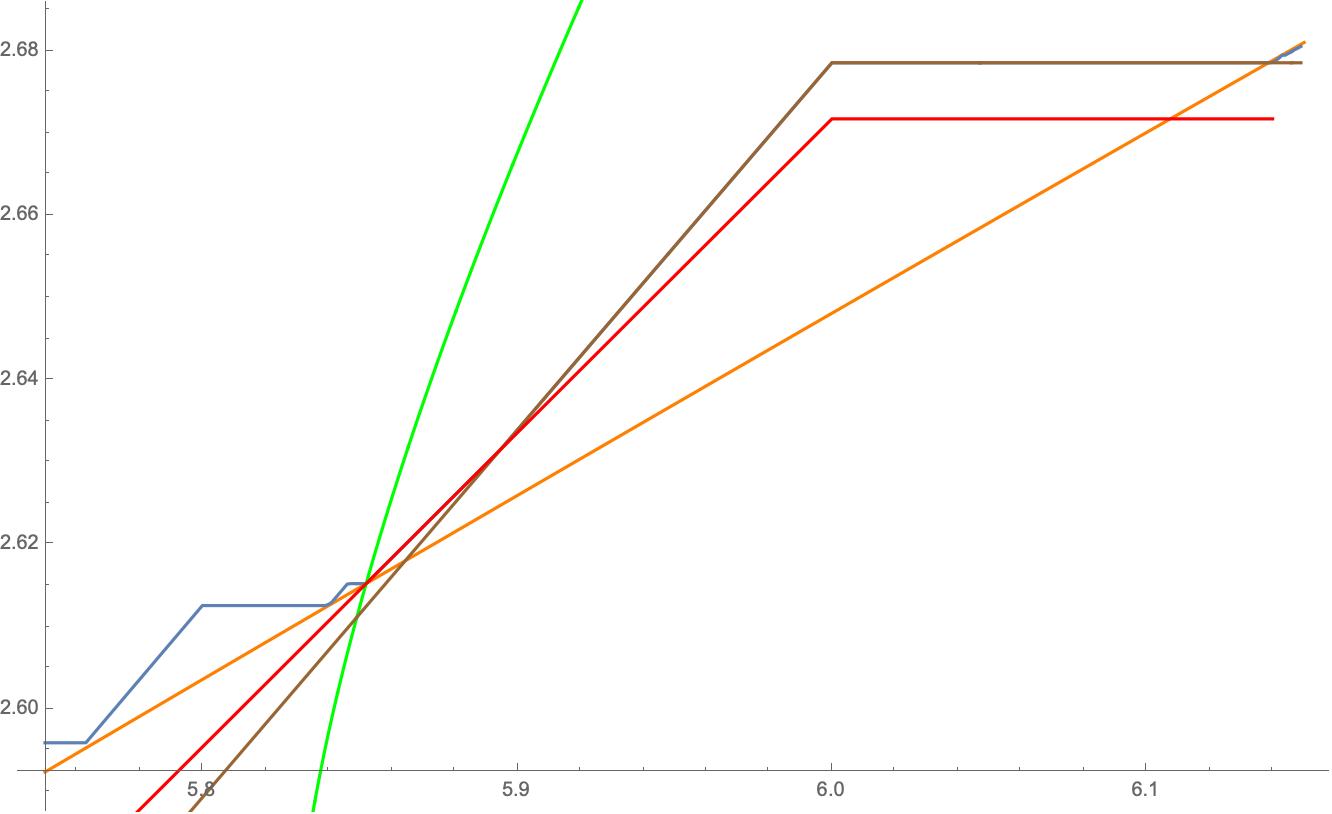}
}
\caption{In all plots the orange curve indicates the volume obstruction, the function $c_{H_b}^\le$ is in dark blue, the obstruction $\mu_{\bB^U_0,b}$ is in brown, the obstruction from $(3,1;2,1^{\times5})$ is in red, and the curve $(\acc(b),V_b(\acc(b)))$ is in bright green. 
Note the differing scales; in all diagrams, both the brown and red curves break  at $z=6$.
}
\label{fig:bleJE0}
\end{figure}


\end{document}